\documentclass{scrartcl}

\usepackage{amsmath, amssymb, amsthm, marginnote, ltxcmds, adjustbox, stmaryrd, graphicx, bbold, extarrows}
\usepackage{esvect}

\usepackage{mathtools, stackengine, changepage, ragged2e}
\usepackage{enumitem}

\usepackage{todonotes}

\definecolor{cite}{HTML}{11871E}
\definecolor{url}{HTML}{698996}
\definecolor{link}{HTML}{912F1B}

\usepackage[pdfencoding=unicode, colorlinks=true, linkcolor=link, citecolor=cite, urlcolor=url, linktocpage, breaklinks=true]{hyperref}

\usepackage[backend=biber, style=alphabetic, maxnames=10, maxalphanames=3, minalphanames=3]{biblatex}
\DeclareFieldFormat{eprint:arxiv}{%
  \href{https://arxiv.org/abs/#1}{arXiv:#1}%
}

\usepackage{tikz}
\usetikzlibrary{cd}
\usetikzlibrary{arrows, arrows.meta, positioning, calc}
\tikzcdset{arrow style=tikz, diagrams={>={Straight Barb[scale=0.8]}}}

\tikzstyle{arrow} = [-{Straight Barb[scale=0.8]}, line width=0.2mm]
\tikzset{
math to/.tip={Glyph[glyph math command=rightarrow]},
loop/.tip={Glyph[glyph math command=looparrowleft, swap]},
}

\usepackage{contour}
\usepackage{ulem}

\contourlength{0.5pt}

\newcommand{\myuline}[1]{%
  \uline{\phantom{#1}}%
  \llap{\contour{white}{#1}}%
}

\makeatletter
\newcommand*{\saved@myuline}{}
\let\saved@myuline\myuline

\newcommand*{\mathuline}{%
  \mathpalette{\math@myuline\saved@myuline}%
}
\newcommand*{\math@myuline}[3]{%
  \mbox{#1{$#2#3\m@th$}}%
}

\renewcommand*{\myuline}{%
  \relax  
  \ifmmode
    \expandafter\mathuline
  \else
    \expandafter\saved@myuline
  \fi
}
\makeatother

\usepackage{libertine}
\usepackage{stmaryrd}

\usepackage[capitalise]{cleveref}

\Crefname{prop}{Proposition}{Propositions}
\Crefname{lem}{Lemma}{Lemmas}
\Crefname{cor}{Corollary}{Corollaries}
\Crefname{thm}{Theorem}{Theorems}
\Crefname{alphThm}{Theorem}{Theorems}
\Crefname{defn}{Definition}{Definitions}
\Crefname{notation}{Notation}{Notations}
\Crefname{cons}{Construction}{Constructions}
\Crefname{rmk}{Remark}{Remarks}
\Crefname{obs}{Observation}{Observations}
\Crefname{warning}{Warning}{Warnings}
\Crefname{conj}{Conjecture}{Conjectures}
\Crefname{assump}{Assumption}{Assumptions}
\Crefname{recollect}{Recollection}{Recollections}
\Crefname{terminology}{Terminology}{Terminologies}
\Crefname{question}{Question}{Questions}
\Crefname{example}{Example}{Examples}
\Crefname{setting}{Setting}{Settings}

\crefformat{equation}{(#2#1#3)}
\crefformat{section}{\S#2#1#3}
\crefmultiformat{section}{\S\S#2#1#3}{ and~#2#1#3}{, #2#1#3}{, and~#2#1#3}
\newtheorem{thm}[subsubsection]{Theorem}
\newtheorem{prop}[subsubsection]{Proposition}
\newtheorem{lem}[subsubsection]{Lemma}
\newtheorem{cor}[subsubsection]{Corollary}

\newtheorem{alphThm}{Theorem}

\newcommand{\neutralize}[1]{\expandafter\let\csname c@#1\endcsname\count@}
\makeatother

\newtheorem*{thm*}{Theorem}
\newtheorem*{prop*}{Proposition}
\newtheorem*{lem*}{Lemma}
\newtheorem*{cor*}{Corollary}
  
\newtheorem{alphConj}{Conjecture}



\newtheorem{alphCor}{Corrollary}



\newtheorem{alphProp}{Proposition}



\theoremstyle{definition}
\newtheorem*{defn*}{Definition}
\newtheorem{defn}[subsubsection]{Definition}
\newtheorem{cons}[subsubsection]{Construction}
\newtheorem{nota}[subsubsection]{Notation}

\newtheorem{recollect}[subsubsection]{Recollections}
\newtheorem{terminology}[subsubsection]{Terminology}

\newtheorem{example}[subsubsection]{Example}
\newtheorem{rmk}[subsubsection]{Remark}
\newtheorem{setting}[subsubsection]{Setting}
\newtheorem{obs}[subsubsection]{Observation}

\theoremstyle{remark}
\newtheorem{fact}[subsubsection]{Fact}

\newcommand{\groupSurjection}{\theta}
\newcommand{\groupInjection}{\iota}

\newcommand{\brauerQuotientFamily}[1]{\Phi^{{#1}}}
\newcommand{\trivialFamily}{\mathcal{T}}

\newcommand{\singularPartTwiddle}[1]{_{\widetilde{{#1}}}}
\newcommand{\singularPartComplement}[1]{_{{#1}^c}}
\newcommand{\sUpperStar}[1]{^{{#1}^c}}

\newcommand{\sTwiddleLowerStar}[1]{^{\Phi\widetilde{{#1}}}}
\newcommand{\sLowerStarInclusion}[1]{^{\widetilde{#1}}}

\newcommand{\uniqueMapX}{{X}}
\newcommand{\uniqueMapY}{Y}
\newcommand{\uniqueMap}[1]{#1}

\newcommand{\GeneralProperTate}[1]{^{t_{\mathcal{P}}{#1}}}
\newcommand{\properTate}{^{t_{\mathcal{P}}A}}

\newcommand{\picardObject}{P}
\newcommand{\abelianGroups}{\mathrm{Ab}}

\newcommand{\family}{\mathcal{F}}
\newcommand{\terminalTCat}{\underline{\ast}}

\newcommand{\beckChevalley}{\mathrm{BC}}
\newcommand{\projectionformula}{\mathrm{PF}}

\newcommand{\pointProjection}{r}

\newcommand{\pointProjectiontopIndex}[1]{r^{\hspace{0.3mm}#1}}

\newcommand{\eilenbergMacLaneFp}{\mathrm{H}\mathbb{F}_p}
\newcommand{\eilenbergMacLaneCoeff}{\mathrm{H}\mathbb{Z}}
\newcommand{\thom}{\mathrm{Th}}

\newcommand{\Pic}{\mathcal{P}\mathrm{ic}}
\newcommand{\susps}{\Sigma^{\infty}}

\newcommand{\norm}{\mathrm{N}}
\newcommand{\calg}{\mathrm{CAlg}}
\newcommand{\spc}{\mathcal{S}}

\newcommand{\ind}{\mathrm{Ind}}
\DeclareMathOperator{\coind}{Coind}

\DeclareMathOperator{\cgroup}{\mathrm{CGrp}}
\DeclareMathOperator{\cmonoid}{\mathrm{CMon}}
\newcommand{\canonical}{\mathrm{can}}

\newcommand{\borel}{\mathrm{Bor}}
\DeclareMathOperator{\loops}{\Omega^{\infty}}
\newcommand{\inclusion}{\mathrm{incl}}

\newcommand{\constant}{\operatorname{const}}
\newcommand{\cat}{\mathrm{Cat}}
\newcommand{\largecat}{\widehat{\mathrm{Cat}}}

\newcommand{\presentable}{\mathrm{Pr}}

\newcommand{\A}{\mathcal{A}}

\newcommand{\sC}{{\mathcal C}}
\newcommand{\D}{{\mathcal D}}

\newcommand{\B}{\mathcal{B}}
\newcommand{\op}{^{\mathrm{op}}}

\newcommand{\sphere}{\mathbb{S}}
\newcommand{\sets}{\mathrm{Set}}

\newcommand{\E}{\mathcal{E}}

\DeclareMathOperator{\res}{\mathrm{Res}}
\DeclareMathOperator{\mackey}{\mathrm{Mack}}
\DeclareMathOperator{\presheaf}{\mathrm{PSh}}
\newcommand{\orbit}{\mathcal{O}}

\newcommand{\exact}{^{\mathrm{ex}}}
\newcommand{\spectra}{\mathrm{Sp}}

\newcommand{\inflated}{\mathrm{Infl}}
\newcommand{\stable}{\mathrm{st}}

\DeclareMathOperator{\im}{\mathrm{Im}}
\newcommand{\id}{\mathrm{id}}
\DeclareMathOperator{\map}{\mathrm{Map}}
\DeclareMathOperator{\fib}{\operatorname{fib}}
\DeclareMathOperator{\cofib}{\operatorname{cofib}}
\DeclareMathOperator{\func}{\mathrm{Fun}}

\newcommand{\module}{\mathrm{Mod}}
\newcommand{\unit}{\mathbb{1}}
\newcommand{\hearttstructure}{^{\heartsuit}}

\newcommand{\perfectCat}{\mathrm{Perf}}

\newcommand{\trivial}{\mathrm{triv}}
\newcommand{\picardSpace}{\mathcal{P}\mathrm{ic}}

\newcommand{\stmodSmallProper}{\mathrm{stmod}^{\mathcal{P}}}

\newcommand{\stmodSmall}{\mathrm{stmod}}
\newcommand{\proper}{\mathcal{P}}

\newcommand*\cocolon{%
        \nobreak
        \mskip6mu plus1mu
        \mathpunct{}%
        \nonscript
        \mkern-\thinmuskip
        {:}%
        \mskip2mu
        \relax
}

\newcommand{\internalfunc}[1]{\udl{\func}}

\newcommand{\ambi}[3]{{{#1} \cap_{#2} {#3}}}
\newcommand{\collapse}{\mathrm{clps}}

\newcommand{\udlcatC}{\udl{\category{C}}}
\newcommand{\udlcatD}{\udl{\category{D}}}
\newcommand{\udlcatE}{\udl{\category{E}}}

\newcommand{\catgrpcol}[2]{\cat_{#1,#2}}

\newcommand{\udl}[1]{\underline{{#1}}}

\newcommand{\burnside}[1]{H^0({#1};\udl{A})}

\newcommand{\hrep}{\mathcal{V}}

\def\colim{\qopname\relax m{colim}}

\newcommand{\category}[1]{\mathcal{#1}}
\DeclareMathOperator{\spancategory}{Span}
\DeclareMathOperator{\cospancategory}{Cospan}
\newcommand{\funcex}{\func^{\mathrm{ex}}}
\newcommand{\catex}{\cat^\mathrm{ex}}
\newcommand{\bbZ}{\mathbb{Z}}
\newcommand{\bbF}{\mathbb{F}}

\DeclareMathOperator{\infl}{Infl}
\DeclareMathOperator{\pr}{pr}
\DeclareMathOperator{\const}{Const}

\DeclareMathOperator{\Hom}{Hom}

\DeclareMathOperator{\stab}{Stab}

\newcommand{\obj}[1]{\udl{#1}}

\newcommand{\prst}[1]{\presentable_{#1}^{L, \stable}}
\newcommand{\prGst}[1]{\presentable_{#1}^{L, #1-\stable}}
\newcommand{\catst}[1]{\cat^{\stable}_{#1}}
\newcommand{\catGst}[1]{\cat^{#1-\stable}_{#1}}
\newcommand{\catptd}[1]{\cat^{\mathrm{ptd}}_{#1}}
\newcommand{\catGstfam}[2]{\cat^{#1-\stable}_{#1, #2}}

\DeclareMathOperator{\rcoind}{Coinfl}
\newcommand{\stilde}[1]{\widetilde{s}^*}

\DeclareMathOperator{\degree}{deg}
\newcommand{\Nm}[1]{\widetilde{\operatorname{Nm}}_{#1}}
\newcommand{\stcoind}[2]{\operatorname{Coind}_{#1}^{\sim #2}}

\newcommand{\pointedcotensor}{\wedge}
\newcommand{\pointedtensor}{\udl{\hom}_*}

\newcommand{\arrdisp}{0.33ex}
\newcommand{\arrdisplacementsp}{0.72ex}

\newcommand{\ardis}{\ar@<\arrdisp>}
\newcommand{\ardissp}{\ar@<\arrdisplacementsp>}

\title{Parametrised Poincar\'{e} duality and equivariant fixed points methods}
\author{\textsc{Kaif Hilman}\thanks{kaif@mpim-bonn.mpg.de} \and \textsc{Dominik Kirstein}\thanks{kirstein@mpim-bonn.mpg.de} \and \textsc{Christian Kremer}\thanks{kremer@mpim-bonn.mpg.de}}
\date{\today}

\begin{document}

\maketitle
\begin{abstract}
    In this article, we introduce and develop the notion of parametrised Poincar\'{e} duality in the formalism of parametrised higher category theory by Martini--Wolf, in part generalising Cnossen's theory of twisted ambidexterity to the nonpresentable setting. We prove several basechange results, allowing us to move between different coefficient categories and ambient topoi. We then specialise the general framework to yield a good theory of equivariant Poincar\'{e} duality spaces for compact Lie groups and apply our basechange results to obtain a suite of isotropy separation methods. Finally, we employ this theory to perform various categorical Smith--theoretic manoeuvres to prove, among other things, a generalisation of a theorem of Atiyah--Bott and Conner--Floyd on group actions with single fixed points.
\end{abstract}

\tableofcontents

\section{Introduction}

Poincar\'{e} duality has a distinguished history that goes back right to the birth of algebraic topology at the hands of Henri Poincar\'{e}. Broadly speaking, it says that there is often a hidden symmetry between homology and cohomology, and arguably beyond Poincar\'{e}'s wildest dreams, it is a phenomenon that is  not just endemic to algebraic topology but also pervasive in fields as far as algebraic geometry, arithmetic geometry, and even representation theory. Essentially, it tends to show up in any context in which homological algebra is present. Perhaps a reason as to why it is such a useful principle is that it may be exploited both  computationally as well as theoretically: the former because it halves the amount of homological computations to be made and the latter because, for instance, it may be used to produce ``wrong--way'' maps which opens the way to powerful transfer arguments. No less importantly in the way of theoretical significance, it would also be remiss of us not to mention that Poincar\'{e} duality constitutes one of the starting points of the surgery theory of manifolds.  In either case, it would be fair to summarise that Poincar\'{e} duality provides strong structural constraints on homological invariants which lends a rigidity not seen for a bare homotopy type.

On another front, group actions on manifolds have attracted the attention of topologists nearly since the beginning of the subject. A deep vein in this line of work is the hope of finding algebro--topological constraints on group actions in order to rule out the existence of certain group actions on manifolds. Smith theory, for example, predicts that if a $p$-group $G$ acts on an $\bbF_p$-homology sphere, then the fixed points of the action must also be an $\bbF_p$-homology sphere. In a similar spirit, the Conner--Floyd conjecture, resolved affirmatively by Atiyah--Bott, predicts in a simple version that if $p$ is an odd prime and the group $C_{p^k}$ acts smoothly on a smooth, closed, orientable, positive--dimensional manifold, then the fixed point set of the action cannot consist of a single isolated point.

Given the motivations above, it should come as no surprise that a theory of equivariant Poincar\'{e} duality is desirable in order to incorporate the strong homological constraints aforementioned in the equivariant context. Indeed, so natural  is this a question that there is a very rich corpus of contributions -- too many to mention exhaustively -- in this line of investigation coming from a wide variety of schools of thought. From our point of view, the strand of work that is most pertinent to us (and on which we build, either directly or indirectly), is the parametrised category theoretic one of Costenoble--Waner \cite{costenoble_waner_finite_group_PD,costenoble_waner_book} and May--Sigurdsson \cite{MaySigurdsson2006}, which were informed by the work of tom Dieck in \cite{tomDieck_Transformation_Gruyter}.  More specifically, our work builds heavily upon Cnossen's work on twisted ambidexterity \cite{Cnossen2023}, which is, in turn, built upon the insights of the preceding work. A more detailed account of the relationships between the present work and the ones just mentioned will be given at the end of the introduction.

The main goal of this article is to develop the theory of equivariant Poincar\'{e} duality for compact Lie groups from an $\infty$--categorical perspective. As a proof of concept, we will then apply it to a selection of concrete problems in equivariant geometric topology, some of which have been resolved through different methods before. Our categorical formalism of choice is the parametrised $\infty$--category theory of \cite{Martini2022Cocartesian,Martini2022Yoneda,MartiniWolf2022Presentable,MartiniWolf2024} and a central role will be played by equivariant stable homotopy theory as first systematically developed in \cite{lewis_may_steinberger} and later in \cite{greenlees-may_tate,MNN17}. The rest of the introduction will give an overview of our methods and highlighted results. 

\vspace{1mm}

\noindent \textbf{Notations and conventions:} We work in the setting of $\infty$--categories as developed by Joyal and Lurie without referring to any particular model such as quasicategories. 
To avoid notational clutter, we will refer to $\infty$--categories as just categories, while classical categories (for which there is a set of morphisms between objects as opposed to a space of such) will be referred to as $1$-categories.
 We fix three Grothendieck universes $U \in V \in W$ called small, large and very large.
We denote the large category of small categories by $\cat$ and the very large category of large categories by $\largecat$. The term ``category" will be reserved for small categories. Furthermore, left adjoints will always be written on top of right adjoints in our diagrams.

\subsection*{Equivariant and parametrised homotopy theory}

Let $G$ be compact Lie group. It has long been understood that in order to have good access to inductive methods in equivariant  homotopy theory, the fixed points spaces for all closed subgroups of $G$ should be recorded as part of the structure of a $G$--space. The earliest categorical articulation of this principle is the theorem of Elmendorf's which says that there is an equivalence of ($\infty$--)categories $\spc_G \simeq \func(\orbit(G)\op,\spc)$ between the category of $G$--spaces and presheaves on the orbit category of $G$. 

In the same way that the category of spectra is the universal homology theory on spaces, it has been identified in \cite{segalNice} and fully developed in \cite{lewis_may_steinberger}  that the appropriate replacement of spectra in the equivariant setting is the stable category $\spectra_G$ of \textit{genuine $G$--spectra}. In the finite group case, we may even view  $\spectra_G$ as $G$--Mackey functors valued in spectra. To each $G$--space ${X} \in \spc_G$ we may associate an ``equivariant stable homotopy type" $\Sigma^\infty_+ {X} \in \spectra_G$. From this, we may, among other things, recover the stable homotopy type of all fixed point spaces $X^H$ for subgroups $H \leq G$  via the \textit{geometric fixed points functor} $\Phi^H \colon \spectra_G \rightarrow \spectra$ to obtain $\Phi^H(\susps_+X) \simeq \susps_+(X^H)$. In fact, more generally, there is a geometric fixed points functor  associated to a \textit{family} $\family$ of closed subgroups of $G$, and it should be thought as a functor which universally kills equivariant cells $G/H$ where $H\in\family$.

It turns out to be fruitful to treat questions about $G$--spaces not only through a categorical lense, but rather work with equivariant versions of categories themselves. Two equivalent approaches have been developed, the first by Barwick--Dotto--Glasman--Nardin--Shah \cite{parametrisedIntroduction,expose1Elements,shahThesis,Nardin2017Thesis,nardinShah} and the second by Martini--Wolf \cite{Martini2022Cocartesian,Martini2022Yoneda,MartiniWolf2022Presentable,MartiniWolf2024}. Each formalism has their advantages, and for our purposes in this article, we have chosen to work mainly in the second one since it affords us the flexibility of working over an arbitrary topos: this will allow us to give uniform and streamlined proofs. In either case, the appropriate replacement for $\cat$ in the equivariant setting is the category $\cat_G \coloneqq \func(\orbit(G)\op,\cat)$ of \textit{$G$--categories}. We will write $\udl{\sC}$ for an object in $\cat_G$ and $\sC(G/H)$ for the evaluation of $\udl{\sC}$ at $G/H$. Some $G$--categories of special interest to us, viewed as presheaves on the orbit category, are
\begin{align*}
    \udl{\spc} \colon & \hspace{3mm} G/H \mapsto \spc_H & \text{the $G$--category of $G$--spaces};\\
    \myuline{\spectra} \colon & \hspace{3mm} G/H \mapsto \spectra_H & \text{the $G$--category of $G$--spectra};\\
    \udl{\Pic}(\myuline{\spectra}) \colon  & \hspace{3mm} G/H \mapsto \Pic(\spectra_H) & \text{the $G$--space of invertible $G$--spectra}.
\end{align*}
Crucially, for a large part of this article, we will rely upon a good theory of parametrised presentable categories, whereupon we may speak of, for instance, the category $\presentable^{L,G-\mathrm{st}}_G$ of $G$--stable presentable $G$--categories (in which $\myuline{\spectra}_G$ is the symmetric monoidal unit). Now for a closed normal subgroup $N\leq G$ and writing $Q\coloneqq G/N$ for the quotient group, there is a fully faithful inclusion $\mathrm{incl}\colon \widehat{\cat}_Q\hookrightarrow \widehat{\cat}_G$ of large $Q$--categories into large $G$--categories and it admits a left adjoint $(-)^N$ given by forgetting all information from the subgroups of $G$ that do not contain $N$. These two functors restrict to give functors $\mathrm{incl}\colon \presentable_{Q}^{L,Q-\mathrm{st}}\hookrightarrow \presentable_{G}^{L,G-\mathrm{st}}$ and $(-)^N\colon \presentable_{G}^{L,G-\mathrm{st}}\rightarrow\presentable_{Q}^{L,Q-\mathrm{st}}$ respectively. However, the adjunction on large categories does \textit{not} descend to an adjunction on the  presentable categories because the adjunction unit in $\widehat{\cat}_G$ is not a morphism in $\presentable_G^{L,G\mathrm{-st}}$. Nevertheless, we show the following:

\begin{alphThm}[{\cref{prop:categorical_brauer_quotients,prop:stability_for_quotient_groups}}]\label{alphThm:brauer_quotients}
    Let $G$ and $Q$ be as above. Then the inclusion  $\presentable_{Q}^{L,Q-\mathrm{st}}\hookrightarrow\presentable_G^{L,G-\mathrm{st}}$  admits a left adjoint $\brauerQuotientFamily{N}$ which is a smashing localisation.
\end{alphThm}

We call the functor $\brauerQuotientFamily{N}$ above  the \textit{Brauer quotient } functor, borrowing the term from classical Mackey functor theory. In fact, in the precise versions of the result, we prove it  more generally for families and we also prove this for small $G$--stable categories when the group is finite. The result above should be viewed as a categorification of the geometric fixed points functors aforementioned. We may indeed recover the usual geometric fixed points functors by considering the adjunction unit evaluated on $\myuline{\spectra}$. Using \cref{alphThm:brauer_quotients}, we may functorially construct geometric fixed points for \textit{any} $G$--stable category, a construction that will be important to us in performing isotropy separation arguments for equivariant Poincar\'{e} duality, as we shall see below.

\subsection*{Equivariant and parametrised Poincar\'{e} duality}

Poincaré duality is usually formulated as the statement that for a closed $d$-manifold $M$, there exists an infinite cyclic local coefficient system $\orbit$ on $M$ and a class $[M] \in H_d(M;\orbit)$ such that the the cap product with $[M]$ induces, for every local coefficient system $\eta$ on $M$, isomorphisms
\[ [M] \cap - \colon H^*(M;\eta) \longrightarrow H_*(M; \eta \otimes \orbit). \]
We will briefly recall a different formulation due to \cite{Klein2001} in terms of local systems of spectra (c.f. \cite[App. A]{markusPoincareSpaces} for a nice and detailed exposition of this point of view). It will let us arrive at an equivariant  version (even a parametrised one, in general) with little creativity. 

Following the notation of \cite{Cnossen2023}, let $M \colon M \rightarrow *$ be the unique map. We get adjunctions
\begin{equation*}
\begin{tikzcd}
    \spectra^M \arrow[rr, "M_!" description, bend left] \arrow[rr, "M_*"' description, bend right] 
    &  
    & \spectra \arrow[ll, "M^*"'description]
\end{tikzcd}
\end{equation*}
where $M_!$ (resp. $M_*$) associates to each local system $\xi \in \spectra^M$ its colimit (resp. limit), and the resulting spectrum should be viewed as the homology (resp. cohomology) of $M$ with coefficients in $\xi$.
For a smooth manifold $M$, the Spivak normal fibration can be used to construct a local system $D_M\in\Pic(\spectra)^M\subset \spectra^M$. The stable Pointryagin--Thom collapse map can be viewed as a map $c \colon \sphere \rightarrow M_!D_M$, which deserves the name ``fundamental class" for $M$. It is possible\footnote{This is conveniently descibed in \cite{Lurie_LTheorySurgery} or \cite{markusPoincareSpaces}, which we generalise in \cref{subsection:parametrised_spivak_data}.} to describe the cap product with the fundamental class as a morphism in $\func(\spectra^M,\spectra)$
\begin{equation}
    \label{eq:PD}
    c \cap - \colon   M_*(-) \rightarrow M_!(- \otimes D_M)
\end{equation}
and Poincar\'{e} duality may be interpreted as demanding that this transformation be an equivalence. It turns out that using general Morita theory, it is possible to construct a unique $D_X \in \spectra^X$ and $c \colon \sphere \rightarrow \uniqueMapX_!D_X$ for \textit{any compact space $X$} such that the associated map $c \cap - \colon   \uniqueMapX_*(-) \rightarrow \uniqueMapX_!(- \otimes D_X)$ is an equivalence (since it is such a crucial property, we term the property of this map being an equivalence as \textit{twisted ambidexterity}, inspired by \cite{Cnossen2023}).  The local system $D_X$ is referred to as the dualising sheaf of $X$.
We call the compact space $X$ a \textit{Poincaré space} if the dualising sheaf $D_X$ takes values in invertible spectra, i.e. in $\Pic(\spectra)$. Poincar\'{e} duality, as formulated above, shows that closed manifolds are Poincar\'{e} spaces. It bears mentioning that Wall, in his seminal paper \cite{WallPD}, introduced the notion of \textit{Poincar\'{e} complexes} and we show in \cref{example:wall_poincare_complex} that his notion coincides with Poincar\'{e} spaces as defined above (see also \cite[Prop. A.12]{markusPoincareSpaces} for a proof for finite spaces).

Now the theory of parametrised higher categories as introduced in \cite{MartiniWolf2022Presentable,MartiniWolf2024} affords us the latitude of considering the situation just presented, but working internally to an arbitrary base topos $\B$ (e.g. the topos $\spc_G$ of $G$--spaces). In this setting, one may speak of $\B$--functor categories, $\B$--adjunctions, $\B$--Kan extensions,  $\B$--(co)limits, $\B$--Morita theory, etc. For example, in the equivariant situation, we may take parametrised colimits with respect to a diagram indexed by a $G$--category (and so in particular, diagrams indexed by $G$--spaces). In light of this, we may just transpose the discussions of the previous paragraph into the parametrised setting with relative ease and make sense of the notion of parametrised Poincar\'{e} duality.

\vspace{1mm}

However, in our theory, we have chosen to strictly generalise the well--known presentation above in two main ways: (1) we do not just consider the coefficient category of spectra (or rather $\B$--spectra), but we allow for arbitrary symmetric monoidal coefficient categories (which was also done in \cite{Cnossen2023} in the twisted ambidextrous setting); (2) we do not just consider \textit{presentable} coefficient $\B$--categories, as in \cite{Cnossen2023}, but also arbitrary $\B$--categories. As we shall see in our geometric applications later, both of these extra flexibilities will play  important and conceptually  natural roles. Crucially, point (2) precludes us from having access to Morita theory, by which token being a $\B$--Poincar\'{e} space is a property. Hence, we will have to declare more structures in order to be able to speak of Poincar\'{e} duality in arbitrary coefficient categories. We axiomatise this as  follows:

\begin{defn*}[Spivak data, {\cref{defn:spivak_data}}]
    Let $\udl{X}$ be a $\B$--space and $\udl{\sC}$ a symmetric monoidal $\B$--category admitting $\udl{X}$--shaped colimits. A \textit{$\udl{\sC}$--Spivak datum for $\udl{X}$} is defined to be a pair $(\xi,c)$ where $\xi\in\udl{\func}(\udl{X},\udl{\sC})$ is called the \textit{dualising sheaf} and $c\colon \unit_{\udl{\sC}}\rightarrow \uniqueMapX_!\xi$ is a morphism in $\udl{\sC}$ called the \textit{fundamental class}.
\end{defn*}

From this datum, provided $\udl{\sC}$ satisfies a standard condition called the $\udl{X}$--projection formula (c.f. \cref{terminology:projection_formula}), we may construct from $(\xi,c)$ a transformation
\begin{equation}\label{eqn:capping_map_intro}
    \ambi{c}{\xi}{-}\colon \uniqueMapX_*(-)\longrightarrow \uniqueMapX_!(-\otimes \xi)
\end{equation}
as in \cref{eq:PD}, called the capping transformation. This is a morphism in $\udl{\func}(\udl{\sC}^{\udl{X}},\udl{\sC})$. We then say that the $\udl{\sC}$--Spivak datum $(\xi,c)$ is \textit{twisted ambidextrous} if \cref{eqn:capping_map_intro} is an equivalence, and we say that it is \textit{Poincar\'{e}} if additionally $\xi\colon \udl{X}\rightarrow \udl{\sC}$ factors through $\udl{\picardSpace}(\udl{\sC})$. As will become clear in the article, one advantage of studying such a structural axiomatisation of the situation is that it provides us with a finer control over the specific fundamental class and capping equivalences at play.
Also note that this approach is very close to traditional formulations of Poincar\'e duality and even covers general duality groups in the sense of Bieri-Eckmann \cite{Bieri_Eckmann}, which we hope helps clarify the relation of modern works such as Cnossen's \cite{Cnossen2023} with classical literature.

Having set up the primitive notions of the paper, we  focus on the equivariant setting (i.e. working over the base topos $\B=\spc_G$ of $G$--spaces) for the rest of the introduction and  point out the more general parametrised versions along the way, as appropriate. 

We now state one of the main theorems of our abstract equivariant Poincar\'{e} duality theory. For ease of statement in this introductory section, we only state it for presentable $G$--categories, where being Poincar\'{e} is a property of a $G$--space (i.e. the Spivak datum is unique, if it exists).

\begin{alphThm}[Poincar\'{e} isotropy basechange, {\cref{thm:poincare_isotropy}}]\label{alphThm:poincare_isotropy_basechange}
    Let $G$ be a compact Lie group, $N$ a  closed normal subgroup, $\udl{X}$ a $G$--space, and $\udl{\sC}$ a presentably symmetric monoidal fibrewise stable $G$--category. If $\udl{X}$ is Poincar\'{e} with coefficient $\udl{\sC}$, then the $G/N$--fixed points space $\udl{X}^N$ is Poincar\'{e}  with coefficient in the fibrewise stable Brauer quotient $G/N$--category $\brauerQuotientFamily{N}\udl{\sC}$ of \cref{alphThm:brauer_quotients}.
\end{alphThm}

In the full statement, the result above works in the generality of a fixed family of closed subgroups and we also provide a version of the theorem  for small categories with a weaker conclusion, but which is nevertheless strong enough for our applications in \cref{subsection:pulling_back_fixed_points}. Furthermore, \cref{alphThm:poincare_isotropy_basechange} will be the key tool for our categorified Smith--theoretic proof of \cref{alphThm:atiya_bott}. 

It should also be mentioned that the theorem above is an immediate consequence of a much more general set of basechange results for arbitrary base topoi (c.f. \cref{thm:base_change_of_tw_amb_spivak_data,thm:small_base_change,thm:omnibus_geometric_pushforward_of_capping}). These general results constitute the main theorems in our theory of parametrised Poincar\'{e} duality. The operating philosophy of these results, and thus of the paper by extension, is that many important inductive manoeuvres on Poincar\'{e} duality may be casted as instances of basechanging the coefficient categories and basechanging the ambient topoi.

The most important coefficient category for us will be that of genuine $G$--spectra $\myuline{\spectra}$, and we say that a $G$--space is $G$--Poincar\'{e} if it is Poincar\'{e} with respect to $\myuline{\spectra}$. As a straightforward consequence of \cref{alphThm:poincare_isotropy_basechange}, we obtain the following result which says that being Poincar\'{e} is compatible with taking fixed points. It should be viewed as a spectral enhancement of the homological statement \cite[Prop. 2.4]{costenoble2017equivariant} of Costenoble--Waner.

\begin{alphThm}[{\cref{thm:fixed_points_poincare_duality}}]\label{alphThm:fixed_points_PD}
    Let $G$ be a compact Lie group, $H\leq G$ a closed subgroup, and $N\leq G$ a closed normal subgroup. If $\udl{X}$ is a $G$--Poincar\'{e} space, then $X^H$ is a Poincar\'{e} space and $\udl{X}^N$ is a $G/N$--Poincar\'{e} space. 
\end{alphThm}

In fact, we also provide a conditional converse to the result above in \cref{thm:PD_fixed_point_recognition} where we give a recognition principle for equivariant Poincar\'{e} duality in terms of nonequivariant Poincar\'{e} duality by way of the geometric fixed points functors. We warn the reader that the converse - that a compact $G$-space $\udl{X}$ is $G$-Poincar\'e if $X^H$ is Poincar\'e for each closed subgroup $H \leq G$ - is \textit{not} true, and we will comment on this below.

\vspace{1mm}

Next, to justify the theory of equivariant Poincar\'{e} spaces, we first give a large collection of examples of such as encapsulated by the following (the smooth manifolds case is certainly not new and has been proven in various forms for example in \cite{MaySigurdsson2006,costenoble_waner_book,costenoble2017equivariant,horevKlangZou}):
\begin{alphThm}[{\cref{thm:closed_smooth_G_manifolds_are_poincare,thm:generalised_homotopy_representations}}] \label{thm:examples_G_pd}
    Let $G$ be a compact Lie group. Then smooth closed $G$--manifolds and tom Dieck's generalised homotopy representations are  $G$--Poincar\'{e}  spaces.
\end{alphThm}

Here, by tom Dieck's generalised homotopy representations,  we mean a compact $G$--space $\udl{\hrep}$ such that all fixed points $\hrep^H$ have the homotopy type of a sphere of some dimension. They are a class of $G$--spaces strictly distinct from smooth $G$--manifolds. For example, Bredon \cite{bredonTrans} gave an example of a generalised $C_2$--homotopy representation $\udl{\hrep}$ such that $\hrep^{C_2}$ and $\hrep^e$ are spheres of the same dimension, although the map $\hrep^{C_2} \rightarrow \hrep^e$ is not an equivalence. Of course, this cannot arise as the underlying $C_2$--space of a smooth closed $C_2$--manifold. 

With plenty of naturally interesting examples in hand, we then provide a suite of construction principles to construct new examples of equivariant Poincar\'{e} spaces from old ones in \cref{subsection:construction_principles}. Among other things, we show that equivariant Poincar\'{e} duality is preserved under various standard equivariant operations such as inflations, restrictions, inductions, coinductions, and Borelifications. We also make contact with the nilpotence theory of Mathew--Naumann--Noel \cite{MNN17,MNNDerivedInduction} and show that Poincar\'{e} duality interacts nicely with nilpotence with respect to families. Furthermore, we also show the following equivariant generalisation of Klein's well--known result \cite[Cor. F]{Klein2001}.
\begin{alphThm}[{\cref{thm:poincare_integration}}]\label{alphThm:poincare_integration}
    Let $\udl{\sC}$ be a presentably symmetric monoidal $G$--category and $f\colon \udl{X}\rightarrow \udl{Y}$  a map of $G$--spaces. If $\udl{Y}$ is $\udl{\sC}$--Poincar\'{e} and for every closed subgroup $H\leq G$, the fibres of $f$ over every $H$--point of $\udl{Y}$ is $\res^G_H\udl{\sC}$--Poincar\'{e}, then $\udl{X}$ is $\udl{\sC}$--Poincar\'{e} too. 
\end{alphThm}

Having set up a robust and nonempty abstract theory, we now ask ourselves: what does it all mean and what is it useful for?

\subsection*{Phenomena and applications}

It turns out that equivariant Poincar\'{e} duality for a $G$--space $\udl{X}$ offers quite a lot ``hidden'' homotopical information about $\udl{X}$ that is not obvious  from merely having all its fixed points satisfying Poincar\'{e} duality. To put it in a slogan, this is essentially because there is a global fundamental class which ties together the local fundamental classes of the various fixed points spaces in nontrivial ways. This is certainly not a new observation and is one that has been appreciated by many of the forerunners to this story. For the remainder of this introduction, we highlight three applications of a geometric flavour of our theory which exploit this principle in one form or another and which illustrate the rigidity of Poincar\'{e} spaces hinted at before.

\vspace{1mm}

Let  $p$ be a prime and $G$ a finite group. Now for a $G$--space $\underline{X}$, we may view the cohomology group $H^*(X^G;\mathbb{F}_p)$ as a count of the fixed points of $\underline{X}$ mod $p$. As aptly interpreted  by Browder in \cite{browder_fixed_points}, if a map of $G$--spaces $f\colon \underline{X}\rightarrow \underline{Y}$ induces an injection on $H^*(-;\mathbb{F}_p)$, then we may ``pull back'' the mod $p$ fixed points of $\udl{Y}$ to those of $\udl{X}$ and, among many things, he studied situations in which one can upgrade this cohomological statement to an actual surjection on the fixed points as topological spaces. Cohomological injection results of this type were first proved by Bredon as \cite[Thm. 5.1]{bredon_poincare_duality} for the group $G=C_p$ purely homotopy--theoretically and later on generalised by Browder as \cite[Thm. 1.1]{Browder_Poincare} to arbitrary finite abelian $p$--groups under stronger manifold assumptions. This question has also been studied for instance in \cite{ewingStong,hanke_puppe_pulling_back_fixed_points}. In this line, we employ our categorical technology in the generality of Poincar\'{e} duality for small, non--presentable coefficient categories to prove the following version of the aforementioned results: 

\begin{alphThm}[{\cref{thm:bredonBrowderInjection}}]\label{alphThm:bredon_browder_injection}
    Let $A$ be an elementary abelian $p$--group. Let $f\colon \underline{X}\rightarrow \underline{Y}$ be a map of compact $A$--spaces. Suppose $X^e, Y^e$ are $\eilenbergMacLaneFp$--Poincar\'{e} spaces such that $f^e\colon X^e\rightarrow Y^e$ is of degree one.  Then for any $\eilenbergMacLaneFp$--local system  $\zeta\in\func(Y^A,\perfectCat_{\eilenbergMacLaneFp})$ for the fixed point space $Y^A$, the map $f^A$ induces an injection $H^*(Y^{A};\zeta)\rightarrow H^*(X^{A};f^*\zeta)$.
\end{alphThm}

Unlike the cited works above, our methods avoid manifold assumptions altogether and apart from one preliminary standard argument, we also avoid spectral sequences entirely and use instead formal categorical and stable homotopy theoretic manipulations. It should be noted also that our result  works for arbitrary  twisted coefficient systems, which as far as we know, is new and depends crucially on the categorical nature of our approach.  

As an application of \cref{alphThm:bredon_browder_injection} (in fact, the version proved by Bredon suffices), we obtain the following rigidity result for equivariant Poincar\'{e} duality spaces. 

\begin{alphThm}[{\cref{thm:Cp_PD_contractible_underlying_space}}]\label{alphThm:Cp_PD_contractible_underlying_space}
    Let $G$ be a solvable finite group (e.g. a group of odd order) and $\udl{X} \in \spc_{G}^\omega$ a  compact $G$--Poincar\'e space with $X^e \simeq \ast \in\spc$.    Then $\udl{X} \simeq \udl{\ast}\in\spc_G$.
\end{alphThm}

This is a slightly surprising result in light of Bredon's examples mentioned after \cref{thm:examples_G_pd} which demonstrated that Poincar\'{e}  spaces can be rather counterintuitive when the underlying space is noncontractible. Combining this with the celebrated theorem of Jones \cite{Jones} on the converse to Smith theory, we construct an example of a compact $C_p$--space whose underlying and fixed points spaces are Poincar\'{e} but which is not itself $C_p$--Poincar\'{e}, making good on our warning after \cref{alphThm:fixed_points_PD}.

\vspace{1mm}

We mention one more application, whose investigation was one of the main goals of this project. In \cite{ConnerFloyd64}, Conner--Floyd made the following conjecture:
\begin{quote}
    \textit{There cannot exist a periodic differentiable map of odd prime power period acting on a closed oriented manifold $V^n$, $n>0$ preserving the orientation and possessing exactly one fixed point.}
\end{quote}
The first proof of this statement (in fact, a slightly more general version) was given by Atiyah-Bott in \cite{AtiyahBottLefschetzII} and soon after by Conner--Floyd \cite{ConnerFloydMapsOfOddPeriod} themselves. Many variations have been proven since then, and we mention \cite{lueckDegree, AssadiBarlowKnop92} as further examples. Atiyah-Bott's argument uses Atiyah--Singer's index theory, whereas Conner--Floyd's proof uses a particular  bordism spectrum. In all these cases, local structures in the geometric settings were used in essential ways.

Inspired by the notion of a ``gluing class'' due to L\"{u}ck which measures how the singular part of a $G$--space is glued into the whole space, we consider such a construction in our setting and use it to prove a very general version of the Conner--Floyd conjecture which in particular yields the theorem of Atiyah--Bott as an immediate corollary.

\begin{alphThm}[{\cref{thm:generalised_atiyah_bott}}]\label{alphThm:atiya_bott}
    Let $p$ be an odd prime, $G = C_{p^k}$ for some $k$, and suppose $\underline{X}\in\spc_G^{\omega}$ is a $G$--Poincare space such that the underlying space $X^e\in\spc^{\omega}$ is connected, $\mathbb{Z}$--orientable, and has formal dimension (in the classical sense) $d>0$.   Then $X^G\not\simeq \ast$.
\end{alphThm}

Our proof uses categorified Smith--theoretic methods afforded to us by \cref{alphThm:poincare_isotropy_basechange} which reduces the problem to various forms of Tate cohomology considerations.  In particular, it is fully homotopical and thus is of a ``global'' nature. We hazard a suggestion here that, apart from being a new generalisation of a very classical result,  \cref{alphThm:atiya_bott} locates the explanation of such phenomena in the global realm of homotopy theory as opposed to the local one of geometry. Finally, let us point out that when paired together, \cref{alphThm:Cp_PD_contractible_underlying_space,alphThm:atiya_bott} give a curious partial dichotomy for $C_{p^k}$--Poincar\'{e} spaces delineated by whether or not the underlying space is contractible.

\vspace{1mm}

Before closing the introduction, we mention that since the theory of Poincar\'{e} duality here was developed in the generality of Martini--Wolf's parametrised category theory, it might be interesting to explore the theory presently developed in the context of  topoi other than the equivariant ones. It could be said that the  defining feature of our work is in exploiting various kinds of geometric morphisms of topoi central to equivariant homotopy theory, and one can imagine that this might also lead to fruitful lines of pursuit in other contexts.

\subsection*{Relations to other work}

The following works are some of the milestones that made this article possible. Wall introduced the notion of a Poincar\'e space motivated by surgery theory, and developed their theory in \cite{WallPD}. Klein built up an impressive amount of theory related to Poincar\'e spaces, one of his most influential concepts being that of the dualising spectrum \cite{Klein2001, Klein2007}. His approach was revisited by Lurie \cite[Lecture 26]{Lurie_LTheorySurgery}, Nikolaus-Scholze \cite[Sec. I.4.1.]{nikolausScholze} and Land \cite[Appendix A]{markusPoincareSpaces}, also providing an account of the ``universality" of Klein's construction. An advantage of Klein's approach is that the stable Spivak fibration of a Poincar\'e space admits a categorically more natural description (as the dualising spectrum) than in Wall's original work, where it had to be constructed.  The theory of dualising spectra in general was coined ``twisted ambidexterity" by Cnossen in \cite{Cnossen2023}.

Cnossen develops twisted ambidexterity in a general topos in terms of parametrised homotopy theory, and his approach is what we most closely follow. His  motivation is a characterisation of the $G$-category of $G$-spectra as the initial presentable, fibrewise stable $G$-category in which all compact $G$-spaces are twisted ambidextrous. An important predecessor in the equivariant context is the book of May-Sigurdsson \cite{MaySigurdsson2006}, which also gives an account of equivariant Poincar\'e duality. A more classical approach to equivariant Poincar\'e duality in terms of equivariant homology and cohomology (over the Burnside ring), much in the spirit of Wall's original definition can be found in the work of Costenoble-Waner \cite{costenoble_waner_finite_group_PD, costenoble2017equivariant}. For finite groups $G$, a nonabelian version of equivariant Poincar\'{e} duality for so--called ``$V$--framed manifolds'' has also been studied in \cite{zouNonabelianPD,horevKlangZou} which in particular implies the homological version of Poincar\'{e} duality for such objects, c.f. \cite[Prop. 4.1.4]{horevKlangZou}. An approach to Poincar\'e duality in the context of six-functor-formalisms is to be found for instance in \cite[Lecture V]{ScholzeSix}.

\subsection*{Organisation of the paper}
In \cref{section:preliminaries}, we introduce and develop the categorical underpinnings that will support the later sections. In  more detail, we recall in \cref{sec:parametrised_category_theory} the Martini--Wolf theory of parametrised higher categories and take the opportunity to record some elementary observations about geometric morphisms that we need. In \cref{subsection:equivariant_categories_and_families}, we specialise the preceding discussions to the equivariant context and recall the standard gamut of equivariant operations on categories; this will lead to the proof of \cref{alphThm:brauer_quotients}, categorifying the well--known geometric fixed points functor. This will be used later to articulate our results about fixed points of equivariant Poincar\'{e} spaces.

Having set up the requisite language, we turn to the matter of defining and studying Poincar\'{e} duality in \cref{section:parametrised_PD} in the general context of parametrising over arbitrary topoi. We define and work out the basic properties of Spivak data in \cref{subsection:parametrised_spivak_data}; we then use this structure to define twisted ambidexterity and Poincar\'{e} duality in \cref{subsection:parametrised_twisteed_ambidex} with respect to arbitrary coefficient $\B$--categories. In \cref{subsection:constructions_with_spivak_data}, we give several constructions one can perform on Spivak data and prove the main results of the section in the form of \cref{thm:base_change_of_tw_amb_spivak_data,thm:small_base_change} on basechanging coefficient categories and \cref{thm:omnibus_geometric_pushforward_of_capping} on basechanging the base topoi. Finally, we set up a  theory of degrees for maps between Poincar\'{e} spaces in \cref{sec:degree_theory}.

We specialise the general parametrised theory in \cref{section:parametrised_PD} to the equivariant situation in \cref{section:equivariant_PD_elements} for compact Lie groups. After recording the specialisations of the notions in \cref{subsection:setting_the_stage}, we state and prove several isotropy separation statements including \cref{alphThm:poincare_isotropy_basechange,alphThm:fixed_points_PD} in \cref{subsection:fixed_points_methods}, which will form  our main suite of techniques for dealing with fixed points of equivariant Poincar\'{e} spaces. Following that, we provide a set of construction principles in \cref{subsection:construction_principles} to generate new equivariant Poincar\'{e} spaces from old ones  and we supply in \cref{sec:examples_G_PD} geometrically natural examples of Poincar\'{e} spaces. We then introduce the notion of gluing classes in \cref{subsection:gluing_classes} that will form the main obstruction class for our applications in \cref{subsection:theorem:single_fixed_points}, and we lay down a rudimentary theory of equivariant degrees in \cref{sec:equivariant_degree}.

In the final \cref{section:geometric_applications}, we use categorified Smith--theoretic methods supported by the abstract theory developed in the article to give two strands of applications: in \cref{subsection:bredonPoincareDuality}, we use degree theory to show \cref{alphThm:bredon_browder_injection}, which is in turn used to show \cref{alphThm:Cp_PD_contractible_underlying_space}; then, in \cref{subsection:theorem:single_fixed_points}, we use the gluing classes to prove \cref{alphThm:atiya_bott}.

Beyond the main body of the article, we record in \cref{sec:G_stability_presentable_categories} several characterisations of $G$--stability for presentable categories when $G$ is a compact Lie group, and we prove a standard observation about reflecting pushout squares in \cref{section:reflecting_pushout_squares}.

\subsection*{Acknowledgements}

We are grateful to Bastiaan Cnossen, Ian Hambleton, Markus Land, Sil Linskens, Wolfgang L\"uck, and Shmuel Weinberger for numerous helpful conversations and encouragements on this project. All three authors are supported by the Max Planck Institute for Mathematics in Bonn.
The second author thanks the Department of Mathematics at the University of Chicago for its hospitality, where parts of this project were written.

\section{Preliminaries}\label{section:preliminaries}

The present section reviews the techniques that are essential to our approach to parametrised Poincar\'e duality. In \cref{sec:parametrised_category_theory}, we first recall some preliminaries on category theory parametrised by a topos $\category{B}$. Special attention is given to presentable $\category{B}$-categories, and basechange methods that allow us to switch the base topos along a geometric morphism. These basechange results will be essential for the isotropy separation arguments in the equivariant context.

After that, \cref{subsection:equivariant_categories_and_families} specialises to topoi related to the category $\spc_G$ of $G$-spaces, where $G$ is a compact Lie group. This also features various change-of-group functors like induction, restriction and coinduction along a homomorphism of compact Lie groups $\alpha \colon H \rightarrow G$, as well as the theory of families. 
We give a quick recollection on the basics of $G$-spaces. 
After that, we record some facts on equivariant stability, preservation of equivariant stability under change-of-group functors, and multiplicative properties of these constructions. We then prove the main result of this section, namely \cref{alphThm:brauer_quotients} on  Brauer quotients which categorifies the geometric fixed points functors. The section ends with some remarks on free actions
that will be used later on.

\subsection{Parametrised category theory} \label{sec:parametrised_category_theory}

For the rest of this section, let $\category{B}$ be a topos. For us, most topoi of interest will actually be presheaf topoi, and for the purpose of this article the most important such will be that of presheaves on the orbit category $
\orbit(G)$ for a compact Lie group $G$.  As category theory internal to $
\category{B}$ is essential to our considerations, we give a short recollection of the formalism developed by Martini and Martini-Wolf in the series of articles \cite{Martini2022Yoneda, Martini2022Cocartesian,MartiniWolf2022Presentable, MartiniWolf2024}.
Let us mention here that the theory of categories parametrised by a topos was preceded by about a decade by the theory of categories parametrised by presheaf topoi, pioneered by Barwick--Dotto--Glasman--Nardin--Shah in \cite{expose1Elements,parametrisedIntroduction,nardinExposeIV,shahThesis}.

\begin{defn}
    A \textit{$\category{B}$--category} $\udl{\category{C}}$ is a limit preserving functor $\udl{\category{C}} \colon \category{B}\op \to \cat$, i.e. a sheaf of categories on $\category{B}$.
    Denote by $\cat_{\category{B}} \subseteq \presheaf_\cat(\category{B})$ the full subcategory on $\category{B}$--categories.
    Maps in $\cat_{\category{B}}$ are called \textit{$\category{B}$-functors}.
\end{defn}

In \cite{Martini2022Yoneda}, Martini produces an equivalence of categories
\[  \presheaf_{\cat}(\category{B}) \supseteq \{ \text{$\cat$-valued sheaves on  $\category{B}$} \}  \simeq \{ \text{Complete Segal objects in $\category{B}$} \} \subseteq s\category{B}. \]
It is worthwhile to study $\category{B}$--categories from both perspectives, the \textit{parametrised} point of view on the left as well as the \textit{internal} point of view on the right.

As our arguments will often require us to work with unparametrised categories and $\category{B}$--categories at the same time, we follow the convention of underlining $\category{B}$ categories, so a generic $\category{B}$ category is denoted $\udl{\category{C}}, \udl{\category{D}}, \udl{\category{E}}, \dots$ and so on. For example, the category of $G$--spaces will be denoted by $\spc_G$ while the $G$--category of $G$--spaces is written $\udl{\spc}$ (or $\udl{\spc}_G$ if we want to emphasise that this is happening for the group $G$).

\begin{example}[Presheaf topoi]
    For a small category $T$, consider the presheaf topos $\presheaf(T) = \func(T\op, \spc)$.
    We write $\cat_T \coloneqq \cat_{\presheaf(T)}$ and call it the category of \textit{$T$--categories}.
    Restriction along the Yoneda embedding $T \hookrightarrow \presheaf(T)$ induces an equivalence $\cat_{T} \xrightarrow{\simeq} \func({T\op}, \cat)$ so a $T$--category is simply a functor $T\op \to \cat$.
    In particular for $T= *$, we see $\cat_{\spc} \simeq \cat$ and an $\spc$--category is just an ordinary ($\infty$-)category.
    In the special case where $G$ is a finite group (or a compact Lie group) and $T = \orbit(G)$ is its orbit category, the category of transitive $G$-sets (or homogeneous $G$--spaces), we obtain the category $\cat_G \coloneqq \cat_{\orbit(G)}$ of $G$--categories.
\end{example}

\begin{example}[$\category{B}$-groupoids]
    The Yoneda embedding $\category{B} \to \func(\category{B}\op, \cat)$ restricts to a limit preserving fully faithful functor $\category{B} \to \cat_{\category{B}}$. An object in the essential image will be referred to as a \textit{$\category{B}$-groupoid}.
\end{example}

Switching to the internal picture, Martini \cite[Section 3.1]{Martini2022Yoneda} also characterised $\category{B}$-groupoids as the complete Segal objects that are equivalent to constant simplicial objects in $\category{B}$. 
We will not distinguish between $\category{B}$-groupoids and objects in $\category{B}$, so to avoid confusion we also denote both with an underline by $\udl{X},\udl{Y},\udl{Z},\dots$, except in the spectial case $\category{B} = \spc$.

With equivariant applications in mind, it will be important for us to change the base topos $\category{B}$. This can be done along \textit{geometric morphisms} of topoi, defined as: 
\begin{defn}\label{defn:geometric_morphisms_of_topoi}
    A geometric morphism is an adjunction between topoi $f^* \colon \category{B} \rightleftharpoons \category{B}' \cocolon f_*$ whose left adjoint $f^*$ is left exact, i.e. commutes with finite limits.
\end{defn}

\begin{cons}[Basechange along geometric morphisms]\label{cons:base_change_cat_geometric_morphism}
    A geometric morphism $f^* \colon \category{B} \rightleftharpoons \category{B}' \cocolon f_*$ of topoi induces an adjoint pair $f^* \colon \cat_{\category{B}} \rightleftharpoons \cat_{\category{B}'} \cocolon f_*$ where 
    the right adjoint $f_*$ is given by restriction along $(f^*)\op \colon \category{B}\op \to (\category{B}')\op$. 
\end{cons}

In the internal picture, the functor $f^* \colon \category{B} \rightarrow \category{B}'$ applied entrywise induces a functor on simplicial objects $s\category{B} \rightarrow s\category{B}'$, which commutes with finite limits. By \cite[Lem. 3.3.1]{Martini2022Yoneda}, it restricts to a functor on complete Segal objects, and this is how to obtain $f^* \colon \cat_{\category{B}} \rightarrow \cat_{\category{B}'}$. 

\begin{lem}
    \label{lem:geometric_morphism_induce_left_exact_morphisms_on_complete_segal_objects}
    For a geometric morphism $f^* \colon \category{B} \rightleftharpoons \category{B}' \cocolon f_*$ of topoi, the functor    $ f^* \colon \cat_{\category{B}} \rightarrow \cat_{\category{B}}$    preserves finite limits. In fact, it preserves all limits if $f^* \colon \category{B} \rightarrow \category{B}'$ does.
\end{lem}
\begin{proof}
    The induced functor $sf^* \colon s\category{B} \rightarrow s\category{B}'$ commutes with finite limits, and complete Segal objects are closed under limits. Thus, $f^*$ preserves finite limits by being a restriction of a finite limit preserving functor to a category closed under (finite) limits.
\end{proof}

\begin{example}[Geometric morphisms from presheaves]
A good supply of examples of geometric morphisms is given by considering a functor $f \colon S \rightarrow T$ of small categories.
Then restriction and right Kan extension along $f\op \colon S\op \to T\op$ induce a geometric morphism $f^* \colon \presheaf(T) \rightleftharpoons \presheaf(S) \cocolon f_*$.
\end{example}

\begin{example}[\'Etale geometric morphisms]
    If $\category{B}$ is a topos and $\obj{X} \in \category{B}$, then so is the slice category $\category{B}_{/\obj{X}}$. We have an ajunction
    \begin{equation}
        \label{eq:etale_geometric_morphism}
        (\pi_{\obj{X}})^* \colon  \category{B} \rightleftharpoons \category{B}_{/\obj{X}} \cocolon (\pi_{\obj{X}})_*
    \end{equation}
    whose left adjoint takes $\obj{A} \in \category{B}$ to $\obj{A} \times \obj{X} \rightarrow \obj{X} $.
    Now $\category{B}_{/\obj{X}}$ itself is a topos, and the adjunction above is in fact a geometric morphism of topoi.
    Geometric morphisms equivalent (in the category of topoi) equivalent to such of this kind are called \textit{\'etale geometric morphisms}, see \cite[Sec. 6.3.5.]{lurieHTT} for a detailed account. A special feature of \'etale geometric morphisms is that in the adjunction \cref{eq:etale_geometric_morphism} a further left adjoint exists, and so $(\pi_{\obj{X}})^{*}$ commutes with all limits.  The further left adjoint forgets the map to $\obj{X}$ and is denoted by $(\pi_{\obj{X}})_!$. A useful characterisation of \'etale geometric morphisms is given in \cite[Prop. 6.3.5.11.]{lurieHTT}.
\end{example}

Note that if a geometric morphism $f^* \colon \cat_{\category{B}} \rightleftharpoons \cat_{\category{B}'} \colon f_*$ is \'etale, then \cref{lem:geometric_morphism_induce_left_exact_morphisms_on_complete_segal_objects} shows that $f^* \colon \cat_\category{B} \rightarrow \cat_{\category{B}'}$ preserves all limits.

\begin{example}[Constant categories and global sections]\label{example:constant_and_global_sections}
    Recall from \cite[Prop. 6.3.4.1]{lurieHTT} that there is a unique geometric morphism $\const \colon \spc \rightleftharpoons \category{B} \colon \Gamma$. It induces an adjunction
    $\const \colon \cat \rightleftharpoons \cat_{\category{B}} \colon \Gamma$.
    We refer to $\Gamma$ as the \textit{global sections} functor. Explicitly, it is given by evaluation at the terminal object in $* \in \category{B}$. Since geometric morphisms from $\spc$ are unique, for any geometric morphism $f^*\colon \B\rightleftharpoons \B' : f_*$, we get a triangle of geometric morphisms
    \begin{center}
        \begin{tikzcd}
            \cat \ar[drrr, shift left = 2, "\constant_{\B'}"]\ar[drrr, shift right = 0, "\Gamma_{\B'}"'{xshift=-4pt,yshift=4pt},leftarrow]\dar["\constant_{\B}"', shift right = 1]\dar["\Gamma_{\B}", shift left = 1,leftarrow]\\
            \cat_{\B} \ar[rrr, shift left = 1, "f^*"]\ar[rrr, shift right = 1, "f_*"', leftarrow] &&& \cat_{\B'}
        \end{tikzcd}
    \end{center}
    In particular, note that $\constant_{\B'}\simeq f^*\constant_{\B}$ and $\Gamma_{\B}f_*\simeq \Gamma_{\B'}$.
\end{example}

\begin{example}[Internal functor categories and 2-categorical structures]
    The category $\cat_{\category{B}}$ is cartesian closed.
    This means that for any $\category{B}$--category $\udl{\category{C}}$ the product functor $- \times \udl{\category{C}} \colon \cat_{\category{B}} \to \cat_{\category{B}}$ admits a right adjoint $\internalfunc{\category{B}} (\udl{\category{C}}, \--) \colon \cat_{\category{B}} \to \cat_{\category{B}}$.
    We call $\internalfunc{\category{B}} (\udl{\category{C}}, \udl{\category{D}})$ the \textit{$\category{B}$--category of $\category{B}$-functors} and denote its global sections (in the sense of \cref{example:constant_and_global_sections}) by $\func_{\category{B}} (\udl{\category{C}}, \udl{\category{D}})$.
    Maps in $\func_{\category{B}} (\udl{\category{C}}, \udl{\category{D}})$ are called \textit{$\category{B}$-natural transformations}. $\func_{\category{B}}$ can be enhanced to a $\cat$-enrichment of $\cat_{\category{B}}$ making $\cat_{\category{B}}$ into a 2-category, see \cite[Remark 3.4.3]{Martini2022Yoneda}.
\end{example}

\begin{defn}[Adjoint functors]
    Using the 2-categorical structure on $\cat_{\category{B}}$, one can define an adjoint pair of $\category{B}$-functors as an internal adjunction in $\cat_{\category{B}}$.
    Explicitly, an adjunction consists of a pair of $\category{B}$-functors $ L \colon \udlcatC \rightleftharpoons \udlcatD \colon R$  as well as a pair of natural transformations
    $ \eta \colon \id_{\udlcatC} \rightarrow RL, \:\epsilon \colon LR \rightarrow \id_{\udlcatD}$.
    satisfying the triangle identities in the sense that $\epsilon_L \circ L\eta$ and $R\epsilon \circ \eta_R$ are equivalent to the respective identities. 
\end{defn}

We recall here the key standard categorical concept that will underpin most  this article.

\begin{cons}\label{defn:beck_chevalley_transformation}
    Suppose we have a commuting square of $\B$--categories
    \begin{center}
        \begin{tikzcd}
            \udl{\sC} \dar["g"']\rar["f"] & \udl{\D}\dar["g'"]\\
            \udl{\sC'}\rar["f'"] & \udl{\D'}
        \end{tikzcd}
    \end{center}
    such that $f$, $f'$ admit a right adjoints $h$, $h'$ respetively. Then we obtain a  transformation
    \[\beckChevalley_*\colon gh\xlongrightarrow{\eta_{gh}}h'f'gh \simeq h'g'fh \xlongrightarrow{h'g'\epsilon} h'g'\] called the \textit{right Beck--Chevalley transformation}. Similarly, if $f, f'$ admit left adjoints $\ell$, $\ell'$ respectively, then we obtain a transformation
    \[\beckChevalley_!\colon \ell'g' \xlongrightarrow{\ell'g'\eta}\ell'g'f\ell \simeq \ell'f'g\ell  \xlongrightarrow{\epsilon_{g\ell}} g\ell\]
    called the \textit{left Beck--Chevalley transformation}. We will often omit the words ``left'' and ``right'' when the context is clear. These transformations enjoy excellent functoriality properties, and we refer the reader to \cite[$\S2.2$]{CarmeliSchlankYanovski2022} for a good source on these matters.
\end{cons}

The following is an important class of Beck--Chevalley transformations.

\begin{terminology}[Projection formula]\label{terminology:projection_formula}
    Let $\udl{J}$ be a $\B$--category and $\pointProjection\colon \udl{J}\rightarrow \obj{\ast}$ the unique map. We say that a symmetric monoidal $\category{B}$--category $\udl{\category{C}}$   \textit{satisfies the $\udl{J}$-projection formula} if it admits $\udl{J}$-shaped colimits and the Beck--Chevalley transformation 
    \[ \projectionformula^{{J}}_! \colon \pointProjection_!(\xi \otimes \pointProjection^*(\--)) \rightarrow \pointProjection_!\xi\otimes (\--) \]
    of functors $\udl{\category{C}} \to \udl{\category{C}}$ is an equivalence for all $\xi\in\underline{\sC}^{\udl{J}}$.
\end{terminology}

In \cite[Prop. 3.2.9.]{MartiniWolf2024} it is shown, using work on relative adjunctions due to Lurie, that a functor $R \colon \udl{\category{C}} \to \udl{\category{D}}$ admits a left adjoint if and only if the following conditions are satisfied:
    \begin{enumerate}
        \item For every object $\obj{X} \in \category{B}$ the map $R(\obj{X}) \colon \category{C}(\obj{X}) \to \category{D}(\obj{X})$ admits a left adjoint $L(\obj{X})$.
        \item For every map $f \colon X \to Y$ in $\category{B}$ the Beck--Chevalley transformation $f^* L(\obj{X}) \to L(\obj{Y}) f^*$  is an equivalence.
    \end{enumerate}

\begin{example}[Limits and colimits]
    A $\category{B}$--category $\udl{\category{C}}$ is said to admit \textit{$\udl{I}$-shaped $\category{B}$-(co)limits} if the restriction functor $I^* \colon \udl{\category{C}} \to \internalfunc{\category{B}}(\udl{I}, \udl{\category{C}})$ along $I \colon \udl{I} \rightarrow \udl{*}$ admits a right (resp. left) adjoint. The right adjoint will usually be denoted by $I_*$, the left adjoint by $I_!$. 
    Note that for example, the adjunction unit for $I_! \dashv I^*$ produces for each $F \in \func_{\category{B}}(\udl{I},\udlcatC)$ a natural transformation
    \[ F \rightarrow I^* I_! F \in \func_{\category{B}}(\udl{I},\udlcatC)  \]
    which should be thought of as analogous to the diagram defining a colimit in unparametrised category theory. 
\end{example}

\begin{example}[Symmetric monoidal categories]
    A \textit{symmetric monoidal $\category{B}$--category} is a commutative monoid in $\cat_{\category{B}}$. 
   So $\cmonoid(\cat_{\category{B}})$ is the category of symmetric monoidal $\category{B}$--categories and symmetric monoidal functors.
    Notice that $\cmonoid(\cat_{\category{B}})$ is equivalent to the category of $\cmonoid(\cat)$--valued sheaves on $\category{B}$.
\end{example}

\subsubsection*{Geometric and \'{e}tale morphisms of topoi}

We record here further miscellaneous elementary observations about geometric and \'{e}tale morphisms that will be relevant to us later. Since this will just be a litany of minor technical results, the reader is advised to skip this on first reading and return to it as needed.

\begin{lem}\label{lem:pushforwar_fun_manoeuvre}
    Let $f^* \colon \category{B} \rightleftharpoons \category{B}' \cocolon f_*$ be a geometric morphism of topoi. 
    There is an equivalence, natural in $\obj{X} \in \B$ and $\udl{\category{C}} \in \cat_{\category{B}'}$ of functors $f_* \internalfunc{\category{B'}}(f^* \udl{X}, \udl{\category{C}}) \simeq \internalfunc{\category{B}}(\udl{X}, f_* \udl{\category{C}})$. Moreover, if $\udl{\category{C}}$ were a symmetric monoidal $\B$--category, then this equivalence naturally upgrades to a symmetric monoidal one.
\end{lem}
\begin{proof}
    Note that the diagram
    \begin{equation}\label{diag:Base_change_pullback_commutation}
    \begin{tikzcd}
        \cat_{\category{B}'} \ar[r, "(\pi_{f^*\obj{X}})^*"] \ar[d, "f_*"]
        & \cat_{\category{B}'_{/f^*\obj{X}}} \ar[r, "(\pi_{f^*\obj{X}})_*"] \ar[d, "f_*"]
        & \cat_{\category{B}'} \ar[d, "f_*"]
        \\
        \cat_{\category{B}} \ar[r, "(\pi_{\obj{X}})^*"]
        & \cat_{\category{B}_{/\obj{X}}} \ar[r, "(\pi_{\obj{X}})_*"]
        & \cat_{\category{B}}
    \end{tikzcd}       
    \end{equation}
    commutes as the corresponding diagram
    \begin{equation*}
    \begin{tikzcd}
        \category{B}' \ar[r, "(\pi_{f^*\obj{X}})^*"] \ar[d, "f_*"]
        & \category{B}'_{/f^*\obj{X}} \ar[r, "(\pi_{f^*\obj{X}})_*"] \ar[d, "f_*"]
        & \category{B}' \ar[d, "f_*"]
        \\
        \category{B} \ar[r, "(\pi_{\obj{X}})^*"]
        & \category{B}_{/\obj{X}} \ar[r, "(\pi_{\obj{X}})_*"]
        & \category{B}
    \end{tikzcd}       
    \end{equation*}
    of topoi commutes (this can be checked after passing to left adjoints everywhere where it is easy to see, see e.g. \cite[Remark 6.3.5.8.]{lurieHTT}).
    Now the composite $\cat_{\category{B}} \xrightarrow{(\pi_{\obj{X}})^*} \cat_{\category{B}_{/\obj{X}}} \xrightarrow{(\pi_{\obj{X}})_*} \cat_{\category{B}}$ sends $\udl{\category{C}} \in \cat_{\category{B}}$ to $\internalfunc{\category{B}}(\udl{X}, \udl{\category{C}}) = \lim_{\udl{X}} \udl{\category{C}}$ which proves the first part of the statement.

    For the part about symmetric monoidality, note that all functors in \cref{diag:Base_change_pullback_commutation} are finite limit preserving.
    As the forgetful functors $\cmonoid(\cat_{\category{B}}) \to \cat_{\category{B}}$ are limit preserving, this shows that the equivalence $f_* \internalfunc{\category{B'}}(f^* \udl{X}, \udl{\category{C}}) \simeq \internalfunc{\category{B}}(\udl{X}, f_* \udl{\category{C}})$ from the first part naturally refines to a symmetric monoidal one.
\end{proof}

\begin{lem}\label{lem:restriction_fun_manoeuvre}
    Let $f^* \colon \category{B} \rightleftharpoons \category{B}' \cocolon f_*$ be an \'etale morphism of topoi. 
    There is an equivalence, natural in $\obj{X} \in \B$ and $\udl{\category{C}} \in \cat_{\category{B}}$ of functors $f^* \internalfunc{\category{B}}(\udl{X}, \udl{\category{C}}) \simeq \internalfunc{\category{B'}}(f^* \udl{X}, f^* \udl{\category{C}})$. Moreover, if $\udl{\category{C}}$ were a symmetric monoidal $\B$--category, then this equivalence naturally upgrades to a symmetric monoidal one.
\end{lem}
\begin{proof}
    The proof is similar to \cref{lem:pushforwar_fun_manoeuvre}.
    As $f$ is \'etale, it is equivalent to a functor of the form $(\pi_{\obj{Y}})^* \colon \category{B} \rightleftharpoons \category{B}_{/\obj{Y}} \cocolon (\pi_{\obj{Y}})_*$ for some $\obj{Y} \in \category{B}$.
    Observe that there is a commutative diagram
    \begin{equation}\label{diag:fun_etale_base_change_manouver}
    \begin{tikzcd}
        \cat_{\category{B}} \ar[r, "(\pi_{\obj{X}})^*"] \ar[d, "(\pi_{\obj{Y}})^*"]
        & \cat_{\category{B}_{/\obj{X}}} \ar[r, "(\pi_{\obj{X}})_*"] \ar[d, "(\pi_{\obj{Y}})^*"]
        & \cat_{\category{B}} \ar[d, "(\pi_{\obj{Y}})^*"]
        \\
        \cat_{\category{B}_{/\obj{Y}}} \ar[r, "(\pi_{\obj{X}})^*"']
        & \cat_{\category{B}_{/\obj{X \times Y}}} \ar[r, "(\pi_{\obj{X}})_*"']
        & \cat_{\category{B}_{/\obj{Y}}}
    \end{tikzcd}
    \end{equation}
    coming from the commutative diagram of topoi
    \begin{equation*}
    \begin{tikzcd}
        \category{B} \ar[r, "(\pi_{\obj{X}})^*"] \ar[d, "(\pi_{\obj{Y}})^*"]
        & \category{B}_{/\obj{X}} \ar[r, "(\pi_{\obj{X}})_*"] \ar[d, "(\pi_{\obj{Y}})^*"]
        & \category{B} \ar[d, "(\pi_{\obj{Y}})^*"]
        \\
        \category{B}_{\obj{Y}} \ar[r, "(\pi_{\obj{X}})^*"']
        & \category{B}_{/\obj{X \times Y}} \ar[r, "(\pi_{\obj{X}})_*"']
        & \category{B}_{/\obj{Y}}.
    \end{tikzcd}
    \end{equation*}
    The left square here obviously commutes and commutativity of the right square is easy to check after passing to left adjoints.
    Now the top right composite sends $\udl{\category{C}}$ to $f^* \internalfunc{\category{B}}(\udl{X}, \udl{\category{C}})$ while the bottom left composite sends it to $\internalfunc{\category{B'}}(f^*\udl{X}, f^*\udl{\category{C}})$.    For the statement about symmetric monoidality, again note that all functors in \cref{diag:fun_etale_base_change_manouver} are product preserving and use the same argument as in the proof of \cref{lem:pushforwar_fun_manoeuvre}.
\end{proof}

\begin{lem}[Pushforward of parametrised (co)limits]\label{lem:parametrised_colimits_base_change}
    Let $f^* \colon \category{B} \rightleftharpoons \category{B}' \cocolon f_*$ be a geometric morphism of topoi.
    Consider $X \in \category{B}$ and a $\category{B}'$--category $\udl{\category{C}}$ which admits $f^*\udl{X}$-shaped limits and colimits.
    Then $f_* \udl{\category{C}}$ admits $\udl{X}$-shaped limits and colimits.
    Furthermore, the equivalence from \cref{lem:pushforwar_fun_manoeuvre} induces an identification of adjoint triples 
    \begin{equation}
    \begin{tikzcd} \label{diag:base_change_twisted_ambidexterity_geometric_morphism}
        f_* \internalfunc{\category{B'}}(f^* \obj{X}, \udl{\category{C}}) \dar[equal] \rar[shift left = 4, "f_* \pointProjection_!"] \rar[shift right = 4, "f_* \pointProjection_*"]
        & f_* \udl{\category{C}} \dar[equal] \lar["f_* \pointProjection^*"']
        \\
        \internalfunc{\category{B}}(\obj{X}, f_* \udl{\category{C}}) \rar[shift left = 4, "\pointProjection_!"] \rar[shift right = 4, "\pointProjection_*"]
        & f_* \udl{\category{C}} \lar["\pointProjection^*"'] .
    \end{tikzcd}
    \end{equation}
\end{lem}
\begin{proof}
    First note that \cref{diag:base_change_twisted_ambidexterity_geometric_morphism} commutes with the leftwards pointing arrows.
    Since the functor $f_* \colon \cat_{\category{B}'} \to \cat_{\category{B}}$ preserves adjunctions (see e.g. \cite[Cor. 3.1.9.]{MartiniWolf2024}), it follows that $f_* \pointProjection_!$ and $f_* \pointProjection_*$ define left and right adjoints to $\pointProjection^*$ and \cref{diag:base_change_twisted_ambidexterity_geometric_morphism} also commutes with the rightwards pointing arrows.
\end{proof}

\begin{lem}[Pullback of parametrised (co)limits]\label{lem:parametrised_colimits_etale_base_change}
    Let $f^* \colon \category{B} \rightleftharpoons \category{B}' \cocolon f_*$ be an \'etale morphism of topoi.
    Consider $X \in \category{B}$ and a $\category{B}$--category $\udl{\category{C}}$ which admits $\udl{X}$-shaped limits and colimits.
    Then $f^* \udl{\category{C}}$ admits $f^*\udl{X}$-shaped limits and colimits.
    Furthermore, the equivalence from \cref{lem:restriction_fun_manoeuvre} induces an identification of adjoint triples 
    \begin{equation}
    \begin{tikzcd} \label{diag:base_change_twisted_ambidexterity_etale_morphism}
        f^* \internalfunc{\category{B}}(\obj{X}, \udl{\category{C}}) \dar[equal] \rar[shift left = 4, "f^* \pointProjection_!"] \rar[shift right = 4, "f^* \pointProjection_*"]
        & f^* \category{C} \dar[equal] \lar["f^* \pointProjection^*"']
        \\
        \internalfunc{\category{B'}}(f^*\obj{X}, f^* \udl{\category{C}}) \rar[shift left = 4, "\pointProjection_!"] \rar[shift right = 4, "\pointProjection_*"]
        & f^* \category{C} \lar["\pointProjection^*"'] .
    \end{tikzcd}
    \end{equation}
\end{lem}
\begin{proof}
    The proof is identical to \cref{lem:parametrised_colimits_base_change},  using \cref{lem:restriction_fun_manoeuvre} instead of \cref{lem:pushforwar_fun_manoeuvre}.
\end{proof}

\begin{lem}\label{lem:natural_equivalences_preserved_under_geometric_morphisms}
    Let $f^*\colon \B\rightleftharpoons \B' \cocolon f_*$ be a geometric morphism of topoi. Let $\udl{J}\in\cat_{\B}$ and $\udl{\sC},\udl{\D}\in\cat_{\B'}$, and let $\alpha\colon \udl{\func}(f^*\udl{J},\udl{\sC})\times \constant_{\B'}\Delta^1\rightarrow\udl{\D}$ be a natural transformation and $f_*\alpha\circ (\id\times\eta)\colon \udl{\func}(\udl{J},f_*\udl{\sC})\times \constant_{\B}\Delta^1\rightarrow \udl{\func}(\udl{J},f_*\udl{\sC})\times f_*f^*\constant_{\B}\Delta^1\xrightarrow{f_*\alpha} f_*\udl{\D}$ the associated transformation.  Then $\alpha$ is a natural equivalence if and only if $f_*\alpha\circ(\id\times\eta)$ is a natural equivalence.
\end{lem}
\begin{proof}
    By using that $\Gamma_{\B}f_*\simeq \Gamma_{\B'}$ from \cref{example:constant_and_global_sections}, the two putatively equivalent statements are equivalent to the condition that the natural transformation $\func_{\B'}(f^*\udl{J},\udl{\sC})\times \Delta^1\rightarrow \Gamma_{\B'}\udl{\D}$ of unparametrised categories is a natural equivalence. 
\end{proof}

\begin{nota}
    Recall that there is the \textit{Picard space functor} $\picardSpace\colon \cmonoid(\cat)\rightarrow\cgroup(\spc)$ which takes as input a symmetric monoidal category and outputs a space of invertible objects. This functor is right adjoint to the inclusion $\cgroup(\spc)\hookrightarrow \cmonoid(\spc)\hookrightarrow \cmonoid(\cat)$ and is corepresented as $\picardSpace(-)\simeq \map_{\cmonoid(\cat)}(\loops\sphere,-)$.
\end{nota}

\begin{lem}\label{lem:picard_of_geometric_morphisms}
    Let $f^*\colon \B\rightleftharpoons \B' \cocolon f_*$ be a geometric morphism. Then we have an equivalence of functors $\udl{\picardSpace}(f_*-)\simeq f_*\udl{\picardSpace}(-)\colon \cmonoid(\cat_{\B'})\rightarrow \cgroup(\B)$.
    If $f^*\dashv f_*$ were moreover \'{e}tale, then we also have an equivalence $f^*\udl{\picardSpace}(-)\simeq \udl{\picardSpace}(f^*-)\colon \cmonoid(\cat_{\B})\rightarrow \cgroup(\B')$.
\end{lem}
\begin{proof}
    The first part is an immediate consequence of the fact that the diagram of left adjoints 
    \begin{center}
        \begin{tikzcd}
            \cmonoid(\cat_{\B})\dar["f^*"'] & \cgroup(\B)\lar[hook]\dar["f^*"]\\
            \cmonoid(\cat_{\B'}) & \cgroup(\B')\lar[hook]
        \end{tikzcd}
    \end{center}
    commutes, which is clear. For the second part, we note that $\udl{\picardSpace}(-)\colon \cmonoid(\cat_{\B})\rightarrow\cgroup(\B)$ is given by $\myuline{\map}_{\cat_{\B}}(\constant_{\B}\loops\sphere, -)$. Thus, since $f^*\dashv f_*$ was \'{e}tale, we get that $f^*\udl{\picardSpace}(-)\simeq f^*\myuline{\map}_{\cat_{\B}}(\constant_{\B}\loops\sphere, -)\simeq \myuline{\map}_{\cat_{\B'}}(\constant_{\B'}\loops\sphere, f^*-)\simeq \udl{\picardSpace}(f^*-)$.
\end{proof}

\begin{cor}\label{cor:factoring_through_picard_geometric_morphisms}
    Let $f^*\colon \B\rightleftharpoons \B' \cocolon f_*$ be a geometric morphism of topoi, $\udl{X}\in\B$,  $\udl{\D}\in\cmonoid(\cat_{\B'})$, and $\udl{\E}\in\cmonoid(\cat_{\B})$.
    \begin{enumerate}[label=(\arabic*)]
        \item A functor $\udl{X}\rightarrow f_*\udl{\D}$ has the property that it factors through $\udl{\picardSpace}(f_*\udl{\D})\hookrightarrow f_*\udl{\D}$ if and only if the associated functor $f^*\udl{X}\rightarrow \udl{\D}$ factors through $\udl{\picardSpace}(\udl{\D})\hookrightarrow \udl{\D}$,
        \item Suppose $f^*\dashv f_*$ is moreover \'{e}tale. If a functor $\udl{X}\rightarrow \udl{\E}$ factors through $\udl{\picardSpace}(\udl{\E})$, then $f^*\udl{X}\rightarrow f^*\udl{\E}$ factors through $\udl{\picardSpace}(f^*\udl{\E})$.
    \end{enumerate}
\end{cor}
\begin{proof}
    Part (1) is an immediate consequence of the equivalence \[\map_{\cat_{\B}}(\constant_{\B}\ast,f_*\udl{\func}(f^*\udl{X},\udl{\picardSpace}(\udl{\D})))\simeq \map_{\cat_{\B'}}(\constant_{\B'}\ast,\udl{\func}(f^*\udl{X},\udl{\picardSpace}(\udl{\D})))\] and the computation
    \[\udl{\func}(\udl{X},\udl{\picardSpace}(f_*\udl{\D}))\simeq \udl{\func}(\udl{X},f_*\udl{\picardSpace}(\udl{\D}))\simeq f_*\udl{\func}(f^*\udl{X},\udl{\picardSpace}(\udl{\D}))\]
    where the first equivalence is by \cref{lem:picard_of_geometric_morphisms} and the second by \cref{lem:pushforwar_fun_manoeuvre}. For part (2), if we have a factorisation $\udl{X}\rightarrow \udl{\picardSpace}(\udl{\E})\hookrightarrow \udl{\E}$, then applying $f^*$ to this and using the second part of \cref{lem:picard_of_geometric_morphisms} gives the required factorisation.
\end{proof}

\begin{prop}\label{prop:projection_formula_geometric_pushforwards}
    Let $f^*\colon \B\rightleftharpoons \B' \cocolon f_*$ be a geometric morphism of topoi, $\udl{X}\in \B$,  $\udl{\D}\in \cmonoid(\cat_{\B'})$, and $\udl{\E}\in\cmonoid(\cat_{\B})$.
    \begin{enumerate}[label=(\arabic*)]
        \item The symmetric monoidal $\B'$--category $\udl{\D}$ satisfies the $f^*\udl{X}$--projection formula if and only if the symmetric monoidal $\B$--category $f_*\udl{\D}$ satisfies the $\udl{X}$--projection formula,
        \item If $f_*$ is fully faithful, then the colimit $\uniqueMapX_!\colon \udl{\func}(\udl{X},f_*\udl{\D})\rightarrow f_*\udl{\D}$ (resp. limit $\uniqueMapX_*$) exists if and only if the colimit $(f^*\uniqueMapX)_!\colon \udl{\func}(f^*\udl{X},\udl{\D})\rightarrow \udl{\D}$ (resp. limit $(f^*\uniqueMapX)_*$) does,
        \item If $f^*\dashv f_*$ is moreover \'{e}tale, then if $\udl{\E}$ satisfies the $\udl{X}$--projection formula, then $f^*\udl{\E}$ satisfies the $f^*\udl{X}$--projection formula.
    \end{enumerate}
\end{prop}
\begin{proof}
    For (1), by the symmetric monoidal identification \cref{lem:pushforwar_fun_manoeuvre} and the identification of adjunctions \cref{lem:parametrised_colimits_base_change}, we see that for a fixed $A\in\udl{\D}$, applying $f_*$ to the  projection formula transformation on the left in
    \begin{center}
        \begin{tikzcd}
            \udl{\func}(f^*\udl{X},\udl{\D})\ar[rrr,phantom, "\Downarrow \mathrm{PF}"] \ar[rrr, bend left = 30, "(f^*\uniqueMapX)_!(-\otimes (f^*\uniqueMapX)^*A)"{xshift=0pt}]\ar[rrr, bend right = 30, "(f^*\uniqueMapX)_!(-)\otimes A"']&&& \udl{\D} & \udl{\func}(\udl{X},f_*\udl{\D})\ar[rrr,phantom, "\Downarrow \mathrm{PF}"] \ar[rrr, bend left = 28, "\uniqueMapX_!(-\otimes \uniqueMapX^*A)"{xshift=0pt}]\ar[rrr, bend right = 28, "\uniqueMapX_!(-)\otimes f_*A"']&&& f_*\udl{\D}
        \end{tikzcd}
    \end{center}
    yields the projection formula transformation on the right. Thus, by \cref{lem:natural_equivalences_preserved_under_geometric_morphisms}, we see that the left projection formula transformation  is an equivalence if and only if the right one is.

    For (2), that the existence of $(f^*\uniqueMapX)_!$ implies the existence of $\uniqueMapX_!$ is by \cref{lem:parametrised_colimits_base_change}. For the converse, we use again the diagram \cref{lem:parametrised_colimits_base_change} together with the fact $f^*$ preserves adjunctions by \cite[Cor. 3.1.9]{MartiniWolf2024} and that $f^*f_*\simeq \id$ by fully faithfulness.

    Part (3) is proved similarly as in (1), but using \cref{lem:restriction_fun_manoeuvre} and \cref{lem:parametrised_colimits_etale_base_change} instead.
\end{proof}

\subsubsection*{Presentability}

Presentable categories are useful for many reasons, among them being that they have all (co)limits, fulfill the adjoint functor theorem, and have a symmetric monoidal structure coming from the Lurie tensor product. Presentability in the parametrised context was first studied in \cite{Nardin2017Thesis} and later on in \cite{kaifPresentable}. Subsequently,  Martini-Wolf  \cite{MartiniWolf2022Presentable} introduced and developed a much more general theory for $
\category{B}$--categories, and this is the theory that we will use. 
Recall that presentable categories are usually large categories. To talk about presentable $\category{B}$--categories, we define a \textit{large $\category{B}$--category} to be a sheaf of large categories on $\category{B}$, i.e. a limit preserving functor $\category{B}\op \rightarrow \largecat$. The very large category of large $\category{B}$--categories will be denoted by $\largecat_{\category{B}}$.  

\begin{defn}\label{def:parametrised_presentablility}
    A $\category{B}$--category $\udl{\category{C}}$ is called \textit{fibrewise presentable} if the map $\udl{\category{C}} \colon \category{B}\op \to \largecat$ factors through $\presentable^L \subset \largecat$.
    Furthermore, $\udl{\category{C}}$ is called \textit{presentable} if it is fibrewise presentable and the following conditions hold:
    \begin{enumerate}
        \item For any map $f \colon X \to Y$ in $\category{B}$ the map $f^* \colon \category{C}(\udl{Y}) \to \category{C}(\udl{X})$ admits a left adjoint $f_!$.
        \item For any pullback square
        \begin{equation}\label{diag:base_change_square}
        \begin{tikzcd}
            \udl{X'} \ar[r, "g'"] \ar[d, "f'"]
            & \udl{X} \ar[d, "f"] 
            \\
            \udl{Y'} \ar[r, "g"]
            & \udl{Y}
        \end{tikzcd}
        \end{equation}
        in $\category{B}$ the Beck--Chevalley transformation $ f'_! (g')^* \to g^* f_!$ between functors $\category{C}(\udl{X}) \to \category{C}(\udl{Y'})$ is an equivalence.
        \end{enumerate}
\end{defn}

The above definition was chosen because it is easy to state, but there are many useful ways to characterise presentable $\category{B}$--categories, see \cite[Thm. 6.2.4]{MartiniWolf2022Presentable}.

\begin{defn}\label{def:parametrised_colimit_preserving}
    A map $F \colon \udl{\category{C}} \to \udl{\category{D}}$ between presentable $\category{B}$--categories is said to \textit{preserve $\category{B}$-colimits} if it satisfies the following conditions:
    \begin{enumerate}
        \item For any object $\udl{X} \in \category{B}$ the map $F(\udl{X}) \colon \category{C}(\udl{X}) \to \category{D}(\udl{X})$ preserves colimits.
        \item For any map $f \colon \udl{X} \to \udl{Y}$ in $\category{B}$ the Beck--Chevalley transformation $f_! F(\udl{X}) \to F(\udl{Y}) f_!$ is an equivalence.
    \end{enumerate}
\end{defn}

If $\udl{\category{C}}$ is presentable, then $f^* \colon \category{C}(\udl{Y}) \to \category{C}(\udl{X})$ also admits a right adjoint $f_*$.
By passing to right adjoints one can see that for any pull back square \cref{diag:base_change_square} the Beck--Chevalley transformation $g^* f_* \to (f')_* (g')^*$ is an equivalence, see e.g. \cite[Observation 1.6.2]{Haine2022}.

\begin{defn}
    Denote by $\presentable^L_{\category{B}}$ the (nonfull) subcategory of $\largecat_{\category{B}}$ of presentable $\category{B}$--categories and $\category{B}$-colimit preserving functors.
    We write $\udl{\func}^L_{\category{B}}(\category{C}, \category{D})$ for the full $\category{B}$-subcategory of $\internalfunc{\category{B}}(\category{C}, \category{D})$ of colimit preserving functors.
    The subcategories $\presentable^L_{\category{B}_{/X}} \subset \largecat_{\category{B}_{/X}}$ assemble into the $\category{B}$--category $\udl{\presentable}^L_{\category{B}} \subset \udl{\largecat}_{\category{B}}$ of presentable $\category{B}$--categories.
\end{defn}

The $\category{B}$--category $\udl{\presentable}^L_{\category{B}}$ admits all $\category{B}$-limits and $\category{B}$-colimits \cite[Cor. 6.4.11.]{MartiniWolf2022Presentable}. Moreover, for two presentable $\category{B}$--categories $\udlcatC$ and $\udlcatD$, their functor category $\internalfunc{\category{B}}^L(\udlcatC,\udlcatD)$ is presentable.

\begin{cons}[Tensor product of presentable categories]
    Given two presentable $\category{B}$--categories $\udl{\category{C}}$ and $\udl{\category{D}}$, their tensor product $\udl{\category{C}} \otimes \udl{\category{D}}$ is a presentable $\category{B}$--category together with a functor $\udl{\category{C}} \times \udl{\category{D}} \to \udl{\category{C}} \otimes \udl{\category{D}}$ which preserves colimits in each variable such that precomposition along it induces an equivalence $\func_{\category{B}}^L(\udl{\category{C}} \otimes \udl{\category{D}}, \udl{\category{E}}) \xrightarrow{\simeq} \func_{\category{B}}^L(\udl{\category{C}}, \internalfunc{\category{B}}^L(\udl{\category{D}}, \udl{\category{E}}))$ for any presentable $\category{B}$--category $\udl{\category{E}}$. This equips $\presentable_{\category{B}}^L$ with the structure of a closed symmetric monoidal category.   It can even be extended to a symmetric monoidal structure on the $\category{B}$--category $\udl{\presentable}_{\category{B}}^L$, see \cite[Proposition 8.2.9]{MartiniWolf2022Presentable}.
    Furthermore, the tensor product $\-- \otimes \-- \colon \udl{\presentable}_{\category{B}}^L \times \udl{\presentable}_{\category{B}}^L \to \udl{\presentable}_{\category{B}}^L$ preserves $\category{B}$-colimits in each variable. 
\end{cons}

\begin{defn}
    A \textit{presentably symmetric monoidal $\category{B}$--category} is a commutative algebra $\udl{\category{C}} \in \calg(\presentable^L_{\category{B}})$.
    Explicitly, this means that $\udl{\category{C}}$ is a symmetric monoidal $\udl{B}$--category which is presentable such that the tensor product $\-- \otimes \-- \colon \udl{\category{C}} \times \udl{\category{C}} \to \udl{\category{C}}$ preserves $\category{B}$-colimits in both variables.
    The latter condition means that:
    \begin{enumerate}
        \item For all $\obj{X} \in \category{B}$ the tensor product $\category{C}(\obj{X}) \times \category{C}(\obj{X}) \to \category{C}(\obj{X})$ preserves colimits.
        \item For all maps $f \colon \obj{X} \to \obj{Y}$ in $\category{B}$ and all $A \in \category{C}(\obj{X})$ and $B \in \category{C}(\obj{Y})$ the Beck--Chevalley transformation $f_!(A \otimes f^* B) \to f_! A \otimes B$ is an equivalence.
    \end{enumerate}
\end{defn}

To construct examples of presentable $\category{B}$--categories, the following proposition is useful.

\begin{prop}[{\cite[Section 8.3]{MartiniWolf2022Presentable}}]\label{cons:embedding_presentable_into_parametrised_presentable_categories}
    There is a symmetric monoidal colimit preserving fully faithful functor $\-- \otimes_{\category{B}} \Omega \colon \module_{\category{B}}(\presentable^L) \hookrightarrow \presentable^L_{\category{B}}$
    whose right adjoint $\Gamma^{\mathrm{lin}}$ refines the global sections functor $\Gamma \colon \presentable^L_{\category{B}} \to \presentable^L$.
\end{prop}

\begin{lem}[Presentablility and basechange]\label{lem:lax_monoidal_base_change_geometric_morphism}
    Let $f^* \colon \category{B} \rightleftharpoons \category{B}' \cocolon f_*$ be a geometric morphism.
    Then the functor $f_* \colon \largecat_{\category{B}'} \to \largecat_{\category{B}}$ restricts to a functor $f_* \colon \presentable^L_{\category{B}'} \to \presentable^L_{\category{B}}$.
    It admits a unique lax symmetric monoidal refinement $f_*^\otimes \colon \presentable^{L, \otimes}_{\category{B}'} \to \presentable^{L, \otimes}_{\category{B}}$ lifting the lax symmetric monoidal functor $\presentable^{L, \otimes}_{\category{B}'} \to \largecat_{\category{B}'}^\times \xrightarrow{f_*} \largecat_{\category{B}}^\times$ along the lax symmetric monoidal functor $\presentable^{L, \otimes}_{\category{B}} \to \largecat_{\category{B}}^\times$.
\end{lem}
\begin{proof}
    Recall that $f_* \colon \largecat_{\category{B}'} \to \largecat_{\category{B}}$ is given by precomposition with $f^* \colon \category{B}\op \to (\category{B}')\op$. Hence, the first condition in \cref{def:parametrised_presentablility} is immediate. The second condition follows from the fact that $f^*\colon \category{B}\op \to (\category{B}')\op$ also preserves finite limits, and in particular pullbacks.

    For the statement about lax symmetric monoidality, we will freely use the terminologies from \cite{MartiniWolf2022Presentable}. First observe that as $f_* \colon \largecat_{\category{B}'} \to \largecat_{\category{B}}$ preserves products, it is symmetric monoidal with respect to the cartesian symmetric monoidal structure on both sides.
    Recall from \cite[Section 8.2]{MartiniWolf2022Presentable} that $\presentable^{L, \otimes}_{\category{B}'} \looparrowright \largecat_{\category{B}'}^\times$ is the subcategory generated by presentable $\category{B}'$--categories and locally multilinear functors.
    We know from the first part that $f_* \colon \largecat_{\category{B}'} \to \largecat_{\category{B}}$ preserves presentable categories and multilinear functors between those.
    As $f^* \colon \largecat_{\category{B}} \to \largecat_{\category{B}'}$ preserves colimits and finite limits it preserves effective epimorphisms.
    From this it follows that $f_*$ also preserves locally multilinear functors.
\end{proof}

\begin{lem}[Presentability and fully faithful basechange]\label{lem:base_change_presentable_fully_faithful_geometric_morphism}
    Let $f^* \colon \category{B} \rightleftharpoons \category{B}' \cocolon f_*$ be a geometric morphism of topoi and assume that $f_*$ is fully faithful.
    Then the square
    \begin{equation*}
    \begin{tikzcd}
        \presentable^L_{\category{B}'} \ar[r, "f_*"] \ar[d, loop->]
        & \presentable^L_{\category{B}} \ar[d, loop->]
        \\ \largecat_{\category{B}'} \ar[r, "f_*", hook]
        & \largecat_{\category{B}}
    \end{tikzcd}
    \end{equation*}
    is cartesian.
    In particular, $f_* \colon \presentable^L_{\category{B}'} \to \presentable^L_{\category{B}}$ is fully faithful.

    If, in addition, the image of $f_*  \colon \presentable^L_{\category{B}'} \to \presentable^L_{\category{B}}$ is closed under $\otimes$, then the maps $f_* \udl{\category{C}} \otimes f_* \udl{\category{D}} \to f_* (\udl{\category{C}} \otimes \udl{\category{D}})$ coming from the lax symmetric monoidal structure are equivalences.
\end{lem}
\begin{proof}
    We have to show that a $\category{B}'$--category $\udl{\category{C}}$ is presentable if $f_* \udl{\category{C}}$ is presentable and similarly that a $\category{B}'$-functor $F \colon \udl{\category{C}} \to \udl{\category{D}}$ is $\category{B}'$-colimit preserving if $f_* F$ is $\category{B}$-colimit preserving.
    Notice that the counit map $f^* f_* \to \id$ is an equivalence as $f_*$ is fully faithful.
    This implies that for $\obj{X} \in \category{B}$ we have $f_* \udl{\category{C}}(f_*\obj{X}) = \category{C}(f^* f_* \obj{X}) \simeq \category{C}(\obj{X})$.
    The statements about presentablility and colimit preservation now directly follow from the definitions. 
    
    For the statement about the lax monoidal multiplication map, observe that for presentable $\category{B}'$--categories $\udl{\category{C}}, \udl{\category{D}}, \udl{\category{E}}$, the functor $f_*$ induces an equivalence between $\category{B}'$-multilinear functors $\udl{\category{C}} \times \udl{\category{D}} \to \udl{\category{E}}$ and $\category{B}$-multilinear functors $f_*\udl{\category{C}} \times f_*\udl{\category{D}} \to f_*\udl{\category{E}}$:
    It is clear that if $g$ is multilinear, then $f_* g$ is multilinear while the converse follows from essential surjectivity of $f^*$.
    In particular, precomposition along the map $f_* \udl{\category{C}} \otimes f_* \udl{\category{D}} \to f_* (\udl{\category{C}} \otimes \udl{\category{D}})$ induces an equivalence
    \begin{equation*}
        \func^L(f_* (\udl{\category{C}} \otimes \udl{\category{D}}), f_* \udl{\category{E}}) \xrightarrow{\simeq } 
        \func^L(f_* \udl{\category{C}} \otimes f_*\udl{\category{D}}, f_* \udl{\category{E}})
    \end{equation*}
    from which the claim follows.
\end{proof}

\begin{lem}[Presentability and \'etale basechange]\label{lem:lax_monoidal_base_change_etale_morphism}
    For $X\in\category{B}$, the basechange adjunction $\pi_X^* \colon \largecat_{\category{B}} \rightleftharpoons \largecat_{\category{B}_{/X}} \cocolon (\pi_X)_*$ along the \'etale geometric morphism $\pi_X^* \colon \category{B} \rightleftharpoons \category{B}_{/X} \cocolon (\pi_X)_*$ restricts to an adjunction $\pi_X^* \colon \presentable^L_{\category{B}} \rightleftharpoons \presentable^L_{\category{B}_{/X}} \cocolon (\pi_X)_*$.
    The left adjoint $\pi_X^* \colon \presentable^L_{\category{B}} \to \presentable^L_{\category{B}_{/X}}$ admits a unique symmetric monoidal refinement which lifts the lax symmetric monoidal functor $\presentable^L_{\category{B}} \to \largecat_{\category{B}} \xrightarrow{\pi_X^*} \largecat_{\category{B}_{/X}}$.
\end{lem}
\begin{proof}
    That $\pi_X^* \colon \largecat_{\category{B}} \rightleftharpoons \largecat_{\category{B}_{/X}} \cocolon (\pi_X)_*$ restricts to an adjuction $\pi_X^* \colon \presentable^L_{\category{B}} \rightleftharpoons \presentable^L_{\category{B}_{/X}} \colon (\pi_X)_*$ is shown in \cite[Corollary 2.14]{Cnossen2023}.
    For the statement about symmetric monoidality, note that $(\pi_X)_* \colon \largecat_{\category{B}_{/X}} \to \largecat_{\category{B}}$ is product preserving as it admits the left adjoint $(\pi_X)_!$.
    In particular, we obtain the symmetric monoidal unit map $\udl{\presentable}^L_{\category{B}} \to (\pi_X)_* \pi_X^* \udl{\presentable}^L_{\category{B}}$ which on global sections gives the desired symmetric monoidal refinement of $(\pi_X)^* \colon \presentable^L_{\category{B}} \to \presentable^L_{\category{B}_{/X}}$.
\end{proof}

\subsubsection*{\texorpdfstring{$\udl{\category{C}}$}{C}-linear categories}

Here we recall some facts about $\udl{\category{C}}$-linear categories and the classification of $\udl{\category{C}}$-linear functors from \cite[Section 2.2]{Cnossen2023}.

\begin{defn}($\udl{\category{C}}$-linear categories)
    Consider a presentably symmetric monoidal category $\udl{\category{C}} \in \calg(\presentable^L_{\category{B}})$.
    A \textit{$\udl{\category{C}}$-linear category} is a left $\udl{\category{C}}$-module in $\presentable^L_{\category{B}}$. 
    The category of $\udlcatC$-linear categories and $\udlcatC$-linear functors is defined as the category $\module_{\udl{\category{C}}}(\presentable^L_{\category{B}})$.
\end{defn}

The categories $\module_{\pi_X^* \udl{\category{C}}}(\presentable^L_{\category{B}_{/X}})$ assemble into the $\category{B}$--category $\udl{\module}_{\udl{\category{C}}}(\udl{\presentable}^L_{\category{B}})$.
The relative tensor product from \cite[Proposition 7.2.7]{MartiniWolf2022Presentable} equips this with the structure of a symmetric monoidal $\category{B}$--category which is $\category{B}$-complete and $\category{B}$-cocomplete such that the tensor product is bilinear.
This symmetric monoidal structure on $\module_{\udl{\category{C}}}(\presentable^L_{\category{B}})$ is closed and we denote the internal mapping object by $\udl{\func}_{\udl{\category{C}}}(\--, \--)$.
As in \cite[Remark 3.4.3]{Martini2022Yoneda} it endows $\module_{\udl{\category{C}}}(\presentable^L_{\category{B}})$ with a 2-categorical structure.
This allows us to talk about internal adjunctions in $\module_{\udl{\category{C}}}(\presentable^L_{\category{B}})$. The following lemma gives convenient criteria for a $\udlcatC$-linear functor to be an internal left adjoint.

\begin{lem}[\cite{Cnossen2023}, Lem. 2.21]
    A $\udlcatC$-linear functor $F \colon \udlcatD \rightarrow \udlcatE$ is an internal left adjoint in $\module_{\udl{\category{C}}}(\presentable^L_{\category{B}})$ if and only if its right adjoint $G$ preserves fibrewise colimits and
         satisfies the projection formula, i.e. for each $\udl{X} \in \category{B}$ and $e \in \udlcatE(\udl{X})$ the map
        $\mathrm{PF}_* \colon c \otimes G(e) \rightarrow G(c \otimes e)$     is an equivalence.
\end{lem}

\begin{defn}[{Free and cofree categories, \cite[Definition 2.23]{Cnossen2023}}]
    For $\udl{\category{D}} \in \module_{\udl{\category{C}}}(\udl{\presentable}^L_{\category{B}})$ and an object $\obj{X} \in \category{B}$ we define the \textit{cofree $\udl{\category{C}}$-linear $\category{B}$--category on $\obj{X}$} by $\udl{\category{D}}^{\obj{X}} \coloneqq \lim_{\udl{X}} \udl{\category{D}}$    where the $\category{B}$-limit is formed over the constant diagram $\udl{X} \to \udl{\module}_{\udl{\category{C}}}(\udl{\presentable}^L_{\category{B}})$ with value $\udl{\category{D}}$.
\end{defn}
Note that after forgetting the $\udl{\category{C}}$-linear structure, $\udl{\category{D}}^{\obj{X}}$ is given by $\internalfunc{\category{B}}(\udl{X}, \udl{\category{D}})$ as the forgetful functor $\udl{\module}_{\udl{\category{C}}}(\udl{\presentable}^L_{\category{B}}) \to \udl{\cat}_{\category{B}}$ preserves $\udl{\category{B}}$-limits. If $\udl{\category{D}} \in \calg(\module_{\udl{\category{C}}}(\udl{\presentable}^L_{\category{B}}))$ is a $\udl{\category{C}}$-algebra, then $\udl{\category{D}}^{\obj{X}}$ has a canonical pointwise symmetric monoidal structure as the forgetful functor $\udl{\calg}(\udl{\module}_{\udl{\category{C}}}(\udl{\presentable}^L_{\category{B}})) \to \udl{\module}_{\udl{\category{C}}}(\udl{\presentable}^L_{\category{B}})$ preserves $\category{B}$-limits.

The following theorem is crucial for the development of Poincar\'e duality in a presentable context, even in classical terms. It should be viewed as a generalisation of the fact that for a small category for a space $X$ there is an equivalence $\func^L(\spc^X,\spc) \simeq \spc^X$, that is sometimes referred to as Morita theory.

\begin{thm}[{Classification of $\udl{\category{C}}$-linear functors, \cite[Theorem 2.32]{Cnossen2023}}]\label{thm:classification_of_C_linear_functors}
    Let $\udl{\category{C}} \in \calg(\presentable^L_{\category{B}})$ and $\obj{X} \in \udl{\category{C}}$. Then there is an equivalence of $\udl{\category{C}}$-linear $\category{B}$--categories
    \begin{equation*}
        \udl{\category{C}}^{\udl{X}} \to \udl{\func}_{\udl{\category{C}}}(\udl{\category{C}}^{\udl{X}}, \udl{\category{C}}),\hspace{3mm}  \zeta \mapsto r_! (\-- \otimes \zeta).
    \end{equation*}
\end{thm}
Here, the map in the statement of the theorem is adjoint to the map
$\udl{\category{C}}^{\udl{X}} \otimes \udl{\category{C}}^{\udl{X}} \to \udl{\category{C}}^{\udl{X}} \xrightarrow{r_!} \udl{\category{C}}$.

\begin{prop}[Basechange of module categories, {\cite[Prop. 7.2.7]{MartiniWolf2022Presentable}}]\label{prop:base_change_module_categories}
    Suppose that $f \colon \udl{\category{C}} \to \udl{\category{D}}$ is a map in $\calg(\presentable^L_{\category{B}})$.
    Then the restriction functor $f^* \colon \udl{\module}_{\udl{\category{D}}}(\presentable^L_{\category{B}}) \to \udl{\module}_{\udl{\category{C}}}(\presentable^L_{\category{B}})$
    admits a symmetric monoidal left adjoint $\-- \otimes_{\udl{\category{C}}} \udl{\category{D}} \colon \udl{\module}_{\udl{\category{C}}}(\presentable^L(\category{B})) \to \udl{\module}_{\udl{\category{D}}}(\presentable^L(\category{B}))$
\end{prop}
\begin{proof}
    Apply \cite[Proposition 7.2.7]{MartiniWolf2022Presentable} to the case $R = \udl{\category{D}} \in \calg(\udl{\module}_{\udl{\category{C}}}(\presentable^L_{\category{B}}))$.
\end{proof}

\subsection{Equivariant categories and the theory of families}\label{subsection:equivariant_categories_and_families}

\subsubsection*{Change of group functors}
We now specialise the previous general theory to the equivariant setting for a compact Lie group $G$. We set $\spc_G \coloneqq\presheaf(\orbit(G))$, the \textit{category of $G$--spaces}, where $\orbit(G)$ is the \textit{orbit category of $G$}. This is a topos, and we write $\cat_G \coloneqq \cat_{\spc_G} \simeq \func(\orbit(G)\op, \cat)$ for the category of $G$--categories.
The value of a $G$-category $\udlcatC$ at an orbit $G/H$ will be denoted by $\sC(\myuline{G/H})$ or $\category{C}^H$.

\begin{recollect}\label{rec:orbit_category}
    Recall that the category of locally compact Hausdorff topological $G$--spaces is enriched over topological spaces by employing the compact--open topology on morphism sets. 
    The full subcategory on the homogenous $G$--spaces, that is Hausdorff spaces with a transitive $G$-action, is equivalent to the full subcategory spanned by the orbits $G/H$, where $H \leq G$ is a closed subgroup. By $\orbit(G)$ we denote the associated ($\infty$-) category which we call the \textit{orbit category of $G$}.
    
    Later we will need the following standard facts:
    For any morphism $\alpha \colon H \to G$ of compact Lie groups, there is an induction functor
    \begin{equation*}
        \ind_\alpha^{\orbit} \colon \orbit(H) \longrightarrow \orbit(G), \hspace{3mm} S \mapsto G \times_H S.
    \end{equation*}
    If $\alpha \colon H \longrightarrow G$ is an epimorphism, this admits the restriction functor
    \begin{equation*}
        \res^{\orbit}_\alpha \colon \orbit(G) \to \orbit(H)
    \end{equation*}
    as a fully faithful right adjoint.
    Both functors and the adjunction can be constructed on the level of topological categories.
    For a closed subgroup $H \leq G$, induction induces an equivalence of categories $\orbit(H) \xlongrightarrow{\simeq} \orbit(G)_{/(G/H)}$ whose inverse sends $T \rightarrow G/H$ to the homogeneous $H$-space given as the fibre over $eH \in G/H$.

    More information on orbit categories of compact Lie groups can be found in \cite[Sec. 6]{LinskensNardinPol2022global} or \cite[Chapters I.3 and I.4]{bredonTrans}.
\end{recollect}

\begin{nota}[Restrictions, (co)inductions, and (co)inflations]
\label{cons:res_ind_coind}
    Consider a continuous homorphism $\alpha \colon K \to G$ of compact Lie groups.
    We obtain the two adjunctions
    \begin{equation*}
    \begin{tikzcd}[row sep = large, column sep = large]
        \cat_H
        \ar[r, bend right = 40, "\coind_\alpha"']
        \ar[r, bend left = 40, "\ind_\alpha"]
        & \cat_G, \ar[l, "\res_\alpha" description]
    \end{tikzcd}
    \end{equation*}
    called the \textit{induction, restriction}, and \textit{coinduction} functors, respectively. Here, $\res_{\alpha}$ is given by restriction along  $\ind_\alpha^{\orbit} \colon \orbit(K) \to \orbit(G)$ and $\ind_{\alpha}$ and $\coind_{\alpha}$ are given by left and right Kan extensions. The functors $\res_{\alpha}$ and $\coind_{\alpha}$ restrict to a geometric morphism of topoi $\res_{\alpha} \colon \spc_G \rightleftharpoons \spc_K \cocolon \coind_{\alpha}$. The two main classes of examples are: 
    \begin{enumerate}[label=(\alph*)]
        \item If $\alpha$ were an injection $\groupInjection\colon H \rightarrowtail G$, then the geometric morphism $\res_{\groupInjection}\colon\spc_G\rightleftharpoons \spc_H \cocolon \coind_{\groupInjection}$ is \'{e}tale.    We will often also write $\ind_{\groupInjection}, \res_{\groupInjection}, $ and $\coind_{\groupInjection}$ as $\ind^G_H$, $\res^G_H$, and $\coind^G_H$ respectively;
        
        \item If $\alpha$ were an epimorphism $\groupSurjection\colon G\twoheadrightarrow G/N =: Q$ (so that $N\leq G$ is a closed normal subgroup), then $\coind_{\groupSurjection}$ admits a further right adjoint which we write as $\rcoind_{\groupSurjection}$ given by right Kan extension along the  fully faithful right adjoint $\res^{\orbit}_{\groupSurjection}$ to $\ind^{\orbit}_{\groupSurjection}$. In particular, $\res_{\groupSurjection} = (\ind^{\orbit}_{\groupSurjection})^*\simeq  (\res^{\orbit}_{\groupSurjection})_!$ and $\rcoind_{\groupSurjection} = (\res^{\orbit}_{\groupSurjection})_*$ are fully faithful in this epimorphic case. Note also that in this case, $\coind_{\groupSurjection}\simeq (\res^{\orbit}_{\groupSurjection})^*$, i.e. $\coind_{\groupSurjection}$ may be computed by restricting along $\res^{\orbit}_{\groupSurjection}\colon \orbit(G)\rightarrow \orbit(K)$. We summarise in the following diagram the special notations we will also use in the epimorphic case as follows:
    \begin{center}
        \begin{tikzcd}
        \cat_Q \ar[rrrr, bend right =50, "\rcoind_{\groupSurjection} "'description, hook]
        \ar[rrrr, bend right = 25, "(-)^N\coloneqq \coind_{\groupSurjection}"'description,leftarrow, shift left = 1]
        \ar[rrrr, bend left = 25, "N\backslash(-)\coloneqq\ind_{\groupSurjection}"description,leftarrow, shift left = 1] \ar[rrrr, "\inflated_{\groupSurjection}\coloneqq\res_{\groupSurjection}" description ,hook, shift left = 1]
        &&&& \cat_G, 
        \end{tikzcd}
    \end{center}
    The maps $N\backslash(-),\: \inflated_{\groupSurjection},\: (-)^N, $ and $\rcoind_{\groupSurjection}$ are called the \textit{genuine quotient, inflation, genuine fixed points, } and \textit{coinflation}, respectively. We often also write $\inflated_{\groupSurjection}$ and $\rcoind_{\groupSurjection}$ as $\inflated^Q_G$ and $\rcoind^Q_G$ respectively.
    \end{enumerate}
\end{nota}

\begin{rmk}
    From the left Kan extension formula defining the genuine quotient, we obtain
    \begin{equation}
        \label{eq:formula_for_genuine_quotient}
        (N \backslash \udlcatC) (\myuline{Q/H}) \simeq \colim_{G/K,\hspace{1mm} \myuline{Q/H} \rightarrow N\backslash(\myuline{G/K}) } \category{C}(\myuline{G/K}).
    \end{equation}
\end{rmk}

\subsubsection*{Stability}

In the parametrised theory, the theory of stability is more subtle than in the nonparametrised setting.
The most naive version is the following, equivalent characterisations of which can be found in \cite[Section 7.3]{MartiniWolf2022Presentable}.
\begin{defn}
    A $\category{B}$--category $\udl{\category{C}}$ is called fibrewise pointed (resp. stable) if the functor $\udl{\category{C}} \colon \category{B}\op \to \cat$ factors through the subcategory $\cat_* \subset \cat$ of pointed categories and pointed functors (resp. $\cat^{\stable} \subset \cat$ of stable categories and exact functors).
    We denote by $\catptd{\category{B}}$ (resp. $\catst{\category{B}}$) the category of fibrewise pointed (resp. stable) categories and pointed (resp. exact) functors.
\end{defn}

As a parametrised analogue of the category of spectra, there is the $G$-category of $G$-spectra $\myuline{\spectra}_G$ whose value at an orbit $G/H$ is given by the category $\myuline{\spectra}_G(\myuline{G/H}) = \spectra_H$ of genuine $H$-spectra together with restriction maps between them, see \cref{def:G_spectra} for a definition.
If $G$ is clear form the context, we will also just write $\myuline{\spectra}$ for $\myuline{\spectra}_G$.
In addition to being fibrewise stable, $\myuline{\spectra}_G$ satsifies some form of the Wirthmüller isomorphism in the sense that indexed products and coproducts over orbits $G/H$ are canonically equivalent.
This was used by Nardin in \cite{Nardin2017Thesis} to define the notion of $G$-stability for finite groups $G$.
We will not recall the definition here and refer the interested reader to \cite{Nardin2017Thesis} or \cite[Section 4.1]{kaifPresentable} for an exposition of this theory.
In \cref{sec:G_stability_presentable_categories} we generalise this to define $G$-stability for presentable $G$-categories.
This will be sufficient for our purposes.

\begin{nota}
    For a compact Lie group $G$, we denote by $\prGst{G} \subseteq \prst{G} \subseteq \presentable^L_G$ the full subcategories on $G$-stable and fibrewise stable presentable $G$-categories.
    For a finite group $G$, we also denote by $\catGst{G} \subset \cat_G$ the subcategory of $G$-stable $G$-categories and $G$-exact functors.
\end{nota}

Now we study the behaviour of $G$-stability with respect to  standard equivariant operations.

\begin{lem}[Coinduction and stability]\label{lem:lax_monoidal_structure_coind}
    Let $\alpha \colon K \to G$ be a continous group homomorphism of compact Lie groups.
    The lax symmetric monoidal functor $\coind_\alpha \colon \presentable^L_K \to \presentable^L_G$ from \cref{lem:lax_monoidal_base_change_geometric_morphism} restricts to a lax symmetric monoidal functor $\coind_\alpha \colon \prGst{K} \to \prGst{G}$
\end{lem}
\begin{proof}
    We apply \cref{prop:characterisation_G_stability} (3) to show that $\coind_\alpha$ sends $K$-stable to $G$-stable categories.
    Suppose that $\udl{\category{C}}$ is a $K$-stable category and $V$ is a finite dimensional $G$-representation.
    By \cref{lem:parametrised_colimits_base_change} we can identify the maps $S^V \otimes -$ and $\coind_\alpha(S^{\res_\alpha V} \otimes -)$ on $\coind_\alpha \udl{\category{C}}$.
    But as $\res_\alpha V$ is a finite dimensional $K$-representation, $K$-stability of $\udl{\category{C}}$ implies that the second map is invertible.
\end{proof}

\begin{lem}[Restriction and stability]\label{lem:restriction_presentable_categories}
    For an injective continuous homomorphism $\alpha \colon H \rightarrowtail G$ of compact Lie groups, the adjunction $\res_\alpha \colon \Pr^{L}_G \rightleftharpoons \Pr^{L}_H \cocolon \coind_\alpha$ from \cref{lem:lax_monoidal_base_change_etale_morphism} restricts to an adjunction $\res_\alpha \colon \prGst{G} \rightleftharpoons \prGst{H} \cocolon \coind_\alpha$ with symmetric monoidal left adjoint
\end{lem}
\begin{proof}
    By \cref{lem:lax_monoidal_structure_coind} $\coind_H^G$ restricts to a functor between equivariantly stable categories. To show that restriction $\res_H^G$ preserves equivariantly stable categories we employ \cref{prop:characterisation_G_stability}.
    Recall that, by the Peter-Weyl theorem, for any finite dimensional $H$-representation $W$ there is a finite dimensional $G$-representaiton $V$ such that $W$ is a summand of $\res_H^G V$.
    Now, if $\udl{\category{C}}$ is a $G$-stable category, then $S^{V} \otimes \--$ is invertible on $\udl{\category{C}}$.
    By \cref{lem:parametrised_colimits_etale_base_change}, we can identify the two maps $\res_H^G (S^{V} \otimes -)$ and $S^{\res_H^G V} \otimes -$ on $\res_H^G \udl{\category{C}}$.
    This shows that $S^{\res_H^G V}$ and thus also $S^W$ act invertibly on $\res_H^G \udl{\category{C}}$.
\end{proof}

\begin{lem}[Coinflation and stability]\label{lem:coinflation_stability}
    Let $\groupSurjection \colon G \twoheadrightarrow Q = G/N$ be a continuous epimorphism of compact Lie groups.
    Then the lax symmetric monoidal functor $\rcoind_{\groupSurjection} \colon \presentable^L_Q \to \presentable^L_G$ from \cref{lem:lax_monoidal_base_change_geometric_morphism} restricts to a lax symmetric monoidal functor $\rcoind_{\groupSurjection} \colon \prGst{Q} \to \prGst{G}$.    
\end{lem}
\begin{proof}
    We apply \cref{prop:characterisation_G_stability} to show that $\rcoind_\groupSurjection$ sends $Q$-stable to $G$-stable categories.
    Suppose that $\udl{\category{C}}$ is a $Q$-stable category and $V$ is a finite dimensional $G$-representation.
    By \cref{lem:parametrised_colimits_base_change} we can identify the maps $S^V \otimes -$ and $\rcoind_\groupSurjection(\coind_\groupSurjection S^{V} \otimes -)$ on $\rcoind_\groupSurjection \udl{\category{C}}$.
    Note that $\coind_\groupSurjection S^{V} \simeq S^{V^N}$ is the representation sphere of the finite dimensional $Q$-representation carrying the residual action.
    By $Q$-stability of $\udl{\category{C}}$, this map is invertible.
\end{proof}

\begin{cons}[Spectral restriction]\label{cons:spectral_restriction}
    Let $\alpha \colon K \rightarrow G$ be a homomorphism of compact Lie groups.
    \cref{lem:lax_monoidal_structure_coind}
    endows $\coind_\alpha \colon \Pr^{L, K-\stable}_K \to \Pr^{L, G-\stable}_G$ with a lax symmetric monoidal structure.
    This endows $\coind_\alpha \myuline{\spectra}_K$ with the structure of a commutative algebra in $\Pr^{L, G-\stable}_G$.
    In particular, using that $\myuline{\spectra}_G$ is the initial commutative algebra in $\presentable^{L,G-\stable}_G$, we obtain the symmetric monoidal $G$-colimit preserving functor $\res_\alpha \colon \myuline{\spectra}_G \to \coind_\alpha \myuline{\spectra}_K$ called the \textit{restriction map}.
    If $\alpha=\groupSurjection \colon G \twoheadrightarrow Q$ is an epimorphism, we also call $\res_\groupSurjection = \infl_\groupSurjection \colon \myuline{\spectra}_Q \to \coind_\groupSurjection \myuline{\spectra}_G$ the \textit{inflation map}.
\end{cons}

\subsubsection*{Categorical isotropy separation}

At various places in this article we will use isotropy separation arguments.
For this, we recall here some constructions on $G$--categories given a family $\family$ of subgroups of $G$. Recall that a family of subgroups of a compact Lie group $G$ is a collection of closed subgroups of $G$ which is closed under subgroups and conjugation. 

Note that conjugacy classes of subgroups of $G$ correspond bijectively to isomorphism classes of objects in $\orbit(G)$. Given any collection $S$ of closed subgroups of $G$ that is closed under conjugacy, we set $\orbit_S(G) \subset \orbit(G)$ to be the full subcategory on those $G/H$ with $H \in S$.
One important example of this is the collection $S = \family^c$ given by the collection of all subgroups which lie in the complent of a family $\family$. This never forms a family, except in the extreme cases of the empty family or the family of all subgroups.

\begin{example}[A family for quotients]
    \label{ex:normal_subgroup_family}
    Suppose that $ N \leq G$ is a closed normal subgroup of $G$. An interesting family is provided by $\Gamma_N \coloneqq \{ H \leq G \mid N \nleq H \}$. Then $\Gamma_N^c$ consists of those $H \leq G$ with $N \leq H$. 
    Let $\alpha \colon G \rightarrow G/N$ denote the quotient homomorphism. Observe that the adjunction 
    $\ind_\alpha^{\orbit} \dashv \res_\alpha^{\orbit}$ restricts to an equivalence of categories
    \[ \ind_\alpha^{\orbit} \colon  \orbit_{\Gamma_N^c}(G) \simeq \orbit(G/N) \colon \res_\alpha^{\orbit}. \]
\end{example}

\begin{example}[A family for free actions]
    \label{ex:normal_subgroup_another_family}
    Suppose again that $N \leq G$ is a closed normal subgroup of $G$. Another family is given as $\family_N \coloneqq \{ H \subset G \mid H \cap N = \{1\} \}$.    Note that when $N \neq \{1\}$, there is an inclusion of families $\family_N \subseteq \Gamma_N$. Thus, there is an inclusion $\Gamma_N^c \subseteq \family_N^c$.
\end{example}

\begin{defn}
    Let $G$ be a compact Lie group and $S$ a collection of subgroups, closed under conjugacy. Then we write $\catgrpcol{G}{S} \coloneqq \func(\orbit_S(G)\op,\cat)$    for the \textit{category of $S$--categories}.
\end{defn}

If $\family$ is a family of closed subgroups of $G$, we have the following variant of the standard isotropy separation sequence relating the categories $\cat_G$, $\catgrpcol{G}{\family}$ and $\catgrpcol{G}{\family^c}$. This will allow us to ``separate'' our problems into orthogonal pieces, one part concentrated in $\family^c$ and the part which is $\family$--local.

\begin{cons}[Isotropy separation for $G$--categories]\label{cons:isotropy_separation_recollement}
    Let $\family$ be a family of subgroups of a compact Lie group $G$ and denote by $b \colon \orbit_{\family}(G) \hookrightarrow \orbit(G)$ and $s \colon \orbit_{\family^c}(G) \hookrightarrow \orbit(G)$ the inclusions.    We obtain the adjoint triples 
    \begin{equation*}
    \begin{tikzcd}
        \cat_{G, \family^c} \ar[r, "s_!" description, hook, bend left=40] 
        %\ar[r, phantom, "\perp" description, {yshift = 15}]
        \ar[r, "s_*" description, hook, bend right=40] 
        %\ar[r, phantom, "\perp" description, {yshift = -15}]
        & \cat_G \ar[l, "s^*" description] 
        & 
        &
        \cat_{G}  
        %\ar[r, phantom, "\perp" description, {yshift = 15}]
        \ar[r, "b^*" description] 
        %\ar[r, phantom, "\perp" description, {yshift = -15}]
        & \cat_{G, \family} 
        \ar[l, "b_!" description, hook, bend right=40]
        \ar[l, "b_*" description, hook, bend left=40] 
    \end{tikzcd}
    \end{equation*}
    by restriction and Kan extension along $s$ and $b$.
    Without making this precise, let us mention that these can be made into an unstable recollement using the cofibre sequence $\underline{E\family}_+ \xrightarrow{b} \udl{S}^0 \to \widetilde{\underline{E\family}}$ of pointed $G$--spaces. For example, the map $b^*\colon \cat_G\rightarrow \cat_{G,\family}$ is equivalently given by taking the global sections on the map $b^*\colon \udl{\cat}\rightarrow\udl{\func}(\udl{E\family},\udl{\cat})$.
    
    Unwinding the right Kan extension formula, one obtains for example that $s_*$ is given by
    \begin{equation*}
        (s_* \category{C})^H = \begin{cases}
            \category{C}^H, & H \notin \family \\
            *, & H \in \family.
        \end{cases}
    \end{equation*}

    Notice that the adjunction $b^* \dashv b_*$ is the basechange adjunction associated to the \'etale morphism $\pi_{E\family}^* \colon \spc_G \rightleftharpoons (\spc_G)_{/E\family} \cocolon (\pi_{E\family})_*$.
    In particular, it restricts to an adjunction $\presentable^L_G \rightleftharpoons \presentable^L_{G,\family}$ by \cref{lem:lax_monoidal_base_change_etale_morphism}.
    Similarly, the adjunction $s^* \dashv s_*$ is the basechange adjunction associated to the geometric morphism $s^* \colon \spc_G \rightleftharpoons \spc_{G,\family^c} \cocolon s_*$.
\end{cons}

\begin{example}[Isotropy separation and coinduction]\label{ex:isotropy_separation_coinduction}
    Consider a continuous epimorphism $\groupSurjection \colon G \twoheadrightarrow G/N$ of compact Lie groups.
    Recall from \cref{ex:normal_subgroup_family} that there is the  functor $\res^\orbit_\groupSurjection \colon \orbit(Q) \hookrightarrow \orbit(G)$ which restricts to an equivalence $\orbit(G/N) \simeq \orbit_{\Gamma_N^c}(G)$.
    This identifies the adjunctions $\coind_\groupSurjection \colon \cat_G \rightleftharpoons \cat_{Q} \cocolon \rcoind_\groupSurjection$ and $s^* \colon \cat_G \rightleftharpoons \cat_{G, \Gamma_N^c} \cocolon s_*$.
\end{example}

\begin{cons}[Singular part]\label{defn:singular_part_inclusion}
    Consider the inclusion  $s\colon \orbit_{\family^c}(G)\hookrightarrow \orbit(G)$. Then we get the Bousfield  colocalisation $s_!\colon \spc_{G,\family^c} \rightleftharpoons \spc_G \cocolon s^*$ such $s_!s^*(\myuline{G/H}) = \myuline{G/H}$ for every  and $s_!s^*(\myuline{G/H})=\varnothing$ . Since $s_!s^*$ picks out the isotropy of a $G$--space $\udl{X}$ not in $\family$, we shall also use the notations (which will be part of a larger notational package in \cref{nota:isotropy_separation_package})
    \[\udl{X}_{\family^c}\coloneqq s^*\udl{X}\in\spc_{\family^c}\quad\quad\quad \udl{X}_{\widetilde{\family}}\coloneqq s_!s^*\udl{X} = s_!\udl{X}_{\family^c}\in \spc_G.\]
    The  adjunction counit $\epsilon\colon \udl{X}_{\widetilde{\family}}\rightarrow \udl{X}$ thus admits the classical interpretation as the inclusion of the $\family$--singular part of the $G$--space $\udl{X}$.  It is the identity map on $G/H$ for $H\notin\family$ and the map $\varnothing\rightarrow \myuline{G/H}$ for  $H\in\family$. We refer to $\epsilon \colon \udl{X}_{\widetilde{\family}} \rightarrow \udl{X}$ as the \textit{inclusion of the $\family$--singular part of $\udl{X}$}.
\end{cons}

\begin{example}
    For $\family = \proper$, the family of proper subgroups, $\udl{X}_{\widetilde{\proper}}$ is given by the fixed points space $X^G$, considered as a $G$--space with trivial action. For $\family = \{ e \}$, the trivial family, the intuition for $\udl{X}_{\widetilde{\family}}$ is that is gives the $G$--space of all points in $\udl{X}$ with nontrivial isotropy.
\end{example}

Having recounted the constructions relevant to the complementary part $\family^c$, we now recall some language associated to the $\family$--local part. Recall that for a family $\family$, we denoted by $b\colon \udl{E\family}\rightarrow\terminalTCat$ the unique map.

\begin{defn}
    We say that a $G$--category $\underline{\sC}$ is $\family$\textit{--Borel} if that the map $\underline{\sC}\rightarrow b_*b^*\underline{\sC}$ is an equivalence.
    A $G$--category $\udl{\category{C}}$ will be called $\family$--\textit{coBorel} if the map $b_! b^* \udl{\category{C}} \to \udl{\category{C}}$ is an equivalence. 
\end{defn}

\begin{example}[Borel categories]\label{cons:Borel_categories}
    For the trivial family $\family=\{1\}$, we will also write $\underline{E\family}$ as $\underline{EG}$ and  write $
        \udl{\borel} \coloneqq b_* \colon \cat_{G,\{1\}}\simeq \cat^{BG} \hookrightarrow \cat_G$. We call $\udl{\borel}(\category{C})$ the \textit{Borel-$G$-category} associated to $\category{C}$.
    Explicitly, $\udl{\borel}(\category{C})(\myuline{G/H}) = \category{C}^{hH}$. In this case, the adjunctions \cref{cons:isotropy_separation_recollement} produce the Borelification Bousfield (co)localisations studied in \cite[$\S2.4$]{kaifNoncommMotives}. While we will not need it in this article, we mention that there, it was shown in the case of finite groups $G$ that $b^*\colon \cat_G\rightarrow \cat^{BG}$ naturally assemble to a $G$--symmetric monoidal Bousfield localisation and thus interacts well with the multiplicative norms.
\end{example}

\begin{fact}\label{lem:recoginition_borel_cocomplete_categories}
    An alternative description for $b_!b^*\underline{\sC}$  and $b_*b^*\underline{\sC}$ are $\underline{E\family}\times \underline{\sC}$ and $\underline{\func}(\underline{E\family},\underline{\sC})$ respectively. 
\end{fact}

\begin{nota}
    For a subgroup $K\leq G$, we write $\family_K$ for the family of subgroups of $K$ which belong to $\family$. Note that, in particular, the equivalence $\orbit(G)_{/(G/K)}\simeq \orbit(K)$ induces an equivalence $\orbit_{\family}(G)_{/(G/K)}\simeq \orbit_{\family_K}(K)$.
\end{nota}

\begin{prop}[Characterisations of (co)Borelness]\label{prop:characterisation_coborelness}
    Let $\udl{\sC}\in\cat_G$. Then:
    \begin{enumerate}[label=(\alph*)]
        \item $\udl{\sC}$ is $\family$--coBorel if and only if $\sC^H\simeq \varnothing$ for all $H\in\family^c$,
        \item $\udl{\sC}$ is $\family$--Borel if and only if for all $K\leq G$, the canonical map $\sC^K \rightarrow \lim_{G/H\in\orbit_{\family_K}(K)\op}\sC^H$        induced by restrictions is an equivalence.
    \end{enumerate}
\end{prop}
\begin{proof}
    Part (a) is immediate using the description $b_!b^*\udl{\sC}\simeq \udl{E\family}\times \udl{\sC}$. For part (b), the  comma category used to compute the value of the right Kan extension $b_*\colon \func(\orbit_{\family}(G)\op,\cat)\rightarrow \func(\orbit(G)\op,\cat)$ at $G/K$ is 
    \[\big(\orbit_{\family}(G)\op\big)_{(G/K)/}\simeq \big(\orbit_{\family}(G)_{/(G/K)}\big)\op\simeq \orbit_{\family_K}(K)\op\]
    whence the claim.
\end{proof}

\begin{example}[Modules over $\family$--nilpotent rings]\label{example:modules_over_F-nilpotent-rings}
    Suppose $G$ is a finite group and $\family$ is a family of subgroups. By \cite[Prop. 6.38 (1), Thm. 6.42]{MNN17} and the concrete characterisation of $\family$--Borelness from \cref{prop:characterisation_coborelness} (b), we learn that if $R \in \calg(\spectra_G)$ is $\family$--nilpotent, then $\udl{\module}_{\myuline{\spectra}_G}(R)$ is an $\family$--Borel $G$--category.
\end{example}

\subsubsection*{Categorified Brauer quotients}

Consider a family $\family$ of closed subgroups of $G$.
The adjunction $s^* \colon \cat_G \rightleftharpoons \catgrpcol{G}{\family^c} \cocolon s_*$ from \cref{cons:isotropy_separation_recollement} does \textit{not} restrict to an adjunction between presentable or (fibrewise) stable categories as the adjunction unit does not preserve $G$-colimits.
The main result of this section shows that the restriction $s_* \colon \prst{G, \family^c} \hookrightarrow \prst{G}$ (which is fully faithful by \cref{lem:base_change_presentable_fully_faithful_geometric_morphism}) admits a symmetric monoidal left adjoint $\widetilde{s}^*$.
We do this by showing that it is a smashing localisation.

\begin{cons}[(Co)tensoring over pointed groupoids]\label{cons:cotensoring_over_pointed_groupoids}
    Let $\udl{\category{E}}$ be a pointed $\B$--category admitting all parametrised (co)limits. Then $\udl{\E}$ is naturally tensored and cotensored over pointed $\B$--groupoids $\B_*$ as follows: for $\terminalTCat\rightarrow \udl{X}$ in $\B_*$ and $E\in\udl{\E}$, we define 
    \[
    \obj{X} \pointedcotensor E \coloneqq \cofib\left( E \simeq \colim_{\terminalTCat} E \rightarrow \colim_{\udl{X}} E \right) 
    \quad\quad 
    \pointedtensor(\obj{X},E) \coloneqq \fib\left( \lim_{\udl{X}}E \rightarrow \lim_{\terminalTCat}E \simeq E \right)
    \]
    These exhibit $\udl{\E}$ as being tensored and cotensored over $\B_*$, respectively, since for example, for a fixed $F\in \udl{\E}$, we have
    \begin{equation*}
        \begin{split}
            \map_{\udl{\E}}(\obj{X} \pointedcotensor E, F) & \simeq \fib\big(\map_{\udl{\E}}(\colim_{\udl{X}}E,F)\rightarrow\map_{\udl{\E}}(E,F)\big)\\
            &\simeq \map_{\B}(\udl{X},\myuline{\map}_{\udl{\E}}(E,F))\times_{\map_{\B}(\udl{\ast},\myuline{\map}_{\udl{\E}}(E,F))}\{\ast\}\\
            &\simeq \map_{\B_*}(\udl{X},\myuline{\map}_{\udl{\E}}(E,F))
        \end{split}
    \end{equation*}
    Observe also that these constructions give us an adjunction $\obj{X} \pointedcotensor \-- \colon \udl{\E} \rightleftharpoons \udl{\E} \cocolon \pointedtensor(\udl{X},-)$. Moreover, it is easy to see that for $\obj{X},\obj{Y}\in\B_*$, we have $\obj{X}\pointedcotensor (\obj{Y}\pointedcotensor E)\simeq (\obj{X}\wedge \obj{Y})\pointedcotensor E$ where $\obj{X}\wedge \obj{Y} \simeq \cofib(\obj{X}\vee \obj{Y}\rightarrow \obj{X}\times \obj{Y})$.
\end{cons}

\begin{obs}\label{obs:cotensor_commutes_with_geometric_morphisms}
    Let $\family$ be a family of closed subgroups of $G$ and $s^* \colon \cat_{G,*}\rightleftharpoons \cat_{\family^c,*} \cocolon s_*$ the associate Bousfield localisation of the geometric morphism $s^*\colon \spc_G\rightleftharpoons \spc_{\family^c}$. Let $\udl{X}\in\spc_{G,*}$ and $\udl{\sC}\in\cat_{\family^c,*}$. Then there is an equivalence $\pointedtensor(\udl{X},s_*\udl{\sC})\simeq s_*\pointedtensor(s^*\udl{X},\udl{\sC})$ by virtue of the following computation
    \[
    \pointedtensor(\udl{X},s_*\udl{\sC}) \simeq \fib(\lim_{\udl{X}}s_*\udl{\sC} \rightarrow
     \lim_{\terminalTCat}s_*\udl{\sC}) \simeq s_*\fib(\lim_{s^*\udl{X}}\udl{\sC} \rightarrow \lim_{\terminalTCat}\udl{\sC}) \simeq s_*\pointedtensor(s^*\udl{X},\udl{\sC}).
     \]
     Here we have used that $s_*$ commutes with limits and the equivalence $\lim_{\udl{X}}s_*\udl{\sC}\simeq s_*\lim_{s^*\udl{X}}\udl{\sC}$ coming from the identifications of adjunctions in \cref{lem:parametrised_colimits_base_change}.
\end{obs}

We introduce now the key notion of \textit{Brauer quotients} of categories with respect to a fixed family. As will be clear from the next terminology, they will be a special case of the standard categorical construction of Verdier quotients. However, since they will play such a key role in this article and are so specific to the equivariant situation, we have chosen to dignify them with a special name, borrowing from the classical theory of Mackey functors. 

\begin{terminology}[$\family$--Brauer quotients]\label{terminology:family_brauer_quotient}
    For a finite group $G$, we  define the $\family$--Brauer quotient $\udl{\D}/\langle\family\rangle$  of a small $G$--stable category $\udl{\D}$ as a $G$--stable category admitting a $G$--exact functor $\Phi^{\family}\colon \udl{\D}\rightarrow
    \udl{\D}/\langle \family\rangle$ which, for all $G$--stable categories $\udl{\E}$, induces an equivalence 
    \[(\Phi^{\family})^*\colon \udl{\func}\exact(\udl{\D}/\langle \family\rangle, \udl{\E})\xlongrightarrow{\simeq}\udl{\func}^{\mathrm{ex},\family=0}(\udl{\D},\udl{\E})\]
    where $\udl{\func}^{\mathrm{ex},\family=0}(\udl{\D},\udl{\E})\subseteq \udl{\func}^{\mathrm{ex}}(\udl{\D},\udl{\E})$ is the full $G$--subcategory of $G$--exact functors $F\colon \udl{\D}\rightarrow \udl{\E}$ such that $\res^G_HF\colon\res^G_H\udl{\D}\rightarrow\res^G_H\udl{\E}$ is the zero functor for all $H\in\family$. Observe that $\udl{\D}/\langle\family\rangle$ must be unique if it exists.
    We denote by $\catGstfam{G}{\family^c} \subseteq \catGst{G}$ the full subcategory given by those $G$-stable categories lying in the image of $s_* \colon \catgrpcol{G}{\family^c} \hookrightarrow \cat_G$, i.e. those with value 0 on $\orbit_\family(G)$.

    Analogously in the presentable setting, for a compact Lie group $G$ and a family $\family$ of closed subgroups, we may define the $\family$--Brauer quotient of an object $\udl{\sC}\in\presentable_G^{L,\mathrm{st}}$ as a presentable $G$--category $\udl{\sC}/\langle\family\rangle$ equipped with a parametrised colimit--preserving functor $\Phi^{\family}\colon \udl{\sC}\rightarrow
    \udl{\sC}/\langle \family\rangle$ inducing for every fibrewise stable presentable $G$--category $\udl{\E}$ an equivalence
    \[(\Phi^{\family})^*\colon \udl{\func}^L(\udl{\sC}/\langle \family\rangle, \udl{\E})\xlongrightarrow{\simeq}\udl{\func}^{L,\family=0}(\udl{\sC},\udl{\E})\]
\end{terminology}

\begin{thm}[Categorified Brauer quotients]\label{prop:categorical_brauer_quotients}
    Let $G$ be a compact Lie group, $H$ a finite group, $\family$ a family of closed subgroups of $G$ and $\E$ a family of subgroups of $H$. Then the fully faithful inclusions 
    \[
    s_* \colon \prst{G, \family^c} \hookrightarrow \prst{G} \quad \quad 
    s_* \colon \catGstfam{H}{\E^c} \hookrightarrow \catGst{H} 
    \]
    all admit symmetric monoidal left adjoints $\widetilde{s}^*$ which are smashing localisations at the idempotent algebra $\widetilde{EF}$.
    In the first case, the induced lax symmetric monoidal structure on $s_*$ agrees with the one from \cref{lem:lax_monoidal_base_change_geometric_morphism}.
    Moreover, $\widetilde{s}^*$ satisfies the universal property of the $\family$--Brauer quotient.
\end{thm}
\begin{proof}
    We only prove the first case since the second one can be done entirely analogously. 
    $\udl{\presentable}^{L, \stable}_G$ is a pointed category admitting all parametrised (co)limits and is thus tensored over $\spc_{G,*}$ by \cref{cons:cotensoring_over_pointed_groupoids} (in the presentable case, this also comes from pointedness being classified by the idempotent algebra $\spc_{G,*} \in \presentable^L_G$).
    
    We claim that the left adjoint to $s_*$ is given by $\udl{\widetilde{E\family}}\pointedcotensor \--$.
    To see that this functor does indeed take values in $\prst{G, \family^c} \hookrightarrow \prst{G}$ note that for all $K \in \family$ and $\udl{\category{C}} \in \prst{G}$ we have
    \begin{equation*}
        \res^G_K(\udl{\widetilde{E\family}}\pointedcotensor \udl{\category{C}}) \simeq \res^G_K \udl{\widetilde{E\family}} \pointedcotensor \res^G_K \udl{\category{C}} \simeq * \pointedcotensor \res^G_K \udl{\category{C}} \simeq 0.
    \end{equation*}
    To see that it is the left adjoint to $s_*$, let $\udl{\category{D}} \in \prst{G, \family^c}$.
    Observe that since $s^*\udl{\widetilde{E\family}} \simeq \udl{S}^0$, we have equivalences $\pointedtensor(\udl{\widetilde{E\family}},s_*\udl{\D})\simeq s_*\pointedtensor(\udl{S}^0,\udl{\D})\simeq s_*\udl{\D}$, where the first equivalence is by \cref{obs:cotensor_commutes_with_geometric_morphisms}. 
    Thus, the computation
    \begin{equation*}
        \map_{\prst{G}}(\udl{\widetilde{E\family}} \pointedcotensor\underline{\sC}, s_*\udl{\D}) 
        \simeq \map_{\prst{G}} (\udl{\sC}, \pointedtensor(\udl{\widetilde{E\family}}, s_*\udl{\D}))
        \simeq \map_{\prst{G}}(\udl{\sC}, s_*\udl{\D}).    
    \end{equation*}
    shows that $\udl{\widetilde{E\family}}\pointedcotensor\--$ is indeed the left adjoint to $s_*$ as claimed.
    
    Next, since $\udl{\widetilde{E\family}}$ is an idempotent algebra in $\spc_{G,*}$, the left adjoint $\widetilde{s}^*(-)\simeq \udl{\widetilde{E\family}}\pointedcotensor\--$ is a smashing localisation and in particular attains a canonical symmetric monoidal structure.
    To show that the induced lax symmetric monoidal structure on $s_*$ is equivalent to the one from \cref{lem:lax_monoidal_base_change_geometric_morphism}, by \cref{lem:base_change_presentable_fully_faithful_geometric_morphism} 
    we only have to show that the map $u \colon \widetilde{s}^* \unit_{\prst{G}} \to \unit_{\prst{G, \family^c}}$ adjoint to the lax unit $\unit_{\prst{G}} \to s_* \unit_{\prst{G, \family^c}}$ is an equivalence.
    But by construction of the lax symmetric monoidal structure on $s_*$, we have for $\udl{\category{C}} \in \prst{G, \family^c}$ the commutative diagram
    \begin{equation*}
    \begin{tikzcd}
        \func^L_{\family}\left(\unit_{\prst{G, \family^c}}, \udl{\category{C}} \right) \ar[d, "u^*"] \ar[rr, "\simeq"] 
        & 
        & \Gamma(\udl{\category{C}}) \ar[d, "\simeq"] \\
        \func^L_{\family}\left(\widetilde{s}^* \unit_{\prst{G}}, \udl{\category{C}} \right) \ar[r, "\simeq"]
        & \func^L_{G}\left(\unit_{\prst{G}}, s_* \udl{\category{C}} \right) \ar[r, "\simeq"]
        & \Gamma(s_* \udl{\category{C}})
    \end{tikzcd}
    \end{equation*}
    showing that the left vertical map is an equivalence.
    
    Finally, for the statement about $\family$--Brauer quotients, notice that the unit map $\udl{\category{C}} \to \udl{\widetilde{E \family}} \pointedcotensor \udl{\category{C}}$ has trivial restriction to each group $H \in \family$ as $\res_H^G \udl{\widetilde{E \family}} \pointedcotensor \udl{\category{C}} \simeq \res^G_H\udl{\widetilde{E \family}} \pointedcotensor \res^G_H\udl{\category{C}} \simeq * \pointedcotensor \res^G_H\udl{\category{C}} \simeq 0$.
    We have to show that for all $\udl{\D} \in \prst{G}$, the induced map $\udl{\func}^L(s_* \widetilde{s}^*\udl{\category{C}},\udl{\D}) \rightarrow \udl{\func}^{L,\family=0}(\udl{\category{C}},\udl{\D})$ is an equivalence.
    From the cofibre sequence $\udl{E \family}_+ \to \udl{S^0} \to {\udl{\widetilde{E\family}}}$ in $\spc_{G, *}$ we obtain the fibre sequence
    \begin{equation*}
        \udl{\func}^L(\udl{\widetilde{E\family}} \pointedcotensor \udl{\category{C}}, \udl{\category{D}}) 
        \to \udl{\func}^L(\udl{\category{C}}, \udl{\category{D}})
        \to \udl{\func}^L(\udl{E\family}_+ \pointedcotensor \udl{\category{C}}, \udl{\category{D}})
    \end{equation*}
    in $\prst{G}$.
    This shows that $\udl{\func}^L(\udl{\widetilde{E\family}} \pointedcotensor \udl{\category{C}}, \udl{\category{D}})$ is a full $G$-subcategory of $\udl{\func}^L(\udl{\category{C}}, \udl{\category{D}})$ (this is true for any fibre sequence of stable categories).
    Now suppose that $f \colon \udl{\category{C}} \to \udl{\category{D}}$ vanishes on $\family$.
    Denote by $\langle \im(f) \rangle \subseteq \udl{\category{D}}$ the full presentable fibrewise stable $G$-subcategory generated by the image of $f$.
    The assumption on $f$ guarantees that $\res_H^G \langle \im(f) \rangle = 0$ whenever $H \in \family$, i.e. $\langle \im(f) \rangle$ lies in the image of $s_*$.
    Consider the commutative diagram
    \begin{equation*}
    \begin{tikzcd}
        \udl{\category{C}} \ar[r, "f"] \ar[d] 
        & \langle \im(f) \rangle \ar[r, hook] \ar[d, "\simeq"] 
        & \udl{\category{D}} \\
        \udl{\widetilde{E\family}} \pointedcotensor \udl{\category{C}} \ar[r, "f"] 
        & \udl{\widetilde{E\family}} \pointedcotensor \langle \im(f) \rangle \ar[ur, dashed] 
        &
    \end{tikzcd}.
    \end{equation*}
    As $\langle \im(f) \rangle$ lies in the image of $s_*$, the first part shows that the middle vertical arrow is an equivalence from which we obtain the dashed factorisation.
    This concludes the proof of the theorem.
\end{proof}

\begin{rmk}\label{rmk:F-Verdier-quotients_agree_with_quigley_shah}
    In the setting of finite groups $G$, for a $G$--stable category $\udl{\D}$, since $\widetilde{s}^*\udl{\D}$ is the $\family$--Brauer quotient, it satisfies the universal property of the Verdier quotient articulated in \cite[Thm. 5.23]{quigleyShahParametrisedTate}. Thus, by \cite[Def. 5.21]{quigleyShahParametrisedTate}, it may alternatively be described as a fibrewise Verdier quotient in the nonequivariant sense.
\end{rmk}

\begin{nota}\label{nota:isotropy_separation_package}
    Now that we have all the fixed points functors that will concern us, let us collect and summarise them, introducing some new notations along the way. While the notations $\{s_!, s^*, s_*, \widetilde{s}^*\}$ are compact and lithe, useful to prove results, we believe that the notations presently introduced  have more intuitive appeal. The starting point will be the inclusion $s\colon \orbit_{\family^c}(G)\op\hookrightarrow \orbit(G)\op$ from \cref{cons:isotropy_separation_recollement}.
    \begin{enumerate}[label=(\alph*)]
        \item Recall the notations from \cref{defn:singular_part_inclusion} which gives us the top adjunctions in
        \begin{center}
            \begin{tikzcd}
                \spc_{\family^c}\ar[dd,hook] \ar[rrr, bend  left = 20, "(-)\singularPartTwiddle{\family}\coloneqq s_!" description, hook]\ar[rrr, bend  left = -20, "(-)\sLowerStarInclusion{\family}\coloneqq s_*"' description, hook]& &&\spc_G \ar[lll, "(-)\sUpperStar{\family} \coloneqq s^*" description]\ar[dd,hook]\\
                \\
                \cat_{G,\family^c} \ar[rrr, bend  left = 20, "(-)\singularPartTwiddle{\family}\coloneqq s_!" description, hook]\ar[rrr, bend  left = -20, "(-)\sLowerStarInclusion{\family}\coloneqq s_*"' description, hook]& &&\cat_G \ar[lll, "(-)\sUpperStar{\family} \coloneqq s^*" description]
            \end{tikzcd}
        \end{center}
        and that we have the commuting squares of adjunctions is an easy check using that $s^*$ commutes with the vertical maps and their adjoints. Since  $(-)\singularPartTwiddle{\family}$ and $(-)\sLowerStarInclusion{\family}$ are fully faithful, we  also write $(-)\singularPartTwiddle{\family}$ and $(-)\sLowerStarInclusion{\family}$ for $s_!s^*$ and $s_*s^*$ respectively. In particular, for $\udl{X}\in\spc_G$, the  counit gives us a map $\epsilon\colon \udl{X}\singularPartTwiddle{\family}=s_!s^*\udl{X}\rightarrow \udl{X}$ as in \cref{defn:singular_part_inclusion}.  Moreover, by \cref{cons:isotropy_separation_recollement}, we have for $\udl{\sC}\in\cat_{G,\family^c}$ the description
        \begin{equation*}
            \udl{\sC}\sLowerStarInclusion{\family} = \begin{cases}
                \sC(\myuline{G/H}) & \text{ if } H\in \family^c;\\
                \ast & \text{ if } H\in \family.
            \end{cases}
        \end{equation*}
        \item We also have the following solid commuting squares 
        \begin{center}
            \begin{tikzcd}
                \presentable_G^{\mathrm{st}} \ar[d, loop->]\ar[rr, shift left =2, "\brauerQuotientFamily{\family}(-)\coloneqq \widetilde{s}^*",dashed ]& &\presentable_{G,\family^c}^{\mathrm{st}} \ar[d, loop->]\ar[ll, "(-)\sTwiddleLowerStar{\family}\coloneqq s_*", hook] && \cat_G\exact \ar[d, loop->]\ar[rr, shift left =2, "\brauerQuotientFamily{\family}(-)\coloneqq \widetilde{s}^*",dashed]&& \cat_{G,\family^c}\exact\ar[d, loop->]\ar[ll, "(-)\sTwiddleLowerStar{\family}\coloneqq s_*", hook]\\
                \widehat{\cat}_G && \widehat{\cat}_{G,\family^c}\ar[ll, "(-)\sLowerStarInclusion{\family}\coloneqq s_*"', hook] && \cat_G && \cat_{G,\family^c}\ar[ll, "(-)\sLowerStarInclusion{\family}\coloneqq s_*"', hook]
            \end{tikzcd}
        \end{center}
        where the top maps admit the dashed left adjoints. Here, the left diagram holds for general compact Lie groups $G$ whereas the right diagram is only defined for finite groups $G$ from \cref{prop:categorical_brauer_quotients}. As above, since the functors $(-)\sTwiddleLowerStar{\family}$ are fully faithful, we will also write $(-)\sTwiddleLowerStar{\family}$ to denote $s_*\widetilde{s}^*$. The adjunction unit $\id\rightarrow s_*\widetilde{s}^*$ will be denoted by $\Phi^{\family}\colon (-)\rightarrow (-)\sTwiddleLowerStar{\family}$ or just $\Phi\colon (-)\rightarrow (-)\sTwiddleLowerStar{\family}$ when the family $\family$ is understood.
    \end{enumerate}
\end{nota}

\subsubsection*{Stability for quotient groups}

Let $N \le G$ be a closed normal subgroup of the compact Lie group $G$ and denote by $\groupSurjection \colon G \twoheadrightarrow G/N = Q$ the quotient map.
We will use the categorified Brauer quotient from \cref{prop:categorical_brauer_quotients} for the family $\Gamma_N$ from \cref{ex:normal_subgroup_family} to relate $G$-- and $Q$--stable categories.

\begin{prop}\label{prop:stability_for_quotient_groups}
    Suppose that $\groupSurjection \colon G \twoheadrightarrow G/N = Q$ is a continuous epimorphism of compact Lie groups.
    Then there is an adjunction
    \begin{equation*}
    \begin{tikzcd}
            \prGst{G} \ar[rr, shift left = 1, "\stcoind{\alpha}{}"]
            && \prGst{Q} \ar[ll, shift left = 1, "\rcoind_\alpha", hook]
    \end{tikzcd}
    \end{equation*}
    which is a smashing localisation.
    The lax symmetric monoidal structure on $\rcoind_\alpha$ from \cref{lem:coinflation_stability} is equivalent to the lax symmetric monoidal structure from this smashing localisation.
    We thus may view $G/N$--stable presentable categories precisely as $G$--stable categories which vanish for all subgroups $H \leq G$  not containing $N$.
\end{prop}
\begin{proof}
    By combining \cref{ex:isotropy_separation_coinduction} and \cref{prop:categorical_brauer_quotients}, $\rcoind_\alpha \colon \prst{Q} \hookrightarrow \prst{G}$ admits a symmetric monoidal left adjoint $\stcoind{\alpha}{} = \coind_\alpha(\udl{\widetilde{E \Gamma}_N} \wedge \--)$ which is a smashing localisation.
    We only have to show that this restricts to an adjunction between $G$- and $Q$-stable categories.
    But this follows by combining \cref{lem:lax_monoidal_structure_coind}, \cref{lem:coinflation_stability} and observing that $\udl{\widetilde{E \Gamma}_N} \wedge \--$ preserves $G$-stable categories.
\end{proof}

\begin{cor}\label{lem:geometric_fixed_point_spectrum_weil_group}
    Writing $\groupSurjection\colon G\twoheadrightarrow G/N$ for the quotient map by a closed normal subgroup, the symmetric monoidal unit map $\myuline{\spectra}_{G/N} \to \stcoind{\alpha}{} \myuline{\spectra}_G$ is an equivalence.
\end{cor}
\begin{proof}
    This is a direct consequence of symmetric monoidality of the adjunction in \cref{prop:stability_for_quotient_groups}.
\end{proof}

\begin{cons}[Geometric fixed points]\label{cons:geometric_fixed_points}
    Let $\groupSurjection_G\colon G\twoheadrightarrow 1$ be the quotient map. The symmetric monoidal $G$-colimit preserving unit map 
    \begin{equation*}
        \Phi^G \colon \myuline{\spectra}_G \longrightarrow \rcoind_{\groupSurjection_G} \stcoind{\groupSurjection_G}{} \myuline{\spectra}_G \simeq \rcoind_{\groupSurjection_G} \spectra
    \end{equation*}
    restricts to 
    a symmetric monoidal colimit preserving functor $\Phi^G \colon \spectra_G \to \spectra$.
    There is an equivalence equivalence $\Phi^G \circ \Sigma^\infty_G (\--) \simeq \Sigma^\infty (\--)^G$ as, by construction, $\Phi_G$ is $\spc_{G, *}$-linear and sends the unit to the unit.
    This shows that $\Phi^G$ recovers the classical geometric fixed points functor which is uniquely determined by these properties.

    If $H \le G$ is a closed subgroup, we have the symmetric monoidal $G$-colimit preserving functor $\Phi^H \colon \myuline{\spectra}_G \to \coind_H^G \myuline{\spectra}_H \to \coind_H^G \rcoind_{\groupSurjection_H} \spectra$ which on global sections recovers the classical geometric fixed point functors $\Phi^H \colon \spectra_G \to \spectra$.
\end{cons}

\begin{defn}
    A collection of $\B$--functors $\{F_s\colon \udl{\sC}\rightarrow\udl{\D}_s\}_{s\in S}$ is \textit{jointly conservative} if for all $X\in \B$, the collection $\{F_s(X)\colon \sC(\udl{X})\rightarrow \D_s(\udl{X})\}_{s\in S}$ is jointly conservative.
\end{defn}

\begin{obs}\label{obs:functors_preserve_joint_conservativity}
    Let $\{F_s\colon \udl{\sC}\rightarrow\udl{\D}_s\}_{s\in S}$ be a jointly conservative collection of $\B$--functors and $\udl{X}\in\B$. Then the collection $\{F_s\colon \udl{\func}(\udl{X},\udl{\sC})\rightarrow\udl{\func}(\udl{X},\udl{\D}_s)\}_{s\in S}$ is also a jointly conservative collection. This is an immediate consequence of the definition and that the evaluation at $Y\in\B$ for the $\B$--category $\udl{\func}(\udl{X},\udl{\sC})$ is $\sC(\udl{Y}\times \udl{X})$.
\end{obs}

\begin{rmk}
    \label{rem:joint_conservativity_detectable_on_generators}
    If $\{ \udl{X}_i\}_{i \in I}$ is a set of objects generating $\category{B}$ under colimits, then a collection of $\B$--functors $\{F_s\colon \udl{\sC}\rightarrow\udl{\D}_s\}_{s\in S}$  is jointly conservative if $\{F_s(X)\colon \sC(\udl{X}_i)\rightarrow \D_s(\udl{X}_i)\}_{s\in S}$ is jointly conservative for each $i$. Indeed, let $\udl{X}$ be an object. Then $\sC(\udl{X}) \simeq \lim_{ (i,f \colon \udl{X}_i \rightarrow \udl{X}) } \sC(\udl{X}_i)$  and the collection $T = \{  \sC(\udl{X}) \xrightarrow{f^*} \sC(\udl{X}_i) \mid i \in I, f \colon \udl{X}_i \rightarrow \udl{X} \}$ is jointly conservative. Suppose that $h$ is a morphism in $\sC(\udl{X})$ which maps to an equivalence in $\category{D}_s(\udl{X})$ for each $s \in S$. For $(i,f) \in T$, we observe that in the commutative diagram
    \begin{center}
        \begin{tikzcd}
            \category{C}(\udl{X}) \ar[r, "f^*"] \ar[d] & \category{D}_s(\udl{X}) \ar[d] \\
            \category{C}(\udl{X}_i) \ar[r]  & \category{D}_s(\udl{X}_i)
        \end{tikzcd}
    \end{center}
    the morphism $h$ maps to an equivalence in the lower right corner for each $s \in S$, so by assumption it mapped to an equivalence in the lower left corner. As that holds true for each $(i,f) \in T$ and the collection $T$ was jointly conservative, we see that $h$ was an equivalence to start with, as desired.
\end{rmk}

\begin{prop}[Joint conservativity of geometric fixed points]\label{prop:joint_conservativity_geometric_fixed_points}
    The collection of $G$--functors
    \begin{equation*}
        \resizebox{.95\hsize}{!}{$\Big\{\Phi^H\colon \myuline{\spectra}\xrightarrow{\eta} \coind^G_H\res^G_H\myuline{\spectra} \xrightarrow{\coind^G_H\res^G_H\Phi^{\proper_H}} \coind^G_H\res^G_H\myuline{\spectra}\sTwiddleLowerStar{\proper_H} \quad \mid \quad H\leq G \text{ closed  }\Big\} $}
    \end{equation*}
    is jointly conservative.
\end{prop}
\begin{proof}
    By \cref{rem:joint_conservativity_detectable_on_generators}, it suffices to show that the collection is a jointly conservative collection of functors when evaluated at each $\myuline{G/K}\in\spc_G$. Let $H\leq K$ be a subgroup. Since $\myuline{G/H}\simeq \ind^G_H\terminalTCat$, by the triangle identity, the counit $\ind^G_H\res^G_H\myuline{G/H}\xrightarrow{\epsilon_{\myuline{G/H}}} \myuline{G/H}$ admits a section. Using this and the map $\myuline{G/H}\rightarrow \myuline{G/K}$ we obtain in total a map $h\colon \myuline{G/H}\rightarrow \ind^G_H\res^G_H\myuline{G/K}$. Hence, since we have  $(\coind^G_H\res^G_H\udl{\sC})(G/K)\simeq \func_G(\ind^G_H\res^G_H\myuline{G/K},\udl{\sC})$ and $\sC(G/H)\simeq \func_G(\myuline{G/H},\udl{\sC})$ for any $G$--category $\udl{\sC}$,  we get a transformation $h^* \colon (\coind^G_H\res^G_H\udl{\sC})(G/K)\rightarrow \sC(G/H)$ natural in $\udl{\sC}$. Therefore, for $H\leq K$, evaluating the functor $\Phi^H$ in the statement at $G/K$ together with the transformation above gives the following commuting diagram
    \begin{center}
        \begin{tikzcd}
            \spectra_K \rar["\eta"]\dar[equal] & (\coind^G_H\res^G_H\myuline{\spectra})(G/K)\ar[rrr,"\coind^G_H\res^G_H\Phi^{\proper_H}"]\dar["h^*"] & & &(\coind^G_H\res^G_H\myuline{\spectra}\sTwiddleLowerStar{\proper_H})(G/K)\dar["h^*"] \\
            \spectra_K \rar["\res^K_H"] & \spectra_H \ar[rrr,"\Phi^H"] && &\spectra
        \end{tikzcd}
    \end{center}
    But since the bottom compositions are jointly conservative when we let $H$ vary over all closed subgroups of $K$ (this is well--known, see for example \cite[Prop. 3.3.10]{Schwede_Global}), we thus get similarly that the top compositions are too. This completes the proof.
\end{proof}

\subsubsection*{Free actions}

Here we review a few geometric facts on $G$-spaces on which a normal subgroup $N$ acts freely. It will be needed later on to argue for example, if $\udl{X}$ is a $G$-Poincar\'e space with free $N$-action, then also the quotient $N\backslash \udl{X}$ is $G/N$-Poincar\'e.

\begin{defn}[Free actions]\label{def:free_action}
    Consider a group $G$ together with a normal subgroup $N \le G$.
    We say that the action of $N$ on $\obj{X} \in \spc_G$ is \textit{free} if $\obj{X}$ is coBorel with respect to the family $\family_N$. 
\end{defn}

That is, $N$ acts freely on $\udl{X}$ if whenever $N \cap H \neq \{1\}$ we have $X^H = \varnothing$ as is shown in \cref{lem:recoginition_borel_cocomplete_categories}.

\begin{rmk}[Quotients of free $G$--spaces]\label{rem:quotient_free_action}
    Consider $\obj{X} \in \spc_G$ and isotropy separation with respect to the trivial family $\family = 1$ (which is the case $N = G$ in \cref{def:free_action}). Recalling the operation of genuine quotients from \cref{cons:res_ind_coind}, note that for this family we obtain $G\backslash b_!(-) \simeq (-)_{hG}$ as the first functor is left adjoint to the composite $b^* \infl_G^1 \colon \spc \to \spc^{BG}$ which is the restriction functor along the projection $BG \to *$.
    In particular, we obtain a map
    \begin{equation*}
        X_{hG} \simeq G\backslash (b_! b^* \obj{X}) \to G\backslash \obj{X}.
    \end{equation*}
    This is an equivalence if $G$ acts freely on $\obj{X}$ since $X\simeq b_!Y$ for some $Y\in \spc^{BG}$.
\end{rmk}

We will now prove two lemmas about quotients by free actions. Together, they are useful in studying the fibres of the map $\udl{X} \rightarrow \infl_G^Q N \backslash \udl{X}$, as we will see in \cref{cor:fibre_of_quotient}. 

\begin{lem}
    \label{lem:quotient_cartesian_natural_transformation}
    Let $G$ be a group and let $N \subset G$ be a closed normal subgroup. Let $f \colon \udl{X} \rightarrow \udl{Y}$ be a map of $G$--spaces, where $\udl{X}$ and $\udl{Y}$ have free $N$-actions. Then the square
    \begin{equation}
        \label{diag:quotient_by_free_action}
        \begin{tikzcd}
            \udl{X} \ar[r] \ar[d] & \udl{Y} \ar[d] \\
            \infl^{G/N}_G N\backslash \udl{X} \ar[r] &\infl^{G/N}_G N\backslash \udl{Y} 
        \end{tikzcd}
    \end{equation}
    is cartesian.
\end{lem}

\begin{proof}
    We want to check that the square \cref{diag:quotient_by_free_action} is cartesian
    and we will do so in three steps\footnote{An (arguably shorter) proof is possible if one recalls the model from \cite[Prop. B.7.]{Schwede_Global} and observes that given an $N$-free topological $G$-CW space $\mathcal{X}$, the map $\mathcal{X} \rightarrow \infl_G^{G/N} N \backslash \mathcal{X}$ is a fibration with point-set fibre $N$.}. First, a computation shows that it is cartesian whenever $\udl{X}$ and $\udl{Y}$ are $N$-free $G$-orbits. Second, writing $\udl{Y} = \colim_{H \leq G, \hspace{1mm}\myuline{G/H} \rightarrow \udl{Y}} \myuline{G/H}  $ an application of \cite[6.1.3.9.(4)]{lurieHTT} shows that the claim is true for $\udl{X}$ an $N$-free $G$-orbit and $\udl{Y}$ an  $N$-free $G$--space. It is easy to see that the claim now also holds when $\udl{X}$ is a disjoint union of $N$-free $G$-orbits. For the general statement, note that by \cref{lem:fixed_points_of_quotients}, we may find a collection of $N$-free $G$-orbits $\myuline{G/H_i}$ together with maps $q_i \colon \myuline{G/H_i} \rightarrow \udl{X}$ such that the map
    \[ \coprod \infl_G^{G/N} N \backslash q_i \colon \coprod \infl_G^{G/N} N \backslash \myuline{G/H_i} \rightarrow \infl_G^{G/N} N \backslash\udl{X} \]
    induces a $\pi_0$-surjection on all fixed points\footnote{Such morphisms are effective epimorphisms.}. In the diagram
    \begin{center}
        \begin{tikzcd}
            \coprod \myuline{G/H_i} \ar[r] \ar[d] &\udl{X} \ar[r] \ar[d] & \udl{Y} \ar[d] \\
            \coprod \infl^{G/N}_G N\backslash \myuline{G/H_i} \ar[r]&\infl^{G/N}_G N\backslash \udl{X} \ar[r] &\infl^{G/N}_G N\backslash \udl{Y} 
        \end{tikzcd}
    \end{center}
    we know that the outer square and left square are cartesian. As the bottom left map induces a $\pi_0$-surjection on all fixed points, this implies that the right square is cartesian as well.
\end{proof}

\begin{lem}
    \label{lem:fixed_points_of_quotients}
    Let $\udl{X}$ be a $G$--space with free $N$-action. Then, for each map $f \colon \myuline{Q/H} \rightarrow N \backslash \udl{X}$ there exists a subgroup $K \in \family_N$ and a commutative diagram
    \begin{center}
        \begin{tikzcd}
            & \myuline{G/K} \ar[d] \ar[r] & \udl{X} \ar[d] \\
            \infl_G^Q \myuline{Q/H} \ar[r] & \infl_G^Q N \backslash \myuline{G/K}  \ar[r] & \infl_G^Q N \backslash \udl{X}
        \end{tikzcd}
    \end{center}
    where the lower composition is $\infl_G^Q(f)$.
\end{lem}

\begin{proof}
    Using the explicit formula from \cref{eq:formula_for_genuine_quotient} we compute
    \[ \map_{\spc_Q}(\myuline{Q/H},N \backslash \udl{X}) \simeq \colim_{G/K,\hspace{1mm} Q/H \rightarrow N\backslash(G/K)  } \map_{\spc_G}(\myuline{G/K},\udl{X}).  \]
    The map from the left hand side takes $g \colon \myuline{G/K} \rightarrow \udl{X}$, applies $N\backslash(-)$ to it and precomposes with $Q/H \rightarrow N \backslash \myuline{G/K}$. Thus, $f$ eviently factors through some map $N \backslash \myuline{G/K} \rightarrow N \backslash \udl{X}$ that is of the from $N\backslash g$ for some $g \colon \myuline{G/H} \rightarrow \udl{X}$. Now as $\map(\myuline{G/H},\udl{X})$ can only be nonempty if $H \in \family$, we have proved the assertion.
\end{proof}

\begin{cor}
    \label{cor:fibre_of_quotient}
    Let $\udl{X}$ be a $G$--space on which the closed normal subgroup $N \leq G$ acts freely.
    Consider any map $f \colon \myuline{G/H} \rightarrow \infl_G^Q N \backslash \udl{X}$. Then there exists a cartesian diagram
    \begin{center}
        \begin{tikzcd}
            \myuline{G/H} \times_{\infl_G^Q N \backslash \udl{X}} \udl{X} \ar[r] \ar[d, "\text{proj}"] & \myuline{G/K_0} \ar[d]\\
            \myuline{G/H} \ar[r] & \myuline{G/K_1}
        \end{tikzcd}
    \end{center}
    where $K_0 \in \family$ and $\myuline{G/K_1} \simeq \infl N \backslash \myuline{G/K_0}$.
\end{cor}

\begin{proof}
    Set $H' = H/(H\cap N) \subset Q$.
    The map $f \colon \myuline{G/H} \rightarrow \infl_G^Q N \backslash \udl{X}$ factors through the adjunction unit $\myuline{G/H} \rightarrow \infl_G^Q N \backslash \myuline{G/H} \simeq \infl_G^Q \myuline{Q/H'}$. As the functor $\infl_G^Q$ is fully faithful, we can apply \cref{lem:fixed_points_of_quotients} to the corresponding map $\myuline{Q/H'} \rightarrow N \backslash \udl{X}$ and obtain a commutative diagram
    \begin{center}
        \begin{tikzcd}
            && \myuline{G/K} \ar[d] \ar[r] & \udl{X} \ar[d] \\
            \myuline{G/H} \ar[r] &\infl_G^Q \myuline{Q/H} \ar[r] & \infl_G^Q N \backslash \myuline{G/K}  \ar[r] & \infl_G^Q N \backslash \udl{X}
        \end{tikzcd}
    \end{center}
    in which the square is cartesian by \cref{lem:quotient_cartesian_natural_transformation}.
    Completing the cospan involving $\myuline{G/K}$ and $\myuline{G/H}$ to a pullback gives the desired pullback. 
\end{proof}

\section{Parametrised Poincar\'e duality}\label{section:parametrised_PD}

In this section we start developing the basic formalism of Poincaré duality within the context of  categories parametrised over a topos as summarised in \cref{sec:parametrised_category_theory}. This general theory will later be specialised to the equivariant setting for compact Lie groups in \cref{section:equivariant_PD_elements}.

As a motivation for the definitions appearing in this section recall that, for a closed smooth manifold $M^d$, an embedding $M \hookrightarrow \mathbb{R}^N$ gives rise to a collapse map
\begin{equation*}
    c \colon S^N \rightarrow \thom(\nu_{M\subset \mathbb{R}^N})
\end{equation*}
where $\nu_{M \subset \mathbb{R}^N}$ is the normal bundle of $M$ in $\mathbb{R}^N$.
It turns out that neither the stable homotopy type of the Thom space $\thom(\nu_{M\subset \mathbb{R}^N})$ nor the stable homotopy class of the collapse map $c$ depend on the choice of embedding.
The collapse map defines a class $[c] \in H_N(\thom(\nu_{M\subset \mathbb{R}^N})) \simeq H_d(M; \mathcal{O}_\nu)$, where the isomorphism is the Thom isomorphism and $\mathcal{O}_\nu$ denotes the orientation local system of $\nu$.
Classical Poincar\'e duality now says that 
\begin{equation}\label{eq:classical_PD}
    [c] \cap \-- \colon H^k(M) \to H_{k-d}(M; \mathcal{O}_\nu)
\end{equation}
is an isomorphism.

We start by axiomatising in \cref{subsection:parametrised_spivak_data} such stable collapse maps as \textit{Spivak data} with respect to a fixed coefficient category, upon which we may demand the further condition of being twisted ambidextrous and Poincar\'{e} in \cref{subsection:parametrised_twisteed_ambidex}, generalising the situation sketched above. We then investigate in \cref{subsection:constructions_with_spivak_data} various operations one can perform on Spivak data, proving along the way the main results of the section (c.f. \cref{thm:base_change_of_tw_amb_spivak_data,thm:small_base_change})  about basechanging coefficient categories, which will be the key inputs to our equivariant theory. We then end the section with a discussion of degree theory which will serve as the foundation for our theory of equivariant degrees in \cref{sec:equivariant_degree} and our geometric applications in \cref{section:geometric_applications}.

\subsection{Spivak data}\label{subsection:parametrised_spivak_data}
For an object $\udl{X} \in \category{B}$ we denote by $\uniqueMapX \colon \udl{X} \rightarrow \udl{\ast}$
the map to the final object. Recall that a $\category{B}$--category $\udl{\category{C}}$ admits $\udl{X}$--shaped limits (resp. colimits) if
$\uniqueMapX^* \colon \udl{\sC} \simeq \internalfunc{\category{B}}(\terminalTCat,\udl{\category{C}}) \rightarrow \internalfunc{\category{B}}(\udl{X},\udl{\category{C}})$ admits a right adjoint $\uniqueMapX_*$ (resp. left adjoint $\uniqueMapX_!$).

\begin{defn}\label{defn:spivak_data}
    Let $\udl{X} \in \category{B}$ and $\udl{\category{C}}$ a
    symmetric monoidal $\category{B}$--category which admits $\obj{X}$-shaped colimits.
    A \textit{$\udl{\sC}$--Spivak datum for $\udl{X}$} consists of
    \begin{enumerate}[label=(\arabic*)]
        \item an object $\xi \in \udl{\func}(\udl{X}, \udl{\sC})$ called the dualising sheaf;
        \item a map $c \colon \unit_{\udlcatC} \rightarrow  \uniqueMapX_!\xi$ in $\udl{\sC}$, called the fundamental class (or collapse map).
    \end{enumerate}
\end{defn}

The importance of Spivak data comes from the following construction, which allows us to compare the $\udl{X}$--shaped limit functor with a twisted $\udl{X}$--colimit functor.
It is a generalisation of the map \cref{eq:classical_PD} given by capping with the fundamental class appearing in classical Poincar\'e duality.

\begin{cons}[Capping map]\label{cons:capping_map}
    Let $\udlcatC$ be a symmetric monoidal $\category{B}$--category  which admits $\udl{X}$--shaped limits and colimits and satisfies the $\udl{X}$--projection formula (c.f. \cref{terminology:projection_formula}).
    For each $\udl{\sC}$--Spivak datum $(\xi,c)$ on $\udl{X}$ we can construct a natural transformation
    \begin{equation*}
        \ambi{c}{\xi}{-} \colon 
        \uniqueMapX_*(-) \xlongrightarrow{c \otimes \--}
        \uniqueMapX_! \xi \otimes \uniqueMapX_*(-) \xleftarrow[\simeq]{\projectionformula^X}
        \uniqueMapX_!(\xi \otimes \uniqueMapX^* \uniqueMapX_* (\--)) \xlongrightarrow{\uniqueMapX_!(\id\otimes\epsilon)}
        \uniqueMapX_!(\xi \otimes -)
    \end{equation*}
    which is a morphism in $\udl{\func}(\udl{\category{C}}^{\udl{X}},\udl{\category{C}})$ where $\epsilon \colon \uniqueMapX^* \uniqueMapX_* \to \id$ denotes the adjunction counit. To avoid notational clutter, we will often omit the $\xi$ from $\ambi{c}{\xi}{-}$ when the context is clear.
\end{cons}

There is also a construction in the other direction,  which produces a fundamental class for $\xi$ from a natural transformation $\uniqueMapX_*(-) \rightarrow \uniqueMapX_!(\xi \otimes -)$.

\begin{cons}
    Given a natural transformation $t \colon \uniqueMapX_*(-) \rightarrow \uniqueMapX_!(\xi \otimes -)$ and writing $\eta \colon \id \rightarrow \uniqueMapX_* \uniqueMapX^*$ for the adjunction unit, we obtain a collapse map as the composite
    \begin{equation*}
        \collapse_\xi(t) \colon \unit_{\udl{\sC}} \xlongrightarrow{\eta}
        \uniqueMapX_* \uniqueMapX^* \unit_{\udl{\sC}} \xlongrightarrow{t}
        \uniqueMapX_!(\xi \otimes \uniqueMapX^* \unit_{\udl{\sC}}) \simeq \uniqueMapX_!\xi.    
    \end{equation*}
\end{cons}

\begin{lem}\label{lem:collapse_of_capping}
    There is an equivalence $\collapse_{\xi} (\ambi{c}{\xi}{-}) \simeq c\in\map_{\udl{\sC}}(\unit_{\udl{\sC}},\uniqueMapX_!\xi)$. 
\end{lem}

\begin{proof}
    Consider the commutative diagram
    \begin{equation*}
    \begin{tikzcd}
        \unit_{\udl{\sC}} \ar[r, "\eta"] \ar[d, "c"]
        & \uniqueMapX_* \uniqueMapX^* \unit_{\udl{\sC}} \ar[d, "c \otimes \--"]
        &
        &
        \\
        \uniqueMapX_! \xi \ar[r, "\-- \otimes \eta"]
        & \uniqueMapX_! \xi \otimes \uniqueMapX_* \uniqueMapX^* \unit_{\udl{\sC}}
        &
        &
        \\
        \uniqueMapX_!(\xi \otimes \uniqueMapX^* \unit_{\udl{\sC}}) \ar[u, "\simeq", "\projectionformula^X"'] \ar[r, "\uniqueMapX^*\eta"]
        & \uniqueMapX_!(\xi \otimes \uniqueMapX^* \uniqueMapX_* \uniqueMapX^* \unit_{\udl{\sC}}) \ar[u, "\simeq", "\projectionformula^X"'] \ar[r, "\epsilon_{\uniqueMapX^*}"]
        & \uniqueMapX_!(\xi \otimes \uniqueMapX^* \unit_{\udl{\sC}}) \simeq \uniqueMapX_! \xi.
    \end{tikzcd}
    \end{equation*}
    The composite $\unit_{\udl{\sC}} \to \uniqueMapX_! \xi$ going through the upper right corner of the rectangle is by definition equal to $\collapse_{\xi} (\ambi{c}{\xi}{-})$.
    The composite $\unit_{\udl{\sC}} \to \uniqueMapX_! \xi$ going through the bottom left corner of the rectangle is equivalent to $c$ using the triangle identity $\epsilon_{\uniqueMapX^*} \circ \uniqueMapX^* \eta \simeq \id$.
\end{proof}

\subsubsection*{Intertwining capping with module maps}
As we shall see throughout the article, the capping maps produced from Spivak data often intertwine the left and right Beck--Chevalley transformations. Our aim now is to give the first expression of this principle in the form of \cref{prop:intertwining_principle_for_capping}, the other one being \cref{lem:umkehr_map_degree_one}.

\begin{setting}[Module pushforwards from multiplicative basechanges]\label{setting:master_setting}
    Suppose we have:
    \begin{itemize}
        \item symmetric monoidal $\B$--categories $\underline{\sC}, \underline{\D}$,
        \item a symmetric monoidal parametrised colimit--preserving functor $U\colon \underline{\sC}\rightarrow \underline{\D}$ as well as a $\underline{\sC}$--linear functor $F\colon \underline{\sC}\rightarrow \underline{\D}$  using the $\underline{\sC}$--linear structure on $\underline{\D}$ coming from $U$,
        \item a map $r\colon \underline{J}\rightarrow \underline{K}$  in $\cat_{\B}$ (to disambiguate notations, we will write $\rho \coloneqq r$ when we use it in the context of the category $\underline{\D}$),
        \item $\underline{\sC}$ and $\udl{\D}$ admit left Kan extensions along $\udl{J}\rightarrow{\udl{K}}$.
    \end{itemize}
   For  $(\xi,c)$  a $\underline{\sC}$--Spivak datum for $r$,   we define \[\big(\zeta,\:\: d\big)\coloneqq\Big( U(\xi),\:\:  U(c)\colon \unit_{\udl{\category{D}}}\rightarrow U(r_!\xi)\simeq \rho_!\zeta\Big)\] as the associated $\underline{\D}$--Spivak datum for $\rho$. From the data above, we also obtain symmetric monoidal functors $U\colon \underline{\sC}^{\underline{K}}\rightarrow\udl{\sC}^{\udl{J}}$ and $U\colon \underline{\D}^{\underline{K}}\rightarrow\udl{\D}^{\udl{J}}$, using which we may upgrade the functors $F\colon \underline{\sC}^{\underline{K}}\rightarrow \underline{\D}^{\underline{K}}, F\colon \underline{\sC}^{\underline{J}}\rightarrow \underline{\D}^{\underline{J}}$ to a $\underline{\sC}^{\underline{K}}$-- and a $\underline{\sC}^{\underline{J}}$--linear one, respectively. Note that by virtue of $\underline{\sC}$--linearity in all its guises as explained in the previous sentence, we have for any $A\in \{\underline{\sC},\underline{\sC}^{\underline{J}}, \underline{\sC}^{\underline{K}}\}$ a natural map $UA\otimes F(-)\rightarrow F(A\otimes-)$ which is an equivalence. Furthermore, note also that we clearly have equivalences $\rho^*F\simeq Fr^*$. Since $U$ was parametrised colimit preserving, we have an equivalence $Ur_!\simeq \rho_!U$.
\end{setting}

\begin{example}\label{example:master_setting}
    The following will be the examples of the abstract \cref{setting:master_setting} that will be important for us:
    \begin{enumerate}[label=(\alph*)]
        \item In the case $F=U$, the $\underline{\sC}$--linear structure on $F=U$ will be given by the symmetric monoidality structure $UA\otimes U(-)\xrightarrow{\simeq} U(A\otimes-)$;
        \item In the case when $\underline{\D}=\underline{\sC}$, $U=\id_{\underline{\sC}}$, and $F = a\otimes-$ for some fixed object $a\in \underline{\sC}$, the $\underline{\sC}$--linear structure on $F$ is the tautological one given by $\id(A)\otimes a\otimes- \simeq a\otimes A\otimes -$ coming from the symmetric monoidal structure on $\underline{\sC}$.
    \end{enumerate}
\end{example}

\begin{lem}\label{lem:beck_chevalley_and_linearity}
    Suppose we are in the \cref{setting:master_setting}. For all $A\in\underline{\sC}^{\underline{J}}$, writing $B\coloneqq U(A)\in\underline{\D}^{\underline{J}}$, we have a commuting diagram
    \begin{center}
        \begin{tikzcd}
             F(-)\otimes {\rho}_!B  && F(-\otimes r_!A)  \ar[ll,"\mathrm{linearity}"', "\simeq", leftarrow]\\
             && Fr_!(r^*(-)\otimes A)\ar[u,"F(\beckChevalley_!)"']\\
             {\rho}_!(\rho^*F(-)\otimes B)\ar[uu,"\beckChevalley_!"]&& \rho_!F(r^*(-)\otimes A)\ar[ll,"\rho_!(\mathrm{linearity})"', "\simeq", leftarrow]\ar[u, "\beckChevalley_!"']
        \end{tikzcd}
    \end{center}
\end{lem}
\begin{proof}
    Let $x\in \udl{\sC}^{\udl{K}}$ be an arbitrary object. Consider the  diagram
    \begin{center}
        \begin{tikzcd}
            & & & \udl{\D}^{\udl{K}}\ar[dd, "\rho^*"]\\
            \udl{\sC}^{\udl{K}} \ar[rr, "x\otimes -"'{xshift=0pt}, crossing over]\ar[dd, "r^*"'{yshift=0pt}]\ar[urrr, "Fx\otimes U(-)"]& & \udl{\sC}^{\udl{K}}\ar[ur, "F"']\\
            & & & \udl{\D}^{\udl{J}}\\
            \udl{\sC}^{\udl{J}}\ar[urrr,"\rho^*Fx\otimes U(-)"] \ar[rr, "r^*x\otimes -"'] & & \udl{\sC}^{\udl{J}} \ar[uu, "r^*"{yshift=0pt}, leftarrow, crossing over]\ar[ur, "F"']
        \end{tikzcd}
    \end{center}
    where the commuting triangles come from the $\underline{\sC}$--linearity of the functor $F$ with the $\udl
    {\sC}$--linear structure on $\underline{\D}$ coming from the symmetric monoidal colimit--preserving functor $U\colon \underline{\sC}\rightarrow\underline{\D}$.  By passing to the left adjoints $r_!\dashv r^*$ and $\rho_!\dashv \rho^*$ of the vertical functors and Beck--Chevalley pasting \cite[Lem. 2.2.4]{CarmeliSchlankYanovski2022}, we obtain the required commuting diagram.
\end{proof}

\begin{obs}
    A funny consequence of the preceding lemma is that if we supposed that $\underline{\sC}$ satisfied the $r$--projection formula and $\underline{\D}$ the $\rho$--projection formula so that the left vertical $\beckChevalley_!$ map and $F(\beckChevalley_!)$ are equivalences, then  $\beckChevalley_!\colon \rho_!F(r^*(-)\otimes A)\rightarrow Fr_!(r^*(-)\otimes A)$ is automatically an equivalence.
\end{obs}

\begin{prop}[Linear intertwining principle]\label{prop:intertwining_principle_for_capping}
    Suppose we are as in \cref{setting:master_setting} and that $\udl{\sC}$ and $\udl{\D}$ admit right Kan extensions along $\udl{J}\rightarrow\udl{K}$. Then we have a commuting square
    \begin{center}
        \begin{tikzcd}
            Fr_*\ar[d, "\beckChevalley_*"'] \ar[rr,"F(c\cap-)"]&& Fr_!(\xi\otimes-)\\
            \rho_*F \ar[r,"d\cap F"']& \rho_!(\zeta\otimes F-)\rar["\simeq", "\mathrm{linearity}"'] & \rho_!F(\xi\otimes-)\uar["\beckChevalley_!"']
        \end{tikzcd}
    \end{center}
\end{prop}
\begin{proof}
    Consider the following large commuting diagram
    \begin{center}\scriptsize
        \begin{tikzcd}
            Fr_*(-)\ar[ddd, "\beckChevalley_*"']\rar["F(\id\otimes c)"]&F(r_*(-)\otimes r_!\xi)\ar[rdd, phantom, "\mathrm{(A)}"]& Fr_!(r^*r_*(-)\otimes\xi)\lar["F(\beckChevalley_!)"',"\simeq"]\rar["Fr_!(\epsilon\otimes\id)"]& Fr_!(-\otimes\xi)\\ 
            && \rho_!F(r^*r_*(-)\otimes \xi)\uar["\beckChevalley_!"]\\
            &Fr_*(-)\otimes\rho_!\zeta\dar["\beckChevalley_*\otimes\id"']\ar[uu,"\mathrm{linearity}", "\simeq"']& \rho_!(\rho^*Fr_*(-)\otimes\zeta)= \rho_!(Fr^*r_*(-)\otimes\zeta)\ar[dr, "\rho_!(F\epsilon\otimes\id)", "\mathrm{(B)}"'{xshift=-15pt}]
            \uar["\mathrm{linearity}", "\simeq"']\lar["\simeq", "\beckChevalley_!"'] \dar["\rho_!(\rho^*\beckChevalley_*\otimes\id)"']& \rho_!F(-\otimes\xi)\ar[uu, "\beckChevalley_!"']\\
            \rho_*F(-)\rar["\id\otimes d"']&\rho_*F(-)\otimes \rho_!\zeta& \rho_!(\rho^*\rho_*F(-)\otimes\zeta) \lar["\simeq"', "\beckChevalley_!"]\rar["\rho_!(\epsilon\otimes\id)"']& \rho_!(F(-)\otimes \zeta)\uar["\rho_!(\mathrm{linearity})"', "\simeq"]
        \end{tikzcd}
    \end{center}\normalsize
    where three of the squares clearly commute, square (A) commute by \cref{lem:beck_chevalley_and_linearity}, and triangle (B) commutes since the left triangle in the diagram
    \begin{center}
        \begin{tikzcd}
            \rho^*Fr_* \simeq Fr^*r_*\dar["\rho^*\beckChevalley_*"']\ar[dr, "F\epsilon"]& && Fr_*\ar[dr,"\beckChevalley_*"]\dar["\beckChevalley_*"']\\
            \rho^*\rho_*F\rar["\epsilon_F"] & F && \rho_*F\rar[equal] & \rho_*F
        \end{tikzcd}
    \end{center}
    is adjoint to the right one, which clearly commutes.  Now we may take the outer square of the large diagram to conclude.
\end{proof}

\subsection{Twisted ambidexterity and Poincar\'{e} duality}\label{subsection:parametrised_twisteed_ambidex}\label{subsection:parametrised_poincare_duality}
Our aim in this subsection is to introduce the notion of Poincar\'{e} duality for Spivak data. To this end, it would be beneficial first to isolate a property that we will demand Poincar\'{e} Spivak data to satisfy, namely that of \textit{twisted ambidexterity}, i.e. that the associated capping map is an equivalence. This notion gives the equivalence of homology with cohomology necessary for Poincaré duality. While our definition makes sense in more generality -- a level of flexibility we will need for some of our applications -- we show in \cref{rmk:bastiaan_twisted_ambidex_agrees_with_ours} that our notion of twisted ambidexterity nevertheless coincides with the one given in \cite{Cnossen2023} for presentably symmetric monoidal coefficient categories.

For this subsection, we consider $\udl{X} \in \category{B}$ and $\udl{\category{C}}$ a symmetric monoidal $\category{B}$--category which admits $\udl{X}$-shaped limits and colimits and satisfies the $\udl{X}$-projection formula.
Notice that these conditions are satisfied whenever $\udl{\category{C}}$ is a presentably symmetric monoidal $\category{B}$--category.

\subsubsection*{Twisted ambidexterity}

\begin{defn}\label{def:twisted_ambidexterity}
    A $\udl{\sC}$--Spivak datum $(\xi,c)$ for $\obj{X}$
    is \textit{twisted ambidextrous} if the capping transformation 
    $\ambi{c}{\xi}{(-)} \colon \uniqueMapX_*(-) \rightarrow \uniqueMapX_!(\xi \otimes -)$ from \cref{cons:capping_map} is an equivalence.
\end{defn}

There is also the following relative version of this definition.
Recall that associated to an object $\obj{Y} \in \category{B}$ there is the basechange adjunction $\pi_Y^* \colon \category{B} \rightleftharpoons \category{B}_{/Y} \cocolon (\pi_Y)_*$.

\begin{defn}(Twisted ambidextrous maps)\label{defn:twisted_ambidex_maps}
    Consider a map $f \colon \obj{X} \to  \obj{Y}$ in $\category{B}$ and a symmetric monoidal $\category{B}$--category $\udl{\category{C}}$ such that the $\category{B}_{/Y}$--category $(\pi_Y)^* \udl{\category{C}}$ admits $f$ shaped limits and colimits and satisfies the $f$-projection formula. 
    A \textit{$\udl{\category{C}}$-Spivak datum} for $f$ is a $(\pi_Y)^* \udl{\category{C}}$-Spivak datum for $f \in \category{B}_{/Y}$.
    We say that such a Spivak datum exhibits $f$ as a \textit{$\udl{\category{C}}$--twisted ambidextrous map} if it exhibits $f \in \category{B}_{/Y}$ as $(\pi_Y)^* \udl{\category{C}}$--twisted ambidextrous object.
\end{defn}

We will see in \cref{prop:Poincare_duality_descent} that $f$ being $\udl{\category{C}}$--twisted ambidextrous is closely related to the fibres of $f$ being $\udl{\category{C}}$--twisted ambidextrous, see also \cite[Prop. 3.13]{Cnossen2023}.

Next we set out to show that in the presentable case, twisted ambidextrous Spivak data are unique and demonstrate that our notion of twisted ambidexterity is equivalent to the one defined in \cite[Def. 3.4]{Cnossen2023}.

\begin{lem}\label{lem:adjunction_unit_for_twisted_ambi_spivak_data}
    Let $(\xi,c)$ be a twisted ambidextrous $\underline{\sC}$--Spivak datum for $\underline{X}\in\B$. The adjunction $\uniqueMapX^*\dashv \uniqueMapX_*$ induces an adjunction $\uniqueMapX^*\dashv \uniqueMapX_!(\xi\otimes-)$ whose unit is given by 
        \[\id(-)\xlongrightarrow{\id\otimes c}\id(-)\otimes \uniqueMapX_!\xi\xleftarrow[\simeq]{\beckChevalley_!} \uniqueMapX_!(\uniqueMapX^*(-)\otimes\xi) = \uniqueMapX_!(-\otimes\xi)\circ \uniqueMapX^*(-),\]
\end{lem}
\begin{proof}
    That the equivalence $\uniqueMapX_*(-)\simeq \uniqueMapX_!(\xi\otimes-)$ induces an adjunction $\uniqueMapX^*\dashv \uniqueMapX_!(\xi\otimes-)$ is clear. For the description of the adjunction unit, observe that we have the commuting diagram
    \begin{center}
        \begin{tikzcd}[scale=0.5]
            (-) \rar["\id\otimes c"]\dar["\eta"']& (-)\otimes \uniqueMapX_!\xi \dar["\eta\otimes\id"']& \uniqueMapX_!(\uniqueMapX^*(-)\otimes \xi) \lar["\simeq","\beckChevalley_!"']\dar["\uniqueMapX_!(\uniqueMapX^*\eta\otimes\id)"']\ar[drr, equal]\\
            \uniqueMapX_*\uniqueMapX^*(-) \rar["\id\otimes c"'] & \uniqueMapX_*\uniqueMapX^*(-)\otimes \uniqueMapX_!\xi & \uniqueMapX_!(\uniqueMapX^*\uniqueMapX_*\uniqueMapX^*(-)\otimes \xi)\lar["\simeq"',"\beckChevalley_!"]\ar[rr,"r_!(\epsilon_{\uniqueMapX^*}\otimes\id)"'] && \uniqueMapX_!(\uniqueMapX^*(-)\otimes \xi)
        \end{tikzcd}
    \end{center}\normalsize
    where the bottom composite is the capping equivalence and the right triangle is by the triangle identity.    This shows that the claimed map is compatible with the unit $\eta\colon \id\rightarrow \uniqueMapX_*\uniqueMapX^*$ under the capping equivalence $c\cap -\colon \uniqueMapX_*(-)\xrightarrow{\simeq} \uniqueMapX_!(\xi\otimes-)$ as required.
\end{proof}

\begin{obs}
    Let $\udl{\sC}\in\calg(\presentable^L_{\B})$. If $\uniqueMapX^*\colon \udl{\sC}\rightarrow \udl{\sC}^{\udl{X}}$ is an internal left adjoint in $\module_{\udl{\category{C}}}(\presentable^L_{\category{B}})$, then its right adjoint must be  of the form $\uniqueMapX_!(D_{\obj{X}} \otimes \--)$ for a unique $D_X$ by \cref{thm:classification_of_C_linear_functors}
\end{obs}

\begin{prop}[The presentable case]\label{prop:twisted_ambidextrous_presentable_case}
    Let $\udl{\category{C}} \in \calg(\presentable^L_{\category{B}})$ be a presentably symmetric monoidal $\category{B}$--category and $\udl{X} \in \category{B}$.
    \begin{enumerate}[label=(\arabic*)]
        \item If $\uniqueMapX^*$ is an internal left adjoint in $\module_{\udl{\category{C}}}(\presentable^L_{\category{B}})$ with right adjoint  $\uniqueMapX_!(D_{\obj{X}} \otimes \--)$,
        then the unit map $c: \unit_{\udl{\sC}} \to \uniqueMapX_!(\uniqueMapX^* \unit_{\udl{\sC}} \otimes D_{\obj{X}}) = \uniqueMapX_! D_{\obj{X}}$ forms a $\udl{\category{C}}$--twisted ambidextrous Spivak datum $(D_{\obj{X}},c)$ for $\obj{X}$.
        \item If $(\xi, c)$ is a $\udl{\category{C}}$-twisted ambidextrous Spivak datum for $\obj{X}$, then the map \[(\--) \xrightarrow{\id\otimes c} (\--) \otimes \uniqueMapX_! \xi \simeq \uniqueMapX_!( \uniqueMapX^*(\--) \otimes \xi)\] is the unit map of a $\udl{\category{C}}$-linear adjunction $\uniqueMapX^* \dashv \uniqueMapX_!(\-- \otimes \xi)$.        
    \end{enumerate}
    In particular, if $(\xi,c)$ and $(\xi',c')$ are twisted ambidextrous Spivak data, then there is an equivalence $\xi \simeq \xi'$ so that the composition $\unit_{\underline{\category{C}}} \xrightarrow{c} \uniqueMapX_!\xi \simeq \uniqueMapX_!\xi'$ is equivalent to $c'$.
\end{prop}
\begin{proof}
    For point (1),  suppose that $\uniqueMapX^*$ is an internal left adjoint in $\module_{\udl{\category{C}}}(\udl{\presentable}_{\category{B}}^L)$.
    This means that $\uniqueMapX_*$ has a $\udl{\category{C}}$-linear structure making the adjunction $\uniqueMapX^* \dashv \uniqueMapX_*$ into a $\udl{\category{C}}$-linear one.
    By \cref{thm:classification_of_C_linear_functors}, there is an object $D_{\obj{X}} \in \udl{\category{C}}^{\udl{X}}$ together with a $\udl{\category{C}}$-linear equivalence $\phi \colon \uniqueMapX_* \simeq \uniqueMapX_!(\-- \otimes D_{\obj{X}})$.
    In light of the $\underline{\sC}$--linearity of the adjunction $\uniqueMapX^* \dashv \uniqueMapX_*$ and $\uniqueMapX_! \dashv \uniqueMapX^*$, we see that the capping transformation
    \begin{equation}\label{eq:capping_equivalence}
            \ambi{c}{}{(-)} \colon \uniqueMapX_*(\--)\xrightarrow{\-- \otimes c} \uniqueMapX_*(\--) \otimes \uniqueMapX_! D_{\udl{X}} \xleftarrow{\simeq} \uniqueMapX_!( \uniqueMapX^* \uniqueMapX_*(\--) \otimes D_{\udl{X}}) \to \uniqueMapX_!(\-- \otimes D_{\udl{X}}) 
    \end{equation}
    refines to a $\underline{\sC}$--linear transformation because each constituent map refines canonically to a $\underline{\sC}$--linear transformation: the first map is clear; the second map  is so since the Beck--Chevalley equivalence $\uniqueMapX_!(\uniqueMapX^*(-) \otimes -) \to -\otimes \uniqueMapX_!(-)$ is canonically a $\udl{\category{C}}$-linear equivalence; the third map is so since it is the counit to a  $\udl{\category{C}}$-linear adjunction $\uniqueMapX^* \uniqueMapX_* \to \id$.
    We claim that $\ambi{c}{}{(-)}$ is equivalent to the equivalence $\phi$, which would prove the statement.
    By standard adjunction arguments, it suffices  to show that the transformations $(-) \xrightarrow{\eta} X_*X^*(-)\xrightarrow{\ambi{c}{}{X^*(-)}} X_!(X^*(\--) \otimes D_X)$ and $(-) \xrightarrow{\eta} X_*X^*(-)\xrightarrow{\phi(X^*(-))} X_!(X^*(\--) \otimes D_X)$ are equivalent.
    Employing \cref{thm:classification_of_C_linear_functors}, we can test this after evaluating at $\unit_{\udl{\category{C}}}$.
    Now the composite $\unit_{\udl{\category{C}}} \to X_*X^* \unit_{\udl{\category{C}}} \xrightarrow{\phi} X_! D_X$ is by definition the collapse map $c$.
    The composite $\unit_{\udl{\category{C}}} \to X_*X^* \unit_{\udl{\category{C}}} \xrightarrow{c \cap X^* \unit_{\udl{\category{C}}}} X_! D_X$ is also equivalent to $c$ by \cref{lem:collapse_of_capping}.

    Next, for point (2), suppose that $(\xi, c)$ is a $\udl{\category{C}}$-twisted ambidextrous Spivak datum for $\obj{X}$.
    Then $\uniqueMapX_*$ is $\udl{B}$-colimit preserving.
    First, we check the condition in \cite[Prop. A.5]{Cnossen2023} which guarantees that the adjunction $\uniqueMapX^*\dashv \uniqueMapX_*$ is $\underline{\sC}$--linear.
    For this, we need to show that for $a\in\underline{\sC}$ and $E\in\underline{\sC}^{\underline{X}}$, the Beck--Chevalley map $\beckChevalley_*\colon a\otimes \uniqueMapX_*E\rightarrow \uniqueMapX_*(\uniqueMapX^*a\otimes E)$ is an equivalence.
    By the intertwining square in \cref{prop:intertwining_principle_for_capping} applied to \cref{example:master_setting} (b), we see that $\beckChevalley_*$ is an equivalence because $\beckChevalley_! \colon \uniqueMapX_!(\uniqueMapX^* a \otimes E \otimes \xi) \to a \otimes \uniqueMapX_!(E \otimes \xi)$ is an equivalence by presentably symmetric monoidality of $\udl{\category{C}}$.
    As in part (1) we see that the capping equivalence \cref{eq:capping_equivalence} refines to a $\udl{\category{C}}$-linear equivalence from which we obtain a $\udl{\category{C}}$-linear adjunction $\uniqueMapX^* \dashv \uniqueMapX_!(\-- \otimes \xi)$
    The claimed description of the adjunction unit comes from \cref{lem:adjunction_unit_for_twisted_ambi_spivak_data}.

    For the final statement, since both $\uniqueMapX_!(\xi \otimes -)$ and $\uniqueMapX_!(\xi' \otimes -)$ are $\udl{\category{C}}$-linear right adjoints to $\uniqueMapX^*$ by (2),we see by (1) that there is an equivalence $\xi \simeq D_X \simeq \xi'$. To see the coincidence of $c$ and $c'$, we use \cref{lem:collapse_of_capping} to obtain the two commuting triangles in
    \begin{center}
        \begin{tikzcd}
            \uniqueMapX_!\xi' && \uniqueMapX_*\uniqueMapX^*\unit \ar[rr,"c\cap_{\xi}\uniqueMapX^*\unit", "\simeq"']\ar[ll,"c'\cap_{\xi'}\uniqueMapX^*\unit"', "\simeq"]&& \uniqueMapX_!\xi\\
            && \unit \uar["\eta"]\ar[urr, "c"']\ar[ull, "c'"]
        \end{tikzcd}
    \end{center}
    witnessing that $c\simeq c'$ as required.
\end{proof}

\begin{rmk}\label{rmk:bastiaan_twisted_ambidex_agrees_with_ours}
    By combining \cref{prop:twisted_ambidextrous_presentable_case} and \cite[Prop. 3.8]{Cnossen2023}, we see that $\obj{X}$ is $\udl{\category{C}}$-twisted ambidextrous in the sense of \cref{def:twisted_ambidexterity} if and only if it is so in the sense of \cite[Def. 3.4]{Cnossen2023}.
    If that is the case, the twisted norm map $\Nm{\obj{X}} \colon \uniqueMapX_!(\-- \otimes D_{\obj{X}}) \to \uniqueMapX_*(\--)$ constructed in \cite[Def. 3.3]{Cnossen2023} is an equivalence with inverse the map $\ambi{\Nm{\obj{X}}^{-1}(\unit)}{D_{\obj{X}}}{(\--)}$.
\end{rmk}

\begin{defn}
    Let $\udl{\category{C}}$ be a presentably symmetric monoidal $\category{B}$--category.
    An object of $\category{B}$ is called \textit{$\udl{\category{C}}$--twisted ambidextrous} if it admits a (necessarily unique) twisted ambidextrous Spivak datum with coefficients in $\udl{\category{C}}$.
\end{defn}

\begin{nota}
    The twisted ambidextrous Spivak datum of a twisted $\udl{\category{C}}$-ambidextrous object $\udl{X} \in \category{B}$ will be denoted by $(D^{\udl{\category{C}}}_{\udl{X}}, c)$. If $\udl{\category{C}}$ is clear from the context, we will sometimes abbreviate this to $(D_{\udl{X}}, c)$.
\end{nota}

\subsubsection*{Poincar\'{e} duality}

We now come to the definition of Poincaré duality in the parametrised setting.

\begin{defn}\label{defn:poincare_spivak_datum}
    A Spivak datum $(\xi,c)$ for $\obj{X}$ with coefficients in $\udl{\category{C}}$
    is \textit{Poincaré} if it is twisted ambidextrous and $\xi$ takes values in $\udl{\Pic}(\udl{\category{C}})$.
\end{defn}

\begin{defn}
    \label{defn:PD_object_with_psm_coefficients}
    Let $\udlcatC$ be a presentably symmetric monoidal $\category{B}$--category.
    An object $\obj{X} \in \category{B}$ is called \textit{$\udl{\category{C}}$--Poincar\'{e}} if it is twisted $\udl{\category{C}}$--ambidextrous and the unique twisted ambidextrous Spivak datum $(D_{\obj{X}}, c)$ from \cref{prop:twisted_ambidextrous_presentable_case} is Poincar\'{e}.
\end{defn}

\begin{rmk}
    In \cite{Quinn_Normal_Spaces}, Quinn defines the notion of a \textit{normal space} to be a space together with (the unstable analog of) a Spivak datum $(\xi,c)$, where $\xi$ takes values in $\Pic(\spectra)$. He does not require the Spivak datum to be twisted ambidextrous though. 
\end{rmk}

We again have the following relative version.

\begin{defn}\label{def:pd_map}(Poincar\'{e} duality maps)
    Consider a map $f \colon X \to Y$ in $\category{B}$ and a symmetric monoidal $\category{B}$--category $\udl{\category{C}}$ such that the $\category{B}_{/Y}$--category $(\pi_Y)^* \udl{\category{C}}$ admits $f$-shaped limits and colimits and satisfies the $f$-projection formula. 
    We say that a $\udl{\category{C}}$-Spivak datum for $f$ exhibits $f$ as a \textit{$\udl{\category{C}}$--Poincar\'e duality map} if it exhibits $f \in \category{B}_{/Y}$ as a $(\pi_Y)^* \udl{\category{C}}$--Poincar\'{e} duality object.
\end{defn}

Using Costenoble-Waner duality, one can show the following standard result saying that dualisability of the dualising object implies its invertibility.
We will not use it anywhere in the rest of this article but include it for completeness.
In the setting of $\category{B}$-categories, Costenoble-Waner duality was introduced in \cite[Section 3.3]{Cnossen2023} and we follow the notation used there. 
In the nonparametrised context, the following result appears in \cite[Remark A.9]{markusPoincareSpaces}.

\begin{prop}[Invertibility of dualising objects]
    Let $\udl{\category{C}}$ be a presentable symmetric monoidal $\category{B}$--category.
    Suppose that $\obj{X} \in \category{B}$ is $\udl{\category{C}}$-twisted ambidextrous and that $D_{\obj{X}} \in \udl{\category{C}}^{\udl{X}}$ is dualisable.
    Then $D_{\obj{X}}$ is invertible, i.e. $\obj{X}$ is Poincar\'e.
\end{prop}
\begin{proof}
    By \cite[Proposition 3.29]{Cnossen2023}, the unit $\unit_{X} \in \category{C}({\obj{X}} \times *)$ is left Costenoble-Waner dualisable with left dual $D_{\obj{X}} \in \category{C}(* \times {\obj{X}})$.
    By \cite[Proposition 3.30]{Cnossen2023}, this implies that for $F \in \category{C}({\obj{X}} \times Y)$ and $E \in \category{C}(Y)$ we have equivalences
    \begin{align*}
        \Hom(E, X_! F) &\simeq \Hom(E, F \odot \unit_{X}) 
        \simeq \Hom(E \odot D_{\obj{X}}, F) \\
        &= \Hom(X^* E \otimes D_{\obj{X}}, F) 
        \simeq \Hom(X^* E, D_{\obj{X}}^\vee \otimes F) \\
        &\simeq \Hom(E, X_*(D_{\obj{X}}^\vee \otimes F)) \simeq \Hom(E, X_!(D_{\obj{X}} \otimes D_{\obj{X}}^\vee \otimes F)).
    \end{align*}
    giving a $\udl{\category{C}}$--linear equivalence $X_!(-) \simeq X_!(- \otimes D_{\obj{X}} \otimes D_{\obj{X}}^\vee)$.
    It now follows from \cref{thm:classification_of_C_linear_functors} that $D_{\obj{X}} \otimes D_{\obj{X}}^\vee \simeq \unit_X$ so that $D_{\obj{X}}$ is invertible.
\end{proof}

\begin{example}\label{example:examples_of_ambidex_point_semiadditive}
    The phenomenon of higher semiadditivity introduced by \cite{HopkinsLurie2013} provides many instances of Poincar\'{e} duality with trivial dualising sheaf. 
    \begin{enumerate}[label=(\arabic*)]
        \item For any topos $\category{B}$ and any symmetric monoidal $\category{B}$--category $\udlcatC$, the terminal object $\obj{*}$ has the tautological Poincar\'e $\udl{\sC}$--Spivak datum $(\unit, \id_{\unit})$.
        \item If $\underline{\sC}$ is pointed, then by \cite[Rmk. 4.4.6]{HopkinsLurie2013}, the map $\udl{\varnothing}\rightarrow \udl{X}$ in $\B$  is $\underline{\sC}$--Poincar\'{e}. 
        \item If $\underline{\sC}$ is semiadditive, then by \cite[Prop. 4.4.9]{HopkinsLurie2013}, any finite fold map $\nabla\colon \coprod_{i=1}^n\udl{X}\rightarrow \udl{X}$ is $\underline{\sC}$--Poincar\'{e}. 
        \item More generally, a good supply of Poincar\'e spaces with trivial dualising sheaf comes from the theory of higher semiadditivity of \cite{HopkinsLurie2013,CarmeliSchlankYanovski2022}, as worked out in \cite{Cnossen2023}.
    \end{enumerate}
\end{example}

\begin{example}[Wall's Poincar\'{e} complexes]\label{example:wall_poincare_complex}
    Next, we recount some parts of the classical story that began from Wall's seminal paper \cite{WallPD}. In this setting, our base topos $\B$ will be the category  $\spc$ of spaces. Wall defined a Poincar\'e complex (he used the word complex, because he worked with CW-complexes) to be a compact space $X$ together a Spivak datum $(\xi \in \Pic(\module_{\eilenbergMacLaneCoeff})^X, c \colon \eilenbergMacLaneCoeff \rightarrow \uniqueMapX_!\xi )$ such that for each $\psi \in (\module_{\eilenbergMacLaneCoeff}\hearttstructure)^X$ the map
    \begin{equation}
    \label{eq:classical_PD_Wall}
    \ambi{c}{\xi}{\psi} \colon \uniqueMapX_*\psi \longrightarrow \uniqueMapX_!(\xi \otimes \psi)
    \end{equation}
    is an equivalence. As $X$ was assumed to be compact, both sides of \cref{eq:classical_PD_Wall} commute with all (co)limits and so this also implies that the same transformation is an equivalence for arbitrary $\psi \in \module_\bbZ$.  On the other hand, to compute the value of the $\spectra$--dualising sheaf $D_X$ of a space $X$ at a point $x \colon * \rightarrow X$, one calculates
    \[ D_X(x) = x^*D_{X} \simeq X_!x_!x^*D_{X}\simeq  \uniqueMapX_!(D_X \otimes x_!\sphere) \simeq \uniqueMapX_*x_!\sphere. \]
    Note that $x_!$ preserves connective objects while $\uniqueMapX_*$ preserves bounded below objects if it is a retract of a space admitting a finite-dimensional cell structure . So we see that if $X$ is compact (i.e. a retract of a space having a finite cell structure), then $D_X$ is pointwise bounded below. This implies that if $D_X \otimes \bbZ \in \Pic(\module_\bbZ)^X$, then $D_X$ is pointwise given by shifts of spheres and in particular, $D_X \in \Pic(\spectra)^X$. In conclusion,  by combining the points above, a space $X$ is a Poincar\'e complex in the sense of Wall if and only if it is compact and $\spectra$--Poincar\'e in the sense of \cref{defn:PD_object_with_psm_coefficients}. See also \cite[Prop. A.12]{markusPoincareSpaces} for a proof in the case of finite spaces.
\end{example}

\begin{example}[Weak Poincar\'{e} spaces]
    After Wall, some authors subsequently relaxed the compactness condition in the definition of Poincar\'e complexes. For example, in group theory it is not unusual to completely drop it. We say that a space $X$ is weakly Poincar\'e if it admits a Spivak datum $(\xi \in \Pic(\module_{\eilenbergMacLaneCoeff})^X, c \colon \eilenbergMacLaneCoeff \rightarrow \uniqueMapX_!\xi )$ such that for each $\psi \in \module_{\eilenbergMacLaneCoeff}\hearttstructure$ the map in \cref{eq:classical_PD_Wall} is an equivalence. As $\uniqueMapX_*$ preserves coconnectivity, $\uniqueMapX_!$ preserves connectivity, and both preserve fibre sequences, we see that they restrict to functors
    \[ \uniqueMapX_*, \uniqueMapX_!(- \otimes \xi) \colon (\module_{\eilenbergMacLaneCoeff}^b)^X \rightarrow \module_{\eilenbergMacLaneCoeff}^b \]
    where $\module_\bbZ^b$ denotes the category of bounded $\bbZ$--chain complexes. Being weakly Poincar\'e is seen to be equivalent to admitting a Poincar\'e Spivak datum in the sense of \cref{defn:poincare_spivak_datum} with respect to the symmetric monoidal stable category $\module_{\eilenbergMacLaneCoeff}^b$.
\end{example}

\begin{rmk}
    Having now established the three notions and implications
    \[ \text{$\spectra$--Poincar\'e and compact} \implies \text{$\spectra$--Poincar\'e } \implies \text{weakly Poincar\'e},  \]
    we cannot give a conclusive answer about their precise relation. In \cite{Browder_Poincare}, Browder notes that if $X$ is weakly Poincar\'e with finitely presented fundamental group, then it is even compact, and so by \cref{example:wall_poincare_complex}, also $\spectra$--Poincar\'e. On the other hand, Davis shows in \cite{Davis98} that there are weakly Poincar\'e spaces  whose fundamental groups do not admit a finite presentation.
\end{rmk}

From here on, we will reserve the term \textit{Poincar\'e space} for what we referred to as $\spectra$-Poincar\'e spaces above. In particular, we slightly deviate from Wall's definition.
It is useful to try and port concepts from manifold theory to the theory of Poincar\'e spaces. One concept that has a straighforward analog for Poincar\'e spaces is the dimension of a manifold.

\begin{terminology}[Formal dimensions]\label{terminology:Poincare_dimension}
    Let $X\in\spc$ be a Poincar\'{e}  space. 
    We say that it has \textit{formal dimension $d$} if for every point $x\colon \ast\rightarrow X$, we have $x^*D_X\simeq \Sigma^{-d}\sphere\in\picardSpace(\spectra)$. 
    If for every point $x\colon\ast\rightarrow X$, we have  $x^*D_X\simeq \Sigma^{-k}\sphere$ for some $0 \leq k \leq d$, then we will say that it has \textit{formal dimension at most $d$}.
\end{terminology}

\begin{fact}\label{fact:PD_of_low_dimensions}
    Here are some classical facts about nonequivariant Poincar\'{e} spaces that will be relevant to our investigations later.
    \begin{enumerate}[label=(\arabic*)]
        \item Let $X\in\spc^{\omega}$ be a connected Poincar\'{e}  space of formal dimension $d=0$. Then by \cite[Thm. 4.2]{WallPD}, we have $X\simeq \ast$.  In fact, in the aforementioned theorem, Wall even provided classifications of Poincar\'{e} spaces up to formal dimension 3.
        \item Every connected Poincar\'{e} space has formal dimension a nonnegative number. 
        This is since if $X$ has formal dimension $d$, then taking $\mathbb{F}_2$--homology, we get $H_0(X;\mathbb{F}_2)\cong H^d(X;\mathbb{F}_2)$. 
        Thus if $d<0$, then $H_0(X;\mathbb{F}_2)=0$, i.e. $X$ was the empty space.
    \end{enumerate}
\end{fact}

\subsection{Constructions with Spivak data}\label{subsection:constructions_with_spivak_data}

This subsection constitutes the heart of our parametrised Poincar\'{e} duality theory. We begin by studying compositions of Spivak data. Next, we shall study two types of basechange results, namely basechanging coefficient categories (\cref{thm:base_change_of_tw_amb_spivak_data,thm:small_base_change}) and basechanging the underlying topos \cref{thm:omnibus_geometric_pushforward_of_capping}. These are the main abstract results of this article and they will play a fundamental role in much of our equivariant work in \cref{section:equivariant_PD_elements,section:geometric_applications}. We then end this subsection by proving a descent result for Poincar\'{e} duality.

\subsubsection*{Compositions}
\begin{cons}[Compositions of Spivak data]\label{cons:composition_of_Spivak_data}
    Let $f\colon \obj{X}\rightarrow \obj{Y}$ and $g\colon \obj{Y}\rightarrow \obj{Z}$ be maps in $\B$ equipped with Spivak data
    \[\Big(\xi_f\in\underline{\sC}^{\underline{X}},\quad \unit_{\underline{\sC}^{\underline{Y}}}\xlongrightarrow{c_f}f_!\xi_f\Big)\quad\quad\quad \Big(\xi_g\in\underline{\sC}^{\underline{Y}},\quad \unit_{\underline{\sC}^{\underline{Z}}}\xlongrightarrow{c_g}g_!\xi_g\Big)\]
    We may then define the \textit{composition Spivak datum for the map $gf\colon \obj{X}\rightarrow \obj{Z}$} as 
    \[\Big(\xi_{gf}\coloneqq \xi_f\otimes f^*\xi_g\in \underline{\sC}^{\underline{X}},\quad c_{gf}\colon \unit_{\underline{\sC}^{\underline{X}}}\xlongrightarrow{c_g}g_!\xi_g\xlongrightarrow{g_!(c_f\otimes\id)}g_!(f_!\xi_f\otimes \xi_g)\xleftarrow[{\simeq}]{\beckChevalley}(gf)_!(\xi_f\otimes f^*\xi_g)\Big)\]
\end{cons}

\begin{lem}[Capping map for compositions]\label{lem:capping_for_compositions}
    The capping map for the composition Spivak datum is equivalent to the composition of the constituent capping maps. That is, in the situation of \cref{cons:composition_of_Spivak_data}, we have a commuting diagram
    \begin{center}
        \begin{tikzcd}
            (gf)_*(-) \ar[rrr,"{\ambi{c_{gf}}{}{(-)}}"]\dar["{\ambi{c_g}{}{f_*(-)}}"']&&& (gf)_!(\xi_f\otimes f^*\xi_g)\dar["\beckChevalley", "\simeq"']\\
            g_!( f_*(\--) \otimes \xi_g) \ar[rrr,"g_!((\ambi{c_f}{}{(\--)}) \otimes \id)"'] &&& g_!(f_!(\-- \otimes \xi_f) \otimes \xi_g) 
        \end{tikzcd}
    \end{center}
\end{lem}
\begin{proof}
    First note that we have the  commuting diagram
    \begin{center}\footnotesize
        \begin{tikzcd}
            g_*f_*(-)\ar[d,"c_g\otimes\id"] &&\\
             g_!\xi_g\otimes g_*f_*(-) \dar["g_!(c_f\otimes\id)"']& g_!\big(\xi_g\otimes g^*g_*f_*(-)\big)\lar["\simeq","\beckChevalley"']\dar["g_!(c_f\otimes\id)"'] \rar["g_!(\id\otimes\epsilon)"] & g_!\big(\xi_g\otimes f_*(-)\big) \dar["g_!(c_f\otimes\id)"]\\
             g_!(\xi_g\otimes f_!\xi_f)\otimes g_*f_*(-) & g_!\big(\xi_g\otimes f_!\xi_f\otimes g^*g_*f_*(-)\big)\lar["\simeq","\beckChevalley"'] \rar["g_!(\id\otimes\epsilon)"]& g_!\big(\xi_g\otimes f_!\xi_f\otimes f_*(-)\big)\\
            & g_!f_!\big(f^*\xi_g\otimes \xi_f\otimes f^*g^*g_*f_*(-)\big)\phantom{xx} \uar["g_!(\beckChevalley)", "\simeq"']\rar["g_!f_!(\id\otimes\epsilon)"]\ar[dr, "g_!f_!(\id\otimes\epsilon)"'] & \phantom{xx} g_!f_!\big(f^*\xi_g\otimes \xi_f\otimes f^*f_*(-)\big) \uar["g_!(\beckChevalley)"', "\simeq"]\dar["g_!f_!(\id\otimes\epsilon)"]\\
            && g_!f_!(f^*\xi_g\otimes \xi_f\otimes-).
        \end{tikzcd}
    \end{center}\normalsize
    The required commuting square is then obtained by taking the outer diagram.
\end{proof}

\begin{prop}[Duality composition formula]\label{prop:Poincare_duality_composition}
    Let $f \colon \obj{X} \to \obj{Y}$ and $g \colon \obj{Y} \to \obj{Z}$ be  maps in $\category{B}$ and $\udl{\category{C}}$ be a symmetric monoidal $\category{B}$--category satisfying the $f$-- and $g$--projection formulas. Suppose $f$ and $g$ are equipped with Spivak data $(\xi_f, c_f)$ and $(\xi_g, c_g)$ respectively. Then under the composition Spivak datum on $gf\colon \obj{X}\rightarrow \obj{Z}$ from \cref{cons:composition_of_Spivak_data} with dualising sheaf 
    \[\xi_{gf}\coloneqq \xi_f\otimes f^*\xi_g,\] we have that:
    \begin{enumerate}[label=(\arabic*)]
        \item if $f$ and $g$ are twisted ambidextrous, then so is $gf$,
        \item if $f, g, gf$ are all twisted ambidextrous and $g$ is furthermore Poincar\'{e} duality, then $f$ is Poincar\'{e} duality if and onlf if $gf$ is.
    \end{enumerate}
\end{prop}
\begin{proof}
    By \cref{lem:capping_for_compositions}, we have the commuting square
    \begin{center}
        \begin{tikzcd}
            (gf)_*(-) \ar[rrr,"{\ambi{c_{gf}}{}{(-)}}"]\dar["{\ambi{c_g}{}{f_*(-)}}"']&&& (gf)_!(\xi_f\otimes f^*\xi_g)\dar["g_!\beckChevalley_!", "\simeq"']\\
            g_!( f_*(\--) \otimes \xi_g) \ar[rrr,"g_!((\ambi{c_f}{}{(\--)}) \otimes \id)"'] &&& g_!(f_!(\-- \otimes \xi_f) \otimes \xi_g). 
        \end{tikzcd}
    \end{center}
    Hence, if the left vertical and bottom horizontal maps are equivalences, then so is the top horizontal map. It is also clear that from the formula $\xi_{gf}=\xi_f\otimes f^*\xi_g$ that if two $\xi_g$ is invertible (and so also $f^*\xi_g$), then $\xi_{gf}$ is invertible if and only if $\xi_f$ is.
\end{proof}

\subsubsection*{Change of coefficients}

\begin{cons}[Basechanging Spivak data]\label{cons:basechanging_spivak_data}
    Let  $\udl{\category{C}}, \udl{\category{D}}$ be $\category{B}$--categories admitting $\obj{X}$-shaped (co)limits and satisfying the $\obj{X}$-projection formula. Suppose $F \colon \udl{\category{C}} \to \udl{\category{D}}$ is a symmetric monoidal functor of $\category{B}$--categories which preserves $\obj{X}$--shaped colimits. We define a new $\underline{\D}$--Spivak datum for $\obj{X}$ as follows
    \[ 
    F(\xi,c)  \coloneqq \big(F\xi \colon \obj{X} \xrightarrow{\xi} \udl{\category{C}} \xrightarrow{F} \udl{\category{D}}, 
    \:\: 
    Fc \colon \unit_{\udl{\D}} \simeq F(\unit_{\udl{\sC}}) \xrightarrow{Fc} F(\pointProjection_! \xi) \simeq \pointProjection_! F\xi\big). 
    \]
\end{cons}

\begin{thm}[Poincar\'{e} basechange - presentable version]
    \label{thm:base_change_of_tw_amb_spivak_data}
    Let $F \colon \udl{\category{C}} \rightarrow \udl{\category{D}}$ be a functor of presentably symmetric monoidal $\category{B}$--categories.
    Suppose that $(\xi,c)$ is a twisted ambidextrous Spivak datum with coefficients in $\udl{\category{C}}$ for the object $\obj{X} \in \category{B}$. Then $F(\xi,c)$ is a twisted ambidextrous Spivak datum with coefficients in $\udl{\category{D}}$ for $\obj{X}$. In particular, if $\obj{X}$ is $\udl{\category{C}}$--Poincar\'e, then $\obj{X}$ is also $\udl{\category{D}}$--Poincar\'e.
\end{thm}
\begin{proof}
    Recall from \cref{prop:base_change_module_categories} that $\-- \otimes_{\udl{\category{C}}} \udl{\category{D}} \colon \udl{\module}_{\udl{\category{C}}}(\udl{\presentable}^L_{\category{B}}) \to \udl{\module}_{\udl{\category{D}}}(\udl{\presentable}^L_{\category{B}})$ 
    is symmetric monoidal $\category{B}$-colimit preserving.
    Using that $\udl{\category{C}}^{\udl{X}} \in \module_{\udl{\category{C}}}(\presentable^L(\category{B}))$ is self dual (see \cite[Corollary 2.27]{Cnossen2023}), one sees that  the coassembly map $(\lim_{\udl{X}} \udl{\category{C}}) \otimes_{\udl{\category{C}}} \udl{\category{D}} \to (\lim_{\udl{X}} \udl{\category{C}}) \otimes_{\udl{\category{C}}} \udl{\category{D}}$ is an equivalence (even a symmetric monoidal one).
    The commutative diagram
    \begin{equation*}
    \begin{tikzcd}
        \udl{\category{C}}^{\udl{X}} \ar[r, "- \otimes_{\udl{\category{C}}} \xi"] \ar[d, "F"]
        & \udl{\category{C}}^{\udl{X}} \otimes \udl{\category{C}}^{\udl{X}} \ar[d, "F"] \ar[r, "\-- \otimes_{\udl{\category{C}}} \--"]
        & \udl{\category{C}}^{\udl{X}} \ar[d, "F"]
        \\
        \udl{\category{D}}^{\udl{X}} \ar[r, "- \otimes_{\udl{\category{D}}} F \xi"']
        & \udl{\category{D}}^{\udl{X}} \otimes \udl{\category{D}}^{\udl{X}} \ar[r, "\-- \otimes_{\udl{\category{C}}} \--"']
        & \udl{\category{D}}^{\udl{X}}
    \end{tikzcd}
    \end{equation*}
    together with the equivalence $\func_{\udl{\category{D}}}(\udl{\category{D}}^{\udl{X}}, \udl{\category{D}}^{\udl{X}}) \simeq \func_{\udl{\category{C}}}(\udl{\category{C}}^{\udl{X}}, \udl{\category{D}}^{\udl{X}})$ then gives us an equivalence $(- \otimes \xi) \otimes_{\udl{\category{C}}} \udl{\category{D}} \simeq (- \otimes F \xi)$.

    By standard arguments, the functor $\-- \otimes_{\udl{\category{C}}} \udl{\category{D}} \colon \module_{\udl{\category{C}}}(\presentable^L_{\category{B}}) \to \module_{\udl{\category{D}}}(\presentable^L_{\category{B}})$  preserves internal adjunctions.
    Hence, we see that $(\uniqueMapX_{\udl{\category{C}}})_! \otimes_{\udl{\category{C}}} \udl{\category{D}}\colon \udl{\sC}^{\udl{X}}\otimes_{\udl{\sC}}\udl{\D}\simeq \udl{\D}^{\udl{X}}\rightarrow \udl{\sC}\otimes_{\udl{\sC}}\udl{\D}\simeq \udl{\D}$ is an internal left adjoint of $(\uniqueMapX_{\udl{\category{D}}})^* = (\uniqueMapX_{\udl{\category{C}}})^* \otimes_{\udl{\category{C}}} \udl{\category{D}}$ from which we obtain an equivalence $(\uniqueMapX_{\udl{\category{C}}})_! \otimes_{\udl{\category{C}}} \udl{\category{D}} \simeq (\uniqueMapX_{\udl{\category{D}}})_!$.
    Together with the first part, the internal right adjoint $(\uniqueMapX_{\udl{\category{C}}})_!(\xi \otimes \--)$ to $(\uniqueMapX_{\udl{\category{C}}})^*$ basechanges to an internal right adjoint 
    \begin{equation*}
        (\uniqueMapX_{\udl{\category{C}}})_!(\xi \otimes \--) \otimes_{\udl{\category{C}}} \udl{\category{D}} \simeq (\uniqueMapX_{\udl{\category{D}}})_!(F \xi \otimes \--)\colon \udl{\D}^{\udl{X}}\longrightarrow \udl{\D}
    \end{equation*}
    of $(\uniqueMapX_{\udl{\category{D}}})^*$. 
    Because the internal adjunction $(\uniqueMapX_{\udl{\D}})^*\dashv (\uniqueMapX_{\udl{\D}})_!(F\xi\otimes-)$ on $\udl{\D}$ is basechanged from the internal adjunction $\uniqueMapX_{\udl{\sC}}^*\dashv (\uniqueMapX_{\udl{\sC}})_!(\xi\otimes-)$ on $\udl{\sC}$, we see that the $\udl{\D}$--fundamental class, which is the unit of the former internal adjunction, is given by the composite
    \begin{equation*}
        \unit_{\udl{D}} \simeq F(\unit_{\udl{\category{C}}}) \xlongrightarrow{Fc} F(\uniqueMapX_{\udl{\sC}})_!\xi \simeq (\uniqueMapX_{\udl{\D}})_! F\xi. 
    \end{equation*}
    The final statement about Poincar\'{e} duality is clear since $F$ is symmetric monoidal and so preserves invertibility.
\end{proof}

\begin{cor}\label{cor:transferring_colimit_BC_to_limit_BC_by_ambidexterity}
     Let  $\Phi \colon \udl{\category{C}} \rightarrow \udl{\category{D}}$ be a symmetric monoidal functor of presentably symmetric monoidal $\category{B}$--categories and let $\udl{X}\in \B$ be a $\udl{\sC}$--twisted ambidextrous  space. Then the Beck--Chevalley transformation $\beckChevalley_*\colon \Phi\uniqueMapX_*(-)\rightarrow \uniqueMapX_*\Phi(-)$ is an equivalence.
\end{cor}
\begin{proof}
    To disambiguate notations, we will denote by $\uniqueMapX_!$ and $\uniqueMapX_*$ for the $\udl{X}$--colimit and limit for the category $\udl{\D}$. Now, by \cref{prop:intertwining_principle_for_capping} applied to the case of \cref{example:master_setting} (1), we obtain a commuting square
    \begin{center}
        \begin{tikzcd}
            \Phi\uniqueMapX_*(-) \ar[rr,"\Phi(c\cap_{D_X}-)", "\simeq"'] \dar["\beckChevalley_*"]&& \Phi\uniqueMapX_!(D_X\otimes-)\\
            \uniqueMapX_*\Phi(-) \ar[rr,"\Phi c\:\cap_{\Phi D_X}\Phi-"', "\simeq"] && \uniqueMapX_!\Phi(D_X\otimes-)\simeq \uniqueMapX_!(\Phi D_X\otimes \Phi(-)) \uar["\beckChevalley_!", "\simeq"']
        \end{tikzcd}
    \end{center}
    where the top map is an equivalence by $\udl{\sC}$--twisted ambidexterity, the bottom an equivalence by \cref{thm:base_change_of_tw_amb_spivak_data}, and the right vertical is an equivalence since $\Phi$ preserves parametrised colimits by hypothesis. Thus the left vertical map is an equivalence too, as desired.
\end{proof}

In a limited sense, it is possible to exploit that an object $\obj{X} \in \category{B}$ admits a Poincar\'e Spivak datum with coefficients in $\udlcatC$ to get a Spivak datum with coefficients in $\udlcatD$ with interesting properties. To this end, it would be convenient to establish the following terminology:

\begin{terminology}\label{terminology:poincare_duality_for_a_functor}
    Let $\udl{X}\in\B$, $\udl{\D}\in\cmonoid(\cat_{\B})$ satisfying the $\udl{X}$--projection formula, and $\Phi\colon \udl{\sC}\rightarrow \udl{\D}$ a functor of $\B$--categories. Suppose we have a $\udl{\D}$--Spivak datum $(\zeta,d)$ for $\udl{X}$.  We say that the Spivak datum $(\zeta,d)$ is:
    \begin{enumerate}[label=(\alph*)]
        \item $\Phi$\textit{--twisted ambidextrous} if the capping transformation 
        \[\uniqueMapX_*\Phi(-) \xlongrightarrow{ d\cap_{\zeta}\Phi-}\uniqueMapX_!\big(\zeta\otimes \Phi(-)\big)\]
        of functors $\udl{\sC}^{\udl{X}}\rightarrow\udl{\D}$ is an equivalence,
        \item $\Phi$\textit{--Poincar\'{e} duality} if it is $\Phi$--twisted ambidextrous and $\zeta$ takes values in invertible objects.
    \end{enumerate}
\end{terminology}

\begin{thm}[Poincar\'{e} basechange - general version]\label{thm:small_base_change}
    Suppose $\Phi \colon \udl{\category{C}} \to \udl{\category{D}}$ is a symmetric monoidal functor of $\category{B}$--categories such that
    \begin{itemize}
        \item the $\category{B}$--categories $\udl{\category{C}}, \udl{\category{D}}$ admit $\udl{X}$-shaped (co)limits and satisfy the $\udl{X}$-projection formula;
        \item the $\category{B}$-functor $\Phi$ preserves $\udl{X}$-shaped limits and colimits.
    \end{itemize}
    If $(\xi,c)$ is a twisted ambidextrous $\underline{\sC}$--Spivak datum for $\udl{X}$, then  $\Phi(\xi,c)$ is a twisted ambidextrous  $\Phi$--Spivak datum for $\udl{X}$. In particular, if $(\xi,c)$ is a Poincar\'{e} duality $\underline{\sC}$--Spivak datum for $\udl{X}$, then $\Phi(\xi,c)$ is a  Poincar\'{e} duality $\Phi$--Spivak datum for $\udl{X}$.    
\end{thm}
\begin{proof}
    To distinguish from the Kan extensions $\uniqueMapX_*, \uniqueMapX_!$ associated to the category $\underline{\sC}$, we will write $\uniqueMapX_*, \uniqueMapX_!$ for the functors $\underline{\D}^{\underline{X}}\rightarrow \underline{\D}$. By \cref{prop:intertwining_principle_for_capping}, we have a commuting square
    \begin{center}
        \begin{tikzcd}
            \Phi\uniqueMapX_*(-) \ar[rr,"\Phi(c\cap_{\xi}-)", "\simeq"'] \dar["\beckChevalley_*", "\simeq"']&& \Phi\uniqueMapX_!(\xi\otimes-)\\
            \uniqueMapX_*\Phi(-) \ar[rr,"\Phi c\:\cap_{\Phi \xi}\Phi-"'] && \uniqueMapX_!\Phi(\xi\otimes-)\simeq \uniqueMapX_!(\Phi\xi\otimes \Phi(-)) \uar["\beckChevalley_!", "\simeq"']
        \end{tikzcd}
    \end{center}
    from which the desired result is immediate.  The final statement about Poincar\'{e} duality is an immediate consequence of the fact that $\Phi$ is symmetric monoidal and so preserves invertible objects.
\end{proof}

We will exploit this later to reprove an injectivity result of Bredon and Browder as \cref{thm:bredonBrowderInjection}.

\subsubsection*{Change of base topoi}

\begin{nota}\label{notation:square_bracket}
    To state the next construction and result, it will be convenient to adopt the following notation: let $f^*\colon \B\rightleftharpoons \B' \cocolon f_*$ be a geometric morphism of topoi. If $\udl{\sC}\xrightarrow{F}\udl{\D}\xrightarrow{G} \udl{\E}$ are functors of $\B'$--categories, we write $f_*[G\circ F]$ to mean the composite $f_*\udl{\sC}\xrightarrow{f_*F}f_*\udl{\D}\xrightarrow{f_*G}\nolinebreak f_*\udl{\E}$ and similarly for $f^*$. Furthermore, since both $f^*$ and $f_*$ are product--preserving functors and so enhance to symmetric monoidal functors, we see that for an object $A\in \udl{\D}$, writing $f_*A\in f_*\udl{\D}$ under the equivalence $\map_{\cat_{\B}}(\constant_{\B}\ast,f_*\udl{\D})\simeq \map_{\cat_{\B'}}(\constant_{\B'}\ast,\udl{\D})$, the map $-\otimes A\colon \udl{\D}\rightarrow\udl{\D}$ is sent to $-\otimes f_*A\colon f_*\udl{\D}\rightarrow f_*\udl{\D}$, and similarly in the case when we apply $f^*$.
\end{nota}

\begin{cons}[Pushing Spivak data along geometric morphisms]\label{cons:pushing_forward_spivak_data_geometric_morphisms}
    Let $f^* \colon \category{B} \rightleftharpoons \category{B}' \cocolon f_*$ be a geometric morphism of topoi and consider $\obj{X} \in \category{B}$ and $\udl{\category{C}}$ a symmetric monoidal $\category{B}'$--category which admits $f^*\udl{X}$--indexed colimits. By \cref{lem:parametrised_colimits_base_change}, we know that $f_*\udl{\sC}$ admits $\udl{X}$--colimits.

    Suppose we are given a $\udl{\category{C}}$-Spivak datum $(\xi,c)$ for $f^*{\obj{X}}$ and  a $f_* \udl{\category{C}}$--Spivak datum $(\zeta, d)$ for ${\obj{X}}$.
    Using the symmetric monoidal identification from \cref{lem:pushforwar_fun_manoeuvre}, we obtain a $f_*\udl{\sC}$--Spivak datum $f_* (\xi, c)$ for ${\obj{X}}$ and  a $\udl{\category{C}}$--Spivak datum $f^* (\zeta, d)$ for $f^*{\obj{X}}$. Observe in particular that, by construction, we have $f^*f_*(\xi,c)\simeq (\xi,c)$ and $f_*f^*(\zeta,d)\simeq (\zeta,d)$.
    
    Here, for instance, $f_* \xi$ corresponds to $\xi$ under the equivalence $f_*\udl{\func}(f^* {\obj{X}}, \udl{\category{C}}) \simeq \udl{\func}({\obj{X}}, f_*\udl{\category{C}})$ and $f_* c \colon \unit_{f_*\udl{\category{C}}} \to \uniqueMapX_! f_* \xi$ corresponds to $c$ under the identification of adjunctions in \cref{diag:base_change_twisted_ambidexterity_geometric_morphism}. Explicitly, these new Spivak data are given by
    \begin{equation*}
        f_* (\xi,c) 
        \coloneqq \left( 
        f_* \xi \colon \udl{X} \xrightarrow{\eta} f_* f^* \udl{X} \xrightarrow{f_*\xi} f_* \udl{\category{C}}, 
        \hspace{1mm}
        f_* c \colon \unit_{f_*\udl{\category{C}}} = f_*[\unit_{\udl{\sC}}] \rightarrow \uniqueMapX_! f_*\xi = f_*[(f^*\uniqueMapX)_!\xi]
        \right),
    \end{equation*}
    \begin{equation*}
        f^* (\eta, d) 
        \coloneqq \left(
        f^* \udl{X} \xrightarrow{f^*\zeta} f^* f_* \udl{\category{C}} \xrightarrow{\epsilon} \udl{\category{C}},\hspace{1mm} f^*d \colon \unit_{\udl{\sC}} = f^*[\unit_{f_*\udl{\sC}}]\rightarrow (f^*\uniqueMapX)_!f^*\zeta = f^*[\uniqueMapX_!\zeta]
        \right).
    \end{equation*}
\end{cons}

\begin{cons}[Pushing Spivak data along \'{e}tale morphisms]\label{cons:pushing_spivak_data_etale_morphisms}
    Let $f^* \colon \category{B} \rightleftharpoons \category{B}' \cocolon f_*$ be an \'etale morphism of topoi, $\udl{\sC}\in\cmonoid(\cat_{\B})$, and $\udl{X}\in\B$. Suppose $(\xi,c)$ is a $\udl{\sC}$--Spivak datum for $\udl{X}$. Then we can construct a $f^*\udl{\sC}$--Spivak datum $f^*(\xi,c)$ for $f^*\udl{X}$ given by
    \[\Big(f^*\xi\colon f^*\udl{X}\xrightarrow{f^*\xi}f^*\udl{\sC}, \: f^*c\colon \unit_{f^*\udl{\sC}}\simeq f^*[\unit_{\udl{\sC}}]\xrightarrow{f^*[c]} (f^*X)_!f^*\xi\simeq f^*[\uniqueMapX_!\xi]\Big)\]
    where in the last equivalence, we have used $f^*[\uniqueMapX_!]\simeq (f^*\uniqueMapX)_!$ from \cref{lem:parametrised_colimits_etale_base_change}.
\end{cons}

For part (e) of the next result, see \cref{terminology:poincare_duality_for_a_functor}. 

\begin{thm}[Omnibus geometric basechange of Spivak data]\label{thm:omnibus_geometric_pushforward_of_capping}\label{prop:base_change_twisted_ambidexterity_geometric_morphism}\label{prop:base_change_twisted_ambidextrous_etale_morphism}
    Let $f^*\colon \B\rightleftharpoons \B' \cocolon f_*$ be a geometric morphism of topoi, $\udl{X}\in \B$,  $\udl{\D}\in \cmonoid(\cat_{\B'})$  satisfying the $f^*\udl{X}$--projection formula,  and $\udl{\E}\in\cmonoid(\cat_{\B})$ satisfying the $\udl{X}$--projection formula. Let $(\xi,c)$ be a $\udl{\D}$--Spivak datum for $f^*\udl{X}$ and $(\zeta,d)$ a $f_*\udl{\D}$--Spivak datum for $\udl{X}$. Then:
    \begin{enumerate}[label=(\alph*)]
        \item There is a  commuting square of capping maps
        \begin{center}
            \begin{tikzcd}
                \uniqueMapX_*(-)\ar[rr,"\ambi{f_*c}{f_*\xi}{-}"]\dar["\simeq"'] && \uniqueMapX_!(f_*\xi\otimes-)\dar["\simeq"]\\
                f_*\big[(f^*X)_*(-)\big] \ar[rr,"f_*{[\ambi{c}{\xi}{-}]}"] && f_*\big[(f^*X)_!(\xi\otimes-)\big] 
            \end{tikzcd}
        \end{center}
        of functors $\udl{\func}(\udl{X},f_*\udl{\D})\simeq f_*\udl{\func}(f^*\udl{X},\udl{\D})\rightarrow f_*\udl{\D}$
        \item Suppose that $f_*\colon\cat_{\B'}\rightarrow\cat_{\B}$ is fully faithful (resp. that $f^*\colon \B\rightleftharpoons \B' \cocolon f_*$ is \'{e}tale). Then there is a commuting square 
        \begin{center}
            \begin{tikzcd}
                f^*\big[\uniqueMapX_*(-)\big] \ar[rr,"f^*{[\ambi{d}{\zeta}{-}]}"] \dar["\simeq"']&& f^*\big[\uniqueMapX_!(\zeta\otimes-)\big]\dar["\simeq"]\\
                (f^*X)_*(-) \ar[rr, "\ambi{f^*d}{f^*\zeta}{-}"]&& (f^*X)_!(f^*\zeta\otimes-) 
            \end{tikzcd}
        \end{center}
        of functors $\udl{\func}(f^*\udl{X},\udl{\D}) \simeq f^*\udl{\func}(\udl{X},f_*\udl{\D}) \rightarrow \udl{\D}$ (resp. $\udl{\func}(f^*\udl{X},f^*\udl{\E})  \rightarrow f^*\udl{\E}$),

        \item If $(\xi,c)$ is a twisted ambidextrous (resp. Poincar\'{e}) $\udl{\D}$--Spivak datum for $f^*\udl{X}$, then $f_*(\xi,c)$ is a twisted ambidextrous (resp. Poincar\'{e})  $f_* \udl{\category{D}}$--Spivak datum for $\udl{X}$. 
        
        \item Suppose either that $f_*$ is fully faithful or that $\udl{\category{D}}$ is presentably symmetric monoidal. If $(\zeta,d)$ is a twisted ambidextrous (resp. Poincar\'{e}) $f_*\udl{\D}$--Spivak datum for $\udl{X}$, then $f^*(\zeta,d)$ is a twisted ambidextrous (resp. Poincar\'{e}) $\udl{\category{D}}$--Spivak datum for $f^*\udl{X}$.

        \item More generally: suppose  $f_*$ is fully faithful. Let $\udl{\sC}\in\cat_{\B}$ and $\Phi\colon \udl{\sC}\rightarrow f_*\udl{\D}$ be a  functor between $\B$--categories. Then $(\zeta,d)$ is $\Phi$--twisted ambidextrous (resp. --Poincar\'{e}) for $\udl{X}$ if and only if $f^*(\zeta,d)$ is $(f^*\Phi\colon f^*\sC\rightarrow \udl{\D})$--twisted ambidextrous (resp. --Poincar\'{e}) for $f^*\udl{X}$,
        \item Suppose the geometric morphism $f^*\colon \B\rightleftharpoons \B' \cocolon f_*$ is \'{e}tale. If $(\zeta,d)$ is a twisted ambidextrous (resp. Poincar\'{e}) $\udl{\E}$--Spivak datum for $\udl{X}$, then $f^*(\zeta,d)$ is a twisted ambidextrous (resp. Poincar\'{e}) $f^*\udl{\E}$--Spivak datum for $f^*\udl{X}$. 
    \end{enumerate}
\end{thm}
\begin{proof}
    First note by \cref{prop:projection_formula_geometric_pushforwards} (1, 3) that $f_*\udl{\D}$ satisfies the $\udl{X}$--projection formula and $f^*\udl{\E}$ satisfies the $f^*\udl{X}$--projection formula, and so the squares in (a) and (b) make sense. Now to prove  part (a), we have the commutative diagram
    \begin{center}
        \begin{tikzcd}
            \uniqueMapX_*(-) \ar[r, "\simeq"]  \ar[d, "\id\otimes f_*c"']&  f_*\big[(f^*\uniqueMapX)_* (\--)\big] \ar[d, "{f_*[\id\otimes c]}"]
            \\
            \uniqueMapX_* (\--) \otimes \uniqueMapX_! f_* \xi \ar[r, "\simeq"] &
            f_*\big[(f^*\uniqueMapX)_*(\--) \otimes  (f^*\uniqueMapX)_! \xi\big]
            \\
            \uniqueMapX_! (\uniqueMapX^* \uniqueMapX_* (\--) \otimes f_* \xi) \ar[r, "\simeq"] \ar[d, "\epsilon_*"'] \ar[u, "\beckChevalley_!","\simeq"'] &
            f_*\big[(f^*\uniqueMapX)_!((f^*\uniqueMapX)^* (f^*\uniqueMapX)_*(\--) \otimes \xi)\big] \ar[d, "f_*{\epsilon_*}"] \ar[u, "f_*{\beckChevalley_!}"', "\simeq"]
            \\
            \uniqueMapX_! (\-- \otimes f_* \xi) \ar[r, "\simeq"] &
            f_*\big[(f^*\uniqueMapX)_! (\-- \otimes \xi)\big]
        \end{tikzcd}
    \end{center}
    The horizontal arrows come from the identification in \cref{lem:parametrised_colimits_base_change}; the top square commutes by symmetric monoidality of the identification; the middle and bottom squares commute as the respective adjunction (co-)units are identified by \cref{diag:base_change_twisted_ambidexterity_geometric_morphism}. The required square is now obtained by extracting the outer square of the diagram above.

    The proof for (b) in the case that $f_*$ is fully faithful is done similarly as for (a), but using now the commuting squares of adjunctions obtained by applying $f^*$  to  \cref{lem:parametrised_colimits_base_change} ($f^*$ preserves adjunctions by \cite[Cor. 3.1.9]{MartiniWolf2024}) and that $f^*f_*\simeq \id$. The case of \'{e}tale morphisms is also done similarly, using instead the squares of adjunctions from \cref{lem:parametrised_colimits_etale_base_change}. 

    Next, we prove part (c). If $(\xi,c)$ is twisted ambidextrous, then by \cref{lem:natural_equivalences_preserved_under_geometric_morphisms}, the bottom map in the square from (a) is an equivalence, and so the top map is an equivalence too, i.e. $f_*(\xi,c)$ is twisted ambidextrous. The statements about being Poincar\'{e} is a straightforward consequence of the twisted ambidexterity statements we just proved and the characterisation of factoring  through invertible objects in  \cref{cor:factoring_through_picard_geometric_morphisms} (1).

    For the proof of (d), suppose now that $f_*$ is fully faithful and that $(\zeta,d)$ is twisted ambidextrous. 
    Then since $f^*f_*\simeq \id$, the top map in the square from (b) is an equivalence, and so the bottom map is an equivalence too, i.e. $f^*(\zeta,d)$ is twisted ambidextrous. Poincar\'{e} duality is then handled similarly as in (c).  
    
    Next, assume that $\udl{\category{D}}$ is presentably symmetric monoidal. We  show that the capping transformation $\ambi{f^*d}{f^*\zeta}{(-)}\colon \udl{\func}(f^*\udl{X},\udl{\D})\rightarrow \udl{\D}^{\Delta^1}$ is a natural equivalence in unparametrised categories when evaluated at every $W\in\B'$. Firstly,  note that the $\udl{\D}$--Spivak datum for $f^*\udl{X}$ at level $W$ is obtained via the symmetric monoidal biadjoint $\B'$--functor $W^* \colon \udl{\D}\rightarrow\udl{\D}^{\udl{W}}$, and so the transformation evaluated at $W\in \B'$ is given by applying global sections  
    $\Gamma_{\B'}$ to 
    \begin{equation}\label{eqn:capping_at_level_W}
        \ambi{W^*(f^*d)}{W^*\circ f^*\zeta}{(-)}\colon \udl{\func}(f^*\udl{X},\udl{\D}^{\udl{W}})\longrightarrow (\udl{\D}^{\udl{W}})^{\Delta^1}.
    \end{equation}
    Next, applying \cref{thm:base_change_of_tw_amb_spivak_data} along  $f_*[W^*] \colon f_* \udl{\category{D}} \to f_* (\udl{\category{D}}^{\udl{W}})$ shows that the $f_* (\udl{\category{D}}^{\udl{W}})$--Spivak datum $f_*[{W}^*](\zeta, d)$ is twisted ambidextrous, i.e.  \[\ambi{f_*[W^*](d)}{f_*[W^*]\circ\zeta}{(-)}\colon \udl{\func}(\udl{X},f_*(\udl{\D}^{\udl{W}})) \rightarrow (f_*(\udl{\D}^{\udl{W}}))^{\Delta^1}\] is a natural equivalence.
    But then by the square in part (a), this capping transformation is equivalent to $f_*[\ambi{W^*(f^*d)}{W^*\circ f^*\zeta}{(-)}]\colon f_*\udl{\func}(f^*\udl{X},\udl{\D}^{\udl{W}})\rightarrow f_*((\udl{\D}^{\udl{W}})^{\Delta^1})$, where we have also used that $f_*f^*(\zeta,d)\simeq (\zeta,d)$ from \cref{cons:pushing_forward_spivak_data_geometric_morphisms}. Thus, by using that $\Gamma_{\B}f_*\simeq \Gamma_{\B'}$ from \cref{example:constant_and_global_sections}, we may apply $\Gamma_{\B}$ to $f_*[\ambi{W^*(f^*d)}{W^*\circ f^*\zeta}{(-)}]$ to get that applying $\Gamma_{\B'}$ to \cref{eqn:capping_at_level_W} yields a natural equivalence, as desired. And as usual, Poincar\'{e} duality is handled by \cref{cor:factoring_through_picard_geometric_morphisms} (1).
    
    Now, the proof of parts (c, d) clearly goes through straightforwardly to yield a proof of (e).    Finally, the proof of (f) in the twisted ambidexterity case is done similarly as in the proof of (c) using the square from (b), and the Poincar\'{e} duality case is handled by  \cref{cor:factoring_through_picard_geometric_morphisms} (2).
\end{proof}

\subsubsection*{Descent}

For the next result, we briefly recall the notion of effective epimorphisms in a topos $\category{B}$. Given a morphism $f \colon \obj{X} \rightarrow \obj{Y}$ in $\category{B}$, its \textit{Čech nerve} is the simplicial object
\[ \check{\mathrm{C}}(f) \colon \Delta\op \to \category{B}, \hspace{3mm} [n] \mapsto \obj{X} \times_{\obj{Y}} \obj{X} \times_{\obj{Y}} \dots \times_{\obj{Y}} \obj{X} \hspace{3mm} \text{ ($n+1$ factors)}.  \]
Now $f$ is called an \textit{effective epimorphism} if the canonical map $\colim_{\Delta\op} \check{\mathrm{C}}(f) \to Y$ is an equivalence.

\begin{example}\label{ex:effective_epimorphism_presheaf_topos}
    A map in the topos $\spc$ of spaces is an effective epimorphism if and only if it is a $\pi_0$-surjection (see e.g. \cite[Corollary 7.2.1.15]{lurieHTT}).
    Applying this criterion pointwise, one sees that a map $f \colon X \to Y$ in a presheaf topos $\presheaf(T)$ over some category $T$ is an effective epimorphism if and only if for all $t \in T$ the map $f(t) \colon X(t) \to Y(t)$ is a $\pi_0$-surjection. 
    For example, a map of $G$--spaces $f \colon \udl{X} \rightarrow \udl{Y}$ is an effective epimorphism if and only if for each closed subgroup $H \leq G$ the map $f^H \colon X^H \rightarrow Y^H$ is a surjection on path components.
\end{example}

\begin{prop}[Poincar\'e duality and descent]\label{prop:Poincare_duality_descent}
    Let $\udl{\category{C}}$ be a presentably symmetric monoidal $\category{B}$--category.
    Consider a pullback square
    \begin{equation*}
    \begin{tikzcd}
        P \ar[r, "f'"] \ar[d, "g'"'] \ar[dr, phantom, "\lrcorner", very near start]
        & Z \ar[d, "g"]
        \\
        X \ar[r, "f"]
        & Y.
    \end{tikzcd}
    \end{equation*}
    in $\category{B}$.
    If $f$ is $\udl{\category{C}}$-twisted ambidextrous, then $f'$ is $\udl{\category{C}}$-twisted ambidextrous.
    Furthermore, there is an equivalence $(g')^* D_f \simeq D_{f'}$ where $D_f \in \category{C}(X)$ denotes the dualising object.
    In particular, if $f$ is a $\udl{\category{C}}$--Poincar\'e duality map, then $f'$ is  a $\udl{\category{C}}$--Poincar\'e duality map.

    The converse to both statements is true if $g$ is an effective epimorphism.
\end{prop}
\begin{proof}
    Basechange along $g$ defines an \'etale morphism of topoi $g^* \colon \category{B}_{/Y} \rightarrow \category{B}_{/Z} \colon g_*$ where $g^*$ is given by pullback along $g$.
    Now suppose that $f$ is $\udl{\category{C}}$-twisted ambidextrous.
    This means that $f \in \nolinebreak \category{B}_{/Y}$ is $(\pi_Y)^* \udl{\category{C}}$-twisted ambidextrous.
    Applying \cref{prop:base_change_twisted_ambidextrous_etale_morphism} shows that $\nolinebreak{g^* f = f' \in  \category{B}_{/Z}}$ is $g^* \pi_Y^* \udl{\category{C}} = \pi_Z^* \udl{\category{C}}$-twisted ambidextrous with Spivak datum $(g^* \xi, g^* c)$.

    If $f$ is a $\udl{\category{C}}$--Poincar\'e duality map, then $D_f$ is invertible.
    By symmetric monoidality of the restriction map $(g')^* \colon \category{C}(X) \to \category{C}(P)$ we obtain that $D_{f'} = (g')^* D_f$ is invertible.

    Now suppose that $g$ is an effective epimorphism and that $f'$ is $\udl{\category{C}}$-twisted ambidextrous.
    It is shown in \cite[Proposition 3.13 (5)]{Cnossen2023} that $f$ is $\udl{\category{C}}$-twisted ambidextrous.
    As $g' \colon P \to X$ is an effective epimorphism and the map $\mathbb{\Delta}_{\mathrm{inj}}\op \to \mathbb{\Delta}\op$ is colimit cofinal, the map $\colim_{\mathbb{\Delta}_{\mathrm{inj}}\op}\check{\mathrm{C}}(g') \to X$ is an equivalence from which we obtain the symmetric monoidal equivalence $\category{C}(X) \xrightarrow{\simeq} \lim_{\mathbb{\Delta}_{\mathrm{inj}}} \category{C}(\check{\mathrm{C}}(g'))$.  Next,  suppose furthermore that $f'$ is a Poincar\'e duality map.
    As invertibility in limits can be checked pointwise, we have to show that each restriction of $D_f$ to $\category{C}(\check{\mathrm{C}}_n(g'))$ is invertible.
    Note that a restriction map $\category{C}(X) \to \category{C}(\check{\mathrm{C}}_n(g'))$ factors into the symmetric monoidal restriction maps $\category{C}(X) \to \category{C}(P) \to \category{C}(\check{\mathrm{C}}_n(g'))$ and the first part shows the restriction $D_{f'} = (g')^* D_f$ to $P$ is invertible.
\end{proof}

\begin{cor}[Finite products]\label{cor:pd_finite_products}
    Let $\udl{X}, \udl{Y}\in\category{B}$ and $\udl{\sC}$ be a presentably symmetric monoidal $\category{B}$--category.
    If $\udl{X}$ and $\udl{Y}$ are $\udl{\sC}$--Poincar\'e, then $\obj{X} \times \obj{Y}$ is $\udl{\sC}$--Poincar\'e and there is an equivalence
    \begin{equation*}
        D_{X \times Y} \simeq \pr_X^* D_X \otimes \pr_Y^* D_Y
    \end{equation*}
    where $\pr_X \colon \udl{X} \times \udl{Y} \to \udl{X}$ and $\pr_Y \colon \udl{X} \times \udl{Y} \to \udl{Y}$ denote the projections. 
\end{cor}
\begin{proof}
    As $\obj{X}$ is $\udl{\sC}$--Poincar\'e, it follows from \cref{prop:Poincare_duality_descent} that the map $\pr_Y \colon \udl{X} \times \udl{Y} \to \udl{Y}$ is a $\udl{\sC}$--Poincar\'e map. Since $\obj{Y}$ is $\udl{\sC}$--Poincar\'e, \cref{prop:Poincare_duality_composition} implies that the composite $ \obj{X} \times \obj{Y} \to \obj{Y} \to *$
    is a $\udl{\sC}$--Poincar\'e map showing that $\obj{X} \times \obj{Y}$ is $\udl{\sC}$--Poincar\'e.
    \cref{prop:Poincare_duality_descent,prop:Poincare_duality_composition} then give the identifications $D_{X \times Y} \simeq \pr_Y^* D_Y \otimes D_{\pr_Y} \simeq \pr_Y^* D_Y \otimes \pr_X^* D_X$.
\end{proof}

\begin{lem}\label{lem:pd_map_infinite_coproducts}
    Let $\udl{\category{C}} \in \calg(\presentable^L_{\category{B}})$ be a symmetric monoidal $\category{B}$--category and $\nolinebreak{(f_i \colon \obj{X}_i \to \obj{Y}_i)_{i \in I}}$ be a collection of maps in $\category{B}$.
    Then the map $f= \coprod_i f_i \colon \coprod_i \obj{X}_i \to \coprod_i \obj{Y}_i$ is $\udl{\category{C}}$-twisted ambidextrous (or $\udl{\category{C}}$--Poincar\'e duality) if and only if for all $i \in I$ the map $f_i$ is $\udl{\category{C}}$-twisted ambidextrous (or $\udl{\category{C}}$--Poincar\'e duality). If this is so, then under the identification $\sC(\coprod_iX_i)\simeq \prod_i\sC(X_i)$, we have an equivalence $D_f \simeq (D_{f_i})_i$.
\end{lem}
\begin{proof}
    The ``only if''-direction follows from \cref{prop:Poincare_duality_descent}.
    The ``if''-direction in the twisted ambidexterity case is \cite[Proposition 3.13(3)]{Cnossen2023}.
    If in addition all $f_i$ are Poincar\'e duality maps, then $\coprod_i f_i$ is a Poincar\'e duality map as $D_f = (D_{f_i})_i$ under the equivalence $\category{C}(\coprod_i \obj{X}_i) = \prod_i \category{C}(\obj{X}_i)$. 
\end{proof}

\begin{cor}\label{cor:poincare_duality_of_disjoint_unions}
    Let $\udl{\category{C}} \in \calg(\presentable^L_{\category{B}})$ be a symmetric monoidal $\category{B}$--category which is semiadditive and $\{\underline{X}_i\}_i$ a finite collection of objects in $\B$.  Then $\coprod_i\underline{X}_i$ is $\underline{\sC}$--Poincar\'{e} duality if and only if each $\underline{X}_i$ is $\underline{\sC}$--Poincar\'{e} duality. In this case,  under the identification $\sC(\coprod_iX_i)\simeq \prod_i\sC(X_i)$, we have an equivalence $D_{\coprod_i\underline{X}_i} \simeq (D_{\underline{X}_i})_i$.
\end{cor}
\begin{proof}
    Suppose $\coprod_i\underline{X}_i$ is $\underline{\sC}$--Poincar\'{e} duality. By \cref{example:examples_of_ambidex_point_semiadditive} (1) and \cref{lem:pd_map_infinite_coproducts}, we see that the inclusion $\underline{X}_j\hookrightarrow\coprod_i\underline{X}_i$ is Poincar\'{e} duality for each $j$. An immediate application of \cref{prop:Poincare_duality_composition} using the triple of maps $\underline{X}_j\hookrightarrow\coprod_i\underline{X}_i, \coprod_i\underline{X}_i\rightarrow\underline{\ast}$, and $ \underline{X}_j\rightarrow \underline{\ast}$ then shows that $\underline{X}_j$ is also Poincar\'{e} duality. Next, suppose each $\underline{X}_j$ is Poincar\'{e} duality. By semiadditivity and \cref{example:examples_of_ambidex_point_semiadditive} (2), the map $\nabla\colon\coprod_i\underline{\ast}\rightarrow \underline{\ast}$ is $\underline{\sC}$--Poincar\'{e} duality with dualising sheaf $(\unit_{\underline{\sC}})_i\in\prod_i\sC(\ast)$.  Thus, a simple combination of \cref{prop:Poincare_duality_composition} and \cref{lem:pd_map_infinite_coproducts} using the triple of maps $\sqcup r_i\colon \coprod_i\underline{X}_i\rightarrow \coprod_{i}\underline{\ast}$, $\nabla$, and $\nabla\circ (\sqcup_ir_i)$ yields the desired conclusion.
\end{proof}

\subsection{Degree theory}\label{sec:degree_theory}

In this subsection we introduce the notion of the degree of a map between Poincar\'e spaces (or more generally objects with Spivak data).
We use this to construct Umkehr squares which will  important for our geometric applications.
In \cref{sec:equivariant_degree} we will specialise this to the case of $G$--spaces for a finite group $G$ which generalises classical constructions of the equivariant degree.

As a motivation for the definition, recall that given a map $f \colon X \to Y$ between closed connected manifolds of the same dimension, one can assign to it a degree if $f$ is compatible with the orientation behaviour of $X$ and $Y$: given an identification $\mathcal{O}_X \simeq f^* \mathcal{O}_Y$ of orientation local systems, the degree is given by the image of $[X]$ under $f_* \colon H_n(X; \mathcal{O}_X) \to H_n(Y; \mathcal{O}_Y) \simeq \bbZ$.
In our setting, we will replace orientation local systems and the  fundamental classes $[X]$ above with the dualising sheaves and fundamental classes from \cref{defn:spivak_data}.

\subsubsection*{The definition}

\begin{cons}[(Co)homological functoriality]\label{cons:cohomological_functoriality}
    Consider a map $f \colon \obj{X} \to \obj{Y}$ in $\category{B}$ and a $\category{B}$--category $\udl{\category{C}}$ which admits $\udl{X}$- and $\udl{Y}$-shaped limits and colimits.
    We obtain transformations
    \[\beckChevalley^f_! \colon \uniqueMapX_! f^* \longrightarrow \uniqueMapY_!\quad\quad\quad\quad \beckChevalley^f_* \colon \uniqueMapY_* \longrightarrow \uniqueMapX_* f^*\]
    of functors $\underline{\sC}^{\underline{Y}}\rightarrow \underline{\sC}$ coming from the left and right Beck--Chevalley transformations, respectively, associated to the commuting triangle
    \begin{center}
        \begin{tikzcd}
            \underline{\sC}^{\underline{X}} & \underline{\sC}^{\underline{Y}} \lar["f^*"']\\
            \underline{\sC} \uar["\uniqueMapX^*"]\ar[ur, "\uniqueMapY^*"']
        \end{tikzcd}
    \end{center}
    We call $\beckChevalley^f_!$ the \textit{homological functoriality map} and $\beckChevalley^f_*$ the \textit{cohomological functoriality map}.
\end{cons}

\begin{defn}[Degree of a map]\label{def:parametrised_degree}
    Consider a map $f \colon \obj{X} \to \obj{Y}$ in $\category{B}$ and a symmetric $\category{B}$--category $\udl{\category{C}}$ which admits $\udl{X}$- and $\udl{Y}$-shaped limits and colimits and satisfies the $\udl{X}$- and $\udl{Y}$-projection formula.
    Suppose we are given Spivak data $(\xi_{\obj{X}},c_{\obj{X}})$ for $\obj{X}$ and $(\xi_{\obj{Y}},c_{\obj{Y}})$ for $\obj{Y}$.
    A \textit{$\udl{\category{C}}$-degree datum} for $f$ is an equivalence $\alpha \colon \xi_{\obj{X}} \xrightarrow{\simeq} f^* \xi_{\obj{Y}}$ in $\func_{\category{B}}(\udl{X}, \udl{\category{C}})$.
    We define the \textit{$\udl{\category{C}}$-degree} of $(f, \alpha)$ as the point $\degree_{\udl{\category{C}}}(f, \alpha) \in  \map(\unit_{\udl{\sC}}, \uniqueMapY_! \xi_{\obj{Y}})$ given by the composite
    \begin{equation*}
        \unit_{\udl{\sC}} \xrightarrow{c_{\obj{X}}} \uniqueMapX_! \xi_{\obj{X}} \xrightarrow{\alpha}
        \uniqueMapX_! f^* \xi_{\obj{Y}} \xrightarrow{\beckChevalley^f_!} 
        \uniqueMapY_! \xi_{\obj{Y}}.
    \end{equation*}
    We say that an equivalence $c_{\udl{Y}} \simeq \deg_{\udl{\category{C}}}(f, \alpha)$ exhibits $f$ as a map of \textit{$\udl{\category{C}}$-degree one}.
\end{defn}

\begin{rmk}
    Note that an equivalence $c_{\udl{Y}} \simeq \deg_{\udl{\category{C}}}(f, \alpha)$ is the same datum as a homotopy rendering the following diagram commutative
        \begin{equation*}
    \begin{tikzcd}
        \unit_{\udl{\sC}} \ar[r, "c_{\obj{X}}"] \ar[drr, "c_{\obj{Y}}", bend right=10]
        & \uniqueMapX_! \xi_{\obj{X}} \ar[r, "\uniqueMapX_!\alpha","\simeq"']
        & \uniqueMapX_! f^* \xi_{\obj{Y}} \ar[d, "\beckChevalley^f_!"] 
        \\
        &
        & \uniqueMapY_! \xi_{\obj{Y}}.
    \end{tikzcd}
    \end{equation*}
\end{rmk}

\begin{cons}\label{cons:monoid_structure_degree_space}
    If the Spivak datum $(\xi_{\obj{Y}},c_{\obj{Y}})$ is $\udl{\category{C}}$-twisted ambidextrous, then the equivalence $\ambi{c_{\obj{Y}}}{\xi_{\obj{Y}}}{\unit_{\udl{\sC}}} \colon \uniqueMapY_* \uniqueMapY^*\unit_{\udl{\sC}} \simeq \uniqueMapY_! \xi_{\obj{Y}}$ endows $\uniqueMapY_! \xi_{\obj{Y}}$ with the structure of a commutative algebra in $\udl{\category{C}}$.
    This gives $\map(\unit_{\udl{\sC}}, \uniqueMapY_! \xi_{\obj{Y}})$ the structure of a commutative monoid in $\spc$ with unit $c_{\obj{Y}}$.
    This explains the name ``degree one'' in the previous definition.
\end{cons}

\begin{example}
    Here are some well-known sources of degree data in the case $\category{B} = \spc$ with respect to a presentably symmetric monoidal coefficients $\category{C}$. Let $f \colon X \rightarrow Y$ be a map of connected Poincar\'e spaces of the same formal dimension $d$ (c.f. \cref{terminology:Poincare_dimension}). We consider situations when a degree datum exists for the map $f$ with the Poincar\'e Spivak data $(D_X,c_X)$ and $(D_Y,c_Y)$ for $X$ resp. $Y$.
    \begin{enumerate}[label=(\arabic*)]
        \item For $\category{C} = \module_{\mathbb{F}_2}$, a degree datum exists uniquely since $\Pic(\module_{\mathbb{F}_2})\simeq \mathbb{Z}\times B\mathrm{Aut}(\mathbb{F}_2)\simeq \mathbb{Z}\times \ast$ has contractible components.
        \item For $\category{C} = \module_{\bbZ}$, writing $w_1(-)$ for the first Stiefel--Whitney class of a space, a  degree datum exists if and only if $f^*w_1(Y)=w_1(X)\in H^1(X;\mathbb{F}_2)$.   On homotopy groups, the composite $X_! D_X \simeq X_! f^* D_Y \to Y_! D_Y$
        then identifies with
        \begin{equation*}
            f_* \colon H_{d+*}(X; \mathcal{O}_X) \to H_{d+*}(Y; \mathcal{O}_Y) 
        \end{equation*}
        where $\mathcal{O}_X$ and $\mathcal{O}_Y$ denote the orientation local systems.
        The degree defined above is then given by $f_*[X] \in H_d(Y; \mathcal{O}_Y) \simeq \pi_0(\map(\unit_{\module_{\bbZ}}, Y_! D_Y))$ and agrees with the classical definition of the degree.        
        \item In surgery theory, one is often provided with a \textit{normal map}, i.e. a commuting diagram
        \begin{center}
        \begin{tikzcd}
            X \arrow[rd] \arrow[rdd, bend right, "D_X"'] \arrow[rr, "f"] 
            &                          
            & Y \arrow[ldd, bend left, "D_Y"] \arrow[ld] 
            \\                     
            & BO \times \bbZ \arrow[d, "J"] 
            &                                     
            \\
            & \Pic(\spectra)
            &                                    
        \end{tikzcd} 
        \end{center}
        and of course restricting to the outer commutative triangle gives rise to a degree datum.
    \end{enumerate}
\end{example}

\subsubsection*{Homological Umkehr squares}

Classically, given a map $f \colon M \to N$ of degree one between closed oriented manifolds, one can construct a ``homological Umkehr map'' $f^! \colon H_*(N) \to H_*(M)$ going the ``wrong way'' using Poincar\'e duality. The following result is a generalisation of this.

\begin{lem}[Umkehr square]\label{lem:umkehr_map_degree_one}
    Consider a map $f \colon \obj{X} \to \obj{Y}$ in $\category{B}$ and a symmetric monoidal $\category{B}$--category $\udl{\category{C}}$ which admits $\udl{X}$- and $\udl{Y}$-shaped limits and colimits and satisfies the $\udl{X}$- and $\udl{Y}$-projection formula.
    Suppose that there is a degree datum $\alpha$ for $f$ which is of degree one.
    Then the diagram
    \begin{equation*}
    \begin{tikzcd}
        \uniqueMapY_* (\--) \ar[rr, "\beckChevalley^f_*"] \ar[d, "\ambi{c_{\obj{Y}}}{\xi_{\obj{Y}}}{(-)}"']
        & & \uniqueMapX_* f^* (\--) \ar[d, "\ambi{c_{\obj{X}}}{\xi_{\obj{X}}}{f^*(-)}"] 
        \\
        \uniqueMapY_!(\xi_{\obj{Y}} \otimes (\--)) 
        & \uniqueMapX_!(f^* \xi_{\obj{Y}} \otimes f^* (\--)) \ar[l, "\beckChevalley^f_!"]
        & \uniqueMapX_!(\xi_{\obj{X}} \otimes f^* (\--)) \ar[l, "\alpha"]
    \end{tikzcd}
    \end{equation*}
    commutes.
\end{lem}
\begin{proof}
    Consider the diagram
    \begin{equation*}
    \begin{tikzcd}
        & \uniqueMapY_*(\--) \ar[r, "\beckChevalley^f_*"] \ar[dl, "c_{\obj{Y}} \otimes \--"', blue] \ar[d, "\alpha c_{\obj{Y}} \otimes \--"]
        & \uniqueMapX_* f^*(\--) \ar[d, "\alpha c_{\obj{Y}} \otimes \--", red]
        \\
        \uniqueMapY_! \xi_{\obj{Y}} \otimes \uniqueMapY_*(\--)
        & \uniqueMapX_! f^* \xi_{\obj{Y}} \otimes \uniqueMapY_*(\--) \ar[l, "\beckChevalley^f_!\otimes\id"] \ar[r, "\id\otimes\beckChevalley^f_*"]
        & \uniqueMapX_! f^* \xi_{\obj{Y}} \otimes \uniqueMapX_* f^*(\--)
        \\
        \uniqueMapY_! (\xi_{\obj{Y}} \otimes \uniqueMapY^* \uniqueMapY_*(\--)) \ar[u, "\projectionformula^{\obj{Y}}_!", "\simeq"', blue] \ar[d, "\epsilon_{\obj{Y}}", blue]
        & \uniqueMapX_! (f^* \xi_{\obj{Y}} \otimes \uniqueMapX^* \uniqueMapY_*(\--)) \ar[u, "\projectionformula^{\obj{X}}_!", "\simeq"'] \ar[l, "\beckChevalley^f_!"] \ar[r, "\beckChevalley^f_*"] \ar[d, "\epsilon_{\obj{Y}}"]
        & \uniqueMapX_! (f^*\xi_{\obj{Y}} \otimes \uniqueMapX^* \uniqueMapX_*f^*(\--)) \ar[u, "\projectionformula^{\obj{X}}_!", "\simeq"', red] \ar[dl, "\epsilon_{\obj{X}}", red]
        \\
        \uniqueMapY_! (\xi_{\obj{Y}} \otimes (\--))
        & \uniqueMapX_! (f^* \xi_{\obj{Y}} \otimes f^* (\--)) \ar[l, "\beckChevalley^f_!"]
        &
    \end{tikzcd}
    \end{equation*}
    The degree one datum makes the top left triangle commute.
    The bottom right triangle commutes using the definition of the restriction map and the triangle identities.   The top right and bottom left squares commute by naturality of $\beckChevalley$.
    The middle two squares commute by naturality of
    By definition, the composite of the blue arrows is given by $\ambi{c_{\obj{Y}}}{\xi_{\obj{Y}}}{(-)}$ and the composite of the red arrows is given by $\ambi{\alpha c_{\obj{X}}}{f^*\xi_{\obj{Y}}}{(-)}$.
    To finish,  observe that the diagram
    \begin{equation*}
    \begin{tikzcd}
        \uniqueMapX_* (\--) \ar[d, "\ambi{\alpha c_{\obj{X}}}{f^*\xi_{\obj{Y}}}{(-)}"'] \ar[dr, "\ambi{c_{\obj{X}}}{\xi_{\obj{X}}}{(-)}"]
        & 
        \\
        \uniqueMapX_! (f^* \xi_{\obj{Y}} \otimes (\--))
        & \uniqueMapX_! (\xi_{\obj{X}} \otimes (\--)) \ar[l, "\alpha"]
    \end{tikzcd}
    \end{equation*}
    commutes.
\end{proof}

\subsubsection*{Basechange}

\begin{cons}\label{cons:Change_of_coefficients_degree_datum}
    Suppose that $F \colon \udl{\category{C}} \to \udl{\category{D}}$ is a symmetric monoidal colimit preserving functor of presentably symmetric monoidal $\category{B}$--categories.
    Consider a map $f \colon \obj{X} \to \obj{Y}$ in $\category{B}$ and consider a $\udl{\category{C}}$-degree datum $\alpha$ for $f$.
    We obtain a $\udl{\category{D}}$-degree datum $F(\alpha) \colon F\xi_{\obj{X}} \xrightarrow{\simeq} Ff^*\xi_{\obj{Y}}\simeq f^*F\xi_Y$ for $f$ and the Spivak data $F(\xi_{\obj{X}},c_{\obj{X}})$ and $F(\xi_{\obj{Y}},c_{\obj{Y}})$ from \cref{cons:basechanging_spivak_data}.
\end{cons}
\begin{lem}\label{lem:map_degree_spaces_symmetric_monoidal_functor}
    In the situation of \cref{cons:Change_of_coefficients_degree_datum}, the image of $\deg_{\udl{\category{C}}}(f, \alpha)$ under the map
    \begin{equation}\label{diag:degree_comparison}
        \map(\unit_{\udl{\sC}}, \uniqueMapY_! \xi_{\obj{Y}}) \xrightarrow{F} \map(\unit_{\udl{\D}}, \uniqueMapY_! F\xi_{\obj{Y}})
    \end{equation}
    is equivalent to $\deg_{\udl{\category{D}}}(f, F(\alpha))$.
    Furthermore, if $(\xi_{\obj{Y}},c_{\obj{Y}})$ is $\udl{\category{C}}$-twisted ambidextrous, then the map \cref{diag:degree_comparison} refines to a map of commutative monoids for the commutative monoid structures from \cref{cons:monoid_structure_degree_space} and we have a commutative diagram
    \begin{equation*}
    \begin{tikzcd}
        \map(\unit_{\udl{\sC}}, \uniqueMapY_! \xi_{\obj{Y}}) \ar[r, "F"] \ar[d, "c_Y", "\simeq"',leftarrow] 
        & \map(\unit_{\udl{\D}}, \uniqueMapY_! F\xi_{\obj{Y}}) \ar[d, "F(c_Y)", "\simeq"',leftarrow] \\
        \map(\unit_{\udl{\sC}}, \uniqueMapY_* \uniqueMapY^* \unit_{\udl{\sC}}) \ar[r, "F"]
        & \map(\unit_{\udl{\D}}, \uniqueMapY_* \uniqueMapY^* \unit_{\udl{\D}}).
    \end{tikzcd}
    \end{equation*}

\end{lem}
\begin{proof}
    Note that as $F$ is colimit preserving, it commutes with the homological functoriality constructed in \cref{cons:cohomological_functoriality}.
    Thus it sends  \[\deg_{\udl{\category{C}}}(f, \alpha)\colon \unit_{\udl{\sC}} \xrightarrow{c_{\obj{X}}} \uniqueMapX_! \xi_{\obj{X}} \xrightarrow[\simeq]{\alpha}
        \uniqueMapX_! f^* \xi_{\obj{Y}} \xrightarrow{\beckChevalley_!} 
        \uniqueMapY_! \xi_{\obj{Y}}\]  \[\text{to}  \hspace{5mm} \deg_{\udl{\category{D}}}(f, F(\alpha))\colon \unit_{\udl{\D}} \xrightarrow{F(c_{\obj{X}})} \uniqueMapX_! F\xi_{\obj{X}} \xrightarrow[\simeq]{\alpha}
        \uniqueMapX_! f^* F\xi_{\obj{Y}} \xrightarrow{\beckChevalley_!} 
        \uniqueMapY_! F\xi_{\obj{Y}}\] as desired. Next, suppose that $(\xi_{\obj{Y}},c_{\obj{Y}})$ is $\udl{\category{C}}$-twisted ambidextrous.
    By \cref{prop:intertwining_principle_for_capping} applied to \cref{example:master_setting} (1), $F$ sends the equivalence $\uniqueMapY_* \uniqueMapY^* \unit_{\udl{\sC}} \simeq \uniqueMapY_! \xi_{\obj{Y}}$ induced by $(\xi_{\obj{Y}},c_{\obj{Y}})$ 
    to the equivalence $\uniqueMapY_* \uniqueMapY^* \unit_{\udl{\D}} \simeq \uniqueMapY_! F(\xi_{\obj{Y}})$ 
    induced by $F(\xi_{\obj{Y}},c_{\obj{Y}})$.
    Thus $\uniqueMapY_! F\xi_{\obj{Y}}\simeq F\uniqueMapY_! \xi_{\obj{Y}}$ as commutative algebras in $\udl{\category{D}}$.
\end{proof}

\begin{cons}\label{cons:degree_one_C_equivalence}
    Let $f \colon \obj{X} \to \obj{Y}$ be a map in $\category{B}$ and consider a symmetric monoidal $\category{B}$--category $\udl{\category{C}}$ which admits $\udl{X}$-shaped limits and colimits and satisfies the $\udl{X}$-projection formula.
    Furthermore, assume that the map $f^* \colon \internalfunc{\category{B}}(\udl{Y}, \udl{\category{C}}) \to \internalfunc{\category{B}}(\udl{X}, \udl{\category{C}})$ is an equivalence.
    It canonically refines to a symmetric monoidal equivalence.
    Hence, by \cref{cons:cohomological_functoriality}, we obtain canonical equivalences $\beckChevalley^f_!\colon \uniqueMapX_!f^* \xrightarrow{\simeq} \uniqueMapY_!$ and $\beckChevalley^f_*\colon \uniqueMapY_* \xrightarrow{\simeq} \uniqueMapX_*f^*$.
    Thus, for any Spivak datum $(\xi_{\obj{X}},c_{\obj{X}})$ for $\obj{X}$ we obtain the Spivak datum
    \begin{equation*}
        \left( \unit_{\udl{\sC}} \xrightarrow{c_{\obj{X}}} \uniqueMapX_! \xi_{\obj{X}} \simeq \uniqueMapY_! (f^*)^{-1} \xi_{\obj{X}} \right)   
    \end{equation*}
    for $\obj{Y}$.
    It is twisted ambidextrous (or Poincar\'e) if and only if $(\xi_{\obj{X}},c_{\obj{X}})$ is.
    Furthermore, note that with respect to these Spivak data, the map $f$ is clearly of degree one.
\end{cons}

\begin{lem}\label{lem:degree_one_base_change}
    Consider a geometric morphism $f^* \colon \category{B} \rightleftharpoons \category{B}' \cocolon f_*$ of topoi and $\udl{\category{C}}$ be a symmetric monoidal $\category{B}'$-category.
    Suppose that we are given a map $g \colon \obj{X} \to \obj{Y}$ in $\udl{\category{B}}$ together with $\udl{\category{C}}$-Spivak data for $f^*\obj{X}$ and $f^*\obj{Y}$.
    Then a $\category{C}$-degree (one) datum for $f^* g \colon f^* \obj{X} \to f^* \obj{Y}$ is equivalent to a $f_*\category{C}$-degree (one) datum for $g$, where we endow $\obj{X}$ and $\obj{Y}$ with the $f_* \udl{\category{C}}$-Spivak data from \cref{cons:pushing_forward_spivak_data_geometric_morphisms}.
\end{lem}
\begin{proof}
    The equivalence of degree data follows from \cref{lem:pushforwar_fun_manoeuvre}.
    The statement about degree one data being equivalent follows from \cref{lem:parametrised_colimits_base_change}.
\end{proof}

\section{Equivariant Poincar\'{e} duality: elements}
\label{section:equivariant_PD_elements}

In this section we will apply the abstract theory of parametrised Poincar\'e duality developed in \cref{section:parametrised_PD} to the topos $\spc_G$ of $G$--spaces for a compact Lie group $G$ and use this as our definition of equivariant Poincar\'e duality spaces.
We begin in \cref{subsection:setting_the_stage} by explaining the definition in this special case in more detail, and then come to one of the key components of the theory in \cref{subsection:fixed_points_methods}, namely fixed points methods. After that, in \cref{subsection:construction_principles} we study how Poincar\'e duality interacts with various kinds of equivariant and homotopical operations such as restrictions, inflations, (co)inductions, fibre sequences, and quotients, and we then give natural examples of $G$--Poincar\'{e} spaces in \cref{sec:examples_G_PD}. Lastly, we will round off this section  by explaining some geometrically meaningful ramifications of the theory of fundamental classes in \cref{subsection:gluing_classes,sec:equivariant_degree}.

The reader who is not too familiar with the abstract categorical language can in most situations safely replace $G$ by a finite group and a presentably symmetric monoidal $G$--category $\udl{\category{C}}$ by the $G$--category $\myuline{\spectra}$ of genuine $G$--spectra or  even $\udl{\module}_{\udl{A}(G)}(\myuline{\spectra}) \simeq D(\udl{\mackey}_G(\abelianGroups))$ (c.f. \cite[$\S5.2$] {greenlees-shipley} or \cite[$\S5$]{patchkoria_derived_mackey}  for this equivalence), the derived category of $G$-Mackey functors with values in abelian groups.
As a sanity check for constructions not involving a change of groups, it might also be helpful to first read the statements for $G = 1$.

\subsection{Setting the stage}\label{subsection:setting_the_stage}

We specialise the definitions in \cref{subsection:parametrised_spivak_data,subsection:parametrised_twisteed_ambidex} from the abstract parametrised setting to the equivariant situation for a compact Lie group $G$. After giving the formal definitions, we will provide more intuition for them by  unraveling what these notions mean in \cref{rmk:unravelling_the_definition}.

\begin{defn}
    Let $\udl{X}\in\spc_G$ and $\udl{\sC}$ a symmetric monoidal $G$--category admitting $\udl{X}$--shaped colimits. A $\udl{\sC}$\textit{--Spivak datum} for $\udl{X}$ is a pair $(\xi,c)$ where $\xi\in {\func}_G(\udl{X},\udl{\sC})$ is called the dualising sheaf and  $c\colon \unit_{\udl{\sC}}\rightarrow\uniqueMapX_!\xi$ is a morphism in $\udl{\sC}$  called the fundamental class.
\end{defn}

Now let $\udlcatC$ be a symmetric monoidal $G$--category and $\udl{X} \in \spc_G$ a $G$--space. Suppose that $\udlcatC$ admits $\udl{X}$--shaped limits and colimits and satisfies the $\udl{X}$--projection formula (c.f. \cref{terminology:projection_formula}). For example, if either: (a) the $G$--category $\udl{\sC}$ were presentably symmetric monoidal, i.e. an object in $\calg(\presentable_G^L)$, or (b) if $G$ were a finite group, $\udl{X}\in\spc_G$ were compact, and $\udl{\sC}$ were a small $G$--stably symmetric monoidal category, i.e. an object in $\calg(\cat\exact_G)$, then these conditions are satisfied. Under these conditions, given a $\udl{\sC}$--Spivak datum, \cref{cons:capping_map} provides a a morphism in $\udl{\func}(\udlcatC^{\udl{X}},\udlcatC)$ 
\begin{equation}
    \label{eq:capping_map_equivariant_case}
    \ambi{c}{\xi}{-} \colon \uniqueMapX_*(-) \longrightarrow \uniqueMapX_!(- \otimes \xi)
\end{equation}
called the capping transformation. We refer the  reader to the preamble of \cref{section:parametrised_PD} for the motivation for these notations.

\begin{defn}
    A $\udl{\sC}$--Spivak datum $(\xi,c)$ for $\udl{X}$ is \textit{twisted ambidextrous} if the capping map \cref{eq:capping_map_equivariant_case} is an equivalence. It is \textit{Poincar\'e} if additionally, $\xi$ takes values in the subcategory $\udl{\Pic}(\udlcatC)$.
\end{defn}

If we take a presentably symmetric monoidal $G$--category $\udlcatC \in \calg(\presentable^L_G)$ as coefficient system, then the situation again simplifies a little, for example because a twisted ambidextrous Spivak datum is unique, if it exists, by \cref{prop:twisted_ambidextrous_presentable_case}.

\begin{defn}
    If $\udl{X}$ is a $G$--space and $\udlcatC$ is a presentably symmetric monoidal $G$--category, we say that $\udl{X}$ is \textit{$\udlcatC$--twisted ambidextrous} if it admits a twisted ambidextrous $\udl{\sC}$--Spivak datum $(D_{\udl{X}},c)$. Furthermore, $\udl{X}$ is \textit{$\udlcatC$--Poincar\'e} if additionally $D_{\udl{X}}$ takes values in $\udl{\Pic}(\udlcatC)$.
\end{defn}

\begin{terminology}
    In the special case where $\udlcatC = \myuline{\spectra}$, we just say that $\udl{X}$ is $G$--twisted ambidextrous or $G$--Poincar\'e.
\end{terminology}

Understanding the case of $\myuline{\spectra}$--Poincar\'{e} duality is our main motivation for this article. Because of its importance, we give a few explanations about this particular space.

\begin{rmk}[Unraveling the definition of $\myuline{\spectra}$-Poincar\'e duality]\label{rmk:unravelling_the_definition}
    To set up a good formalism, we needed to work in a generality that runs the risk of seeming overly abstract. We  stress that the task of checking if a space $\udl{X}$ is $G$--Poincar\'e  closely resembles classical Poincar\'e duality.

    First, one has to find the correct analog of a local system with respect to which $\udl{X}$ is supposed to satisfy Poincar\'e duality. We required a $\xi \in \func_G(\udl{X},\myuline{\spectra})$ that lands in $\udl{\Pic}(\myuline{\spectra})$, which  unravels to providing for each closed subgroup $H \leq G$ a local system of invertible $H$--spectra $\xi^H \colon X^H \rightarrow \Pic(\spectra_H)$
    together with compatibilities that amount to providing for each map $G/K \rightarrow G/H$ a homotopy in the diagram
    \begin{center}
        \begin{tikzcd}
            X^K \ar[r, "\xi^K"] \ar[d, "\res^K_H"'] &  \ar[d, "\res^K_H"] \Pic(\spectra_K)\\
            X^H \rar["\xi^H"]& \Pic(\spectra_H) 
        \end{tikzcd}
    \end{center}
    plus higher coherences between these homotopies. Additionally, we required a fundamental class $c \colon \sphere_G \rightarrow \uniqueMapX_!(\xi)$
    which the reader should think of as an equivariant homology class of $\udl{X}$ with coefficients in the local system $\xi$. The capping map $\ambi{c}{\xi}{-} \colon \uniqueMapX_*(-) \rightarrow \uniqueMapX_!(- \otimes \xi)$  should then be thought of as the cap product with the homology class $c$. Recall from the preamble to \cref{section:parametrised_PD} that classical Poincar\'e duality is the statement that capping with a certain ``fundamental class" induces an isomorphism between cohomology and homology, and what we ask here is exactly the same condition.
\end{rmk}

In the presentable setting, we are in the pleasant situation where we can identify a large class of twisted ambidextrous objects.

\begin{prop}\label{prop:cpt_twisted_ambidextrous}
    Every compact $G$--space $\obj{X}$ is $\myuline{\spectra}$--twisted ambidextrous. Consequently, every compact $G$--space $\obj{X}$ is $\udl{\category{C}}$--twisted ambidextrous for any $G$-stable presentably symmetric monoidal $G$--category.
\end{prop}
\begin{proof}
    The first part is an immediate consequence of \cite[Thm. 4.8 (5)]{Cnossen2023} and \cref{rmk:bastiaan_twisted_ambidex_agrees_with_ours}, and the second part is  by \cref{thm:base_change_of_tw_amb_spivak_data}.
\end{proof}

Next, as may be expected of a well--behaved equivariant notion,  equivariant Poincar\'{e} duality is preserved under restrictions. To show this, first recall from \cref{rec:orbit_category} that  for a closed subgroup $H \leq G$ there is an identification $\orbit(H) \simeq \orbit(G)/_{/(G/H)}$ so that the induction $\spc_H \rightarrow \spc_G$ can be identified with the \'etale geometric morphism $\spc_G \rightleftharpoons (\spc_G)_{/(\myuline{G/H})}$.

\begin{cons}[Pushing Spivak data along restrictions]\label{cons:pushing_spivak_data_along_restrictions}
    Let $\udl{X}\in\spc_G$, $\udl{\sC}\in\cmonoid(\cat_G)$, and $(\xi,c)$ a $\udl{\sC}$--Spivak datum for $\udl{X}$. 
    Then by \cref{cons:pushing_spivak_data_etale_morphisms}, we obtain a $\res^G_H\udl{\sC}$--Spivak datum $\res^G_H(\xi,c)$ for $\res^G_H\udl{X}$ given by \small
    \[\Big(\res^G_H \xi \colon \res^G_H\udl{X} \xrightarrow{\res^G_H\xi} \res^G_H\udl{\sC}, \: \res^G_Hc \colon  \res^G_H{\unit_{\udl{\sC}}} \xrightarrow{\res^G_H[c]} (\res^G_H\uniqueMapX)_! \res^G_H\xi\Big)\]\normalsize
\end{cons}

\begin{prop}[Restriction stability of Poincar\'{e} duality]\label{prop:restriction_stability_of_Poincare_duality}
    Let $\underline{X}\in\spc_G$, $\udl{\sC}$ be a symmetric monoidal $G$--category, and  $(\xi,c)$ be a Poincar\'{e} $\underline{\sC}$--Spivak datum. Then for all closed subgroups $H\leq G$, $\res^G_H(\xi,c)$ is a Poincar\'{e} $\res^G_H\underline{\sC}$--Spivak datum for the $H$-space $\res_H^G \udl{X}$.
\end{prop}
\begin{proof}
    This is a direct consequence of part (e) of \cref{prop:base_change_twisted_ambidextrous_etale_morphism} applied to \'etale the geometric morphism ${\spc}_G \rightleftharpoons ({\spc}_G)_{/(\myuline{G/H})}\simeq \spc_H$.
\end{proof}

\subsection{Fixed points methods}\label{subsection:fixed_points_methods}

Let $G$ be a compact Lie group. In this subsection, we study how $G$--Poincar\'e duality for a $G$--space $\udl{X}$ relates to Poincar\'e duality for its fixed points. In fact, we shall build upon the theory set up in \cref{subsection:equivariant_categories_and_families} and discuss these questions in the generality of isotropy separations with respect to a family of subgroups, of which the case of fixed points against a subgroup is a special case. Therefore, let us fix a family $\family$ of closed subgroups of $G$ throughout this subsection. Recall the notational package from \cref{nota:isotropy_separation_package}.

An important family to keep in mind as an intuitional guide is the following:

\begin{example}[Proper family]\label{example:proper_family}
    Denote by $\proper$ the family of proper closed subgroups of $G$, so that $\proper^{c} = \{G\}$ and $s\colon \ast\simeq \orbit_{\proper^c}(G)\op\hookrightarrow \orbit(G)\op$ is the inclusion of the orbit $G/G$. Note that for any $\udl{J}\in\cat_G$, we thus have $\udl{J}\sUpperStar{\proper} = s^*\udl{J}\simeq J^G\in\cat$.  In this special case, we know that the adjunction unit $\Phi\colon \myuline{\spectra}\rightarrow \myuline{\spectra}\sTwiddleLowerStar{\proper}\simeq s_*\widetilde{s}^*\myuline{\spectra}$ adjoints to the geometric fixed points functor $\Phi^G\colon s^*\myuline{\spectra}=\myuline{\spectra}\sUpperStar{\proper}\simeq\spectra_G\rightarrow \brauerQuotientFamily{\proper}\myuline{\spectra}\simeq\spectra$.
\end{example}

\begin{obs}\label{obs:parametrised_colimits_of_s_*-categories}
    Consider the case of the family of proper closed subgroups $\proper$ of $G$. In particular, we have that   $\udl{\func}(-,-)\sUpperStar{\proper} \simeq \udl{\func}(-,-)^G \simeq \func_G(-,-)$. 
    For a fixed $\sC\in\cat$ having the approparite (co)limits and a $G$-space $\udl{X}$, applying $(-)\sUpperStar{\proper}$ to the commuting diagram in \cref{lem:parametrised_colimits_base_change} and using that $(-)\sLowerStarInclusion{\proper}$ is fully faithful yields the left commuting diagram
    \begin{equation*}
    \begin{tikzcd} 
        \func(X^G, \sC) \dar[equal] \rar[shift left = 1, "X^G_!"] \rar[shift right = 1, " X^G_*"']
        & {\sC} \dar[equal] 
        && \func_G(\udl{X},\myuline{\spectra}) \dar["\Phi^G"']\rar["X_!"] & \spectra_G\dar["\Phi^G"]\\
        \func_G(\udl{X}, {\sC}\sLowerStarInclusion{\proper}) \rar[shift left = 1, "X_!"] \rar[shift right = 1, "X_*"']
        & {\sC} &&\func(X^G,\spectra) \rar["X^G_!"] & \spectra. 
    \end{tikzcd}
    \end{equation*}
    That is, parametrised (co)limits in $G$--categories of the form $\sC\sLowerStarInclusion{\proper}$ is given by the ordinary (co)limits of the fixed points of the indexing diagram. In particular, since $\Phi^G\colon \myuline{\spectra}\rightarrow \myuline{\spectra}\sTwiddleLowerStar{\proper} $ preserves parametrised colimits, the identifications above yield the right commuting square in the diagram above.
\end{obs}

\begin{lem}\label{lem:counit_equivalence_singular_part}
    Let $\udl{X}\in\spc_G$ and consider a family $\family$ of subgroups of $G$. Let $\udl{\category{C}} \in \catgrpcol{G}{\family^c}$.   Then the  map $\epsilon^* \colon \udl{\func}(\udl{X}, \udl{\category{C}}\sLowerStarInclusion{\family}) \to \udl{\func}(\udl{X}_{\widetilde{\family}}, \udl{\category{C}}\sLowerStarInclusion{\family})$ is an equivalence.
\end{lem}
\begin{proof}
    By \cref{lem:pushforwar_fun_manoeuvre}, the equivalence $\udl{\func}_G(\udl{X}, s_* \udl{\category{C}}) \simeq s_*\udl{\func}_{\family^c}(s^* \udl{X}, \udl{\category{C}})$ identifies restriction along $\epsilon \colon s_! s^* \udl{X} \to \udl{X}$ on the left side with restriction along $s^* \epsilon \colon s^* s_! s^* \udl{X} \to s^* \udl{X}$ on the right side.
    But $s^* \epsilon$ is an equivalence.
\end{proof}

\begin{cons}[Isotropy separation for Spivak data]\label{cons:isotropy_separation_spivak_Data}
    Let $\udl{X} \in \spc_G$ and $\udl{\category{C}}$ a symmetric monoidal $\family^c$--category which admits $\udl{X}\sUpperStar{\family}$--indexed colimits. By \cref{lem:parametrised_colimits_base_change}, we know that $\udl{\sC}\sLowerStarInclusion{\family}$ admits $\udl{X}$--colimits. Suppose we are given a $\udl{\category{C}}$-Spivak datum $(\xi,c)$ for ${\obj{X}}\sUpperStar{\family}$ and  a $ \udl{\category{C}}\sLowerStarInclusion{\family}$--Spivak datum $(\zeta, d)$ for ${\obj{X}}$.
    By \cref{cons:pushing_forward_spivak_data_geometric_morphisms}, we obtain a $\udl{\sC}\sLowerStarInclusion{\family}$--Spivak datum $(\xi, c)\sLowerStarInclusion{\family}$ for ${\obj{X}}$ and  a $\udl{\category{C}}$--Spivak datum $(\zeta, d)\sUpperStar{\family}$ for ${\obj{X}}\sUpperStar{\family}$. Observe in particular that, by construction, we have $((\xi,c)\sLowerStarInclusion{\family})\sUpperStar{\family}\simeq (\xi,c)$ and $((\zeta,d)\sUpperStar{\family})\sLowerStarInclusion{\family}\simeq (\zeta,d)$.
\end{cons}

\begin{cor}[Inclusion of singular part is degree one]\label{cor:inclusion_fixed_points_degree_one}
    Let $\udl{\category{D}}$ be a symmetric monoidal $\family^c$--category and $\udl{X} \in \spc_G$. Suppose $\udl{X}$ is equipped with a $\udl{\D}\sLowerStarInclusion{\family}$--Spivak datum. Then $\udl{X}_{\widetilde{\family}}\in\spc_{G}$ inherits a $\udl{\D}\sLowerStarInclusion{\family}$--Spivak datum under which the inclusion $\epsilon \colon \udl{X}_{\widetilde{\family}} \to \udl{X}$ is of $\udl{\D}\sLowerStarInclusion{\family}$--degree one.
\end{cor}
\begin{proof} 
    By \cref{lem:counit_equivalence_singular_part}, the map  $\epsilon^* \colon \udl{\func}(\udl{X}, s_*  \udl{\category{D}}) \xrightarrow{\simeq} \udl{\func}(s_!s^*\udl{X}, s_* \udl{\category{D}})$ is an equivalence.    The result now follows immediately from \cref{cons:degree_one_C_equivalence}.
\end{proof}

\begin{lem}\label{lem:poincare_duality_adjunction}\label{lem:poincare_duality_in_trivial_bottoms}
    Let $\udl{X}\in\spc_G$, $\family$ be a family of closed subgroups of $G$, and $\underline{\D}$ a  symmetric monoidal $\family^c$--category. Then a $\udl{\category{D}}\sLowerStarInclusion{\family}$--Spivak datum $(\xi,c)$ for $\udl{X}$ is Poincar\'e  if and only if the $\udl{\category{D}}$--Spivak datum $(\xi,c)\sUpperStar{\family}$ for $\udl{X}\singularPartComplement{\family}$ is Poincar\'{e}.
\end{lem}
\begin{proof}
    This is a special case of \cref{prop:base_change_twisted_ambidexterity_geometric_morphism} (c).
\end{proof}

We  now come to  the main result of this subsection which says that we may perform isotropy separation on equivariant Poincar\'{e} spaces by appropriately isotropy--separating   the coefficient category. For the second part of the  result, we will need to recall \cref{terminology:poincare_duality_for_a_functor}.

\begin{thm}[Poincar\'e  isotropy basechange]\label{thm:poincare_isotropy} 
Let $\udl{X}\in\spc_G$, $\udl{Y}\in\spc_G^{\omega}$, $\udl{\sC}$ be a presentably symmetric monoidal fibrewise stable $G$--category, and $\udl{\D}$ be a $G$--stably symmetric monoidal category.   
\begin{enumerate}[label=(\arabic*)]
    \item  If $\udl{X}$ is $\udl{\category{C}}$--Poincar\'e, then $\udl{X}\sUpperStar{\family}$ is $\brauerQuotientFamily{\family}\udl{\category{C}}$--Poincar\'e;
    \item If $(\xi,c)$ is a Poincar\'{e} $\udl{\D}$--Spivak datum for $\udl{Y}$, then the Spivak datum $(\Phi\xi,\Phi c)\sUpperStar{\family}$ is a Poincar\'{e} $\nolinebreak{(\Phi\colon \udl{\D}\sUpperStar{\family}\rightarrow \brauerQuotientFamily{\family}\udl{\D})}$--Spivak datum for $\udl{Y}\sUpperStar{\family}$.
\end{enumerate}
\end{thm}
\begin{proof}
    For (1), applying the basechange result \cref{thm:base_change_of_tw_amb_spivak_data} along the symmetric monoidal $G$-colimit preserving unit map $\udl{\category{C}} \to  \udl{\category{C}}\sTwiddleLowerStar{\family}$ shows that $\udl{X}$ is $\udl{\category{C}}\sTwiddleLowerStar{\family}$--Poincar\'e.  Thus \cref{lem:poincare_duality_adjunction} shows that $\udl{X}\sUpperStar{\family}$ is $\brauerQuotientFamily{\family}\udl{\category{C}}$--Poincar\'e.    Point (2) is an immediate consequence of \cref{thm:small_base_change} and \cref{lem:poincare_duality_adjunction}.
\end{proof}

Having set up a general theory of equivariant fixed points for Poincar\'{e} spaces, we now specialise to the most important coefficient category, namely the presentably symmetric monoidal $G$--stable category $\myuline{\spectra}$ of genuine $G$--spectra.

\begin{cons}[Pushing Spivak data along geometric fixed points]\label{cons:pushing_spivak_data_geometric_fix_points}
    Let $\udl{X}\in\spc_G$, $(\xi,c)$ a $\myuline{\spectra}_G$--Spivak datum for $\udl{X}$, and $H\leq G$  a closed subgroup. By \cref{cons:pushing_spivak_data_along_restrictions}, we obtain a $\myuline{\spectra}_H$--Spivak datum 
    $\res^G_H(\xi,c)$ for $\res^G_H\udl{X}$. On the other hand, we may apply  \cref{cons:isotropy_separation_spivak_Data} to $\res^G_H(\xi,c)$ along the symmetric monoidal map $\Phi^H\colon \myuline{\spectra}_H\rightarrow s_*\spectra$ to get a nonequivariant $\spectra$--Spivak datum $\Phi^H(\xi,c)$ for $X^H$. Explicitly, this is given by\small
    \[\Big(\Phi^H \xi \colon X^H \xrightarrow{\xi^H} \spectra_H \xrightarrow{\Phi^H} \spectra, \: \Phi^Hc \colon \unit_{\spectra} = \Phi^H{\unit_{\spectra_H}} \xrightarrow{\Phi^Hc} \Phi^H(\res^G_H\uniqueMapX)_! \res^G_H\xi \simeq \uniqueMapX^H_! \Phi^H\xi \Big)\]\normalsize where the last equivalence is by \cref{obs:parametrised_colimits_of_s_*-categories}.
\end{cons}

Next, we unwind the general \cref{thm:poincare_isotropy} (1) for the geometric fixed points functor on spectra to show that the fixed points of an equivariant Poincar\'{e} space are Poincar\'{e} with the residual Weyl group action (c.f. \cite[Prop. 2.4]{costenoble2017equivariant} for the homological shadow of this).

\begin{thm}[Fixed points of Poincar\'e spaces]\label{thm:fixed_points_poincare_duality}
    Suppose $\udl{X} \in \spc_G$ is $G$--Poincar\'e.
    Then for any closed $H \le G$, $\udl{X}^H \in \spc_{W_G H}$  is a $W_G H$--Poincar\'e space.
    In particular, $X^H \in \spc$ is a nonequivariant Poincar\'e space with dualising sheaf $X^H \xrightarrow{D_{\udl{X}}} \spectra_H \xrightarrow{\Phi^H} \spectra$.
\end{thm}
\begin{proof}
    First consider the case where $H$ is normal in $G$.
    We apply \cref{thm:poincare_isotropy} (1) in the case $\udl{\category{C}} = \myuline{\spectra}_G$ and for the family $\Gamma_H \coloneqq \{K\leq G\: | \: H \nleq K \}$ of subgroups of $G$ not containing $H$.
    Thus, if $\udl{X}$ is a $G$--Poincar\'e space, then $s^* \udl{X}$ is a $\widetilde{s}^* \myuline{\spectra}_G$ Poincar\'e space.
    In \cref{prop:stability_for_quotient_groups} we saw that $\coind \colon \spc_{G/H} \to \spc_{G}$ induces an equivalence $\spc_{G/H} \simeq \spc_{\Gamma_H^c}$ endowing $s^* \udl{X}$ with a $G/H$ action.
    It also follows from \cref{lem:geometric_fixed_point_spectrum_weil_group} that $\widetilde{s}^* \myuline{\spectra}_G \simeq \myuline{\spectra}_{G/N}$ which completes the proof of this case.

    Now suppose that $H \le G$ is a general subgroup.
    We can apply \cref{prop:restriction_stability_of_Poincare_duality} to obtain that $\res^G_{N_{G}H} \udl{X}$ is a $N_{G}H$--Poincar\'e duality space.
    Then the normal subgroup case from above shows that $\res^G_H \udl{X} = \res^{N_{G}H}_H \res^G_{N_{G}H} \udl{X}$ is a $W_GH = N_{G}H/H$--Poincar\'e duality space.

    The first part in particular shows that $X^G$ is a $\spectra$--Poincar\'e duality space.
    Applying this to $X^H \in \spc_{W_G H}$, we see that $X^H$ is a nonequivariant $\spectra$--Poincar\'e duality space.
\end{proof}

To end our discussion on general fixed points methods, we provide a sort of converse to the previous statement. By \cref{prop:twisted_ambidextrous_presentable_case}, we know that in the presentable setting, a twisted ambidextrous Spivak datum is unique if it exists. Via the geometric fixed points functors, the following result gives a full characterisation for a candidate invertible Spivak datum to be the unique one for a twisted ambidextrous space in terms of nonequivariant Poincar\'{e} duality. It will be essential for constructing examples of equivariant Poincar\'e duality spaces in \cref{sec:examples_G_PD}.

\begin{thm}[Fixed point recognition principle of Poincar\'e spaces]\label{thm:PD_fixed_point_recognition}
    Suppose that $\udl{X} \in \spc_G$ is a twisted ambidextrous $G$--space (e.g. a compact $G$--space) and let $(\xi,c)$ be a $\myuline{\spectra}_G$-Spivak datum for $\udl{X}$ such that $\xi \colon \udl{X} \to \myuline{\spectra}_G$ takes values in $\udl{\picardSpace}(\myuline{\spectra}_G)$.  Then $(\xi, c)$ exhibits $\udl{X}$ as a $G$--Poincar\'e duality space if and only if for all closed subgroups $H \le G$, the Spivak datum $\Phi^H(\xi, c)$ from \cref{cons:pushing_spivak_data_geometric_fix_points} exhibits $X^H$ as a nonequivariant $\spectra$--Poincar\'e space.
\end{thm}
\begin{proof}
    The ``only if'' direction is a consequence of \cref{thm:fixed_points_poincare_duality}.
    For the other direction, we have to show that the Spivak datum $(\xi, c)$ is twisted ambidextrous as $\xi$ is invertible by assumption. By \cref{obs:functors_preserve_joint_conservativity} and \cref{prop:joint_conservativity_geometric_fixed_points}, the collection \small
    \[\Big\{\udl{\func}(\udl{X},\myuline{\spectra}) \xrightarrow{\Phi^H} \udl{\func}(\udl{X},\coind^G_Hs_*\widetilde{s}^*\myuline{\spectra})\simeq \prod^G_Hs_*\func(X^H,\spectra) \: | \: H\leq G \text{ closed subgroups }\Big\}\]\normalsize is jointly conservative. Thus,  it suffices to show that the transformations
    \begin{equation}
        \Phi^H(\ambi{c}{\xi}{-}) \colon \Phi^H\uniqueMapX_*(\--) \to \Phi^H\uniqueMapX_!(\-- \otimes \xi) \label{diag:tw_ambi_map_fixed_point_recognition}
    \end{equation}
    are equivalences. By passing to the adjoint $\udl{\func}(\res^G_H\udl{X},\res^G_H\myuline{\spectra})\xrightarrow{\Phi^H}s_*\func(X^H,\spectra)$ to consider everything as $H$--categories, we may without loss of generality just consider the case $\Phi^G$. By \cref{prop:intertwining_principle_for_capping} applied to the case of \cref{example:master_setting} (1), the symmetric monoidal functor of presentably symmetric monoidal $G$--categories $\Phi^G\colon \myuline{\spectra}\rightarrow s_*\spectra$ yields a  square
    \begin{center}
        \begin{tikzcd}
            \Phi^G\uniqueMapX_*(-) \ar[rr,"\Phi^G(c\cap_{\xi}-)"] \dar["\beckChevalley_*", "\simeq"']& &\Phi^G\uniqueMapX_!(\xi\otimes-)\\
            \uniqueMapX_*\Phi^G(-) \ar[rr,"\Phi^G c\:\cap_{\Phi^G\xi}\Phi^G-"'] && \uniqueMapX_!\Phi^G(\xi\otimes-)\simeq \uniqueMapX_!(\Phi^G\xi\otimes \Phi^G(-)) \uar["\beckChevalley_!", "\simeq"']
        \end{tikzcd}
    \end{center}
    where the vertical Beck--Chevalley maps are equivalences, the right one by \cref{obs:parametrised_colimits_of_s_*-categories} and the left one by \cref{cor:transferring_colimit_BC_to_limit_BC_by_ambidexterity} since $\udl{X}$ was assumed to be twisted ambidextrous. By \cref{obs:parametrised_colimits_of_s_*-categories}, the bottom map identifies with $\Phi^G c\:\cap_{\Phi^G\xi}\Phi^G-\colon X^G_*\Phi^G(-)\rightarrow X^G_!(\Phi^G\xi\otimes \Phi^G-)$, which is an equivalence by hypothesis. Thus, in total, we see that the top horizontal map in the square above is an equivalence, as was to be shown.
\end{proof}

\subsection{Construction principles}\label{subsection:construction_principles}

In this section we will study various results on how to build new Poincar\'e duality spaces out of old ones.

\subsubsection*{Change of groups}

We begin by studying the effect of standard equivariant operations on $\udl{X}$. Recall the constructions and notations  from \cref{cons:res_ind_coind} and \cref{cons:isotropy_separation_recollement}.

\begin{prop}[Poincar\'e duality and restriction]\label{prop:restriction_poincare_complex}
    Suppose that $\alpha \colon H \to G$ is a continuous homomorphism of compact Lie groups and $\underline{X} \in \spc_G$.
    If $\underline{X}$ is a $G$--Poincar\'e space, then $\res_\alpha \underline{X}$ is a $H$--Poincar\'e space with Spivak datum $(\res_\alpha c, \res_\alpha D_{\udl{X}})$ where
    \begin{enumerate}
        \item the local system $\res_\alpha D_{\udl{X}}$ is $\res_\alpha \udl{X} \to \res_\alpha \myuline{\spectra}_G \to \myuline{\spectra}_H$ and
        \item the collapse $\res_\alpha c$ is $\unit_{\spectra_H} = \res_\alpha \unit_{\spectra_G} \to \res_\alpha \uniqueMapX_! D_{\udl{X}} \simeq \pointProjectiontopIndex{\res_\alpha \udl{X}}_! \res_\alpha D_{\udl{X}}$.
    \end{enumerate}
\end{prop}
\begin{proof}
    If $\udl{X}$ is $\myuline{\spectra}_G$--Poincar\'e, then applying \cref{thm:base_change_of_tw_amb_spivak_data} for the symmetric monoidal $G$-colimit preserving functor $\res_\alpha \colon \myuline{\spectra}_G \to \coind_\alpha \myuline{\spectra}_H$ from \cref{cons:spectral_restriction} shows that $\udl{X}$ is $\coind_\alpha \myuline{\spectra}_H$--Poincar\'e.
    By \cref{prop:base_change_twisted_ambidexterity_geometric_morphism} (d) we see that $\res_\alpha \udl{X}$ is $\myuline{\spectra}_H$--Poincar\'e with claimed Spivak datum.
\end{proof}

\begin{prop}[Poincar\'e duality and inflation]\label{prop:pd_inflation}
    Consider a closed normal subgroup $N \le G$ and a $G/N$-space $\udl{X}$.
    Then $\udl{X}$ is a $G/N$--Poincar\'e duality space if and only if $\infl_G^{G/N} \udl{X}$ is a $G$--Poincar\'e duality space.
\end{prop}
\begin{proof}
    One direction is a consequence of \cref{prop:restriction_poincare_complex} while the other one follows from \cref{thm:fixed_points_poincare_duality}.
\end{proof}

\begin{prop}[Poincar\'e duality and induction]
    Let $\groupInjection \colon H \rightarrow G$ be an injective homomorphism of compact Lie groups. If $\udl{X}$ is a $H$--Poincar\'e space, then $\ind^G_H \udl{X}$ is $G$--Poincar\'e space.
\end{prop}

\begin{proof}
    We first claim that the map $\ind_H^G \udl{X} \to \myuline{G/H}$ is a $G$--Poincar\'e map.
    Using the equivalence $(\spc_G)_{/\underline{G/H}} \simeq \spc_H$, this is equivalent to $\res_H^G \ind_H^G \udl{X}$ being a $H$--Poincar\'e space.
    Observe that $\res_H^G \ind_H^G \udl{X} \simeq \udl{X} \times \res_H^G \myuline{G/H}$.
    As $\myuline{G/H}$ is a $G$--Poincar\'e space, the claim follows from \cref{cor:pd_finite_products} and \cref{prop:restriction_poincare_complex}.
    Now \cref{prop:Poincare_duality_composition} implies that the composite $\ind_H^G \udl{X} \to \myuline{G/H} \to \udl{*}$ is a $G$--Poincar\'e map meaning that $\ind_H^G \udl{X}$ is a $G$--Poincar\'e space.
\end{proof}

For the next result, we will need to restrict to the case of finite groups since we will need to invoke the theory of $G$--symmetric monoidal structures as introduced in \cite{Nardin2017Thesis} and further developed in \cite{nardinShah}.

\begin{recollect}[Multiplicative norms]\label{recollect:multiplicative_norms}
    Nardin constructed in \cite{Nardin2017Thesis} a $G$--symmetric monoidal structure for the $G$--category of genuine $G$--spectra $\myuline{\spectra}$, packaging the multiplicative norms of \cite{greenleesMayMU, HHR} coherently. For a finite $G$--set, $U = \coprod_iG/H_i$, we write $\cat_U\coloneq \prod_i\cat_{H_i}$ and write $\myuline{\spectra}_U\coloneqq (\myuline{\spectra}_{H_i})_i\in\cat_U$. For a map of finite $G$--sets $f\colon U \rightarrow V$, we get an adjunction $f^* \colon \cat_V\rightleftharpoons \cat_U \cocolon f_*$ where $f^*$ is given by restrictions and $f_*$ is given by coinductions. As part of the $G$--symmetric monoidal structure on $\myuline{\spectra}$, we have a map $f_{\otimes}\colon f_*\myuline{\spectra}_U\rightarrow \myuline{\spectra}_V$ encoding the multiplicative norm along $f$. For example, when $f$ is the map $f\colon G/H\rightarrow
     G/G$, this encodes a map $f_{\otimes}\colon \coind_{H}^G\myuline{\spectra}_H\rightarrow \myuline{\spectra}_G$, upon applying the functor $(-)^G$ to which yields the multiplicative norm $\norm^G_H\colon \spectra_H\rightarrow \spectra_G$. By \cite[$\S3.3$]{nardinShah}, for a fixed $\udl{X}\in\spc_G$, we may obtain a pointwise $G$--symmetric monoidal structure on the functor category $\udl{\func}(\udl{X},\myuline{\spectra})$. From this, we ma for example extract the pointwise multiplicative norm functor
    \[f_{\otimes}\colon f_*\underline{\func}(\underline{X},f^*\myuline{\spectra})\longrightarrow f_{\otimes}\underline{\func}(\underline{X},f^*\myuline{\spectra})\simeq \underline{\func}(f_*\underline{X},\myuline{\spectra})\]
    where the equivalence is by \cite[Cor. 2.2.20]{kaifNoncommMotives}.
\end{recollect}

\begin{prop}[Poincar\'{e} duality and coinductions]\label{prop:coinduction_of_PD_spaces}
    Let $G$ be a finite group and $\{H_i\}_i$ a finite collection of subgroups of $G$. Suppose for each $i$, we have a $H_i$--Poincar\'{e} space $\udl{X}_i\in\spc_{H_i}$ with dualising sheaf $D_{\udl{X}_i}$. Then $\prod_i\coind^G_{H_i}\underline{X}_i\in\spc_G$ is a $G$--Poincar\'{e} space with dualising sheaf $\bigotimes_i\norm^G_{H_i}{D}_{\underline{X}_i}\in \underline{\func}(\prod_i\coind^G_{H_i}\myuline{X},\myuline{\spectra})$. 
\end{prop}
\begin{proof}
    We  consider a map of finite $G$--sets $f\colon U = \coprod_iG/H_i\rightarrow V= G/G$ as in \cref{recollect:multiplicative_norms}. Writing $\udl{X}\coloneqq (\udl{X}_i)_i\in \cat_U$, we have  an equivalence of the two functors
    \begin{equation}\label{eqn:poincareDualityOnHSpectra}
        \begin{tikzcd}
            \underline{\func}(\underline{X},f^*\myuline{\spectra}) \ar[r, "{D}_{\underline{X}}\otimes-"'] \ar[rr, bend left = 20, "\uniqueMapX_*"description]&\underline{\func}(\underline{X},f^*\myuline{\spectra})\rar["\uniqueMapX_!"'] & f^*\myuline{\spectra}
        \end{tikzcd}
    \end{equation}
    by our hypothesis. Now writing $f_*\uniqueMapX\colon f_*\underline{X}\rightarrow \terminalTCat$ for the unique map, note that since $\uniqueMapX_*$ itself has a right adjoint, we may use \cite[Lem. 4.4.3]{kaifPresentable} to see that applying $f_{\otimes}$ preserves the adjunctions $\uniqueMapX_!\dashv \uniqueMapX^*\dashv \uniqueMapX_*$ in the sense that we have the adjunctions
    \begin{center}
        \begin{tikzcd}
            f_{\otimes}\underline{\func}(\underline{X},f^*\myuline{\spectra}) \ar[rr,bend left = 25, "f_{\otimes}(\uniqueMapX_!)"description]\ar[rr,bend right = 25, "f_{\otimes}(\uniqueMapX_*)"description]&& f_{\otimes}f^*\myuline{\spectra} \ar[ll, "f_{\otimes}(\uniqueMapX^*)"'description]
        \end{tikzcd}
    \end{center}
    But then, since  $\underline{\func}(-,\myuline{\spectra}) \colon \underline{\cat}_G^{\underline{\times}}\rightarrow \underline{\presentable}_{L,\mathrm{st}}^{\underline{\otimes}}$ functorial in left Kan extensions is $\underline{\otimes}$--symmetric monoidal by \cite[After Cor. 6.0.11]{nardinShah} together with \cite[Cor. 2.2.20]{kaifNoncommMotives},  we get
    \[f_{\otimes}(\uniqueMapX_!)\simeq (f_*\uniqueMapX)_!\colon f_{\otimes}\underline{\func}(\underline{X},f^*\myuline{\spectra})\simeq \underline{\func}(f_*\underline{X},\myuline{\spectra}) \longrightarrow f_{\otimes}f^*\myuline{\spectra}\simeq \myuline{\spectra}\]
    and thus consequently, also that $f_{\otimes}(\uniqueMapX^*)\simeq (f_*\uniqueMapX)^*$ and $f_{\otimes}(\uniqueMapX_*)\simeq (f_*\uniqueMapX)_*$. Next, note that the functor ${D}_{\underline{X}}\otimes-$ may be written as
    \[f^*\myuline{\spectra}\otimes \underline{\func}(\underline{X},f^*\myuline{\spectra})\xlongrightarrow{{D}_{\underline{X}}\otimes\id} \underline{\func}(\underline{X},f^*\myuline{\spectra})\otimes \underline{\func}(\underline{X},f^*\myuline{\spectra}) \xlongrightarrow{\otimes}\underline{\func}(\underline{X},f^*\myuline{\spectra})\]
    Thus applying $f_{\otimes}$ to this composite and using that $f_{\otimes}$ is itself a symmetric monoidal functor, we get the identification of $f_{\otimes}({D}_{\underline{X}}\otimes-)$ as
    \[\myuline{\spectra}\otimes f_{\otimes}\underline{\func}(\underline{X},f^*\myuline{\spectra})\xlongrightarrow{f_{\otimes}\underline{D}_{\underline{X}}\otimes\id} f_{\otimes}\underline{\func}(\underline{X},f^*\myuline{\spectra})\otimes f_{\otimes}\underline{\func}(\underline{X},f^*\myuline{\spectra}) \xlongrightarrow{\otimes}f_{\otimes}\underline{\func}(\underline{X},f^*\myuline{\spectra})\]
    That is, that $f_{\otimes}({D}_{\underline{X}}\otimes-)\simeq f_{\otimes}{D}_{\underline{X}}\otimes -$. Therefore, all in all, applying $f_{\otimes} $ to the identification in \cref{eqn:poincareDualityOnHSpectra}, we obtain an equivalence
    \[(f_*\uniqueMapX)_*\simeq f_{\otimes}(\uniqueMapX_*) \simeq f_{\otimes}\big(\uniqueMapX_!({D}_{\underline{X}}\otimes-)\big)\simeq (f_*\uniqueMapX)_!(f_{\otimes}{D}_{\underline{X}}\otimes-)\] as was to be shown.
\end{proof}

\begin{prop}[Poincar\'e duality and Borelification]\label{lem:borelification_fixed_points_poincare}
    Let $\sC\in\calg(\presentable^L)$ and $\udl{X} \in \spc_G$ such that $X^e$ is nonequivariantly a $\category{C}$--twisted ambidextrous (resp. Poincar\'e) space. Then $\udl{X}$ is a $\udl{\borel}(\category{C})$--twisted ambidextrous (resp. Poincar\'e) space.
\end{prop}
\begin{proof}
    Since $\ast\rightarrow BG$ is an effective epimorphism, we may apply \cref{prop:Poincare_duality_descent}
    to the fibre sequence $X^e \to X_{hG} \xrightarrow{\pi} BG$ to get that $\pi$ is a $\category{C}$--twisted ambidextrous (resp. Poincar\'e) map.
    Writing $s \colon BG \to \orbit(G)$ for the inclusion and using the identification $\spc_{/BG} = \func(BG, \spc)$ under which $\pi$ corresponds to $s^* \udl{X}$, this means by \cref{def:pd_map} that $s^* \udl{X}$ is $\pi_{BG}^* \category{C}$--twisted ambidextrous (resp. Poincar\'e), where $\pi_{BG}^* \colon \cat \to \cat_{BG}$ denotes the restriction functor.
    Now the basechange result \cref{prop:base_change_twisted_ambidexterity_geometric_morphism} shows that $\udl{X}$ is $s_* \pi_{BG}^* \category{C} = \udl{\borel}(\category{C})$--twisted ambidextrous (resp. Poincar\'e).
\end{proof}

\begin{lem}[Degree one data and Borelification]\label{lem:degree_one_and_borelifications}
    Let $\category{C} \in \calg(\presentable^L)$ be a presentably symmetric monoidal category and $f \colon \udl{X} \to \udl{Y}$ a map of $G$--spaces such that $X^e$ and $Y^e$ are nonequivariantly $\udl{\category{C}}$--twisted ambidextrous.
    Suppose that $\alpha \colon D_{X^e} \simeq f^* D_{Y^e}$ is a $G$--equivariant $\udl{\category{C}}$--degree datum for $f^e \colon X^e \to Y^e$, i.e. $\alpha$ is an equivalence in $\func(X^e, \category{C})^{hG}$.
    Then there is a $\udl{\borel}(\category{C})$--degree datum for $f \colon \udl{X} \to \udl{Y}$.
    If in addition the $G$--equivariant degree datum for $f^e$ is $G$-equivariantly of degree one, i.e. there is an equivalence $c_{Y^e} \simeq f_! c_{X^e}$ in $\map(\unit_{\category{C}}, (Y^e)_! D_{Y^e})^{hG}$, then the $\udl{\borel}(\category{C})$--degree datum for $f$ is of degree one.
\end{lem}
\begin{proof}
    With the notation from \cref{lem:borelification_fixed_points_poincare}, the assumption on $G$-equivariance of the degree datum implies that means that $\alpha$ is a $\pi_{BG}^* \category{C}$-degree datum for the map $f^e \colon X^e \to Y^e$ in $\spc^{BG}$.
    If $\alpha$ is $G$-equivariantly of degree one, then the $\pi_{BG}^* \category{C}$-degree datum for $f^e$ is of degree one.
    Now \cref{lem:degree_one_base_change} provides us with a $\udl{\borel}(\category{C})$ degree datum for $f \colon \udl{X} \to \udl{Y}$ (which is of degree one if $\alpha$ was $G$-equivariantly of degree one).
\end{proof}

\subsubsection*{Family nilpotence}
We now study how Poincar\'{e} duality interacts with the $\family$--nilpotence theory of \cite{MNN17}.

\begin{prop}\label{prop:family_completeness_poincare_duality}
    Let $\family$ be a family of subgroups of $G$, $\underline{\category{C}}$ a presentably symmetric monoidal $G$--category which is $\family$--Borel complete, and $\underline{X}\in\spc_G$. Then $\underline{X}$ satisfies $\underline{\category{C}}$--Poincare duality if and only if $\res^G_H\underline{X}$ satisfies $\res^G_H \underline{\category{C}}$--Poincare duality for all $H\in\family$.
\end{prop}
\begin{proof}
    Using that $\category{C}$ is $\family$-Borel complete, \cref{prop:base_change_twisted_ambidexterity_geometric_morphism}
    shows that $\udl{X}$ is a $\category{C} \simeq b_* b^*\category{C}$--Poincar\'e space if and only if $b^* \udl{X}$ is a $b^*\udl{\category{C}}$--Poincar\'e space. 
    But by definition, this is equivalent to the map $\udl{X} \to \udl{E \family}$ being a $\udl{\category{C}}$--Poincar\'e duality map.

    There is an effective epimorphism $\coprod_{H \in \family} \myuline{G/H} \to \udl{E \family}$.
    The descent result \cref{prop:Poincare_duality_descent} together with \cref{lem:pd_map_infinite_coproducts} now shows that $\udl{X} \to \udl{E \family}$ is a $\udl{\category{C}}$--Poincar\'e duality map if and only if $\myuline{G/H} \times_{\udl{E \family}} \udl{X} \to \myuline{G/H}$ is a $\udl{\category{C}}$--Poincar\'e duality map for all $H \in \family$.
    Note that under the equivalence $(\spc_G)_{/\myuline{G/H}} \simeq \spc_H$ the map $\myuline{G/H} \times_{\udl{E \family}} \udl{X} \to \myuline{G/H}$ corresponds to $\res_H^G \udl{X}$.
    Similarly, this equivalence identifies $\pi_{\myuline{G/H}}^* \udl{\category{C}}$ and $\res_H^G \udl{\category{C}}$.
\end{proof}

Recall the notion of $\family$--nilpotent ring $G$--spectra from \cite[Def. 6.36]{MNN17}.

\begin{cor}\label{cor:ring_nilpotence_poincare_duality}
    Let $G$ be a finite group, $\family$  a family of subgroups,  $R\in\calg(\spectra_G)$ an $\family$--nilpotent ring $G$--spectrum, and $\udl{X}\in\spc_G$. Then $\udl{X}$ is an $R$--Poincar\'{e} space if and only if for all $H\in\family$, $\res^G_H\udl{X}$ is $\res^G_HR$--Poincar\'{e}. 
\end{cor}
\begin{proof}
    By \cref{example:modules_over_F-nilpotent-rings}, we know that the presentably symmetric monoidal $G$--stable category $\udl{\module}_R(\myuline{\spectra})$ is $\family$--Borel and so we may apply \cref{prop:family_completeness_poincare_duality} to conclude.
\end{proof}

\begin{example}
    We collect here a small list of potentially interesting consequences of \cref{cor:ring_nilpotence_poincare_duality} using known nilpotence results from \cite[Table 2]{MNNDerivedInduction} for finite groups $G$. We invite the reader to consult the cited table for a quite exhaustive list of possibly interesting examples of coefficient ring $G$--spectra to consider. 
    \begin{enumerate}[label=(\arabic*)]
        \item Writing $\mathrm{KO}_G$ and $\mathrm{KU}_G$ for Segal's equivariant topological K--theories, we see that a $G$--space $\udl{X}$ is $\mathrm{KO}_G$-- or $\mathrm{KU}_G$--Poincar\'{e} if and only if $\res^G_C\udl{X}$ is  $\mathrm{KO}_C$-- or $\mathrm{KU}_C$--Poincar\'{e} for all cyclic subgroups $C\leq G$.
        \item A $G$--space $\udl{X}$ is $\borel_G(\eilenbergMacLaneCoeff)$--Poincar\'{e} if and only if $\res^G_E\udl{X}$ is $\borel_E(\eilenbergMacLaneCoeff)$--Poincar\'{e} for all elementary abelian $p$--subgroups $E\leq G$ for all primes $p$.
        \item A $G$--space $\udl{X}$ is $\borel_G(\mathrm{MU})$--Poincar\'{e} if and only if $\res^G_A\udl{X}$ is $\borel_A(\mathrm{MU})$--Poincar\'{e} for all  abelian $p$--subgroups $A\leq G$ for all primes $p$.
        \item A $G$--space $\udl{X}$ is $\borel_G(\mathrm{MO})$--Poincar\'{e} if and only if $\res^G_A\udl{X}$ is $\borel_A(\mathrm{MO})$--Poincar\'{e} for all elementary abelian $2$--subgroups $A\leq G$.
    \end{enumerate}
\end{example}

\subsubsection*{Poincar\'e integration}

In this section we study the equivariant generalisation of a well known result of Klein \cite[Corollary F]{Klein2001}\footnote{In \cite{Klein2001}, Klein mentioned that the result answered a question of Wall and also attributed the result to Quinn from an unpublished announcement and Gottlieb \cite{gottliebPoincare} who proved it in the manifolds setting. } which says  that in a fibration $F \to E \to B$ of finitely dominated spaces, $E$ is a Poincar\'e  space if $F$ and $B$ are Poincar\'e spaces. Since we have  $E\simeq \colim_BF$ by the straightening--unstraightening equivalence, the aforementioned result may be viewed as saying that integrating a Poincar\'{e} space along a diagram which is itself Poincar\'{e} yields a Poincar\'{e} space.

\begin{terminology}[Fibrewise twisted ambidextrous and Poincar\'{e} maps]\label{terminology:fibrewise_PD}
    Let $f\colon\udl{X}\rightarrow\udl{Y}$ be a map of $G$--spaces for $G$ a compact Lie group and $\udl{\category{C}}$ a presentably symmetric monoidal $G$-category. 
    We say that it is a \textit{fibrewise $\udl{\category{C}}$--twisted ambidextrous (resp. Poincar\'{e}) map} if for all closed subgroups $H\leq G$ and all maps $y\colon \myuline{G/H}\rightarrow\udl{Y}$, writing $F_y$ for the pullback $\myuline{G/H}\times_{\udl{Y}}\udl{X}$,
    the map $F_y\rightarrow \myuline{G/H}$ is $\udl{\category{C}}$--twisted ambidextrous (resp. Poincar\'{e}).
    Expanding \cref{defn:twisted_ambidex_maps,def:pd_map}, this means that viewed as an object in $\spc_H\simeq (\spc_G)_{/\myuline{G/H}}$, $F_y$ is $\res_H^G\udl{\category{C}}$--twisted ambidextrous (resp. Poincar\'{e}). 

    In the case where $\udl{\category{C}} = \myuline{\spectra}_G$ this means that $F_y$ is $\res_H^G \myuline{\spectra}_G = \myuline{\spectra}_H$--twisted ambidextrous (resp. Poincar\'e).
    Since $\myuline{G/H}$ is $G$--Poincar\'{e}, we see by \cref{prop:Poincare_duality_composition}, \cref{cor:poincare_duality_of_disjoint_unions}, and \cref{prop:restriction_stability_of_Poincare_duality} that in this case the preceding condition is also equivalent to $F_y$ being $G$--twisted ambidextrous (resp. Poincar\'{e}).
\end{terminology}

\begin{thm}[Equivariant Poincar\'{e} integration]\label{thm:poincare_integration}
    Let  $f \colon \udl{X} \to \udl{Y}$ be a map of $G$--spaces and $\udl{\category{C}}$ a presentably symmetric monoidal $G$-category.
    If $\udl{Y}$ is a $\udl{\category{C}}$--Poincar\'e space and $f$ is a fibrewise $\udl{\category{C}}$--Poincar\'{e} map,  then $\udl{X}$ is a $\udl{\category{C}}$--Poincar\'e space.
    Furthermore, there is an equivalence $D_{\udl{X}} \simeq f^* D_{\udl{Y}} \otimes D_f$, where $D_f \in \func_G(\udl{X}, \udl{\category{C}})$ such that $y^* D_f \simeq D_{F_y}$ is the dualising sheaf of the fibres.

    Conversely, suppose that $\udl{Y}$ is $\udl{\category{C}}$--twisted ambidextrous and that $f$ is fibrewise $\udl{\category{C}}$-twisted ambidextrous.
    Furthermore assume that for all closed subgroups $H \le G$ the map $f^H \colon X^H \to Y^H$ is a $\pi_0$-surjection.
    If $\udl{X}$ is a $\udl{\category{C}}$--Poincar\'e space, then $\udl{Y}$ is also a $\udl{\category{C}}$--Poincar\'e  space and  $f$ is a fibrewise Poincar\'{e} map.
\end{thm}
\begin{proof}
    The map $\coprod_{H \le G} \coprod_{\pi_0(Y^H)} \myuline{G/H} \to \udl{Y}$ is a $\pi_0$ surjection on each fixed point space and thus an effective epimorphism in $\spc_G$ (see \cref{ex:effective_epimorphism_presheaf_topos}).
    It then follows from \cref{prop:Poincare_duality_descent} and \cref{lem:pd_map_infinite_coproducts} that $f$ is a $\udl{\category{C}}$--Poincar\'e  map if and only if the map $F_y \to \myuline{G/H}$ is  $\udl{\category{C}}$--Poincar\'e for all closed subgroups $H\leq G$ and all $y\colon \myuline{G/H}\rightarrow\udl{Y}$.
    This is precisely what it means that $f$ was fibrewise $\udl{\category{C}}$--Poincar\'{e}.
    Moreover, \cref{prop:Poincare_duality_descent} also provides an equivalence $y^* D_f \simeq D_{F_y}$.
    Since in addition $\udl{Y}$ is a $\udl{\category{C}}$--Poincar\'e  space, it follows from \cref{prop:Poincare_duality_composition} that $\udl{X}$ is a $\udl{\category{C}}$--Poincar\'e  space and there is an equivalence $D_{\udl{X}} \simeq y^* D_{\udl{Y}} \otimes D_{F_y}$ as desired.

    For the converse, as in the first we conclude from \cref{prop:Poincare_duality_descent} and \cref{prop:Poincare_duality_composition} that there is an equivalence $D_{\udl{X}} \simeq f^* D_{\udl{Y}} \otimes D_f$.
    If $\udl{X}$ is a $\udl{\category{C}}$--Poincar\'e, then $D_{\udl{X}}$ is invertible which implies that $f^* D_{\udl{Y}}$ and $D_f$ are invertible.
    The $\pi_0$ surjectivity hypothesis on $f$ implies that $D_{\udl{Y}}$ is invertible so $\udl{Y}$ is a $\udl{\category{C}}$--Poincar\'e space.
    From the equivalence $y^* D_f \simeq D_{F_y}$ we see that $(F_y\rightarrow\myuline{G/H})$ is a $\res_H^G\udl{\category{C}}$--Poincar\'e  space.   
\end{proof}

We now use the theorem above to obtain a characterisation of $G$--Poincar\'{e} duality for spaces with free actions in terms of Poincar\'{e} duality for a quotient group.

\begin{cor}[Poincar\'e duality and quotients by free actions] \label{cor:free_actions}
    Let $G$ be a compact Lie group, $N \leq G$ a closed normal subgroup and $Q \coloneqq G/N$.
    If $\udl{X}$ is a $G$--space such that the action of $N$ on $\udl{X}$ is free in the sense of \cref{def:free_action}, then $\udl{X}$ is $G$--Poincar\'e duality space if and only if $N \backslash \udl{X}$ is a $Q$--Poincar\'e duality space.
\end{cor}
\begin{proof}
    We will show that this follows from \cref{thm:poincare_integration}. To do so, it suffices to check that for each map $\myuline{G/H} \rightarrow \infl_G^Q N \backslash \udl{X}$ the space $\myuline{G/H} \times_{\infl_G^Q N\backslash \udl{X}} \udl{X} $ is a $G$--Poincar\'e space. But in \cref{cor:fibre_of_quotient} we have seen that  there exists a cartesian square of the following form.
    \begin{center}
        \begin{tikzcd}
            \myuline{G/H} \times_{\infl_G^Q N \backslash \udl{X}} \udl{X} \ar[r] \ar[d, "\text{proj}"] & \myuline{G/K_0} \ar[d, "f"]\\
            \myuline{G/H} \ar[r] & \myuline{G/K_1}
        \end{tikzcd}
    \end{center}
    Let $S$ denote the point-set fibre of the map of topological $G$--spaces $f \colon G/K_0 \rightarrow G/ K_1$. Then $S$ is a homogenous $K_1$-space, and $f = \ind_{K_1}^G(S \rightarrow *)$.
    Now note that since the right map is a Poincar\'e duality map, so is the left one, and as $\myuline{G/H}$ is $G$--Poincar\'e, \cref{prop:Poincare_duality_composition} implies that $\myuline{G/H} \times_{\infl_G^Q N \backslash \udl{X}} \udl{X}$ is $G$--Poincar\'e, as desired.
\end{proof}

\subsection{Examples}\label{sec:examples_G_PD}

The next paragraphs will introduce two different sources of equivariant Poincar\'e spaces. First, we show that smooth $G$-manifolds are equivariantly Poincar\'e. Their study is one of the main motivations for a theory of equivariant Poincar\'e duality, and equivariant Poincar\'e spaces should be viewed as their homotopical analogue. Let us mention that while the proofs given here \textit{depend} on the Wirthm\"uller isomorphism, the Wirthm\"uller isomorphism can also be proven using a different version of equivariant Poincar\'e duality, as is done for example in \cite{MaySigurdsson2006}. 

Our second source of examples are tom Dieck--Petrie's generalised homotopy representations. Here we will find what we consider to be the strangest equivariant Poincar\'e space we know: a $C_p$--Poincar\'e space $\udl{X}$ such that $X^{C_p}$ and $X^e$ are Poincar\'e of the same dimension, yet the map $X^{C_p} \rightarrow X^e$ is not an equivalence, see \cref{ex:strange_ghrep}. 

A general principle here is that \cref{thm:PD_fixed_point_recognition} provides us with a clear strategy to deduce equivariant Poincar\'e duality from nonequivariant Poincar\'e duality of fixed points, provided an appropriate Spivak datum has been constructed.

\subsubsection*{Smooth $G$-manifolds}

Let $G$ be a compact Lie group. A \textit{smooth $G$-manifold} is a smooth manifold on which $G$ acts such that the action map $G \times M \rightarrow M$ is a smooth map. An \textit{equivariant embedding} of smooth $G$-manifolds is a smooth embedding between smooth $G$-manifolds that is also equivariant. 
An \textit{equivariant vector bundle} on $M$ is a tuple $\xi = (E,p)$, where $E$ is a smooth $G$--manifold and $p \colon E \rightarrow M$ is an equivariant map which is a vector bundle where  $G$ acts by bundle maps. For $x \in M^H$, the vector space $E_x \coloneqq p^{-1}(x)$ carries an $H$-action by restriction. Smooth $G$-manifolds port nicely into our homotopical context by virtue of \cite[Cor. 7.2.]{Illman83} which guarantees that smooth $G$--manifolds admit the structure of $G$--CW complexes which is necessarily finite for compact manifolds.
We recommend \cite[Chapter IV]{bredonTrans} for an introduction to the theory of smooth $G$-manifolds.

\begin{fact}\label{fact:equivariant_smooth_manifold_theory}
We collect here some basic facts from equivariant smooth manifold theory that we will need for our purposes.
\begin{enumerate}[label = (\roman*)]
    \item The tangent bundle of a smooth $G$-manifold  can naturally be considered as an equivariant vector bundle \cite[p. 303]{bredonTrans}. If $f \colon M \rightarrow N$ is an equivariant embedding of smooth $G$-manifolds, then the equivariant tubular neighborhood theorem provides a smooth equivariant embedding of $\nu(f) = f^* TN / TM$ into $N$ \cite[Thm. VI.2.2.]{bredonTrans}.
    \item Let us denote the underlying $G$-homotopy type of $M$ by $\udl{M}$.
    Any $G$-vector bundle $p \colon E \rightarrow M$ over $M$ defines a stable equivariant spherical fibration  of the $G$-vector bundle $p \colon E \rightarrow M$. Furthermore, we can choose a $G$--invariant Riemannian metric for $p$ from which we obtain an associated unit disc bundle $D(p) \subset E$ and unit sphere bundle $S(p) \subset E$.  The fibrewise collapse maps $\udl{S^{E_x}} \rightarrow \cofib(\udl{S(p)}_x \rightarrow \udl{D(p)}_x)$ for each $x\in M$ then assemble into a $G$-equivalence
    \[ M_!(J(p)) \xlongrightarrow{\simeq} \Sigma^\infty \cofib\big[S(p) \rightarrow D(p)\big]. \]
    \item For each $G$-manifold $M$, there exists an equivariant embedding into some $G$-representation $V$. This is the content of the Mostow-Palais theorem, see \cite{Palais57}.
\end{enumerate}
\end{fact}

\begin{prop}
    \label{thm:closed_smooth_G_manifolds_are_poincare}
    Let $M$ be a closed smooth $G$--manifold. Then the underlying $G$--space $\udl{M}$ is a $G$--Poincar\'e space with dualizing object $J(TM)^{-1}$.
\end{prop}

\begin{proof}
    Choose an embedding $f \colon M \rightarrow V$ into some $G$-representation. Denote the normal bundle of $f$ by $\nu = (p\colon E \rightarrow M)$ and pick a tubular neighborhood of $M$ in $V$. 
    
    Consider the Pontryagin-Thom collapse map
    \begin{align*}
        c \colon \sphere \xrightarrow{\simeq} \sphere^V \otimes \sphere^{-V} \rightarrow \Sigma^\infty \cofib\big[S^V \setminus (D(\nu) \setminus S(\nu)) \rightarrow S^V\big] \otimes \sphere^{-V} \\\simeq \Sigma^\infty \cofib\big[S(\nu) \rightarrow D(\nu)\big] \otimes \sphere^{-V} \simeq M_!(J(\nu) \otimes \sphere^{-V}).
    \end{align*}
    We claim that the Spivak datum $(J(\nu)\otimes \sphere^{-V},c)$ is Poincar\'e. Since $M$ is a $G$--compact space and $J(\nu)$ is invertible, by \cref{thm:PD_fixed_point_recognition}, it suffices to check that for every $H \subset G$, the Spivak datum $(\Phi^H J(\nu), \Phi^H c)$ is a Poincar\'e Spivak datum for $M^H$. Recall that $\Phi^H J(\nu)$ is
    \[ \Phi^H J(\nu) \colon M^H \rightarrow \Pic(\spectra), \hspace{3mm} x \mapsto \Phi^H(J(\nu)(x)) = \Phi^H \Sigma^\infty S^{E_x} \simeq \Sigma^\infty S^{E_x^H}.  \]
    But this is just the underlying stable spherical fibration of the normal bundle of $M^H$ in $V^H$. The collapse map $\Phi^H c$ identifies with the geometric Pontryagin-Thom collapse map of the smooth manifold $M^H$ embedded in $V^H$. Thus, by \cite[Cor. A.11]{markusPoincareSpaces} the Spivak datum $(\Phi^H J(\nu),\Phi^H c)$ is Poincar\'e.

    Now note that the equivalence $\constant_V = TV|_{M} \simeq \nu \oplus TM$ shows that
    \[ J(\nu) \otimes \sphere^{-V} \simeq J(\nu) \otimes J(\constant_V)^{-1} \simeq J(\constant_V) \otimes J(TM)^{-1} \otimes J(\constant_V)^{-1} \simeq J(TM)^{-1} \]
    as claimed.
\end{proof}

\begin{rmk}
    We want to mention that versions of \cref{thm:closed_smooth_G_manifolds_are_poincare} are already contained in the literature so we do not claim any originality. In particular, May--Sigurdsson give an account of equivariant Poincar\'e duality and show that closed smooth $G$-manifolds satisfy Poincar\'e duality in their sense \cite[Chapter 18.6.]{MaySigurdsson2006}.
    Depending on which proof of the Wirthm\"uller isomorphism the reader has in mind, the reader might complain that the proof of \cref{thm:closed_smooth_G_manifolds_are_poincare} is circular, as the Wirthm\"uller isomorphism for compact Lie groups itself was proved by showing that smooth $G$-manifolds are $G$-Poincar\'e. 
    Another variant of \cref{thm:closed_smooth_G_manifolds_are_poincare} was given by Costenoble--Waner, see \cite{costenoble2017equivariant}.
\end{rmk}

\subsubsection*{Generalised homotopy representations}

We now turn our attention to another interesting source of equivariant Poincar\'e duality spaces, namely the class of generalised homotopy representations of tom Dieck--Petrie \cite{tomDieckPetrie82}. 

\begin{defn}
    A \textit{generalised homotopy representation} of a compact Lie group $G$ is a compact $G$--space $\underline{\hrep}$ such that for each closed subgroup $H \le G$  the space $\hrep^H$ is equivalent to $S^{n(H)}$ for some $n(H) \in \mathbb{N}$. The function $H \mapsto n(H)$ associated to a generalised homotopy representation is called its \textit{dimension function}.\footnote{Beware that it is also common in the literature to shift the dimension function by one.}
\end{defn}

Examples of generalised homotopy representations are unit spheres of finite dimensional orthogonal $G$-representations or one--point compactifications of finite dimensional linear $G$-representations.

\begin{rmk}
   While it will not play a role in this article, let us mention that  \cite{tomDieckPetrie82,tomDieck86} have also studied what are called \textit{homotopy representations}, namely generalised homotopy representations  for which the fixed points have CW--dimensions those of the respective spheres. A special feature of homotopy representations  is that they satisfy an equivariant Hopf degree theorem, i.e. $G$-homotopy classes of self maps are classified by their degree, an element in a Burnside ring.
\end{rmk}

To show that generalised $G$-homotopy representations are indeed $G$--Poincaré, we first recall a construction of a Poincar\'e Spivak datum for the nonequivariant spheres.

\begin{obs}[Spivak data for spheres]\label{obs:spivak_data_for_spheres}
    We construct a Spivak datum for $S^d\in\spc$. Let $E \coloneqq \fib(\Sigma^\infty_+ S^d \rightarrow \Sigma^\infty_+ * \simeq \sphere)$. Then $E \simeq \sphere^{d} \in \Pic(\spectra)$. Consider the composition
    \[ c \colon \sphere \xrightarrow{\simeq} E \otimes E^\vee \rightarrow \Sigma^\infty_+ S^d \otimes E^\vee \simeq S^d_! ({S^d})^* E^\vee. \]
    We argue now that $(({S^d})^*E^{\vee},c)$ is a Poincar\'e Spivak datum for $S^d$. As $S^d$ is stably parallelisable, we know that its dualising sheaf is constant with value $\sphere^{-d} \simeq E^{\vee}$. Assume $d \geq 1$, the case $d=0$ being easier. Now $\pi_0 S^d_! ({S^d})^* E^\vee \simeq \bbZ$, and $c \in \pi_0 S^d_! ({S^d})^* E^\vee \simeq \bbZ$ gives the collapse map of a Poincar\'e Spivak datum if and only if it corresponds to a generator. This  is indeed the case for $(({S^d})^*E,c)$, so it is Poincar\'e as claimed.
\end{obs}

Having this in mind, we can make an educated guess for the 
a Spivak datum of a generalised homotopy representation. To this end, the following terminology will be useful.

\begin{defn}
    A \textit{homotopical framing} for ${\xi} \in \underline{\func}(\underline{X},\myuline{\spectra})$ is a $G$--spectrum $E$ together with an equivalence $ {\xi} \xrightarrow{\simeq} \uniqueMapX^* E$. A compact  $G$--space $\underline{X}$ is \textit{homotopically parallelisable} if its dualising sheaf ${D}_{\underline{X}} \in \underline{\func}(\underline{X},\myuline{\spectra})$ admits a homotopical framing.
\end{defn}

\begin{thm}\label{thm:generalised_homotopy_representations}
    The dualising sheaf of a generalised homotopy representation $\underline{\hrep}$ admits a canonical homotopical framing ${D}_{\underline{\hrep}} \xrightarrow{\simeq} \hrep^* \fib(\Sigma^\infty_+  {\hrep}^{\vee} \rightarrow \Sigma^\infty_+ * \simeq \sphere)^\vee$.  In particular, generalised homotopy spheres are homotopically parallelisable $G$--Poincaré spaces.
\end{thm}

\begin{proof}
    To prove the theorem, we will construct a Poincar\'e Spivak datum whose underlying parametrised spectrum is constant with value $E^{\vee} \coloneqq \fib(\Sigma^\infty_+  {\hrep}^{\vee} \rightarrow \Sigma^\infty_+ * \simeq \sphere)^\vee$. 
    As in \cref{obs:spivak_data_for_spheres}, we have a map
    $c \colon \sphere \rightarrow E \otimes E^\vee \rightarrow \Sigma^\infty_+ {\hrep} \otimes E^\vee \simeq \hrep_! \hrep^* E^\vee$.   Upon taking geometric fixed points, \cref{obs:spivak_data_for_spheres} identifies the composition
    \[ \Phi^H c \colon \Phi^H\sphere \rightarrow \Phi^H E \otimes \Phi^H E^\vee \rightarrow \Phi^H\Sigma^\infty_+ \underline{\hrep} \otimes \Phi^H E^\vee \simeq \hrep^H_! (\hrep^H)^* \Phi^H E^\vee. \]
    as a Poincar\'e Spivak datum for $\hrep^H$. Thus, by \cref{thm:PD_fixed_point_recognition}, we get that $\hrep^*E^{\vee}$ is a Poincar\'{e} $G$--Spivak datum for $\udl{\hrep}$ and by \cref{prop:twisted_ambidextrous_presentable_case} we get $D_{\udl{\hrep}}\simeq \hrep^*E^{\vee}$ as claimed. 
\end{proof}

\begin{lem}
    Suppose that $X\in \spc_G^\omega$ is homotopically parallelisable and that $X^H$ is a Poincaré space for all $H \le G$. 
    Then $X$ is a $G$--Poincaré space. 
\end{lem}
\begin{proof}
    Since $\udl{X}$ was compact, note that $\uniqueMapX_! {D}_{\underline{X}} \simeq \uniqueMapX_*\uniqueMapX^*\sphere_G$ is a compact $G$--spectrum.
    Now suppose that there is $E \in \spectra_G$ such that ${D}_{\underline{X}} \simeq \uniqueMapX^* E$.
    As $E$ is a retract of $\uniqueMapX_! D_{\underline{X}} \simeq \Sigma^\infty_+ {X} \otimes E$ this implies that $E$ is compact itself.
    If all fixed points of $X$ are Poincaré spaces, then all geometric fixed points of $E$ are invertible.
    Together this shows that $E$ is invertible so that $X$ is a $G$--Poincaré space.
\end{proof}

\begin{example}
    \label{ex:strange_ghrep}
    In \cite[p. 391]{bredonTrans},  Bredon constructs a curious example of a generalised homotopy representation. Namely, he constructs examples of compact $C_p$-spaces $\underline{X}$ which satisfy that $X^{C_p} \simeq X^e \simeq S^2$ such that the map $X^{C_p} \rightarrow X^e$ has degree $q= kp+1$ for $k \in \mathbb{Z}$ arbitrary. Taking the unreduced suspension, examples of this type exist in arbitrary dimensions. From Smith theory we know that each generalised $C_p$-homotopy representation has the property that the dimension of the fixed point sphere does not exceed the dimension of the underlying space.
\end{example}

\subsection{Gluing classes}\label{subsection:gluing_classes}

Our next goal is to hint at nontrivial ways in which the fixed points interact. For this, we construct a certain homology class, the \textit{gluing class}, that should be thought of as passing information between the fundamental class of a Poincar\'e space and fundamental classes of various fixed point spaces.
The gluing class will be one of the main tools for our geometric applications. It is inspired by L\"uck's work on the Nielsen realisation problem, specifically by \cite[Notation 1.8 (H) and Lemma 1.8 (5)]{lueck2022brown}. Much of what we will present here will work for compact Lie groups too, but we nevertheless restrict our attention to finite groups $G$ for this subsection which is sufficient for our geometric purposes later. 

\begin{cons}[Nonsingular part]
    Fix $\underline{X}\in\spc_G$, $\family$ a family of subgroups of $G$, and $\udl{\sC}$ a $G$--stable category. Recall the adjunction counit $\epsilon\colon \underline{X}\singularPartTwiddle{\family}\rightarrow\underline{X}$ from \cref{defn:singular_part_inclusion}. This map then itself induces the adjunction $\epsilon_! \colon \underline{\func}(\underline{X}\singularPartTwiddle{\family},\underline{\sC})\rightleftharpoons \underline{\func}(\underline{X},\underline{\sC}) \cocolon \epsilon^*$. The adjunction (co)unit of \textit{this} adjunction then gives us  functors\small
    \[\underline{\func}(\underline{X},\underline{\sC}) \xlongrightarrow{c}\underline{\func}(\underline{X},\underline{\sC})^{\Delta^1}\: ::\: \xi\mapsto (\epsilon_!\epsilon^*\xi\rightarrow\xi),\quad\underline{\func}(\underline{X},\underline{\sC}) \xlongrightarrow{u}\underline{\func}(\underline{X},\underline{\sC})^{\Delta^1}\: ::\: \xi\mapsto (\xi\rightarrow\epsilon_*\epsilon^*\xi).\]\normalsize
    All in all, we can consider the compositions \small
    \[\alpha\colon \underline{\func}(\underline{X},\underline{\sC}) \xrightarrow{c}\underline{\func}(\underline{X},\underline{\sC})^{\Delta^1} \xrightarrow{\uniqueMapX_!} \underline{\sC}^{\Delta^1}\xlongrightarrow{\cofib}\underline{\sC},\quad\beta\colon \underline{\func}(\underline{X},\underline{\sC}) \xrightarrow{u}\underline{\func}(\underline{X},\underline{\sC})^{\Delta^1} \xrightarrow{\uniqueMapX_*} \underline{\sC}^{\Delta^1}\xlongrightarrow{\fib}\underline{\sC}.\]\normalsize
    Concretely, these take $\xi$ to the objects 
    \[\alpha(\xi)\simeq \cofib\Big((\uniqueMapX\singularPartTwiddle{\family})_!\epsilon^*\xi \longrightarrow \uniqueMapX_!\xi\Big),\quad\quad\quad \beta(\xi)\simeq\fib\Big(\uniqueMapX_*\xi \longrightarrow (\uniqueMapX\singularPartTwiddle{\family})_*\epsilon^*\xi\Big).\]
\end{cons}

\begin{cor}[Nonsingular vanishing]\label{lem:nonsingular_vanishing}
    Let $\underline{X}\in\spc_G^{\omega}$ and $\family$ a family of subgroups of $G$. Let $\nu\colon \underline{\sC}\rightarrow\underline{\D}$ be a $G$--exact functor of $G$--stable categories such that for all $H\in\family$, functor $\res^G_H\nu\colon \res^G_H\underline{\sC}\rightarrow \res^G_H\underline{D}$ is the zero map.  Then the compositions 
    \[\underline{\func}(\underline{X},\underline{\sC}) \xlongrightarrow{\alpha} \underline{\sC}\xlongrightarrow{\nu} \underline{\D}\quad\quad\quad \underline{\func}(\underline{X},\underline{\sC}) \xlongrightarrow{\beta} \underline{\sC}\xlongrightarrow{\nu} \underline{\D}\]
    have the property of being the zero functors. 
\end{cor}
\begin{proof}
    First of all, since $\nu$ was $G$--exact, we have a commuting square
    \begin{center}
        \begin{tikzcd}
            \underline{\func}(\underline{X},\underline{\sC}) \dar["\alpha"']\rar["\nu"]& \underline{\func}(\underline{X},\underline{\D}) \dar["\alpha"]\\
            \udl{\sC}\rar["\nu"] & \udl{\D}.
        \end{tikzcd}
    \end{center}
    Thus it suffices to show that $\alpha\colon \underline{\func}(\underline{X},\underline{\D})\rightarrow\udl{\D}$ is the zero functor. By replacing $\udl{\category{D}}$ by the $G$-stable subcategory generated by the image of $\nu$ we can assume that that $\D^H = 0$ for all $H\in\family$. Therefore, we have that $\udl{\category{D}} \simeq \udl{\D}\sTwiddleLowerStar{\family}$ and \cref{lem:counit_equivalence_singular_part} shows that the functor $\epsilon^*$ is an equivalence, and so the counit $\epsilon_! \epsilon^* \to \id$ and unit  $\id \to \epsilon_* \epsilon^*$ are equivalences in $\udl{\func}(\udl{X}, \udl{\category{D}}) = \udl{\func}(\udl{X}, \udl{\category{D}}\sTwiddleLowerStar{\family})$.
    From this the claim directly follows.
\end{proof}

\begin{nota}
    The family of relevance to us in this subsection will be the singleton family $\trivialFamily$ consisting of the trivial subgroup. To reduce our notational cluttering, we will also write $\udl{X}^{>1}$ for $\udl{X}\singularPartTwiddle{\trivialFamily}$, so that for example, for $\udl{X}\in\spc_G$, we have the inclusion of the singular part $\nolinebreak{\epsilon \colon \udl{X}^{>1}\simeq \udl{X}\singularPartTwiddle{\trivialFamily} \rightarrow \udl{X}}$. The gluing class of $\udl{X}$ will live in $\pi_{-1} (\uniqueMapX^{>1}_! \epsilon^* D_{\udl{X}})_{hG}$. 
\end{nota}

\begin{cons}\label{cons:red_blue_routes}
    Let $\xi\in \func_G(\underline{X},\myuline{\spectra})$ and write $Q\coloneqq \cofib\big(\uniqueMapX^{>1}_!\epsilon^*\xi \rightarrow \uniqueMapX_!\xi\big)$. Consider
    \begin{equation}\label{eqn:black_magic_Atiyah-Bott}
        \begin{tikzcd}
            (\uniqueMapX^{>1}_!\epsilon^*\xi)_{hG}\rar\dar\ar[dr, phantom, very near end, "\ulcorner"] & (\uniqueMapX^{>1}_!\epsilon^*\xi)^{hG} \rar\dar& (\uniqueMapX^{>1}_!\epsilon^*\xi)^{tG} \rar[color =blue]\dar["\simeq",color =blue]& \Sigma(\uniqueMapX^{>1}_!\epsilon^*\xi)_{hG}\\
            (\uniqueMapX_!\xi)_{hG}\rar\dar & (\uniqueMapX_!\xi)^{hG}\rar[color =blue]\dar[color =red] & (\uniqueMapX_!\xi)^{tG}\\
            Q_{hG} \rar["\simeq",color =red]\dar[color =red]& Q^{hG} \\
            \Sigma(\uniqueMapX^{>1}_!\epsilon^*\xi)_{hG}
        \end{tikzcd}
    \end{equation}
    where the  equivalence $Q_{hG}\rightarrow Q^{hG}$ is since $Q^{tG}\simeq 0$ by virtue of \cref{lem:nonsingular_vanishing} applied to the functor $\nu\colon\udl{\sC}\rightarrow\udl{\D}$ given by $\widetilde{EG}\otimes F(EG_+,-)\colon \myuline{\spectra}\rightarrow \udl{\module}_{\widetilde{EG}\otimes F(EG_+,\sphere)}(\myuline{\spectra})$ and the identification $(\widetilde{EG}\otimes F(EG_+,A))^G \simeq A^{tG}$.
    Observe that by \cref{lem:black_magic}, up to a sign change, the red composite is equivalent to the blue composite in \cref{eqn:black_magic_Atiyah-Bott}.
\end{cons}

\begin{cons}[Gluing classes]\label{cons:generalised_gluing_class}
    Let $\underline{X}\in \spc_G^{\omega}$ and $D_{\underline{X}}\in\func_G(\underline{X},\myuline{\spectra})$ its dualising sheaf (which in this generality, need not be invertible). From the fundamental class $\sphere_G\xrightarrow{c} \uniqueMapX_! D_{\underline{X}}$ in $\spectra_G$, we may extract a nonequivariant fundamental class $\sphere \xlongrightarrow{\canonical} \sphere_G^{hG} \xlongrightarrow{c^{hG}} (\uniqueMapX_! D_{\underline{X}})^{hG}$   in $\spectra$ which we also denote by $c$. The \textit{gluing class} is defined to be the composition
    \[\sphere \xlongrightarrow{c} (\pointProjection_! D_{\underline{X}})^{hG} \longrightarrow \Sigma(\uniqueMapX^{>1}_! \epsilon^*D_{\underline{X}})_{hG}\]
    obtained by postcomposing $c$ with the blue route from \cref{eqn:black_magic_Atiyah-Bott}. 
\end{cons}

Our  goal now is to show \cref{cor:vanishing_of_pushforward_gluing_class} which says that under certain orientability assumptions, the gluing class ``adds up to zero" in group homology. This  supplies us with a useful obstruction class which will have meaningful geometric consequences as we shall in our applications in  \cref{subsection:theorem:single_fixed_points}.

\begin{lem}\label{lem:vanishing_of_pushforward_fundamental_class}
    Let $\underline{X}\in\spc_G$ and $\xi\simeq \uniqueMapX^*W\in\func_G(\underline{X},\myuline{\spectra})$ for some $W\in\spectra_G$. Then the composition in 
    $Q \longrightarrow \Sigma \pointProjectiontopIndex{>1}_!\epsilon^*\xi\simeq \Sigma \uniqueMapX_!^{>1}(\uniqueMapX^{>1})^*\xi \xlongrightarrow{\beckChevalley_!^{\uniqueMapX^{>1}}} \Sigma W$ in $\spectra_G$ is nullhomotopic.
\end{lem}
\begin{proof}
    By functoriality of colimits, we have the following map of cofibre sequences
    \begin{center}
        \begin{tikzcd}
            \uniqueMapX^{>1}_!\epsilon^*\uniqueMapX^*W \rar["\beckChevalley^{\epsilon}_!"]\dar["\beckChevalley^{X^{>1}}_!"'] & \uniqueMapX_! \uniqueMapX^*W \rar\dar["\beckChevalley^{\uniqueMapX}_!"] & Q\dar\\
            W \rar[equal] & W \rar & 0
        \end{tikzcd}
    \end{center}
    Thus taking the cofibre of the right horizontal maps gives a factorisation of the composition of interest through $0$.
\end{proof}

\begin{cor}\label{cor:vanishing_of_pushforward_gluing_class}
    Let $\underline{X}\in\spc_G$ and $W\in\spectra_G$. Then the composition
    \[(\uniqueMapX_!\uniqueMapX^*W)^{hG}\xlongrightarrow{\text{ red composite in \cref{eqn:black_magic_Atiyah-Bott}}} \Sigma(\uniqueMapX_!^{>1}(\uniqueMapX^{>1})^*W)_{hG} \xlongrightarrow{\beckChevalley^{\uniqueMapX^{>1}}_!} \Sigma W_{hG}\]
    is nullhomotopic.
\end{cor}
\begin{proof}
    This is an immediate combination of the fact from \cref{cons:red_blue_routes} that the red and blue routes in \cref{eqn:black_magic_Atiyah-Bott} agrees up to a sign with \cref{cor:vanishing_of_pushforward_gluing_class}.
\end{proof}

\begin{rmk}
    The gluing class really is an essential feature of equivariant Poincar\'e duality. It provides some information on how the ``free part" of an equivariant Poincar\'e space is glued to the singular part.
    We will exploit it in the proof of \cref{thm:generalised_atiyah_bott} and plan on clarifying it and its role in relation to L\"uck's work \cite{lueck2022brown} in the future.
\end{rmk}

\subsection{Equivariant degree theory}\label{sec:equivariant_degree}

A nice application of equivariant Poincar\'e duality is a theory of equivariant mapping degrees, as developed in \cite{lueckDegree}. 
For simplicity, we will assume that $G$ is a finite group throughout this section.

\subsubsection*{Recollections on the Burnside ring}

Our aim is to remind the reader of the classical connection between Burnside rings and the equivariant sphere spectrum.

\begin{recollect}[Character maps on the Burnside ring]
    The \textit{Burnside ring of finite $G$-sets} $A(G)$  is the group--completion of the semiring of  isomorphism classes of finite $G$-sets, with disjoint union as addition and the cartesian product as product. For each $H \leq G$, there is a unique ring homomorphism $\chi_H \colon A(G) \rightarrow \bbZ $ sending a finite $G$-set $S$ to the order of the finite set $S^H$, and these assemble into a ring map, called the \textit{character map}, 
    \begin{equation}
        \label{eq:character_map}
        \chi \colon A(G) \rightarrow  \prod_{(H)\leq G} \bbZ
    \end{equation}
    where $(H)$ runs through all conjugacy classes of finite subgroups. 
\end{recollect}

 The following classical theorem may be found  for instance in \cite[Prop. 1.3.5.]{tomDieck_Transformation_and_Repthy}.

\begin{thm}
    The character map \cref{eq:character_map}  is an injective ring homomorphism with finite cokernel. The image can be described through explicit congruences, the \textit{Burnside congruences}.
\end{thm}

We abstain from recalling the Burnside congurences in full generality, the reader may find them in the reference mentioned above. To give some intuition, and as we use it later in the proof of \cref{lem:fixed_point_orientability_from_underlying_orientability}, we describe them in the case of the group $C_p$.

\begin{example}\label{example:C_p_burnside_congruence}
    If $G = C_p$ then the image of the character homomorphism consists of those pairs $(a,b) \in \bbZ \times \bbZ$ that satisfy the congruence
    \[ a \equiv b \mod p. \]
    Indeed, for a finite $G$-set $S$, the orders of $S$ and $S^H$ agree modulo $p$. On the other hand, if $a +kp = b$ for integers $a, b, c$, then $a$ copies of the point and $k$ copies of $C_p$ define an element in $A(G)$ mapping to $(a,b)$.
\end{example}

\begin{cons}
    We may obtain a similar character map for the ring $\pi_0^G\sphere_G$: for each subgroup $H\leq G$, using that $\Phi^H\sphere_G\simeq \sphere\in\spectra$, we may assemble the geometric fixed points functors $\Phi^H$ together with the identification $\deg\colon \pi_0\map_{\spectra}(\sphere,\sphere)\xrightarrow{\cong}\mathbb{Z}$ to  obtain a ring map
    \begin{equation}\label{eqn:spherical_character_map}
        \pi_0^G\sphere_G \cong \pi_0 \map_{\spectra_G}(\sphere_G,\sphere_G) \longrightarrow \prod_{(H)} \bbZ, \hspace{3mm} f \mapsto \deg (\Phi^H f)
    \end{equation}
\end{cons}

\begin{thm}[Segal]
    \label{thm:Burnside_ring_as_pi_0_of_sphere}
    Let $G$ be a finite group. The map \cref{eqn:spherical_character_map}  is an injective ring homomorphism whose image agrees with the image of the character map $\chi \colon A(G) \rightarrow \prod_{(H)} \bbZ$, yielding an identification $\pi_0^G\sphere_G\cong \pi_0 \map_{\spectra_G}(\sphere_G,\sphere_G) \cong A(G)$ as commutative rings.
\end{thm}

Of course, this implies that the set of path components of the selfmaps of any $E \in \Pic(\spectra_G)$ is equivalent to $A(G)$ by the equivalence $\map_{\spectra_G}(E,E)\simeq \map_{\spectra_G}(\sphere_G,\sphere_G)$.

\subsubsection*{The equivariant degree}\

We now want to specialise the abstract definition of the degree from \cref{sec:degree_theory} to the case of maps of $G$--Poincar\'e spaces which should roughly encode the mapping degrees on the various fixed points spaces. Recall that the definition of the degree of a map between $G$--spaces $f \colon \udl{X} \to \udl{Y}$ with Spivak data $(\xi_{\obj{X}},c_{\obj{X}})$ and $(\xi_{\obj{Y}},c_{\obj{Y}})$ depends on an equivalence $\xi_{\obj{X}} \xrightarrow{\simeq} f^*\xi_{\obj{Y}}$.
The existence of such an equivalence is unreasonable to expect with coefficients in $\spectra_G$ but becomes more likely after linearising.
Here we choose to work with coefficients in the Burnside Mackey functor $\udl{A}(G)$.
Recall that for each subgroup $H \subset G$, restriction defines a ring homomorphism $\nolinebreak{A(G) \rightarrow A(H)}$ and induction a transfer map $A(H) \rightarrow A(G)$. These assemble into a Mackey functor, and hence a $G$--spectrum $\udl{A}(G)$ which has values $\udl{A}(G)^H = A(H)$.

\begin{defn}
    Let $\udl{X}$ and $\udl{Y}$ be Poincar\'e $G$--spaces. An \textit{$\udl{A}(G)$--degree datum} is a pair $(f, \psi)$ where $f \colon \udl{X} \rightarrow \udl{Y}$  is a map of $G$--spaces and $\psi \colon D_{\udl{X}} \otimes \udl{A}(G) \xrightarrow{\simeq} f^* D_{\udl{Y}} \otimes \udl{A}(G)$ is an equivalence
    of the $\udl{A}(G)$--linearised dualising sheaves.
\end{defn}

In other words, a $\udl{A}(G)$--degree datum is a $\udl{\module}_{\udl{A}(G)}(\myuline{\spectra})$--degree datum in the sense of \cref{def:parametrised_degree}.

\begin{defn}
    Let $\udl{X},\udl{Y}\in\spc_G$ be $G$--Poincar\'e and $(f, \psi)$ a $\udl{A}(G)$--degree datum.
    We define the \textit{equivariant degree} $\degree_G(f, \psi) \in \pi_0 \map(\sphere_G, \uniqueMapY_* \uniqueMapY^* \udl{A}(G)) \eqqcolon \burnside{\udl{Y}}$ as the composite
    \begin{equation*}
        \sphere_G \xrightarrow{c_X} \uniqueMapX_!(D_{\udl{X}} \otimes \udl{A}(G)) \xrightarrow{\psi} \uniqueMapX_!(f^* D_{\udl{Y}} \otimes \udl{A}(G)) \xrightarrow{\beckChevalley^f_!} \uniqueMapY_!(D_{\udl{Y}} \otimes \udl{A}(G)) \xleftarrow[\simeq]{\ambi{c_Y}{}{\--}} \uniqueMapY_* \uniqueMapY^* \udl{A}(G).
    \end{equation*}
\end{defn}
As explained in \cref{cons:monoid_structure_degree_space}, the commutative algebra structure on $\uniqueMapY_* \uniqueMapY^* \udl{A}(G)$ endows $\burnside{\udl{Y}}$ with the structure of a commutative ring with unit $c_Y$.

Our goal for the rest of this discussion is to  relate the equivariant degree of a map, which lives in $\burnside{\udl{X}}$, to the various degrees induced on fixed points via a character map similar to \cref{eq:character_map} constructed from the geometric fixed points functors. To this end, first recall the Bousfield localisation $\pi_0\colon \spc_G\rightleftharpoons \sets_G \cocolon \inclusion$ from \cref{cons:orbit_component_decomposition}. Notice that $\Omega^\infty \udl{A}(G)$ is levelwise 0--truncated with fixed points $(\Omega^\infty \udl{A}(G))^H = A(H)$.
\begin{lem}
    For   $\udl{X}\in\spc_G$, we have an equivalence
    $\burnside{\udl{X}} \simeq \pi_0 \map_{\spc_G}(\tau_{\leq 0} \udl{X},\Omega^\infty \udl{A}(G))$.
\end{lem}
\begin{proof}
    Consider the computation of $\burnside{\udl{X}}$ as
    \[\pi_0 \map_{\spectra_G}(\sphere_G, \uniqueMapX_* \uniqueMapX^* \udl{A}(G)) \simeq \map_{\spectra_G}(\Sigma^\infty_+ \udl{X}, \udl{A}(G))\simeq \map_{\spc_G}(\tau_{\leq 0} \udl{X},\Omega^\infty \udl{A}(G)) \]
    where the last equivalence uses that $\Omega^\infty \udl{A}(G)$ is levelwise 0-truncated.
\end{proof}

\begin{rmk}\label{rem:Burnside_ring_of_G_space}
    This turns out to be quite simple to compute.
    Note that for two 0-truncated $G$--spaces $\udl{S}$ and $\udl{T}$, the map
    \begin{equation*}
        \map_{\spc_G}(\udl{S}, \udl{T}) \longrightarrow \prod_{(H)\leq G} \map_{\spc}(S^H, T^H)\simeq \map_{\sets}(S^H,T^H),
    \end{equation*}
    is injective with image given by all collections of maps $(f^H)_{(H)}$ compatible with the restrictions coming from  inclusions $K \le H$ or inner automorphsim $K \simeq H$.
    Specialising this to the case of interest, we obtain an injection
    \begin{equation*}
        \burnside{\udl{X}} \hookrightarrow \prod_{(H)} \left( A(H)^{\pi_0(X^H)} \right)^{W_GH}.
    \end{equation*}
    For example,  we have $\map(\tau_{\leq 0} \udl{X},\Omega^\infty \udl{A}(G)) \simeq A(G)$ if all fixed point sets of $\udl{X}$ are nonempty and connected.  For a more complicated example, consider the $C_2$-action on $S^1$ given by complex conjugation.
    Then $(\tau_{\leq_0} \udl{X})^{C_2} \simeq * \coprod *$ while  $(\tau_{\leq_0} \udl{X})^{e} \simeq *$. The set of equivalence classes of maps above identifies with
    the pullback $A(C_2) \times_{A(1)} A(C_2)$ where the two maps $A(C_2) \rightarrow A(1)$ are given by restriction along the group homomorphism $1 \to C_2$.
\end{rmk}

\subsubsection*{Recovering degrees on fixed points}

Now we want to recover different degrees on fixed point spaces from the equivariant degree by base changing along the geometric fixed points functor.
As it is not true that $\Phi^G \udl{A}(G) = \bbZ$, we need a small preparatory lemma.
For this, denote by $\spectra_G^{\geq 0}$ (resp. $\spectra_G^{\leq 0}$) the full subcategory of all $G$--spectra $X$ for which $X^H \in \spectra$ is connective (resp. coconnective) for each $H \leq G$. 
The pair ($\spectra_G^{\geq 0}$, $\spectra_G^{\leq 0}$) forms a $t$-structure on on $\spectra_G$.

\begin{lem}[Geometric fixed points of Mackey functors]\label{lem:geometric_fixed_points_of_mackey_functors}
    Let $X\in\spectra_{G}^{\geq0}$. 
    Then the canonical map $X\rightarrow\tau_{\leq0}X$ induces an isomorphism $\pi_0\Phi^GX \xrightarrow{\cong}\pi_0\Phi^G\tau_{\leq0}X$ of abelian groups. 
\end{lem}
\begin{proof}
    Recall that the geometric fixed points participates in an adjunction $\Phi^G\colon\spectra_G\rightleftharpoons \spectra \cocolon \Xi^G$  where the right adjoint $\Xi^G$ is fully faithful and is given by the formula
    \begin{equation*}
        (\Xi^G Y)^H = \begin{cases}
            Y & \text{if } H = G; \\
            0 & \text{if } H \lneq G.
        \end{cases}
    \end{equation*}
    Observe that $\Phi^G$ preserves connective objects.
    This is because $\Phi^G$ sends $\Sigma^\infty_+ G/G$ to $\sphere$ and $\Sigma^\infty_+ G/H$ to $0$ for $H\lneq G$.
    Since connective $G$--spectra are built as colimits of the orbits $\{\Sigma^\infty_+ G/H\}_{H\leq G}$ and $\Phi^G$ preserves colimits, we see that connective $G$--spectra are sent to connective spectra.
    The formula for $\Xi^G$ shows that it preserves connective and coconnective objects.
    In particular, both restrict to functors $\Xi^G \colon \spectra^{\heartsuit}\hookrightarrow \spectra_G^{\heartsuit}$ and  $\tau_{\leq0}\Phi^G \colon \spectra_G^{\heartsuit} \to \spectra^{\heartsuit}$ and we claim that those are adjoint. 
    To see this, let $\underline{M}\in\spectra_G^{\heartsuit}$ and $N\in \spectra^{\heartsuit}$, and consider\small
    \[\map_{\spectra^{\heartsuit}}(\tau_{\leq0}\Phi^G\underline{M},N)  \simeq \map_{\spectra}(\Phi^G\underline{M},N)\simeq \map_{\spectra_G}(\underline{M},\Xi^G N)\simeq \map_{\spectra_G^{\heartsuit}}(\underline{M},\Xi^G N).\]\normalsize

    To conclude, since the solid square in 
    \begin{center}
        \begin{tikzcd}
            \spectra_G^{\geq0} \dar[shift right = 1, dashed, "\tau_{\leq0}"']\rar[shift left = 1,dashed, "\Phi^G"] & \spectra^{\geq0}\dar[shift right = 1, dashed, "\tau_{\leq0}"']\lar[shift left= 1, "\Xi^G", hook]\\
            \spectra_{G}^{\heartsuit} \rar[shift left = 1, "\tau_{\leq0}\Phi^G",dashed] \uar[hook, shift right = 1]& \spectra^{\heartsuit}\lar[shift left= 1, "\Xi^G",hook]\uar[hook,shift right = 1]
        \end{tikzcd}
    \end{center}
    commutes, so does the dashed square of left adjoints, as was to be shown.
\end{proof}

\begin{rmk}
    \cref{thm:Burnside_ring_as_pi_0_of_sphere} gives an equivalence $\tau_{\le 0} \sphere_G = \udl{A}(G)$.
    By \cref{lem:geometric_fixed_points_of_mackey_functors}, we have $
        \pi_0\Phi^G\underline{A}(G) = \pi_0\Phi^G\tau_{\leq0}\sphere_G \cong \pi_0\Phi^G\sphere_G\cong \pi_0\sphere\cong \mathbb{Z}$.
    Now note that the diagram
    \begin{equation}\label{diag:character_identification}
    \begin{tikzcd}
        \pi_0 (\sphere_G)^G \ar[r, "\chi^H"] \ar[d, "\cong"] 
        & \pi_0 \Phi^G  \sphere_G \ar[r, "\simeq"] \ar[d, "\simeq"] 
        & \pi_0 \tau_{\leq 0} \Phi^G \sphere_G \simeq \bbZ \ar[d, "\simeq"]
        \\
        \pi_0 \udl{A}(G)^G \ar[r] 
        & \pi_0 \Phi^G \udl{A}(G) \ar[r] 
        & \pi_0 \tau_{\leq 0} \Phi^G \udl{A}(G) \simeq \bbZ
    \end{tikzcd}
    \end{equation}
    commutes. 
    The lower horizontal composition thus agrees with the character map.
\end{rmk}

Now we come back to the problem of relating the equivariant degree to the degree on each fixed point space.
We have the symmetric monoidal colimit preserving functor
\begin{align*}
        \phi^H \colon \udl{\module}_{\udl{A}(G)}(\myuline{\spectra}_G) 
        &\xrightarrow{\Phi^H} \coind_H^G \rcoind_H^{1} \module_{\Phi^H \udl{A}(H)}(\spectra) \\
        &\xrightarrow{} \coind_H^G \rcoind_H^{1} \module_{\bbZ}(\spectra)
\end{align*}
where $\Phi^H$ is the parametrised geometric fixed point functor constructed in \cref{cons:geometric_fixed_points} and the second map is induced by the ring map $\Phi^H \udl{A}(H) \to \tau_{\leq 0} \Phi^H \udl{A}(H) \simeq \bbZ$.

\begin{prop}\label{prop:equivariant_degree_fixed_points}
    For a $G$--Poincar\'{e} $\udl{Y}\in\spc_G$, basechange along  $\phi^H$ induces a ring map
    \begin{equation*}
        \chi^H \colon \burnside{\udl{Y}} \longrightarrow H^0(Y^H, \bbZ).
    \end{equation*}
    Given a degree datum $(f \colon \udl{X} \to \udl{Y}, \psi)$, we have
    \begin{equation*}
        \chi^H(\deg_{\udl{A}(G)}(f, \psi)) = \deg_{\bbZ}(f^H \colon X^H \to Y^H, \psi^H).
    \end{equation*}
    For $\udl{Y} = \udl{*}$, $\chi^H$ agrees with the character map from \cref{eq:character_map}.
\end{prop}
\begin{proof}
    Note that we have identifications
    \begin{align*}
        \func_G(\udl{Y}, \coind_H^G \rcoind_H^{1} \module_{\bbZ}(\spectra)) 
        &\simeq \func_H(\res_H^G \udl{Y}, \rcoind_H^{1} \module_{\bbZ}(\spectra)) \\
        &\simeq \func(Y^H, \module_{\bbZ}(\spectra)).
    \end{align*}
    Recall from \cref{obs:parametrised_colimits_of_s_*-categories} that this identifies $\uniqueMapY_!$ with $\uniqueMap{Y^H}_!$ (and also $\uniqueMapY_*$ with $\uniqueMap{Y^H}_*$ as $\udl{Y}$ is Poincar\'e). 
    Now applying \cref{lem:map_degree_spaces_symmetric_monoidal_functor} to basechange along $\phi^H$, we obtain a ring map  \small
    \[\burnside{\udl{Y}} = \pi_0 \map_{\module_{\udl{A}(G)}(\spectra_G)}(\unit, \uniqueMapY_* \uniqueMapY^* \unit) 
        \to \pi_0 \map_{\module_{\bbZ}(\spectra)}(\unit, \uniqueMap{Y^H}_* \uniqueMap{{Y^H}}^* \unit) \simeq H^0(Y^H, \bbZ).\]\normalsize
    The statment about the degrees follows from \cref{lem:map_degree_spaces_symmetric_monoidal_functor}.
    In the case $\udl{Y} = \udl{*}$, this map identifies with the character map by \cref{diag:character_identification}. 
\end{proof}

In the next corollary, we unravel \cref{prop:equivariant_degree_fixed_points} in a special case to illustrate how it can be used to deduce condruences between (nonequivariant) degrees between fixed point sets. 

\begin{cor}[Congruences between degrees on fixed point sets]
    Suppose that $\udl{Y}$ is a $G$--Poincar\'e space and assume that $Y^H$ is nonempty connected for all $H \le G$.
    Given a degree datum $(f\colon \udl{X} \rightarrow \udl{Y},\psi)$, the collection $(\deg_{\bbZ}(f^H, \psi^H))_{(H)}$ lies in the image of the character map
    \begin{equation*}
        \chi \colon A(G) \to \prod_{(H)} \bbZ.
    \end{equation*}
\end{cor}
\begin{proof}
    Any map $f \colon \udl{X} \to \udl{Y}$ of $G$--Poincar\'e spaces induces a commutative diagram
    \begin{equation*}
    \begin{tikzcd}
        \pi_0 \map_{\module_{\udl{A}(G)}(\spectra_G)}(\unit, \uniqueMapY_* \uniqueMapY^* \unit) \ar[r, "\chi^H"] \ar[d, "\beckChevalley_*"]
        & \pi_0 \map_{\module_{\bbZ}(\spectra)}(\unit, \uniqueMap{Y^H}_* \uniqueMap{{Y^H}}^* \unit) \ar[d, "\beckChevalley_*"]
        \\
        \pi_0 \map_{\module_{\udl{A}(G)}(\spectra_G)}(\unit, \uniqueMapX_* \uniqueMapX^* \unit) \ar[r, "\chi^H"]
        & \pi_0 \map_{\module_{\bbZ}(\spectra)}(\unit, \uniqueMap{X^H}_* \uniqueMap{{X^H}}^* \unit).
    \end{tikzcd}
    \end{equation*}
    Applying this to the unique map $\udl{Y} \to \udl{*}$, the vertical maps become equivalences by the assumption on the fixed points of $\udl{Y}$.
    By \cref{prop:equivariant_degree_fixed_points}, this $\chi^H \colon \burnside{\udl{Y}} \to \bbZ$ identifies with the character map $\chi^H \colon A(G) = \burnside{\myuline{G/G}} \to \bbZ$.
    By \cref{prop:equivariant_degree_fixed_points}, this $\chi^H \colon \burnside{\udl{Y}} \to \bbZ$ identifies with the character map $\chi^H \colon A(G) = \burnside{\myuline{G/G}} \to \bbZ$.
    The statement about the degrees is now a consequence of \cref{prop:equivariant_degree_fixed_points}.
\end{proof}

\section{Equivariant Poincar\'{e} duality: applications}\label{section:geometric_applications}

In this section, we employ the general theory developed in the article to investigate some problems of an equivariant geometric topological nature.  In \cref{subsection:pulling_back_fixed_points}, we study cohomological injectivity statements for degree one maps and prove \cref{thm:bredonBrowderInjection} along these lines; we then use it to obtain a rigidity result of equivariant Poincar\'{e} spaces in \cref{thm:Cp_PD_contractible_underlying_space}. Next, in \cref{subsection:theorem:single_fixed_points}, we prove the equivariant Poincar\'{e} generalisation of Atiyah--Bott and Conner--Floyd's theorem on single fixed points for group actions on smooth manifolds.

\subsection{Pulling back twisted fixed points}\label{subsection:bredonPoincareDuality}\label{subsection:pulling_back_fixed_points}

Let $f \colon M \rightarrow N$ be a map closed, connected, oriented manifolds of the same dimension $d$. If the degree of $f$ is nonzero, then $f$ is surjective. The theory of equivariant degrees immediately gives an equivariant application: if $f \colon M \rightarrow N$ is a map of closed, conneced, smooth, oriented $C_p$-manifolds with connected fixed point sets of the same dimension, then if $f$ is of degree coprime to $p$ (when considered as a nonequivariant map), then also the degree of $f^{C_p}$ is coprime to $p$. Thus, $f^{C_p}$ is surjective as well.

To detect if the degree of $f$ is coprime to $p$ it of course suffices to check that $H^*(f^{C_p};\bbF_p)$ is nonzero in the top degree. This line of thought led Browder \cite{browder_fixed_points} to interpret results about the injectivity of $H^*(f^{C_p};\bbF_p)$ as the ``ability to pull back fixed points from $N$ to fixed points of $M$". Browder's strategy is very successful to show actual surjectivity results on fixed points, even if one relaxes the conditions like smoothness, or takes $G$ to be a more general group like an abelian $p$-group. Let us mention \cite{hanke_puppe_pulling_back_fixed_points} for more information, and many interesting variations on this approach. In particular, see \cite[Thm. 4]{hanke_puppe_pulling_back_fixed_points} to see how to pass from cohomological injectivity results to surjectivity on fixed points.
We will content ourselves with showing how our methods can be used to derive cohomological injectivity results of the following type, which was first studied by Bredon \cite{bredon_poincare_duality} and generalised by Browder \cite{browder_fixed_points} under stronger manifold assumptions: 

\begin{thm}[Twisted Bredon--Browder injection]\label{thm:bredonBrowderInjection}
    Let $A$ be an elementary abelian $p$--group $C_p^{\times r}$. Let $f\colon \underline{X}\rightarrow \underline{Y}$ be a map of compact $A$--spaces. Suppose $X^e, Y^e$ are $\module_{\eilenbergMacLaneFp}$--Poincar\'{e} spaces such that $f^e\colon X^e\rightarrow Y^e$  is of $\module_{\eilenbergMacLaneFp}$--degree one (c.f. \cref{def:parametrised_degree}).  Then for any $\zeta\in\func(Y^A,\perfectCat_{\eilenbergMacLaneFp})$, the map induces an injection $H^*(Y^{A};\zeta)\rightarrow H^*(X^{A};f^*\zeta)$.
\end{thm}

Our approach is by default homotopical, and point-set techniques are avoided. As our systematic treatment of Poincar\'e duality allows us to also derive consequences for homology with twisted coefficients, orientability assumptions may even be relaxed. We will also illustrate the usefulness of Browder's cohomological injectivity results to prove a structural result for $G$-Poincar\'e spaces where $G$ is a solvable finite group as \cref{thm:Cp_PD_contractible_underlying_space}: if $X^e$ is contractible, then so is $X^H$ for any $H \leq G$.
This in turn can be applied to give an example of a compact $C_p$-space all of whose fixed points are Poincar\'e  spaces, while itself not being $C_p$-Poincar\'e.

\vspace{1mm}
The basic philosophy of our proof is similar to that of \cite{hanke_puppe_pulling_back_fixed_points}, namely, we proceed via equivariant localisations using the \textit{proper Tate construction}. To this end, we first show that the fixed points of an equivariant space which is underlying Poincar\'{e}  may  most naturally be viewed as a Poincar\'{e} space with coefficients in the \textit{stable module category}, which we now recall. 

\begin{cons}[Proper stable module categories]
    Let $\nolinebreak{R\in\calg(\spectra)}$ and $G$ be a finite group. Consider the $G$--stable category $\underline{\borel}(\perfectCat_R)\in\catGst{G}$ with value $\func(BH,\perfectCat_R)$ at $G/H$. Recall the notion and notations of Brauer quotients from \cref{subsection:equivariant_categories_and_families}. Using the family of proper subgroups $\proper$ of $G$, we may construct a new $G$--stable category $\underline{\stmodSmall}^{\proper}(R)\in \catGst{G}$ defined as $s_*\widetilde{s}^*\underline{\borel}(\perfectCat)$. This $G$--category has value $$\stmodSmallProper_G(R)\coloneqq \func(BG,\perfectCat_R)/\langle R[G/H] \: | \: H\lneq G\rangle$$ at $G/G$ and is trivial elsewhere. Furthermore, there exists a $G$--exact symmetric monoidal functor $\Phi\colon \underline{\borel}(\perfectCat_R)\longrightarrow \underline{\stmodSmall}^{\proper}(R)$.
\end{cons}

\begin{cons}[Descending Poincar\'{e} duality from large to small coefficients]\label{cons:descending_big_to_small_poincare_duality}
     Let  $\underline{X}\in\spc_{G}^{\omega}$ such that $X^e$ is an $R$--Poincar\'{e}  space. Since $\underline{X}$ was a compact $G$--space, the adjunctions $\uniqueMapX_!\dashv \uniqueMapX^* \dashv \uniqueMapX* \colon \underline{\func}(\underline{X},\underline{\borel}(\module_R))\rightleftharpoons \underline{\borel}(\module_R)$ restrict to adjunctions $\uniqueMapX_!\dashv \uniqueMapX^* \dashv \uniqueMapX_* \colon \underline{\func}(\underline{X},\underline{\borel}(\perfectCat_R))\rightleftharpoons \underline{\borel}(\perfectCat_R)$ on the full subcategories. Now by \cref{lem:borelification_fixed_points_poincare}, we know that $\underline{X}$ is $\underline{\borel}(\module_R)$--Poincar\'{e}  and we write $D_{\underline{X}}\in \underline{\func}(\underline{X},\underline{\borel}(\picardSpace_R))$ for the dualising sheaf. Since $D_{\underline{X}}\in \underline{\borel}(\perfectCat_R)$, we even obtain an equivalence
    \[\uniqueMapX_*(-)\simeq \uniqueMapX_!(D_{\underline{X}}\otimes-)\colon \underline{\func}(\underline{X},\underline{\borel}(\perfectCat_R))\longrightarrow \underline{\borel}(\perfectCat_R)\]
    and so $\underline{X}$ is also $\underline{\borel}(\perfectCat_R)$--Poincar\'{e}. We write $D^G_X \colon X^G \rightarrow\perfectCat_R^{BG}$ for the dualising sheaf evaluated at the fixed points.
\end{cons}

Via this construction, we may now prove the following as a simple consequence.

\begin{prop}[Proper stable module Poincar\'{e} duality]\label{prop:sphericalBredon}
    Let $G$ be a finite group and $R\in\calg(\spectra)$. If $\underline{X}\in\spc_{G}^{\omega}$ such that the underlying space $X^e$ is an $R$--Poincare duality space, then $X^{G}$ is a partial $\stmodSmallProper_{G}(R)$--Poincare duality space, i.e. for any $\zeta \in \func(X^G,\perfectCat_R^{BG})$, we have an equivalence in $\stmodSmallProper_G(R)$
    \[ \ambi{\Phi c}{\Phi D_X^G}{\Phi \zeta} \colon \uniqueMapX^{G}_*(\Phi \zeta)\xrightarrow{\sim} \uniqueMapX^{G}_!(\Phi D_{X}^G\otimes \Phi \zeta)\]
\end{prop}
\begin{proof}
   By \cref{cons:descending_big_to_small_poincare_duality}, we know that $\underline{X}$ is  $\underline{\borel}(\perfectCat_R)$--Poincar\'{e} duality.  The statement of the proposition is now an immediate consequence of \cref{thm:poincare_isotropy} (2).
\end{proof}

\begin{rmk}\label{rmk:spherical_bredon}
    While the proposition above looks restrictive and artificial, it already contains some interesting content since the map $\Phi\colon \func_G(\underline{X},\underline{\borel}(\perfectCat_R))\rightarrow \func(X^G,\stmodSmallProper_G(R))$ is symmetric monoidal. In particular, it holds when $\zeta$ is the tensor unit $\unit$. This will then recover the usual untwisted cohomology of $X^{G}$.
\end{rmk}

\begin{example}[Underlying Poincar\'e duality does not imply Poincar\'e duality of the fixed points]
    There exist piecewise linear $C_2$-actions on the sphere $S^d$ whose fixed point sets are submanifolds $M$ which are not homology spheres, see e.g. \cite[p.5]{FarrellLafont} for an exposition. Let $\underline{X}$ be the (unreduced) suspension of such an action. Then $X^e \simeq S^{d+1}$ which is a Poincaré duality space. However, $X^{C_2}$ is the unreduced suspension of a manifold which is not a homology sphere, and hence clearly not Poincaré as Poincaré duality with integer coefficients must fail. 
\end{example}

Next, we recall the {proper Tate construction}. The significance of this to our proof is that combining \cref{prop:sphericalBredon} with the degree theory from \cref{sec:degree_theory}, we may obtain a version of the cohomological injection in proper Tate cohomology. We then extract the desired injection from this version by ``finding'' $\mathbb{F}_p$--cohomology inside proper Tate cohomology.
\begin{recollect}[Proper Tate]\label{recollect:proper_tate}
    Let $G$ be a finite group and $R\in\calg(\spectra)$. One way to define the \textit{$R$--based proper Tate functor} is as the lax symmetric monoidal composite
    \[(-)^{t_{\proper}G}\colon \func(BG,\module_R)\xhookrightarrow{b_*} \mackey_G(\module_R) \xlongrightarrow{\Phi^G}\module_R\]
    This functor kills the proper induced terms, i.e.  those $M\in\module_R^{BG}$ such that $M\simeq \ind^G_HN$ for some $H\lneq G$ and $N\in\module_R^{BH}$ since $\Phi^G$ does. Furthermore, since $(-)\GeneralProperTate{G}$ is a lax symmetric monoidal functor,  $R^{t_{\proper}G}$  canonically attains an $R$--algebra structure. 
    
    Now let $A$ be an elementary abelian $p$--group $A= C_p^{\times r}$. With the trivial action of $A$ on $\eilenbergMacLaneFp$, the $A$--proper Tate $\eilenbergMacLaneFp\properTate$ is a nontrivial $\eilenbergMacLaneFp$--algebra by \cite[Prop. 5.16]{MNNDerivedInduction}. Now let $T\colon\stmodSmallProper_A(\eilenbergMacLaneFp)\longrightarrow \module_{\eilenbergMacLaneFp}$ be the universal functor making the triangle
    \begin{center}
        \begin{tikzcd}
            \func(BA,\perfectCat_{\eilenbergMacLaneFp}) \rar["\canonical"] \ar[dr, "(-)\properTate"']& \stmodSmallProper_{A}(\eilenbergMacLaneFp)\dar["T"]\\
            & \module_{\eilenbergMacLaneFp}
        \end{tikzcd}
    \end{center}
    commute, coming from the universal property of $\stmodSmallProper_A(\eilenbergMacLaneFp)$. 
\end{recollect}

\begin{lem}[Projection formula at dualisables]\label{lem:lax_right_adjoint_preserve_dualisables}
    Let $\A,\sC,\D$ be stably symmetric monoidal categories and $u\colon \A\rightarrow \sC$ and $L\colon\sC\rightarrow \D$ be symmetric monoidal exact functors. Suppose $L$ admits a right adjoint $R$. Then for every $a\in \A$ dualisable and $d\in \D$, the canonical map $ua\otimes Rd \rightarrow R(Lua\otimes d)$ is an equivalence.
\end{lem}
\begin{proof}
    Let $c\in \sC$. By considering the equivalences
    \[\map_{\sC}(c, ua\otimes Rd)\simeq \map_{\sC}(c\otimes ua^{\vee}, Rd)\simeq \map_{\D}(Lc\otimes Lua^{\vee},d)\simeq \map_{\D}(c, R(Lua\otimes d)),\]
    we obtain the desired conclusion by an application of Yoneda's lemma.
\end{proof}

\begin{lem}\label{lem:commuting_proper_Tate_with_compact_limits}
    Let $X\in\spc^{\omega}$, $R\in\calg(\spectra)$, and $\zeta\in\func(X,\perfectCat_R)$. Then viewing $\zeta$ as having the trivial $G$--action, we have an equivalence $(\uniqueMapX_*\zeta)\GeneralProperTate{G}\simeq R\GeneralProperTate{G}\otimes_R\uniqueMapX_*\zeta$.
\end{lem}
\begin{proof}
    Since $X$ was compact, we know that $X_*\zeta\in\perfectCat_R$, i.e. $X_*\zeta$ is a dualisable $R$--module. Setting $\inflated\colon \module_R\rightarrow \mackey_G(\module_R)$ and $b^*\colon\mackey_G(\module_R)\rightarrow \func(BG,\module_R)$ for the functors $u$ and $L$ in \cref{lem:lax_right_adjoint_preserve_dualisables} (and writing $\trivial_G\colon \module_R\rightarrow\func(BG,\module_R)$ for the composite), we see that by the lemma that
    \[(\trivial_GX_*R)\GeneralProperTate{G} = \Phi^Gb_*\big((\trivial_GX_*R)\otimes_R R\big)\simeq \Phi^G\big((\inflated X)_*R\otimes_R b_*R\big)\simeq(X_*R)\otimes_R R\GeneralProperTate{G}\] as was to be shown.
\end{proof}

We now come to the main general proposition.

\begin{prop}[Injection after basechanging to proper Tate]\label{prop:proper_tate_basechanged_injection}
    Consider a finite group $G$, $R\in\calg(\spectra)$, and  $f\colon \underline{X}\rightarrow \underline{Y}$  a map of compact $G$--spaces. Suppose $X^e, Y^e$ are $\module_{R}$--Poincar\'{e} spaces and $f\colon \underline{X}\rightarrow \underline{Y}$ is equipped with a $\udl{\borel}(\module_R)$--degree one datum (c.f. \cref{def:parametrised_degree}). Then for any $\zeta\in \func(Y^G,\perfectCat_R)$, the cohomological functoriality map  in $\module_R$
    \[R\GeneralProperTate{G}\otimes_R Y^G_*\zeta \longrightarrow R\GeneralProperTate{G}\otimes_R X^G_*f^*\zeta\] is a $\pi_*$--split injection.
\end{prop}
\begin{proof}
    By \cref{prop:sphericalBredon} we know that $X^{G}$ and $Y^{G}$ are $\stmodSmallProper_G(R)$--partial Poincar\'{e} duality.  In particular, viewing $\zeta$ as an object in $\func(Y^A,\perfectCat_R^{BG})$ under the symmetric monoidal functor $\trivial_A\colon \perfectCat_R\rightarrow\perfectCat_R^{BG}$, we obtain using \cref{lem:umkehr_map_degree_one} the left commuting square 
    \begin{center}
        \begin{tikzcd}
            \uniqueMapY^A_*(\zeta) \rar["\beckChevalley_*^f"]\dar["\mathrm{PD}", "\simeq"'] & \uniqueMapX^A_*(f^*\zeta)\dar["\mathrm{PD}", "\simeq"'] \\
            \uniqueMapY^A_!(\Phi D_{Y^e}\otimes\Phi\zeta) & \uniqueMapX^A_!(\Phi f^*D_{Y^e}\otimes\Phi f^*\zeta) \lar["\beckChevalley_!^f"'] 
        \end{tikzcd}
    \end{center}
    in $\stmodSmallProper_G(R)$. Hence, the map $\beckChevalley^f_*\colon \uniqueMapY^G_*\zeta\rightarrow \uniqueMapX^G_*f^*\zeta$ is a split inclusion. Finally, applying $T$ to this map and using that $T\circ\Phi\simeq (-)\GeneralProperTate{G}$, we conclude from \cref{lem:commuting_proper_Tate_with_compact_limits} that the map stated in the proposition is a split inclusion in $\module_R$ and in particular is a $\pi_*$--split injection.
\end{proof}

We would like to apply \cref{prop:proper_tate_basechanged_injection} to prove \cref{thm:bredonBrowderInjection}, and for this, a small preliminary calculation will be needed.

\begin{lem}\label{lem:degree_one_calculation_p-groups}
    Let $G$ be a $p$--group, $f\colon\udl{X}\rightarrow \udl{Y}$ a morphism in $\spc_G$. Suppose that $X^e$ and $Y^e$ are $\eilenbergMacLaneFp$--Poincar\'{e} and that $f\colon X^e\rightarrow Y^e$ is equipped with an  $\eilenbergMacLaneFp$--degree one datum. Then this degree one datum lifts to yield a $\udl{\borel}(\module_{\eilenbergMacLaneCoeff})$--degree one datum for the map $f\colon\udl{X}\rightarrow\udl{Y}$.
\end{lem}
\begin{proof}
    First recall from \cref{lem:borelification_fixed_points_poincare} that $\udl{X}$ and $\udl{Y}$ are indeed $\udl{\borel}(\module_{\eilenbergMacLaneFp})$--Poincar\'{e}. So by \cref{lem:degree_one_and_borelifications}, we just need to find $G$--equivariant lifts of the equivalences $D_{X^e}\xrightarrow[\simeq]{\alpha}f^*D_{Y^e}\in\func(X^e,\picardSpace(\eilenbergMacLaneFp))\simeq \map(X^e,\picardSpace(\eilenbergMacLaneFp))$ and $c_Y\simeq \beckChevalley_!^f\circ \alpha\circ c_X\in \uniqueMapY^e_!D_{Y^e}$. That is, we would like to lift these equivalences to ones in $\map(X^e,\picardSpace(\eilenbergMacLaneFp))^{hG}$ and $(Y^e_!D_{Y^e})^{hG}$ respectively. For the first problem, note that $\picardSpace(\eilenbergMacLaneFp)\simeq \bbZ\times B\mathrm{Aut}(\bbF_p)\simeq \bbZ\times B\bbZ/(p-1)$. Thus, by a standard analysis of  the $(-)^{hG}$--spectral sequence \[H^s(G;\pi_t\map(X,\picardSpace(\eilenbergMacLaneFp)))\Rightarrow \pi_{t-s}\map(X,\picardSpace(\eilenbergMacLaneFp))^{hG},\] applying $\pi_0$ yields
    \[\pi_0\map(X,\picardSpace(\eilenbergMacLaneFp))^{hG}\cong (\pi_0\map(X,\picardSpace(\eilenbergMacLaneFp)))^G\longrightarrow\pi_0\map(X,\picardSpace(\eilenbergMacLaneFp))\]
    which in particular is an injection. Thus, since the $G$--equivariant lifts $D_{\udl{X}}$, $f^*D_{\udl{Y}}$ in the source get mapped to $D_{X^e}=f^*D_{Y^e}\in\pi_0\map(X,\picardSpace(\eilenbergMacLaneFp))$, we get that $D_{\udl{X}}=f^*D_{\udl{Y}}$ in the set $ \pi_0\map(X,\picardSpace(\eilenbergMacLaneFp))^{hG}$. That is, the equivalence $\alpha$ lifts to a $G$--equivariant one, as required.

    Next, note by Poincar\'{e} duality that $Y^e_!D_{Y^e}\simeq Y^e_*\unit_{\eilenbergMacLaneFp}$, and so since $\uniqueMapY^e_*$ preserves coconnectivity, we learn that $Y^e_!D_{Y^e}$ is coconnective. Again, by looking at the spectral sequence $H^s(G;\pi_tY^e_!D_{Y^e})\Rightarrow \pi_{t-s}(Y^e_!D_{Y^e})^{hG}$, since no higher cohomologies may contribute to $\pi_0(Y^e_!D_{Y^e})^{hG}$ by coconnectivity, on $\pi_0$ the map $(Y^e_!D_{Y^e})^{hG}\rightarrow Y^e_!D_{Y^e}$ induces the map $(\pi_0Y^e_!D_{Y^e})^G\rightarrow \pi_0Y^e_!D_{Y^e}$, which is an injection. Thus by a similar argument as above, we obtain a $G$--equivariant lift of the equivalence $c_Y\simeq \beckChevalley_!^f\circ \alpha\circ c_X$, as wanted.
\end{proof}

We are now ready to assemble the pieces to prove the theorem.

\begin{proof}[Proof of \cref{thm:bredonBrowderInjection}.]
    Since $\eilenbergMacLaneFp $ was a field, we have the K\"{u}nneth isomorphisms 
    \[\pi_{-*}\big(Y^A_*\zeta\otimes_{\eilenbergMacLaneFp} \eilenbergMacLaneFp\GeneralProperTate{A}\big) \cong H^*(Y^A;\zeta)\otimes_{\mathbb{F}_p}\pi_{-*}(\eilenbergMacLaneFp\GeneralProperTate{A})\]
    \[\pi_{-*}\big(X^A_*f^*\zeta\otimes_{\eilenbergMacLaneFp} \eilenbergMacLaneFp\GeneralProperTate{A}\big) \cong H^*(X^A;f^*\zeta)\otimes_{\mathbb{F}_p}\pi_{-*}(\eilenbergMacLaneFp\GeneralProperTate{A}).\]
    Now consider the commuting square
    \begin{center}
        \begin{tikzcd}
            H^*(Y^{A};\zeta)\otimes_{\mathbb{F}_p}\pi_{-*}(\eilenbergMacLaneFp\properTate) \rar["f^*", tail]& H^*(X^{A};f^*\zeta)\otimes_{\mathbb{F}_p}\pi_{-*}(\eilenbergMacLaneFp\properTate)\\
            H^*(Y^{A};\zeta) \rar["f^*"]\uar[tail]&  H^*(X^{A};f^*\zeta)\uar[tail]
        \end{tikzcd}
    \end{center}
    Here, the vertical arrows are induced by the injection $\mathbb{F}_p= \pi_{-*}(\eilenbergMacLaneFp)\rightarrow \pi_{-*}(\eilenbergMacLaneFp\properTate)$ and so are themselves injections: this is since we are tensoring over a field and so all modules are flat. The top horizontal map is an injection by \cref{prop:proper_tate_basechanged_injection} and the fact that,  by \cref{lem:degree_one_calculation_p-groups}, we have a lift of the given nonequivariant degree one datum to a $\udl{\borel}(\module_{\eilenbergMacLaneFp})$--degree one datum for the map $f\colon \udl{X}\rightarrow \udl{Y}$. Therefore all in all, we see that the bottom map $f^*$ is injective as desired.
\end{proof}

We end this subsection with an application of \cref{thm:bredonBrowderInjection}  where we show \cref{thm:Cp_PD_contractible_underlying_space} that, when $G$ is a $p$--group for an odd prime $p$, equivariant Poincar\'{e} spaces with contractible underlying spaces must already by $G$--contractible. Apart from perhaps being interesting in its own right, this result will also be a crucial ingredient in the inductive proof of the main theorem in the next subsection. We will need several preliminaries on orientations. 

\begin{lem}\label{lem:fixed_point_orientability_from_underlying_orientability}
    Let $\udl{X}$ be a $C_p$--Poincaré duality space, where $p$ is an odd prime. Assume $X^e$ is $\mathbb{Z}$--orientable. Then also $X^{C_p}$ is $\mathbb{Z}$--orientable.
\end{lem}
\begin{proof}
    We check that a class $w_1(X^{C_p}) \in H^1(X^{C_p}; \bbZ/2) \cong \hom(\pi_1(X^{C_p}), \bbZ/2)$, the first Stiefel-Whitney class of $X^{C_p}$, vanishes. Let $\gamma  \colon S^1 \rightarrow X^{C_p}$ be a loop. The value of $w_1$ at the loop $\gamma$ can be computed as the degree of  $\mathrm{Mdrmy}_{\gamma}^{\udl{X}} \colon D_{X^{C_p}}(\gamma(1)) \rightarrow D_{X^{C_p}}(\gamma(1)) \in \Pic(\spectra)$,  the induced monodromy automorphism map. 

    We also have the automorphism $\mathrm{Mdrmy}_\gamma^{\udl{X}} \colon D_{\udl{X}}(\gamma(1)) \rightarrow D_{\udl{X}}(\gamma(1)) \in \Pic(\spectra_G)$.  Using \cref{thm:fixed_points_poincare_duality}, we see that $\Phi^{C_p}\mathrm{Mdrmy}_{\gamma}^{\udl{X}} \simeq \mathrm{Mdrmy}_\gamma^{X^{C_p}}$. Now we know that
    \[ \deg \Phi^{C_p} \mathrm{Mdrmy}_{\gamma}^{\udl{X}} \equiv \deg \Phi^{e} \mathrm{Mdrmy}_{\gamma}^{\udl{X}} \mod p \]
    by \cref{thm:Burnside_ring_as_pi_0_of_sphere} and \cref{example:C_p_burnside_congruence}. But we also know that both degrees are $\pm 1$, as $\mathrm{Mdrmy}_{\gamma}^{\udl{X}}$ is invertible. Thus, in fact we even have $\deg \Phi^{C_p} \mathrm{Mdrmy}_{\gamma}^{\udl{X}} = \deg \Phi^{e} \mathrm{Mdrmy}_{\gamma}^{\udl{X}}$ since $p$ was odd. But $\deg \Phi^{e} \mathrm{Mdrmy}_{\gamma}^{\udl{X}}$ is the value of the first Stiefel Whithney class of the Poincaré space $X^e$ at the loop $S^1 \xrightarrow{\gamma} X^{C_p} \rightarrow X^e$, which  is $1$ as $X^e$ was assumed to be $\mathbb{Z}$--orientable.
\end{proof}

\begin{prop}[Rigidity of orientability]\label{prop:rigidity_or_orientability}
    Let $p$ be an odd prime and $G$ be a $p$--group. Let $\underline{X}$ be a $G$--Poincar\'{e}  space. 
    Suppose that $X^e \in \spc$ is nonequivariantly $\mathbb{Z}$--orientable. 
    Then for each subgroup $H\leq G$, the Poincar\'{e} space $X^H$ is also nonequivariantly $\mathbb{Z}$--orientable.
\end{prop}
\begin{proof}
    We shall prove this by induction on the order of the subgroup, where the base case $|H|=1$ is given by the hypothesis that $X^e$ is nonequivariantly orientable. 
    Suppose we know the statement for all subgroups of order $p^{k-1}$ and consider $H\leq G$ with $\lvert H \rvert = p^k$. 
    Since $H$ is a $p$--group, we may find a normal subgroup $N\lneq H$ such that $H/N\cong C_p$.
    By \cref{prop:restriction_stability_of_Poincare_duality} and \cref{thm:fixed_points_poincare_duality}, we know that $(\res^G_H\underline{X})^{N}$ is a $H/N\cong C_p$--Poincar\'{e} duality space and by induction, we know that $(X^{N})^e$ is orientable. Thus, by \cref{lem:fixed_point_orientability_from_underlying_orientability}, we see that $X^H\simeq (X^N)^{H/N}\simeq (X^N)^{C_p}$ is also orientable as required.  
\end{proof}

\begin{thm}[Poincar\'{e} rigidity of contractible underlying spaces]\label{thm:Cp_PD_contractible_underlying_space}
    Let $G$ be a solvable group and $\udl{X} \in \spc_{G}^\omega$ a  compact $G$--Poincar\'e space with $X^e \simeq *$.    Then $\udl{X} \simeq \udl{\ast}$.
\end{thm}
\begin{proof}
    We prove this reducing to the case of $G=C_p$ using the solvability assumption. To wit, let us suppose we know the statement to be true for all solvable groups with size smaller than $|G|$. Choose a normal subgroup $N$ of $G$ such that $G/N=C_p$. By \cref{prop:restriction_stability_of_Poincare_duality}, we know that $\res^G_N\udl{X}$ is $N$--Poincar\'{e} with $(\res^G_N\udl{X})^e\simeq X^e\simeq\ast$, and so by induction, $\res^G_N\udl{X}\simeq \terminalTCat$. In particular, $X^N\simeq \ast$. Therefore, by \cref{thm:fixed_points_poincare_duality}, we have that $\udl{X}^N$ is a $G/N=C_p$--Poincar\'{e} space with $(\udl{X}^N)^e\simeq X^N\simeq\ast$. Thus, we are left to prove that for a $C_p$--Poincar\'e space $\udl{X}$, $X^e \simeq *$ implies $X^{C^p} \simeq *$.
    
    Observe that $X^{C_p} \neq \varnothing$ as $\udl{EC}_p$ is not compact.  Now pick a map $f \colon \udl{*} \to \udl{X}$.
    It is an equivalence on undelying spaces with $C_p$-action.
    By \cref{thm:bredonBrowderInjection}, $f$ induces an injection
    \begin{equation}\label{eq:Bredon_Browder_single_fixed_point_p-groups}
        f^* \colon H^*(X^{C_p}; \bbF_p) \rightarrowtail H^*(*; \bbF_p).
    \end{equation}
    In degree 0, this shows that $X^{C_p}$ is connected. 
    Furthermore, again by \cref{thm:fixed_points_poincare_duality}, $X^{C_p}$ is a Poincar\'{e} space. To conclude, by the classification of zero-dimensional Poincar\'e spaces (\cref{fact:PD_of_low_dimensions}) it suffices to show that the formal dimension of $X^{C_p}$ is zero.
    Note that $X^{C_p}$ is $\bbF_p$-orientable.
    In the case $p = 2$ this is clear while in the case $p \neq 2$ this follows from \cref{prop:rigidity_or_orientability}.
    Now, injectivity of \cref{eq:Bredon_Browder_single_fixed_point_p-groups} implies that the formal dimension of $X^{C_p}$ is zero, as zero is the highest degree in which $H^*(X^{C_p};\bbF_p)$ does not vanish.
\end{proof}

\begin{rmk}
    By Feit--Thompson's celebrated result, all finite groups of odd order are solvable. Hence, the Poincar\'{e} rigidity result above holds unconditionally for all odd finite groups.
\end{rmk}

\begin{cor}\label{cor:pointwise_PD_which_is_not_equivariant_PD}
    Let $p$ be an odd prime.
    There exists a compact $C_p$-space $\udl{X}$ with
    \begin{enumerate}
        \item the underlying space $X^e$ is contractible and
        \item the fixed point space $X^{C_p}$ is Poincar\'e and
        \item the $C_p$-space $\udl{X}$ is not $C_p$--Poincar\'e.
    \end{enumerate}
\end{cor}

\begin{proof}
    Pick a noncontractible $\bbF_p$-acyclic Poincar\'e space $K$ that is homotopy equivalent to a finite CW complex, for example $\mathbb{R}\mathrm{P}^d$ for $d>0$ an even number.
    By \cite[Thm. 1.1]{Jones}, we may pick a finite $C_p$--CW complex $\udl{X}$ with $X^e \simeq *$ and $X^{C_p} \simeq K$. By \cref{thm:Cp_PD_contractible_underlying_space}, we see that $\udl{X}$ can not be $C_p$--Poincar\'e, as then we would have $K \simeq *$.
\end{proof}

\subsection{The theorem of single fixed points}\label{subsection:theorem:single_fixed_points}
Throughout this subsection, we will fix an odd prime $p$. 

\vspace{1mm}

In \cite{ConnerFloyd64}, Conner-Floyd conjectured that a  smooth action by a cyclic group of odd prime power on a smooth, closed, orientable, positive--dimensional manifold cannot have exactly one fixed point. The first proof of this statement (in fact, a slightly more general version) was given by Atiyah-Bott in \cite{AtiyahBottLefschetzII} and soon after by \cite{ConnerFloydMapsOfOddPeriod} themselves. Many variations have been proven since then, and we mention \cite{lueckDegree, AssadiBarlowKnop92} as further examples. Atiyah-Bott's argument uses Atiyah--Singer's index theory, whereas Conner--Floyd's proof used a particular  bordism spectrum. In either case, and also in \cite{lueckDegree}, local structures of smooth manifolds were used in essential ways. We exemplify such local arguments with the following corollary of \cref{thm:bredonBrowderInjection} which answers the  Conner--Floyd question for elementary abelian $p$--groups. As will be clear from the proof, the result holds  more generally for locally smooth manifolds.

\begin{cor}[Conner--Floyd for elementary abelian groups]\label{cor:conner-floyd_for_elementary_abelians}
    Let $A$ be an elementary abelian $p$--group, and $M$ a closed, orientable, positive--dimensional, smooth $A$--manifold. Then $M^A\neq \ast$.  
\end{cor}
\begin{proof}
    Suppose $M^A=\ast$. Writing $x\in M$ for this single fixed point, we may thus find an $A$--representation $V$ equipped with a codimension zero equivariant embedding $V \subseteq M$ which sends $0 \in V$ to $x \in M$. Consider the collapse map
    $c \colon M \longrightarrow M/(M \setminus Y) \simeq S^V$. It is a map of $A$--Poincar\'e spaces with $D_{M^{e}} \otimes \eilenbergMacLaneCoeff \simeq c^* D_{(S^V)^{e}} \otimes \eilenbergMacLaneCoeff$ as both are orientable. Thus by \cref{thm:bredonBrowderInjection}, we have an injection $H^*((S^{V})^{A};\bbF_p) \rightarrow H^*(M^{A};\bbF_p)$.
    But note that $\nolinebreak{H^*((S^{V})^{A};\bbF_p) \simeq H^*(* \amalg *;\bbF_p)}$ while $H^*(M^{A};\bbF_p) \simeq H^*(*;\bbF_p)$. This is a contradiction.
\end{proof}

In this subsection, we will employ the theory of fundamental classes developed in this article to give a fully homotopical and global proof of the following generalisation of Atiyah--Bott and Conner--Floyd's theorem for $C_{p^k}$--Poincar\'{e}  spaces.
Philosophically, this says that the equivariant fundamental class packages enough structures so as to be able to provide a global obstruction to some naturally interesting geometric questions.

\begin{thm}[Generalised Atiyah--Bott--Conner--Floyd]\label{thm:generalised_atiyah_bott}
    Let $p$ be an odd prime, $G = C_{p^k}$ for some $k$, and suppose $\underline{X}\in\spc_G^{\omega}$ is $G$--Poincare  such that the underlying space $X^e\in\spc^{\omega}$ is connected, $\mathbb{Z}$--orientable, and has formal dimension $d>0$.   Then $X^G\not\simeq \ast$.
\end{thm}

We obtain the theorem of Atiyah--Bott and Conner--Floyd as an immediate consequence.

\begin{cor}[{\cite[Thm. 7.1]{AtiyahBottLefschetzII}, \cite[8.3]{ConnerFloydMapsOfOddPeriod}}]
    Let $p$ be an odd prime and $G=C_{p^k}$. Let $M$ be a closed connected orientable smooth manifold of positive dimension equipped with a smooth $G$--action. Then $M^G\neq \ast$. 
\end{cor}

 The orientability assumption is crucial, as illustrated by the following: 
\begin{example}
\label{ex:counterexample_to_conner_floys_without_orientability}
   For $p$ odd, consider the suspension of the action of $C_p \subset S^1$ on $S^2$ which descends to an action of $C_p$ on $\mathbb{R}\mathrm{P}^2$ with a single fixed point. 
\end{example}

Consequently, we see that these no--go results for single fixed points cannot purely be a product of classical Smith theory since they must incorporate orientations in some fundamental way. From this perspective, our approach may be seen as a way to encode orientations by enconcing the discussion within the formalism of equivariant Poincar\'{e} duality, where Smith--theoretic fixed points methods are also available as afforded by \cref{subsection:fixed_points_methods}. 

Restricting to odd prime powers is essential as well, as the following example illustrates.

\begin{example}[\cite{ConnerFloyd64}, Chapter 45]
    The group $C_4$ acts on $\mathbb{C} \mathrm{P}^2$ with a single fixed point, by letting a generator act via $[z_0:z_1:z_2] \mapsto [\overline{z_0}:- \overline{z_2}:\overline{z_1}]$.
\end{example}  

To start work on \cref{thm:generalised_atiyah_bott}, we record some preliminaries on Tate cohomology which will be the computational input to our proof.

\begin{recollect}[Group (co)homologies]\label{recollection:group_cohomologies}
    Let $n\geq 2$ be an integer and $A$ an abelian group equipped with the trivial $C_n$--action. Then by definition, we have
    \[\pi_{d}HA^{hC_n}\cong H^{-d}(C_n;A), \quad\quad \pi_{d}HA^{tC_n}\cong \widehat{H}^{-d}(C_n;A),\quad\quad \pi_dHA_{hC_n}\cong H_d(C_n;A).\]
    Moreover, using the fibre sequence of spectra
    $HA_{hC_n} \longrightarrow HA^{hC_n} \longrightarrow HA^{tC_n}$, we get a long exact sequence 
    \[\cdots\rightarrow H^{-d}(C_n;A)\longrightarrow \widehat{H}^{-d}(C_n;A)\longrightarrow H_{d-1}(C_n;A) \longrightarrow H^{-(d-1)}(C_n;A) \rightarrow \cdots\]
    giving us
    \begin{equation*}
        \widehat{H}^{-d}(C_n;A) = \begin{cases}
            H^{-d}(C_n;A)\cong A/n & \text{ if } d\leq-1 \text{ and }d \text{ even};\\
            H^{-d}(C_n;A)\cong 0 & \text{ if } d\leq-1 \text{ and }d \text{ odd};\\
            A/n & \text{ if } d=0;\\
            H_{d-1}(C_n;A)\cong A/n & \text{ if } d\geq 1 \text{ and } d \text{ even};\\
            H_{d-1}(C_n;A)\cong 0 & \text{ if } d \geq 1 \text{ and } d \text{ odd}.
        \end{cases}
    \end{equation*}
\end{recollect}

It will be convenient to recall the notations of \cite{greenlees-may_tate} to manipulate the various forms of the Tate constructions.

\begin{lem}\label{lem:tate_of_induced}
    Let $H\leq G$ be a subgroup of a finite group $G$ and $A\in\spectra_H$. Then we have $(\ind^G_HA)^{tG}\simeq A^{tH}$.
\end{lem}
\begin{proof}
    First observe that $\res^G_H\widetilde{EG}\simeq \widetilde{EH}$ and $\res^G_HEG_+\simeq EH_+$. The required result is now obtained from the computation of $(\ind^G_HA)^{tG}$ as \small
    \[ \big(\widetilde{EG}\otimes F(EG_+,\ind^G_HA)\big)^G\simeq \big(\ind^G_H\res^G_H\widetilde{EG}\otimes F(EG_+,A)\big)^G\simeq \big(\widetilde{EH}\otimes F(EH_+,A)
            \big)^H = A^{tH}\]\normalsize
    where the equivalence $(\ind^G_H-)^G\simeq (-)^H$ is since $\ind^G_H\simeq \coind^G_H$ and we have an equivalence of their left adjoints $\infl_H^1\simeq \res^G_H\infl_G^1$. 
\end{proof}

\begin{lem}\label{lem:tate_torsion_computation}
    Let $\udl{Y}\in\spc_{C_{p^k}}^\omega$ such that $Y^{C_{p^k}}\simeq \varnothing$. Then the change of coefficients map $(Y_+\otimes\eilenbergMacLaneCoeff)^{tC_{p^k}}\rightarrow (Y_+\otimes\mathrm{H}\mathbb{Z}/p^{k-1})^{tC_{p^k}}$ is an equivalence. In particular, the groups $\pi_n(Y_+\otimes\eilenbergMacLaneCoeff)^{tC_{p^k}}$ are $p^{k-1}$--torsion for all $n$.
\end{lem}
\begin{proof}
    Note that the map being an equivalence is stable under retracts and finite colimits in the  $Y$--variable.
    As any compact $C_{p^k}$ space $\udl{Y}$ with $Y^{C_{p^k}} = \varnothing$ is a retract of a finite sequence of pushouts of orbits $\myuline{C_{p^k}/C_{p^l}}$ with $l < k$, it thus suffices to show the desired equivalence for each of these orbits. By \cref{lem:tate_of_induced}, we obtain for any $R \in \calg(\spectra)$ the natural equivalence
    \begin{equation*}
        \left((C_{p^k}/C_{p^l})_+ \otimes R \right)^{tC_{p^k}} =\left( \ind_{C_{p^l}}^{C_{p^k}} R \right)^{tC_{p^k}} \simeq R^{tC_{p^l}}.
    \end{equation*}
    Thus, the claim follows from the fact that the quotient map $\eilenbergMacLaneCoeff^{t C_{p^l}} \to \left(\eilenbergMacLaneCoeff/{p^{k-1}} \right)^{tC_{p^l}}$ is an equivalence for $l<k$, see e.g. \cref{recollection:group_cohomologies}.
\end{proof}

For the next lemma, recall the cofibre sequence
$EG_+ \rightarrow \sphere_G \rightarrow 
\widetilde{EG}$ in $\spectra_G$.

\begin{lem}\label{lem:technical_pi_0_isomorphism}
    Let $G$ be an odd finite group and $\picardObject \in\picardSpace(\spectra_G)$ with $\picardObject^e\simeq \Sigma^k\sphere$. Then there is a canonical equivalence $F({EG}_+,\eilenbergMacLaneCoeff)\otimes \picardObject \simeq \Sigma^kF({EG}_+,\eilenbergMacLaneCoeff) \in \spectra_G$. Consequently, the map 
    \[\big({\widetilde{EG}}\otimes F({EG}_+,\eilenbergMacLaneCoeff)\otimes \picardObject\big)^G \longrightarrow \big(\Sigma {EG}_+\otimes F({EG}_+,\eilenbergMacLaneCoeff)\otimes\picardObject\big)^G\]
    from the cofibre sequence ${EG}_+\rightarrow\sphere_G\rightarrow{\widetilde{EG}}$ may be identified with the usual connecting map $\Sigma^k\eilenbergMacLaneCoeff^{tG}\longrightarrow \Sigma^{1+k}\eilenbergMacLaneCoeff_{hG}$.
\end{lem}
\begin{proof}
    We first show that the Borelification map $F({EG}_+,\eilenbergMacLaneCoeff)\otimes\picardObject\rightarrow F({EG}_+,\eilenbergMacLaneCoeff\otimes\picardObject^e)$ is an equivalence. To wit, let ${Y}\in\spectra_G$. Then
    \begin{equation*}
        \begin{split}
            \map_{\spectra_G}({Y},F({EG}_+,\eilenbergMacLaneCoeff)\otimes\picardObject) & \simeq \map_{\spectra_G}\big({Y}\otimes\picardObject^{-1},F({EG}_+,\eilenbergMacLaneCoeff)\big)\\
            &\simeq \map_{\spectra^{BG}}(Y^e\otimes(\picardObject^e)^{-1},\eilenbergMacLaneCoeff)\\
            &\simeq \map_{\spectra^{BG}}(Y^e,\eilenbergMacLaneCoeff\otimes\picardObject^e)\\
            &\simeq \map_{\spectra_G}\big({Y},F({EG}_+,\eilenbergMacLaneCoeff\otimes\picardObject^e)\big)\\
        \end{split}
    \end{equation*}
    as claimed. But then, since $\eilenbergMacLaneCoeff\otimes\picardObject^e\in\func(BG,\picardSpace(\module_{\spectra}(\eilenbergMacLaneCoeff)))$ and $G$ was an odd group and $B\mathrm{Aut}(\eilenbergMacLaneCoeff)\simeq BC_2$, we know that $\eilenbergMacLaneCoeff\otimes\picardObject^e\simeq \trivial_G\Sigma^k\eilenbergMacLaneCoeff$, whence the first statement. The second statement is then immediate from the equivalences $\big({\widetilde{EG}}\otimes F({EG}_+,E)\big)^G \simeq E^{tG}$ and $\big({EG}_+\otimes F({EG}_+,E)\big)^G\simeq E_{hG}$  for all $E\in\spectra_G$.
\end{proof}

For the proof of the theorem, it will also be helpful to record the following:

\begin{cons}[Orbit--component decompositions]\label{cons:orbit_component_decomposition}
    Let $\underline{X}\in \spc_G$. By an easy adjunction computation, we have that $(\pi_0X^e)/G \cong \pi_0(X_{hG})$.  Let $S\sqcup T$ be a decomposition of $(\pi_0X^e)/G \cong \pi_0(X_{hG})$. By considering the triple of adjunctions
    \begin{equation}\label{eqn:triple_of_adjunctions}
        \begin{tikzcd}
            \spc_G \rar[shift left = 1, "\pi_0"] & \sets_G= \func(\orbit(G)\op,\sets) \rar[shift left = 1, "b^*"] \lar[shift left = 1, "\inclusion", hook]& \func(BG,\sets) \lar[shift left = 1, "b_*"]\rar[shift left =1, "r_!"] & \sets, \lar[shift left = 1, "r^*"]
        \end{tikzcd}
    \end{equation}
    we may obtain a decomposition $\udl{X}\simeq \udl{Y}\sqcup \udl{Z}\in\spc_G$ such that $\pi_0Y_{hG} \cong S$ and $\pi_0Z_{hG} \cong T$.
\end{cons}

We now come to the proof of our generalisation of Atiyah--Bott and Conner--Floyd's theorem. For this, recall the notion of formal dimensions from \cref{terminology:Poincare_dimension}.

\begin{proof}[Proof of \cref{thm:generalised_atiyah_bott}]
    We prove this by induction on $k$, where the base case of $k=0$ is trivial. Now suppose we know that it is true for $k-1$.
    To prove the inductive step for the case of $k$, the strategy is to obtain a contradiction using the gluing class.
    For this, note first that $\underline{X}^{C_p}$ is a $\myuline{\spectra}_{G/C_p}$--Poincar\'{e}  space by \cref{thm:fixed_points_poincare_duality}. We claim that there is a decomposition $\udl{X}^{C_p} = \udl{*} \sqcup \udl{Y}$ of $G/C_p$-spaces, where $Y^{G/C_p}\simeq \varnothing$ necessarily since $\ast\simeq X^G = (X^{C_p})^{G/C_p}\simeq \ast\sqcup Y^{G/C_p}$.
    If the component of $X^{C_p}$ containing $\udl{\ast}$ is of formal dimension larger than 0, then the induction hypothesis and \cref{cons:orbit_component_decomposition} gives  such a decomposition as the $G/C_p$-space $\udl{X}^{C_p}$ also satisfies the conditions of the theorem.
    If the component of $X^{C_p}$ containing $\udl{\ast}$ is of formal dimension 0, then it must be $G/C_p$--contractible by \cref{fact:PD_of_low_dimensions} (1) and \cref{thm:Cp_PD_contractible_underlying_space}. Thus, again by \cref{cons:orbit_component_decomposition}, we obtain the desired decomposition. 
    
    We will derive the contradiction by basechanging along $\myuline{\spectra}\rightarrow \underline{\module}_{F(EG_+,\eilenbergMacLaneCoeff)}$. Observe first that we may assume that $d>0$ is even  since we may replace $\underline{X}$ with $\underline{X}\times\underline{X}$ if necessary: this will still be a $\myuline{\spectra}$--Poincar\'{e} space satisfying the hypotheses of the theorem with $(\underline{X}\times\underline{X})^{G}\simeq \ast$ and $X^e\times X^e$ having formal dimension $2d>0$.

    To set up notation, recall the map $\underline{X}^{>1}\xrightarrow{\epsilon} \underline{X}$ from \cref{defn:singular_part_inclusion} and write $W\coloneqq \Sigma^{-d}\eilenbergMacLaneCoeff\in\module_{\eilenbergMacLaneCoeff}$. We write $D^{\bbZ}_{\udl{X}}\coloneqq F(EG_+,\eilenbergMacLaneCoeff)\otimes D_{\udl{X}}\in \udl{\module}_{F(EG_+,\eilenbergMacLaneCoeff)}^{\udl{X}}$. By the hypothesis of $\mathbb{Z}$--orientability, we get $D_{X^e}^{\mathbb{Z}}\simeq \eilenbergMacLaneCoeff\otimes D_{X^e}\simeq \uniqueMapX^*W\in\module_{\eilenbergMacLaneCoeff}^{X^e}$.  By \cref{prop:rigidity_or_orientability}, the $\myuline{\spectra}_{G/C_p}$--Poincar\'{e}  space $\underline{X}^{C_p}$ has $\mathbb{Z}$--orientable underlying dualising sheaf. 
    By \cref{cor:vanishing_of_pushforward_gluing_class}, the composition giving the gluing class
    \footnotesize
    \begin{equation}\label{eqn:complicated_blue_composition}
        \begin{tikzcd}
            \eilenbergMacLaneCoeff \dar["c"']&  (\uniqueMapX_!^{>1}\epsilon^*D_{\underline{X}}^{\mathbb{Z}})^{tG}\simeq (\uniqueMapX_!^{>1}(\uniqueMapX^{>1})^*W)^{tG} \rar["\canonical"]\ar[d,"\simeq"]& \Sigma(\uniqueMapX_!^{>1}\epsilon^*D_{\underline{X}}^{\mathbb{Z}})_{hG}\simeq \Sigma(\uniqueMapX_!^{>1}(\uniqueMapX^{>1})^*W)_{hG}\dar[ "\canonical"]\\
            (\uniqueMapX_!D_{\underline{X}}^{\mathbb{Z}})^{hG}\rar["\canonical"] & (\uniqueMapX_!D_{\underline{X}}^{\mathbb{Z}})^{tG} &\Sigma W_{hG}
        \end{tikzcd}
    \end{equation}\normalsize
    is nullhomotopic. To achieve a contradiction, we show that this composition is also $\pi_0$--surjective onto a nontrivial group, assuming that $X^G\simeq \ast$. We do this in three steps.

    \begin{itemize}
        \item[(1)] To this end, first note that since the composite $\terminalTCat \hookrightarrow \underline{X}^{>1} \xrightarrow{\epsilon} \underline{X}\rightarrow \terminalTCat$  is equivalent to the identity, by functoriality of colimits, the map in $\func(BG,\module_{\eilenbergMacLaneCoeff})$
        \[W \longrightarrow \uniqueMapX^{>1}_!\epsilon^*D_{{X^e}}^{\mathbb{Z}}\simeq \uniqueMapX^{>1}_!(\uniqueMapX^{>1})^*W \longrightarrow W\]
        is also equivalent to the identity. Thus the rightmost vertical map $\canonical$ in \cref{eqn:complicated_blue_composition} is (split) surjective on homotopy groups coming from the summand $\Sigma W_{hG}$ inside $\Sigma(\uniqueMapX_!^{>1}\epsilon^*D_{\underline{X}}^{\mathbb{Z}})_{hG}\simeq \Sigma W_{hG} \oplus \Sigma(\uniqueMapY_!\uniqueMapY^*W)_{hG}$, where we have used that $\udl{X}^{>1}\simeq \inflated_G^{G/C_p}\udl{X}^{C_p}\simeq \terminalTCat\sqcup \inflated_G^{G/C_p}\udl{Y}$ from the first paragraph of the proof. 

        \item[(2)] Next,  the Tate--to--orbit canonical map  breaks up to become
    \[(\uniqueMapX_!^{>1}\epsilon^*D_{\underline{X}}^{\mathbb{Z}})^{tG} \simeq W^{tG}\oplus (\uniqueMapY_!\uniqueMapY^*W)^{tG} \xlongrightarrow{\canonical \oplus \canonical} \Sigma W_{hG}\oplus \Sigma(\uniqueMapY_!\uniqueMapY^*W)_{hG}\]
    By \cref{recollection:group_cohomologies} and since $d-1$ is odd by our assumption, $\canonical\colon W^{tG}\rightarrow \Sigma W_{hG}\simeq \Sigma^{1-d}\eilenbergMacLaneCoeff_{hG}$ is a $\pi_0$--isomorphism onto $\mathbb{Z}/p^k$. Moreover, by \cref{lem:tate_torsion_computation}, the image of $\canonical\colon \pi_0(\uniqueMapY_!\uniqueMapY^*W)^{tG}\rightarrow\pi_0\Sigma(\uniqueMapY_!\uniqueMapY^*W)_{hG}$ is $p^{k-1}$--torsion since $Y^{G/C_p}\simeq \varnothing$.

        \item[(3)] Finally, consider the commuting diagram 
    \begin{center}\small
        \begin{tikzcd}
            \eilenbergMacLaneCoeff\dar["c"']\ar[dr, "c"]\\
            \Big(\widetilde{EG}\otimes \uniqueMapX_!D_{\underline{X}}^{\mathbb{Z}}\Big)^{G}\dar  & \Big(\widetilde{EG}\otimes \uniqueMapX^{>1}_!\epsilon^*D_{\underline{X}}^{\mathbb{Z}}\Big)^{G}\lar["\epsilon_!"', "\simeq"] \dar\\
            \big(\uniqueMapX_!D_{\underline{X}}^{\mathbb{Z}}\big)^{tG} & \big(\uniqueMapX^{>1}_!\epsilon^*D_{\underline{X}}^{\mathbb{Z}}\big)^{tG}\simeq \Big(\widetilde{EG}\otimes F(EG_+,\uniqueMapX^{>1}_!\epsilon^*D_{\underline{X}}^{\mathbb{Z}})\Big)^{G}\lar["\epsilon_!"',"\simeq"]
        \end{tikzcd}
    \end{center}\normalsize
    where  the top right  triangle involving the fundamental classes commutes since by \cref{cor:inclusion_fixed_points_degree_one} and \cref{lem:poincare_duality_in_trivial_bottoms},  the map $\epsilon\colon \underline{X}^{>1}\rightarrow\underline{X}$ is $\underline{\module}_{\widetilde{EG}\otimes F(EG_+,\eilenbergMacLaneCoeff)}$--degree one. Note importantly that the map 
    \[\eilenbergMacLaneCoeff \xlongrightarrow{c}  \big(\widetilde{EG}\otimes \uniqueMapX^{>1}_!\epsilon^*D_{\underline{X}}^{\mathbb{Z}}\big)^{G}\simeq (\widetilde{EG}\otimes \ast_!\epsilon^*D_{\underline{X}}^{\mathbb{Z}})^{G}\oplus (\widetilde{EG}\otimes \uniqueMapY_!\epsilon^*D_{\underline{X}}^{\mathbb{Z}})^{G}\] is the fundamental class of the $\widetilde{EG}\otimes F(EG_+,\eilenbergMacLaneCoeff)$--Poincar\'{e} space $\terminalTCat\sqcup \infl^G_{G/C_p}\underline{Y}$, and so by \cref{lem:pd_map_infinite_coproducts}, it hits the algebra unit $1$ of $\pi_0(\widetilde{EG}\otimes \ast_!\epsilon^*D_{\underline{X}}^{\mathbb{Z}})^{G}$. Next, consider 
    \begin{center}\small
        \begin{tikzcd}
            \Big(\widetilde{EG}\otimes \ast_!\epsilon^*D_{\underline{X}}^{\mathbb{Z}}\Big)^{G} \rar\dar["\simeq"'] & \Big(\Sigma {EG}_+\otimes \ast_!\epsilon^*D_{\underline{X}}^{\mathbb{Z}}\Big)^{G}  \dar["\simeq"]\\
            W^{tG}\simeq \Big(\widetilde{EG}\otimes F(EG_+,\ast_!\epsilon^*D_{\underline{X}}^{\mathbb{Z}})\Big)^{G} \rar& \Big(\Sigma {EG}_+\otimes F(EG_+,\ast_!\epsilon^*D_{\underline{X}}^{\mathbb{Z}})\Big)^{G} \simeq \Sigma W_{hG}.
        \end{tikzcd}
    \end{center}
    where the vertical equivalences are by \cref{lem:technical_pi_0_isomorphism}. Since the bottom horizontal map is a $\pi_0$--isomorphism onto $\mathbb{Z}/p^k$ as in step (2),  the composition $\eilenbergMacLaneCoeff\xrightarrow{c} W^{tG} \rightarrow \Sigma W_{hG}$  is a $\pi_0$--surjection onto the abelian group $\mathbb{Z}/p^k$.
    \end{itemize}

    All in all, putting the three steps together, we see that the image of $1\in\pi_0\eilenbergMacLaneCoeff$ under the composition map $\pi_0\eilenbergMacLaneCoeff\rightarrow \pi_0\Sigma W_{hG}\cong \mathbb{Z}/p^k$ in \cref{eqn:complicated_blue_composition} is of the form $1+p\cdot a$ for some element $a\in\mathbb{Z}/p^k$, and hence is nonzero. This finishes the proof of the claim, and thus also of the theorem.
\end{proof}

\begin{rmk}
    Step (3) in the proof above might seem labyrinthine at first glance, but the basic idea leading to it is quite simple. Namely, we know always from \cref{lem:nonsingular_vanishing} that the map $\epsilon_! \colon (\uniqueMapX_!^{>1}\epsilon^*D_{\underline{X}}^{\mathbb{Z}})^{tG}\rightarrow (\uniqueMapX_!D_{\underline{X}}^{\mathbb{Z}})^{tG}$   is an equivalence. However, this equivalence has no control over the fundamental class $c\colon \eilenbergMacLaneCoeff\rightarrow (\uniqueMapX_!D_{\underline{X}}^{\mathbb{Z}})^{tG}$, essentially because $(-)^{tG}$ is only a lax symmetric monoidal functor. In contrast, the functor $\widetilde{EG}\otimes-$ is a symmetric monoidal one, and so it is more suited to lift the fundamental class by virtue of the theory of degree one  maps as encapsulated in \cref{cor:inclusion_fixed_points_degree_one}.
\end{rmk}

\appendix

\section{\texorpdfstring{$G$}{G}--stability for presentable \texorpdfstring{$G$}{G}-categories}\label{sec:G_stability_presentable_categories}

In this section we study presentable $G$-stable categories for compact Lie groups.
In \cite{Nardin2017Thesis}, Nardin defines for a finite group $G$-stability as a property of fibrewise stable $G$--categories, that roughly translates to requiring certain Wirthmüller isomorphisms to hold.
Instead of developing his theory for compact Lie groups in full generality, we take a different approach following the general phenomenon that certain properties of categories can be classified through idempotent algebras.
For example, a presentable category is stable if and only if it is a module over the category of spectra, see \cite{GepnerGrothNikolaus2015,CarmeliSchlankYanovski2021} for more examples of this type.
We define presentable $G$-stable categories as those presentable $G$--categories which are modules over the $G$--category $\myuline{\spectra}_G$ of $G$--spectra.

Recall that we say that a map $u \colon \unit \to A$ exhibits an object $A$ in a symmetric monoidal category $\category{C}$ as an idempotent object if the map $A \simeq \unit \otimes A \xrightarrow{u \otimes A} A \otimes A$ is an equivalence.
An idempotent object admits a unique structure of a commutative algebra in $\category{C}$ with $u$ as its unit map, see \cite[Proposition 4.8.2.9]{lurieHA}.
Now given an idempotent algebra $A$ in a symmetric monoidal category $\category{C}$, it is a property of an object $X \in \category{C}$ to be a module over $A$, in the sense that the forgetful functor $\module_{A}(\category{C}) \to \category{C}$ is fully faithful.
Its image is characterised by those $X \in \category{C}$ for which the unit map $X \rightarrow A \otimes X$ is an equivalence, or equivalently admits a right inverse,
see \cite[Proposition 4.8.2.10]{lurieHA}.
Taking this point of view, we observe that the $G$--category of $G$--spectra is an idempotent algebra in $\presentable^L_G$ and use this for the definition of presentable $G$-stable categories.

\begin{defn}[{\cite[Definition C.1, Corollary C.7]{GepnerMeier2023}, \cite[Definition 4.1]{Cnossen2023}}]\label{def:G_spectra}
    The category of $G$--spectra is defined as the formal inversion
    \begin{equation*}
    \spectra_G = \spc_{G,*}[\{ S^V \}^{-1}].
    \end{equation*}
    Here $\{ S^V \}^{-1}$ denotes the collection in $\spc_{G,*}$ consisting of representation spheres of all finite dimensional $G$-representations $V$.
\end{defn}
This means that it comes together with a symmetric monoidal colimit preserving functor $\Sigma^\infty_G \colon \spc_{G, *} \to \spectra_G$ sending all representation spheres $S^V$ to invertible objects and is inital among those.
More details on formal inversions of presentably symmetric monoidal categories can be found in \cite[Section 2]{Robalo2014} and \cite[Section 6.1]{Hoyois2017}.
By \cite[Corollary C.7]{GepnerMeier2023}, the canoncial map 
\begin{equation}\label{eq:spectra_as_colimit}
    \stab_{\{S^V\}}( \spc_{G, *}) \to  \spc_{G, *}[\{S^V\}^{-1}]
\end{equation}
is an equivalence.
Here, we denote for a presentably symmetric monoidal category $\category{C}$ together with a small collection of objects $S \subseteq \category{C}$ the stablisiation of $\category{C}$ at $S$ by $\stab_S(\category{C}) = \colim_{F \subseteq S \ \mathrm{finite}} \stab_{\bigotimes F}(\category{C})$, where for an element $x \in \category{C}$ we denote
\begin{equation*}
    \stab_x(\category{C}) = \colim \left( \category{C} \xrightarrow{- \otimes x} \category{C} \xrightarrow{- \otimes x} \dots \right).
\end{equation*}

\begin{defn}[{\cite[Definition 4.2]{Cnossen2023}}]
    We define the $G$--categories of pointed $G$--spaces and of (genuine) $G$--spectra as
\begin{equation*}
    \udl{\spc}_{G, *} = \spc_{G, *} \otimes_{\spc_G} \Omega \hspace{1cm} \text{and} \hspace{1cm} \myuline{\spectra}_G = \spectra_G \otimes_{\spc_G} \Omega,
\end{equation*}
where $\-- \otimes_{\spc_G} \Omega \colon \module_{\spc_G}(\presentable^L) \to \presentable_G^L$ is the symmetric monoidal colimit preserving embedding from \cref{cons:embedding_presentable_into_parametrised_presentable_categories}.
\end{defn}

\begin{lem}
    The $G$--categories $\udl{\spc}_{G,*}$ and $\myuline{\spectra}_G$ are idempotent algebras in $\presentable_G^L$.
\end{lem}
\begin{proof}
    First note that $\spc_{G,*} \in \module_{\spc_G}(\presentable^L)$ is an idempotent algebra as the image of the idempotent algebra $\spc_* \in \presentable^L$ under the symmetric monoidal functor $\-- \otimes \spc_G \colon \presentable^L \to \module_{\spc_G}$.
    It now follows from the definition of formal inversion that $\spectra_G \in \module_{\spc_G}(\presentable^L)$ is an idempotent algebra.
    This proves the claim as the symmetric monoidal functor $\-- \otimes_{\spc_G} \Omega \colon \module_{\spc_G}(\presentable^L) \to \presentable_G^L$ preserves idempotent algebras.
\end{proof}

Having this at hand, we can now give our definition of $G$-stability.

\begin{defn}
    We say that a presentable $G$--category $\udl{\category{C}}$ is \textit{$G$-stable} if it is a module over the idempotent algebra $\myuline{\spectra}_G \in \calg(\presentable^L_G)$.
    We denote by $\presentable^{L, G-\stable}_G \subseteq \presentable^{L}_G$ the full subcategory on $G$-stable presentable $G$--categories.
    It is closed under all limits and colimits.
\end{defn}

Our goal is to prove the following characterisation of presentable $G$-stable categories.
\begin{thm}[Characterisation of {$G$}-stability]\label{prop:characterisation_G_stability}
    For a presentable $G$--category $\udl{\category{C}}$ the following are equivalent:
    \begin{enumerate}
        \item $\udl{\category{C}}$ is $G$-stable.
        \item $\udl{\category{C}}$ is fibrewise pointed and for all closed subgroups $H \le G$ and all finite dimensional $H$-representations $V$ tensoring with $S^V \in \spc_{H,*}$ induces an equivalence $\-- \otimes S^V \colon \category{C}^H \xrightarrow{\simeq} \category{C}^H$.
        \item $\udl{\category{C}}$ is fibrewise pointed and for all finite dimensional $G$-representation $V$ tensoring with $S^V \in \spc_{G,*}$ induces an equivalence $\-- \otimes S^V \colon \udl{\category{C}} \xrightarrow{\simeq} \udl{\category{C}}$.
    \end{enumerate}
\end{thm}

To clarify the statement, recall that the $\udl{\spc}_G$-module structure on $\udl{\category{C}}$ restricts to a $\spc_H$-module structure on $\category{C}^H$ which refines to a $\spc_{H, *}$-module structure as $\category{C}^H$ is pointed.
The map $\-- \otimes S^V \colon \category{C}^H \xrightarrow{\simeq} \category{C}^H$ is now just the multiplication map induced by this module structure.

For the proof of \cref{prop:characterisation_G_stability}, we need the following preliminary result.

\begin{lem}\label{lem:restriction_spectra_spaces_equivalence}
    Suppose that $\udl{\category{D}}$ is a presentable $G$--category  such that $\category{D}^G$ is pointed and $\-- \otimes S^V \colon \category{D}^G \to \category{D}^G$ is an equivalence for any finite dimensional $G$-representation $V$.
    Then the restriction map $\func_G^L(\myuline{\spectra}_G, \category{D}) \to \func_G^L(\udl{\spc}_G, \category{D})$ is an equivalence.
\end{lem}
\begin{proof}
    Using that $\-- \otimes \udl{\spc}_G \colon \presentable^L \to \presentable^L_G$ is left adjoint to $\Gamma$, we obtain an equivalence
    \begin{equation*}
        \func_G^L(\udl{\spc}_{G, *}, \udl{\category{D}}) \simeq \func^L(\spc_*, \category{D}^G) \xrightarrow{\simeq} \func^L(\spc, \category{D}^G) \simeq \func_G^L(\udl{\spc}_{G}, \udl{\category{D}})
    \end{equation*}
    where the middle equivalence uses that $\category{D}^G$ is pointed.
    Similarly, the restriction map
    \begin{equation*}
        \func_G^L(\myuline{\spectra}_{G}, \udl{\category{D}}) \simeq \func^L_{\spc_{G, *}}(\spectra_G, \category{D}^G) \xrightarrow{\simeq} \func^L_{\spc_{G,*}}(\spc_{G, *}, \category{D}^G) \simeq \func_G^L(\udl{\spc}_{G,*}, \udl{\category{D}})
    \end{equation*}
    is an equivalence by employing the colimit description of $\spectra_G = \spc_{G,*}[\{S^V\}^{-1}]$ from \cref{eq:spectra_as_colimit}.
\end{proof}

\begin{proof}[Proof of \cref{prop:characterisation_G_stability}]
    \underline{$1 \implies 2$:} Observe that $\myuline{\spectra}_G$ is fibrewise pointed and satisfies the assumption on invertible actions of representations spheres as $\myuline{\spectra}_G(G/H) = \spectra_H$ is the formal inversion of $\spc_{H,*}$ at representation spheres of finite dimensional $H$-represenations.
    But this also holds for any $G$-stable category $\udl{\category{C}}$ as $\category{C}^H$ then is a module over $\spectra_H$.

    \underline{$2 \implies 3$:}
    Recall that, by the Peter-Weyl theorem, for any finite dimensional $H$-represenation $W$ there is a finite dimensional $G$-representation $V$ such that $W$ is a summand of $\res_H^G V$. 
    In particular, if $\-- \otimes S^V \colon \category{C}^H \xrightarrow{\simeq} \category{C}^H$ is an equivalence, this implies that $\-- \otimes S^{\res_H^G V} \colon \category{C}^H \xrightarrow{\simeq} \category{C}^H$ is an equivalence.
    But then also $\-- \otimes S^{W}$ is an equivalence

    \underline{$3 \implies 1$:}
    We want to construct a right inverse $\udl{\category{C}} \otimes \myuline{\spectra}_G \to \udl{\category{C}}$ to the unit map.
    By adjunction, this is equivalent to finding a factorisation of the unit map $\udl{\spc}_G \to \udl{\func}^L_G(\udl{\category{C}}, \udl{\category{C}})$ through the unit map $\udl{\spc}_G \to \myuline{\spectra}_G$.
    For this, we apply \cref{lem:restriction_spectra_spaces_equivalence} for $\udl{\category{D}} = \udl{\func}^L_G(\udl{\category{C}}, \udl{\category{C}})$.
    It thus remains to show that $\func^L_G(\udl{\category{C}}, \udl{\category{C}}) = \category{D}^G$ is pointed and tensoring with representation spheres is invertible.
    The assumption on $\udl{\category{C}}$ being fibrewise pointed implies that $\func^L_G(\udl{\category{C}}, \udl{\category{C}})$ is pointed.
    Furthermore, $S^V$ acts invertibly on $\func^L_G(\udl{\category{C}}, \udl{\category{C}})$ as it does so on $\udl{\category{C}}$.
\end{proof}

\section{Reflecting pushout squares}\label{section:reflecting_pushout_squares}

Let $A \rightarrow B \rightarrow B/A$ be a cofibre sequence in a stable category. Recall that there is a natural identification of the cofibre of $B \rightarrow B/A$ as follows, constructed as follows. Consider the diagram
\begin{equation}
\begin{tikzcd}
A \arrow[r] \arrow[rr, "\simeq 0", bend left] & B \arrow[r] \arrow[rr, "\simeq 0"', bend right] & B/A \arrow[r] & \cofib(B \rightarrow B/A)
\end{tikzcd}
\end{equation}
and note that the two nullhomotopies of bent arrows - coming from them being the structure of the cofibre sequences - define a map $\Sigma  A \rightarrow \cofib(B/A)$, which turns out to be an equivalence. This equivalence is natural in maps of cofibre sequences, and we will always use it to identify $\cofib(B \rightarrow B/A)$ with $\Sigma A$.

\begin{lem}
    \label{lem:black_magic}
    Consider a pushout square
    \begin{equation*}
    \begin{tikzcd}
        A \ar[r] \ar[d]
        & B \ar[d] \\
        C \ar[r]
        & D
    \end{tikzcd}
    \end{equation*}
    in a stable category.
    Then the two composites \[\phi_C \colon D \to D/C \simeq B/A \to \Sigma A \hspace{3mm}\text{ and }\hspace{3mm} \phi_B \colon D \to D/B \simeq C/A \to \Sigma A\] coming from the following diagram satisfy $\phi_B \simeq \pm \phi_C$.    
    \begin{equation*}
    \begin{tikzcd}
        A \ar[r] \ar[d]
        & B \ar[d] \ar[r]
        & B/A \ar[r] \ar[d, "\simeq"]
        & \Sigma A 
        \\
        C \ar[r] \ar[d]
        & D \ar[r] \ar[d]
        & D/C 
        \\
        C/A \ar[r, "\simeq"] \ar[d]
        & D/B 
        \\
        \Sigma A
    \end{tikzcd}
    \end{equation*}
\end{lem}
\begin{proof}
    To prove this, we consider the universal example of a span in a stable category.
    Denote by $\spancategory(\category{C}) = \func(\bullet \xleftarrow{} \bullet \xrightarrow{} \bullet, \category{C})$ the category of spans in the stable category $\category{C}$.
    Note that there is an equivalence
    \begin{align*}
        \spancategory(\category{C}) &= \func(\bullet \xleftarrow{} \bullet \xrightarrow{} \bullet, \category{C}) \\
        &\simeq \func(\bullet \xleftarrow{} \bullet \xrightarrow{} \bullet, \funcex(\spectra^\omega, \category{C})) \\
        &\simeq \funcex(\spectra^\omega \otimes (\bullet \xleftarrow{} \bullet \xrightarrow{} \bullet), \category{C}),
    \end{align*}
    where $\spectra^\omega \otimes (\bullet \xleftarrow{} \bullet \xrightarrow{} \bullet)$ denotes the tensoring of $\catex$ over $\cat$.
    The construction of the tensoring in \cite[Section 6.4]{Hermitian1} shows that $\spectra^\omega \otimes (\bullet \xleftarrow{} \bullet \xrightarrow{} \bullet)$ is given by the stable subcategory $\cospancategory(\spectra)^f$ of $\cospancategory(\spectra)$ generated by the three objects in the span \cref{diag:black_magic_pushout_universal_case} (which are given as the values of the left Kan extensions of the inclusions of the individual objects in the category $(\bullet \xrightarrow{} \bullet \xleftarrow{} \bullet)$ at the sphere).
    Using this description, the equivalence $\funcex(\cospancategory(\spectra)^f, \category{C}) \simeq \spancategory(\category{C})$ is given by evaluation at the universal span
    \newsavebox{\cospanoso}
    \begin{lrbox}{\cospanoso}
        \begin{tikzcd}[row sep = small, column sep = small]
            & \sphere & \\
            0 \ar[ur] & & 0 \ar[ul]
        \end{tikzcd}
    \end{lrbox}
    \newsavebox{\cospansso}
    \begin{lrbox}{\cospansso}
        \begin{tikzcd}[row sep = small, column sep = small]
            & \sphere & \\
            \sphere \ar[ur, equal] & & 0 \ar[ul]
        \end{tikzcd}
    \end{lrbox}
    \newsavebox{\cospanoss}
    \begin{lrbox}{\cospanoss}
        \begin{tikzcd}[row sep = small, column sep = small]
            & \sphere & \\
            0 \ar[ur] & & \sphere \ar[ul, equal]
        \end{tikzcd}
    \end{lrbox}
    \newsavebox{\cospansss}
    \begin{lrbox}{\cospansss}
        \begin{tikzcd}[row sep = small, column sep = small]
            & \sphere & \\
            \sphere \ar[ur, equal] & & \sphere \ar[ul, equal]
        \end{tikzcd}
    \end{lrbox}
    \newsavebox{\cospansoo}
    \begin{lrbox}{\cospansoo}
        \begin{tikzcd}[row sep = small, column sep = small]
            & 0 & \\
            \sphere \ar[ur] & & 0 \ar[ul]
        \end{tikzcd}
    \end{lrbox}
    \newsavebox{\cospanoos}
    \begin{lrbox}{\cospanoos}
        \begin{tikzcd}[row sep = small, column sep = small]
            & 0 & \\
            0 \ar[ur] & & \sphere \ar[ul]
        \end{tikzcd}
    \end{lrbox}
    \newsavebox{\cospanosigmaso}
    \begin{lrbox}{\cospanosigmaso}
        \begin{tikzcd}[row sep = small, column sep = small]
            & \Sigma \sphere & \\
            0 \ar[ur] & & 0 \ar[ul]
        \end{tikzcd}
    \end{lrbox}
    \begin{equation} \label{diag:black_magic_pushout_universal_case}
    \begin{tikzcd}
        \left(\usebox{\cospanoso}\right) \ar[r] \ar[d] 
        & \left(\usebox{\cospanoss}\right)
        \\
        \left(\usebox{\cospansso}\right).
    \end{tikzcd}
    \end{equation}
    It suffices to prove the claim in this specific case.
    Any span in $\category{C}$ is the image of this universal span under an exact functor and thus also satisfies the statment of the lemma.
    
    The possibilites for $\phi_B$ and $\phi_C$ are limited, since
    \begin{align*}
        &\pi_0 \map\left(\usebox{\cospansss}, \usebox{\cospanosigmaso} \right)  \\ \simeq &\pi_0 \map\left(\sphere, \lim\left( \usebox{\cospanosigmaso} \right) \right)  \simeq \pi_0\map(\sphere,\sphere) \simeq \bbZ.
    \end{align*}
    so $\phi_B$ and $\phi_C$ identify with integers $n_B$ and $n_C$. Note that if $n_B$ is divisible by $k \in \bbZ$, then $\phi_B$ is divisible by $k$ for any pushout in any stable category. But in the case of the following pushout in $\spectra$
    \begin{equation}
        \label{diag:suspension_pushout}
        \begin{tikzcd}
            \sphere \ar[r] \ar[d] & 0 \ar[d] \\
            0 \ar[r] & \Sigma \sphere
        \end{tikzcd}
    \end{equation}
    the map $\phi_B$ is clearly an equivalence, so $n_B =\pm_1$. The same holds for $\phi_C$, so by lack of alternatives we see $\phi_B = \pm \phi_C$.
\end{proof}

\begin{rmk}
    A more careful analysis of \cref{diag:suspension_pushout} in fact yields that $\phi_B = -\phi_C$. 
\end{rmk}

\emergencystretch 4em

\printbibliography
\end{document}